\documentclass[aop]{imsart}

\RequirePackage{amsthm,amsmath,amsfonts,amssymb}
\RequirePackage[numbers]{natbib}
\RequirePackage[colorlinks,citecolor=blue,urlcolor=blue]{hyperref}
\RequirePackage{graphicx}

\usepackage{stmaryrd}
\usepackage{mathtools}
\usepackage{mathrsfs}
\usepackage{enumitem}
\usepackage{tikz,pgfplots}
\pgfplotsset{compat=1.15}
\usetikzlibrary{arrows}
\usepackage{caption,subcaption}

\startlocaldefs
\numberwithin{equation}{section}
\theoremstyle{plain}
\newtheorem{lemma}{Lemma}[section]
\newtheorem{theorem}[lemma]{Theorem}
\newtheorem{proposition}[lemma]{Proposition}
\newtheorem{corollary}[lemma]{Corollary}
\newtheorem{alg}[lemma]{Algorithm}

\theoremstyle{remark}
\newtheorem{notation}[lemma]{Notation}
\newtheorem{remark}[lemma]{Remark}
\newtheorem{definition}[lemma]{Definition}

\newcommand{\noi}{\noindent}
\def\cq{$\hfill \square$}

\newcommand{\un}{{\bf 1}}
\newcommand{\dd}{{\rm d}}
\newcommand{\I}[1]{\boldsymbol{1}_{\{#1\}}}

\newcommand{\Ske}{\operatorname{Sha}}
\newcommand{\dis}{\operatorname{dis}}
\newcommand{\Dsp}{\operatorname{Dsp}}
\newcommand{\base}{\operatorname{base}}

\newcommand{\cN}{\mathcal N}

\newcommand{\cR}{\mathcal R}

\newcommand{\cT}{\mathcal T}

\newcommand{\sT}{\mathscr T}

\newcommand{\bN}{\mathbf N}

\newcommand{\ba}{\mathbf{a}}

\newcommand{\bbC}{\mathbb C}

\newcommand{\bbU}{\mathbb U}
\newcommand{\bbZ}{\mathbb Z}

\newcommand{\bbD}{\mathbb D}

\newcommand{\bbN}{\mathbb N}

\newcommand{\bbR}{\mathbb R}
\newcommand{\bbT}{\mathbb T}
\newcommand{\bbW}{\mathbb W}

\newcommand{\rI}{\mathrm I}

\newcommand{\rX}{\mathrm X}
\newcommand{\rH}{\mathrm H}

\def\cq{$\hfill \square$}

\def\be{\mathbf{e}}

\def\bR{\mathbf{R}}

\def\bT{\mathbf{T}}

\newcommand{\fTau}{\boldsymbol{\mathcal T}}

\newcommand{\bt}{\mathbf{t}}

\newcommand{\rg}{\mathrm{g}}
\newcommand{\rd}{\mathrm{d}}
\newcommand{\ftau}{\boldsymbol{\tau}}

\newcommand{\tHS}{\boldsymbol{\mathcal{S}}}
\newcommand{\HS}{\mathcal{S}}
\newcommand{\GW}{\mathsf{GW}}
\newcommand{\GWw}{\boldsymbol{\mathsf{GW}}}
\newcommand{\GWl}{\overline{\mathsf{GW}}}
\newcommand{\GWwl}{\overline{\boldsymbol{\mathsf{GW}}}}
\newcommand{\aHS}{\mathsf{HS}}
\newcommand{\cdelta}{\mathfrak{d}}
\newcommand{\bigInfo}[2]{\scalebox{1.5}{$#1\mathfrak{s}$}}
\newcommand{\Info}{\mathpalette\bigInfo\relax}
\newcommand{\ctau}{\mathfrak{t}}
\newcommand{\fctau}{\boldsymbol{\ctau}}
\newcommand{\Hl}{\bar{H}}

\def\E{ {\mathbb E }}
\def\P{ {\mathbb P }}
\def\convd{\stackrel{d}{\longrightarrow}}
\def\convP{\stackrel{\P}{\longrightarrow}}

\DeclareMathOperator*{\argmax}{argmax}
\DeclareMathOperator*{\Circledast}{\scalebox{1.5}{\raisebox{-0.2ex}{$\circledast$}}}

\definecolor{ududff}{rgb}{0.30196078431372547,0.30196078431372547,1}
\definecolor{ffxfqq}{rgb}{1,0.4980392156862745,0}
\definecolor{ffqqqq}{rgb}{1,0,0}
\definecolor{qqccqq}{rgb}{0,0.8,0}
\definecolor{xdxdff}{rgb}{0.49019607843137253,0.49019607843137253,1}

\endlocaldefs

\begin{document}

\begin{frontmatter}
\title{Fluctuations of the Horton--Strahler number of stable Galton--Watson trees}
\runtitle{Fluctuations of the Horton--Strahler number of stable GW trees}

\begin{aug}
\author[A]{\fnms{Robin}~\snm{Khanfir}\ead[label=e1]{robin.khanfir@mcgill.ca}}
\address[A]{Department of Mathematics and Statistics, McGill University\printead[presep={,\ }]{e1}}
\end{aug}

\begin{abstract}
The Horton--Strahler number --- also called the register function --- is a combinatorial tool that quantifies the branching complexity of a rooted tree. We study the law of the Horton--Strahler number of the canonical stable Galton--Watson trees conditioned to have size $n$ (including the Catalan trees), which are the finite-dimensional marginals of stable Lévy trees. While these random variables are known to grow as a multiple of $\ln n$ in probability, their fluctuations are not well understood because they are coupled with deterministic oscillations. To rule out the latter, we introduce a real-valued variant of the Horton--Strahler number. We show that a rescaled exponential of this quantity jointly converges in distribution to a measurable function of the scaling limit of the trees, i.e.~the stable Lévy tree. We call this limit the Strahler dilation and we discuss its similarities with the Horton--Strahler number.
\end{abstract}

\begin{keyword}[class=MSC]
\kwd{60C05}
\kwd{60J80}
\kwd{60F05}
\kwd{60F17}
\kwd{05C05}
\kwd{60D05}
\end{keyword}

\begin{keyword}
\kwd{Horton--Strahler number}
\kwd{Register function}
\kwd{Galton--Watson trees}
\kwd{Scaling limit}
\kwd{Stable trees}
\kwd{Random metric spaces}
\end{keyword}

\end{frontmatter}
\tableofcontents


\section{Introduction}

The Horton--Strahler number of a finite rooted tree $t$ is a nonnegative integer $\HS(t)\in\bbN$ that measures its branching complexity. One of its possible formal constructions is the recursive definition below.
\begin{definition}
\label{def_HS_intro_full}
The \textit{Horton--Strahler number} $\HS(t)$ of a finite rooted tree $t$ is as follows.
\begin{itemize}
\item[(a)] If $t$ reduces to a single node, then $\HS(t)=0$.
\item[(b)] Otherwise, $\HS(t)$ is the maximum of the Horton--Strahler numbers of the subtrees $t_1,\ldots,t_k$ that are attached to the root, plus one if that maximum is not uniquely achieved:
\begin{align}
\label{def_HS_intro}
\HS(t)&=\max_{1\leq i\leq k}\HS(t_i)+\un_{\big\{\#\argmax_{1\leq i\leq k}\HS(t_i)\geq 2\big\}}.
\end{align}
\end{itemize}
\end{definition}
Alternatively, $\HS(t) + 1$ is the number of successive Horton pruning operations (removing all leaves then merging each non-branching path into one edge) necessary to completely erase $t$, and $\HS(t)$ is also the height of the largest perfect binary tree embedded into $t$ (see Figure~\ref{ex_weighted_tree} and Section~\ref{preliminary_HS} for details). In this article, we study the fluctuations of the Horton--Strahler number of the so-called \textit{$\alpha$-stable Galton--Watson trees} (or \textit{$\GW_\alpha$-trees} for short), with $\alpha\!\in\!(1,2]$, conditioned to be large. For all $\alpha\!\in\!(1,2]$, the offspring distribution $\mu_\alpha=(\mu_\alpha(k))_{k\in\bbN}$ of a $\GW_\alpha$-tree is critical and is characterized by its generating function:
\begin{equation}
\label{stable_offspring}
\forall s\in[0,1],\quad\varphi_\alpha(s):=\sum_{k\in\bbN}s^k\mu_\alpha(k)=s+\tfrac{1}{\alpha}(1-s)^\alpha.
\end{equation}
See Section~\ref{words} for a formal definition. We highlight that when $\alpha=2$, a $\GW_2$-tree is just a critical binary Galton--Watson tree, namely $\mu_2(0)=\mu_2(2)=\frac{1}{2}$. Before discussing our results, let us present a brief history of the Horton--Strahler number and explain why it is relevant to focus on the specific family of stable Galton--Watson trees. Here, all the random variables that we consider are defined on the same probability space $(\mathsf{\Omega},\mathscr{F},\P)$ whose expectation is denoted by $\E$, and all topological spaces are endowed with their Borel sigma-field. For all $x\in[0,\infty)$, we denote by $\lfloor x\rfloor\in\bbN$ the integer part of $x$.

\paragraph*{Background} The Horton--Strahler number was introduced in the field of hydrogeology, first by Horton~\cite{horton45} in 1945 and exactly as Definition~\ref{def_HS_intro_full} by Strahler~\cite{strahler52} in 1952, to seek quantitative empirical laws about stream systems. Such a system can be seen as a tree whose root corresponds to the basin outlet and whose leaves represent the springs. Many important characteristics of river networks have been linked to the Horton--Strahler number: see Peckham~\cite{peckham}, Fac-Beneda~\cite{facbeneda}, Chavan \& Srinivas~\cite{chavan}, and Bamufleh et al.~\cite{bamufleh} among others. This number appears independently in other scientific disciplines such as anatomy, botany, molecular biology, physics, social network analysis, etc. In computer science, it is sometimes called the \textit{register function} because the minimum number of registers needed to evaluate an expression tree is equal to its Horton--Strahler number. We refer to Viennot~\cite{Viennot} for an overview of those various applications. 
\smallskip

The Horton--Strahler number also shows a wide range of occurrences in mathematics. It can be encountered in mathematical logic, algebra, combinatorics, topology, approximation theory, and more. See Esparza, Luttenberger \& Schlund~\cite{esparza} for an overview of those connections. In the probability area, Flajolet, Raoult \& Vuillemin~\cite{flajolet} and Kemp~\cite{kemp} study the Horton--Strahler number of a uniform rooted ordered binary tree $T_n$ with $n$ leaves (a uniform \textit{$n$-Catalan tree}, which is also a $\GW_2$-tree conditioned to have $n$ leaves) and prove that
\begin{equation}
\label{esperance_HS_Catalan}
\E[\HS(T_n)]\underset{n\to\infty}{=}\log_4 n+ D(\log_4 n)+o(1)\, ,
\end{equation}
where $D$ is a $1$-periodic continuous function. Moreover, Devroye \& Kruszewski~\cite{devroye95} showed that $\HS(T_n)$ is uniformly concentrated around its expected value, and more strongly that
\begin{equation}
\label{tightness_HS_Catalan}
\forall m\geq 0,\quad\sup_{n\geq 0}\, \P\big(\left|\HS(T_n)-\E[\HS(T_n)]\right|\geq m\big)\leq 4^{1-m}.
\end{equation}
Drmota \& Prodinger~\cite{drmota} extended (\ref{esperance_HS_Catalan}) and (\ref{tightness_HS_Catalan}) to $k$-ary trees. More recently, Brandenberger, Devroye \& Reddad~\cite{brandenberger} proved that the Horton--Strahler number of a critical Galton--Watson tree with finite variance offspring distribution conditioned to have $n$ vertices always grows as $\log_4 n$ in probability, thus extending all previous results on first-order behavior. In \cite{companion_1}, we generalize their "law of large numbers" to the cases where the offspring distribution is in the domain of attraction of a stable law (and thus may have infinite variance). In contrast to those works, the present paper initiates the study of the second-order asymptotics of the Horton--Strahler number, with the aim of obtaining a "central limit theorem".
\smallskip

In a parallel direction, Burd, Waymire \& Winn~\cite{burd} analyzed the invariance and attraction properties of critical Galton--Watson trees with finite variance offspring distribution under the \textit{Horton pruning operation} -- erasing leaves and their parental edges, then removing vertices with a single child -- which exactly decrements the Horton--Strahler number. They notably proved that $\GW_2$-trees are invariant under Horton pruning and that the law of their Horton--Strahler number is geometric with parameter $\frac{1}{2}$. Kovchegov \& Zaliapin~\cite{hortonlaws,kovchegov} broadened this study to the infinite variance case: they proved that a critical Galton--Watson tree whose offspring distribution $\mu$ is in the domain of attraction of an $\alpha$-stable law is invariant by Horton pruning if and only if $\mu=\mu_\alpha$ as in (\ref{stable_offspring}). They~\cite[Lemma 10]{kovchegov} additionally showed that
\begin{equation}
\label{law_HS_discrete}
\textit{if $\tau$ is a $\GW_\alpha$-tree with $\alpha\in(1,2]$,  then for all $n\in\bbN$,}\quad \P(\HS(\tau)\geq n)=(1-\tfrac{1}{\alpha})^n.
\end{equation}
Moreover, $\GW_\alpha$-trees are invariant attractors for a large class of other tree reductions: see Duquesne \& Winkel~\cite{duquesne_winkel} and Kovchegov, Xu \& Zaliapin~\cite{kovchegov23}. These trees had already appeared in numerous works such as Zolotarev~\cite{zolotarev}, Neveu~\cite{neveu86}, Le Jan~\cite{lejan91}, or Abraham \& Delmas~\cite{coalescent}. Indeed, they play the role of reference objects among all critical Galton--Watson trees because they describe the finite-dimensional marginals of their universal scaling limits: the stable Lévy trees. We refer to Duquesne \& Le Gall~\cite{levytree_DLG} and to Marchal~\cite{marchal} for details.
\smallskip 

Let us come back to Catalan trees $T_n$. By (\ref{tightness_HS_Catalan}), a natural centering suffices to gain tightness on their Horton--Strahler numbers. However, the estimate (\ref{esperance_HS_Catalan}) shows that $\HS(T_n)$ is subject to deterministic oscillations. Worse, these variations affect real-world applications. For example, characteristics derived from the Horton--Strahler number show great discrepancies due to the threshold selected for the extraction of stream networks: see e.g.~Moussa \& Bocquillon~\cite{moussa} or Chavan \& Srinivas~\cite{chavan}. This could then impact the chosen hydrologic response, as discussed by Bamufleh et al.~\cite{bamufleh}. We further argue that one cannot hope to find a nondegenerate scaling limit for $\HS(T_n)$. We discuss why on a simpler but related model.

\paragraph*{Heuristic discussion}
Construct a random binary tree $T_n'$ by grafting $n+1$ independent $\GW_2$-trees on the same spine of length $n$ as illustrated by Figure~\ref{ex_GW_spinal}. In other words, $T_0'$ is a $\GW_2$-tree, and $T_{n+1}'$ is composed of $T_n'$ and an independent $\GW_2$-tree attached to the root. The tree $T_n'$ is a truncated version of Kesten's limit tree \cite{kesten86} that may be informally seen as a critical binary Galton--Watson tree conditioned to survive: see Lyons, Pemantle \& Peres~\cite{LlogLcriteria}, Aldous \& Pitman~\cite{aldouspitman98}, and Duquesne~\cite{duquesne09} for more information. The trees $T_n'$ and Kesten's limit tree naturally appear during the study of local convergences of $\GW_2$-trees conditioned to be large, such as Catalan trees. We refer to Abraham \& Delmas~\cite{AD2014} for several results of this type. We let $(G_i)_{i\geq 0}$ be a sequence of independent geometric random variables with parameter $\frac{1}{2}$. From (\ref{law_HS_discrete}), we observe that $\HS(T_n')$ has the same law as $S_n$, where the sequence $(S_i)$ is defined by $S_0=G_0$ and $S_{i+1}=\max(G_{i+1},S_i)+\I{G_{i+1}=S_i}$ for all $i\geq 0$. It is then easy to convince oneself that there is an event $A_n$ such that
\[S_n=M_n+\un_{A_n}\quad\text{ where }\quad M_n=\max_{0\leq i\leq n} G_i.\]

\begin{figure}
\begin{center}
\begin{tikzpicture}[line cap=round,line join=round,>=triangle 45,x=2cm,y=0.9cm,scale=0.34]
\clip(-11.4,-0.7) rectangle (7.4,7.2);
\fill[line width=2pt,color=xdxdff,fill=xdxdff,fill opacity=0.1] (-10,5) -- (-9,7) -- (-11,7) -- cycle;
\fill[line width=2pt,color=xdxdff,fill=xdxdff,fill opacity=0.1] (-6,5) -- (-8.457527231046443,6) -- (-5,6) -- cycle;
\fill[line width=2pt,color=xdxdff,fill=xdxdff,fill opacity=0.1] (-4,4) -- (-4.58,7) -- (-2,7) -- cycle;
\fill[line width=2pt,color=xdxdff,fill=xdxdff,fill opacity=0.1] (2,1) -- (7,3) -- (3,3) -- cycle;
\fill[line width=2pt,color=xdxdff,fill=xdxdff,fill opacity=0.1] (0,2) -- (3,5) -- (0,5) -- cycle;
\draw [line width=1.2pt] (0,0)-- (2,1);
\draw [line width=1.2pt] (-2,1)-- (0,2);
\draw [line width=1.2pt] (-6,5)-- (-8,4);
\draw [line width=1.2pt] (-4,4)-- (-6,3);
\draw [line width=1pt,loosely dotted] (-5.2,2.6)-- (-2.8,1.4);
\draw [line width=1.2pt] (-10,5)-- (-8,4);
\draw [line width=1.2pt] (-8,4)-- (-6,3);
\draw [line width=1.2pt] (-6,3)-- (-5.2,2.6);
\draw [line width=1.2pt] (-2,1)-- (-2.8,1.4);
\draw [line width=1.2pt] (-2,1)-- (0,0);
\draw [line width=1pt,loosely dotted] (-3.2,3.6)-- (-0.8,2.4);
\draw [line width=1.2pt,color=xdxdff] (-10,5)-- (-9,7);
\draw [line width=1.2pt,color=xdxdff] (-9,7)-- (-11,7);
\draw [line width=1.2pt,color=xdxdff] (-11,7)-- (-10,5);
\draw [line width=1.2pt,color=xdxdff] (-6,5)-- (-8.457527231046443,6);
\draw [line width=1.2pt,color=xdxdff] (-8.457527231046443,6)-- (-5,6);
\draw [line width=1.2pt,color=xdxdff] (-5,6)-- (-6,5);
\draw [line width=1.2pt,color=xdxdff] (-4,4)-- (-4.58,7);
\draw [line width=1.2pt,color=xdxdff] (-4.58,7)-- (-2,7);
\draw [line width=1.2pt,color=xdxdff] (-2,7)-- (-4,4);
\draw [line width=1.2pt,color=xdxdff] (2,1)-- (7,3);
\draw [line width=1.2pt,color=xdxdff] (7,3)-- (3,3);
\draw [line width=1.2pt,color=xdxdff] (3,3)-- (2,1);
\draw [line width=1.2pt,color=xdxdff] (0,2)-- (3,5);
\draw [line width=1.2pt,color=xdxdff] (3,5)-- (0,5);
\draw [line width=1.2pt,color=xdxdff] (0,5)-- (0,2);
\begin{scriptsize}
\draw [fill=xdxdff] (0,0) circle (3.5pt);
\draw[color=xdxdff] (-0.27100218942895715,-0.45) node {$u_n$};
\draw [fill=xdxdff] (-2,1) circle (3.5pt);
\draw[color=xdxdff] (-2.2225773674840297,0.4) node {$u_{n-1}$};
\draw [fill=xdxdff] (-6,3) circle (3.5pt);
\draw[color=xdxdff] (-6.348928403782291,2.45) node {$u_2$};
\draw [fill=xdxdff] (-8,4) circle (3.5pt);
\draw[color=xdxdff] (-8.294780487473567,3.5) node {$u_1$};
\draw[color=xdxdff] (-10.046047362795713,6.4) node {$\GW_2$};
\draw[color=xdxdff] (-6.234466516506334,5.6) node {$\GW_2$};
\draw[color=xdxdff] (-3.6018431091593155,6.1) node {$\GW_2$};
\draw[color=xdxdff] (3.8496257525055064,2.4) node {$\GW_2$};
\draw[color=xdxdff] (0.8049395509650418,4.1) node {$\GW_2$};
\end{scriptsize}
\end{tikzpicture}
\caption{The random binary tree $T_n'$.}
\label{ex_GW_spinal}
\end{center}
\end{figure}

Let us set aside the event $A_n$ for now because it is quite difficult to precisely describe: for example, it is realized when $(G_i)_{0\leq i\leq n}\!=\!(1,1,2,3,\ldots ,n)$ but not when $(G_i)_{0\leq i\leq n}\!=\!(1,2,1,3,\ldots ,n)$. Nonetheless, the sole observation of $M_n$ already unveils an issue. Indeed, there are no sequences $(a_n)$ and $(b_n)$ such that $(M_n-b_n)/a_n$ converges in distribution to a nondegenerate random variable. In statistics, this is a classic application of the extreme value theorem, see e.g.~Arnold, Balakrishnan \& Nagaraja~\cite[page 217]{order_stats}. We also refer to Eisenberg~\cite{eisenberg} for a detailed study of the distribution of $M_n$. It is important to note that this situation is due to the discrete setting. Indeed, if $(E_i)_{i\geq 0}$ is a sequence of independent exponential random variables with mean $(\ln 2)^{-1}$, so that the integer part $\lfloor E_i\rfloor $ of $E_i$ has the same distribution as $G_i$, then an elementary calculation shows that
\[\forall x\in\bbR,\quad \P\Big(\max_{0\leq i\leq n}E_i-\log_2 n\leq x\Big)\longrightarrow \exp(-2^{-x}),\]
where the limit is the cumulative distribution function of the Gumbel law with location $0$ and with scale $(\ln 2)^{-1}$. Other curiosities happen for geometric order statistics which further complexify the behavior of $S_n$. For example, Bruss \& O'Cinneide~\cite{bruss} proved that the probability that the maximum of $(G_i)_{0\leq i\leq n}$ is uniquely achieved does not converge.

\paragraph*{Strategy and main results} The above observations suggest to slightly modify the Horton--Strahler number into a continuous quantity, more likely to behave regularly. Since one cannot expect two real numbers to be exactly equal, the recursive formula (\ref{def_HS_intro}) is not suitable in this regard. Nevertheless, we see, keeping the notation of Definition~\ref{def_HS_intro_full}, that
\begin{equation}
\label{def_tHS_intro}
\HS(t)=\max_{1\leq i,j\leq k}\max\big(\HS(t_i),\HS(t_j),\I{i\neq j}+\min(\HS(t_i),\HS(t_j))\big),
\end{equation}
by considering the two maximal values among $\HS(t_1),\ldots,\HS(t_k)$. The formula (\ref{def_tHS_intro}) stays workable when the $\HS(t_i)$ are real numbers, so it only remains to redefine the Horton--Strahler number of a tree reduced to a single node. We do so by endowing the leaves with weights.
\begin{definition}[Weighted Horton--Strahler number]
\label{def_tHS_intro_full}
A \textit{weighted tree} $\bt$ is a finite rooted tree whose each leaf $v$ is equipped with a weight $w_v\in[0,1)$. Its \textit{weighted Horton--Strahler number} $\tHS(\bt)\in[0,\infty)$ is defined recursively as follows.
\begin{itemize}
\item[(a)] If $\bt$ consists of a single node with weight $w_{\mathsf{root}}$, then $\tHS(\bt)=w_{\mathsf{root}}$.
\item[(b)] Otherwise, $\tHS(\bt)$ is derived from the weighted Horton--Strahler numbers of the weighted subtrees $\bt_1,\ldots,\bt_k$ that are attached to the root according to the expression (\ref{def_tHS_intro}). Namely,
\[\tHS(\bt)=\max_{1\leq i,j\leq k}\max\big(\tHS(\bt_i),\tHS(\bt_j),\I{i\neq j}+\min(\tHS(\bt_i),\tHS(\bt_j))\big).\]
\end{itemize}
\end{definition}
See Figure~\ref{ex_weighted_tree} for an example and Section~\ref{words} for a more precise definition of weighted trees. The requirement for the weights to be in $[0,1)$ entails the following relation between the classic and the weighted Horton--Strahler numbers (see Proposition~\ref{regularized_HT} for details):
\begin{equation}
\label{weighted_vs_classic}
\text{for all weighted trees }\bt,\quad \lfloor \tHS(\bt)\rfloor=\HS(t),\quad \text{ where }t\text{ is the underlying tree of }\bt.
\end{equation}
Therefore, one does not lose any information about $\HS$ by working solely with $\tHS$. However, we now need to specify the weights that we put on the leaves of stable Galton--Watson trees.

\begin{figure}
\begin{center}
\begin{tikzpicture}[line cap=round,line join=round,>=triangle 45,x=1.2cm,y=1cm,scale=0.6]
\clip(-11.1,-0.434763757560936) rectangle (8.5,6.5);
\draw [line width=1pt] (0,1)-- (0,0);
\draw [line width=1pt] (5,1)-- (0,0);
\draw [line width=1pt] (-5,1)-- (0,0);
\draw [line width=1pt] (2,2)-- (5,1);
\draw [line width=1pt] (5,2)-- (5,1);
\draw [line width=1pt] (8,2)-- (5,1);
\draw [line width=1pt] (-2,2)-- (-5,1);
\draw [line width=1pt] (-8,2)-- (-5,1);
\draw [line width=1pt] (-10,3)-- (-8,2);
\draw [line width=1pt] (-6,3)-- (-8,2);
\draw [line width=1pt] (-8,3)-- (-8,2);
\draw [line width=1pt] (-2,3)-- (-2,2);
\draw [line width=1pt] (2,3)-- (2,2);
\draw [line width=1pt] (3.5,3)-- (5,2);
\draw [line width=1pt] (6.5,3)-- (5,2);
\draw [line width=1pt] (-3,4)-- (-2,3);
\draw [line width=1pt] (-1,4)-- (-2,3);
\draw [line width=1pt] (-11,4)-- (-10,3);
\draw [line width=1pt] (-9,4)-- (-10,3);
\draw [line width=1pt] (-7,4)-- (-6,3);
\draw [line width=1pt] (-5,4)-- (-6,3);
\draw [line width=1pt] (-7.5,5)-- (-7,4);
\draw [line width=1pt] (-6.5,5)-- (-7,4);
\draw [line width=1pt] (-3.5,5)-- (-3,4);
\draw [line width=1pt] (-2.5,5)-- (-3,4);
\draw [line width=1pt] (-1.5,5)-- (-1,4);
\draw [line width=1pt] (-1,5)-- (-1,4);
\draw [line width=1pt] (-0.5,5)-- (-1,4);
\draw [line width=1pt] (-2.5,6)-- (-2.5,5);
\draw [line width=1pt] (-7.802891314047956,6)-- (-7.5,5);
\draw [line width=1pt] (-7.195108456051984,6)-- (-7.5,5);
\draw [line width=1pt,dash pattern=on 3pt off 4pt,color=red] (-5,1)-- (-8,2);
\draw [line width=1pt,dash pattern=on 3pt off 4pt,color=red] (-8,2)-- (-10,3);
\draw [line width=1pt,dash pattern=on 3pt off 4pt,color=red] (-8,2)-- (-6,3);
\draw [line width=1pt,dash pattern=on 3pt off 4pt,color=red] (-10,3)-- (-11,4);
\draw [line width=1pt,dash pattern=on 3pt off 4pt,color=red] (-10,3)-- (-9,4);
\draw [line width=1pt,dash pattern=on 3pt off 4pt,color=red] (-6,3)-- (-7,4);
\draw [line width=1pt,dash pattern=on 3pt off 4pt,color=red] (-6,3)-- (-5,4);
\draw [line width=1pt,dash pattern=on 3pt off 4pt,color=red] (-5,1)-- (-2,3);
\draw [line width=1pt,dash pattern=on 3pt off 4pt,color=red] (-2,3)-- (-3,4);
\draw [line width=1pt,dash pattern=on 3pt off 4pt,color=red] (-2,3)-- (-1,4);
\draw [line width=1pt,dash pattern=on 3pt off 4pt,color=red] (-3,4)-- (-3.5,5);
\draw [line width=1pt,dash pattern=on 3pt off 4pt,color=red] (-3,4)-- (-2.5,5);
\draw [line width=1pt,dash pattern=on 3pt off 4pt,color=red] (-1,4)-- (-1.5,5);
\draw [line width=1pt,dash pattern=on 3pt off 4pt,color=red] (-1,4)-- (-0.5,5);
\begin{scriptsize}
\draw [fill=xdxdff] (0,0) circle (2.5pt);
\draw[color=xdxdff] (0.3219725263733974,-0.3) node {$3.2$};
\draw [fill=qqccqq] (0,1) circle (2.5pt);
\draw[color=qqccqq] (0.31105965511597927,1.369844155595521) node {$0.4$};
\draw [fill=xdxdff] (5,1) circle (2.5pt);
\draw[color=xdxdff] (5.320067562270898,1.369844155595521) node {$1.3$};
\draw [fill=xdxdff] (5,2) circle (2.5pt);
\draw[color=xdxdff] (5.0908972658651175,2.420025687631523) node {$1.2$};
\draw [fill=qqccqq] (8,2) circle (2.5pt);
\draw[color=qqccqq] (8.244717059258956,2.2865253412186433) node {$0.1$};
\draw [fill=xdxdff] (2,2) circle (2.5pt);
\draw[color=xdxdff] (2.419027966480914,2.2428738561889707) node {$0.3$};
\draw [fill=xdxdff] (-5,1) circle (2.5pt);
\draw[color=xdxdff] (-5,1.4) node {$3.2$};
\draw [fill=xdxdff] (-2,2) circle (2.5pt);
\draw[color=xdxdff] (-1.5641700424767012,2.2210481136741347) node {$2.2$};
\draw [fill=xdxdff] (-8,2) circle (2.5pt);
\draw[color=xdxdff] (-8.4,1.706566925050243) node {$2.3$};
\draw [fill=xdxdff] (-2,3) circle (2.5pt);
\draw[color=xdxdff] (-1.5,3.0) node {$2.2$};
\draw [fill=xdxdff] (-1,4) circle (2.5pt);
\draw[color=xdxdff] (-0.5165344017645613,4.163539197494561) node {$1.2$};
\draw [fill=xdxdff] (-3,4) circle (2.5pt);
\draw[color=xdxdff] (-2.5617640993572415,4.17445206875198) node {$1.5$};
\draw [fill=xdxdff] (-10,3) circle (2.5pt);
\draw[color=xdxdff] (-9.44600170410484,3.115903556782421) node {$1.3$};
\draw [fill=xdxdff] (-6,3) circle (2.5pt);
\draw[color=xdxdff] (-5.481890823889808,3.0613392004953304) node {$1.4$};
\draw [fill=qqccqq] (-8,3) circle (2.5pt);
\draw[color=qqccqq] (-7.66624923405667,3.2686837543862746) node {$0.7$};
\draw [fill=xdxdff] (-7,4) circle (2.5pt);
\draw[color=xdxdff] (-6.529526464601948,4.071713454979724) node {$1.4$};
\draw [fill=qqccqq] (-5,4) circle (2.5pt);
\draw[color=qqccqq] (-4.64338389575185,4.272667910068742) node {$0.3$};
\draw [fill=qqccqq] (3.5,3) circle (2.5pt);
\draw[color=qqccqq] (3.8250042000046154,3.2905094969011106) node {$0.2$};
\draw [fill=qqccqq] (6.5,3) circle (2.5pt);
\draw[color=qqccqq] (6.815130924537181,3.301422368158529) node {$0.6$};
\draw [fill=qqccqq] (2,3) circle (2.5pt);
\draw[color=qqccqq] (2.2753764814512416,3.2795966256436926) node {$0.3$};
\draw [fill=qqccqq] (-11,4) circle (2.5pt);
\draw[color=qqccqq] (-10.678201701104072,4.261755038811324) node {$0.3$};
\draw [fill=qqccqq] (-9,4) circle (2.5pt);
\draw[color=qqccqq] (-8.659320518481719,4.283580781326161) node {$0.9$};
\draw [fill=qqccqq] (-3.5,5) circle (2.5pt);
\draw[color=qqccqq] (-3.159233404742984,5.22208770946412) node {$0.8$};
\draw [fill=xdxdff] (-2.5,5) circle (2.5pt);
\draw[color=xdxdff] (-2.1552492490605166,5.254826323236373) node {$0.5$};
\draw [fill=qqccqq] (-1.5,5) circle (2.5pt);
\draw[color=qqccqq] (-1.45909819034,5.320303550780882) node {$0$};
\draw [fill=qqccqq] (-1,5) circle (2.5pt);
\draw[color=qqccqq] (-0.9,5.35) node {$0.2$};
\draw [fill=qqccqq] (-0.5,5) circle (2.5pt);
\draw[color=qqccqq] (0,5) node {$0.2$};
\draw [fill=qqccqq] (-6.5,5) circle (2.5pt);
\draw[color=qqccqq] (-6.149360129275551,5.254826323236373) node {$0.4$};
\draw [fill=xdxdff] (-7.5,5) circle (2.5pt);
\draw[color=xdxdff] (-7.196995769987691,5.1) node {$1$};
\draw [fill=qqccqq] (-2.5,6) circle (2.5pt);
\draw[color=qqccqq] (-2.1552492490605166,6.33) node {$0.5$};
\draw [fill=qqccqq] (-7.802891314047956,6) circle (2.5pt);
\draw[color=qqccqq] (-8,6.33) node {$0.1$};
\draw [fill=qqccqq] (-7.195108456051984,6) circle (2.5pt);
\draw[color=qqccqq] (-7.000564087354165,6.33) node {$0$};
\end{scriptsize}
\end{tikzpicture}
\caption{A weighted tree $\bt$ with $\tHS(\bt)=3.2$. The leaves and their weights are in green. Next to each node is written the weighted Horton--Strahler number of the subtree stemming from it. An embedded copy of the perfect binary tree of height $3$ is highlighted in red and dashed.}
\label{ex_weighted_tree}
\end{center}
\end{figure}
\smallskip
In all this article, to lighten notation, we fix and denote
\begin{equation}
\label{alpha_gamma_delta}
\alpha\in(1,2],\ \quad\beta=\tfrac{1}{\alpha-1},\ \quad\gamma=\ln\tfrac{\alpha}{\alpha-1}\quad\text{ and }\quad\delta=\big(\tfrac{\alpha}{\alpha-1}\big)^{\alpha-1}=e^{\gamma(\alpha-1)}\in(1,2].
\end{equation}
For $a>0$, we say a random variable $W\in[0,1)$ has distribution $\mathsf{FExp}(a)$ when
\begin{equation}
\label{FEXP(a)}
\forall r\in[0,1],\quad \P(W\leq r)=\frac{1-e^{-ar}}{1-e^{-a}}.
\end{equation}
If $E$ is an exponential random variable with mean $\frac{1}{a}$ then the fractional part $E-\lfloor E\rfloor$ of $E$ has law $\mathsf{FExp}(a)$. We say a random weighted tree $\ftau$ is an \emph{$\alpha$-stable Galton--Watson weighted tree} (or a \emph{$\GWw_\alpha$-weighted tree} for short) when it is a $\GW_\alpha$-tree $\tau$ whose leaves are endowed with independent weights with law $\mathsf{FExp}(\gamma)$ (more formally, see Definition~\ref{GWwdef}). This choice of weights is justified by what we prove in Proposition~\ref{law_HT_stable}:
\begin{equation}
\label{law_HS_weighted}
\textit{if $\ftau$ is a $\GWw_\alpha$-weighted tree,  then for all $x\in[0,\infty)$,}\quad \P(\tHS(\ftau)\geq x)=e^{-\gamma x}.
\end{equation}
In fact, Proposition~\ref{law_HT_stable} entails that choosing $\mathsf{FExp}(\gamma)$ for the law of the weights is the only way for the law of $\tHS(\ftau)-x$ under $\P(\,\cdot\, |\, \tHS(\ftau)\geq x)$ to weakly converge as $x\to\infty$, and thus hoping to erase the unwanted deterministic oscillations as planned. Note that (\ref{law_HS_weighted}) readily implies (\ref{law_HS_discrete}). The law of $\tHS(\ftau)$ is absolutely continuous but, for all non-integers $x\geq 0$, we can define the law of $\tau$ under $\P(\dd\tau\, |\, \tHS(\ftau)=x)$ by setting
\[\P(\tau=t\ |\ \tHS(\ftau)=x)=\lim_{\varepsilon\to 0^+}\P\big(\tau=t\ \big|\ |\tHS(\ftau)-x|<\varepsilon\big)\]
for all trees $t$ (see Definition~\ref{cond_H=x_def} and Proposition~\ref{cond_H=x_prime} for details and a proof).
\smallskip

Instead of studying the weighted Horton--Strahler number of $\GWw_\alpha$-trees conditioned to have a large number of vertices, we begin by looking for a scaling limit for $\tau$ under $\P(\dd\tau\, |\, \tHS(\ftau)=x)$ as $x$ tends to $\infty$. Indeed, this problem is easier to tackle thanks to the so-called \textit{$r$-weighted Horton pruning operation} that we introduce in Section~\ref{pruning_section}: erasing the parental edges of leaves with weights smaller than some threshold $r\!\in\![0,1]$, then removing vertices with a single child. When $r\!=\!1$, we retrieve the Horton pruning mentioned above. This operation subtracts $r$ from the weighted Horton--Strahler number, and $\GWw_\alpha$-weighted trees are almost invariant under weighted Horton pruning (see Theorem~\ref{invariance} for a precise statement). We then deduce that $x\mapsto\P(\dd\tau\, |\, \tHS(\ftau)=x)$ is Cauchy in some sense, yielding our first main result. The rooted Gromov--Hausdorff--Prokhorov distance, which will be recalled in Section~\ref{topo_function}, gives a sense to convergences of rooted measured compact metric spaces.

\begin{theorem}
\label{scaling_limit_HS_intro}
Let $\ftau$ be a $\GWw_\alpha$-weighted tree. The underlying tree $\tau$ is endowed with its graph distance denoted by $\mathtt{d}_{\mathrm{gr}}$, its root denoted by $\varnothing$, and its counting measure $\sum_{u\in \tau}\delta_u$. There exists a nondegenerate random compact metric space $(\sT^\alpha,d^\alpha)$ endowed with a distinguished point $\rho^\alpha$ and a finite Borel measure $\mu^\alpha$ such that the convergence in distribution
\[\Big(\tau,e^{-\gamma(\alpha-1)x}\mathtt{d}_{\mathrm{gr}},\varnothing,e^{-\gamma\alpha x}\sum_{u\in \tau}\delta_u\Big)\; \text{ under }\P(\dd\tau\, |\, \tHS(\ftau)=x)\, \xrightarrow[x\rightarrow\infty,x\notin\bbN]{d} (\sT^\alpha,d^\alpha,\rho^\alpha,\mu^\alpha)\]
holds for the rooted Gromov--Hausdorff--Prokhorov distance.
\end{theorem}

Before describing the law of this limit space or discussing the fluctuations of the weighted Horton--Strahler numbers of $\GWw_\alpha$-trees conditioned on having total progeny $n$, we need to recall their own scaling limit. Let $(X_n)_{n\in\bbN}$ be a random walk started at $X_0=0$ and with jump law given by $\P(X_1=k)=\mu_\alpha(k+1)$ for all integers $k\geq -1$. Using (\ref{stable_offspring}) and the continuity theorem for (bilateral) Laplace transform (see e.g.~\cite[Appendix A]{yamato}), one can check that
\begin{equation}
\label{domain_attraction}
\tfrac{1}{a_n}X_n\convd \rX_1,\quad\text{ where }\quad a_n=\alpha^{-1/\alpha}n^{1/\alpha}\quad\text{ and }\quad\E\left[\exp\left(-\lambda \rX_1\right)\right]=\exp(\lambda^\alpha)
\end{equation}
for all $n\in\bbN$ and $\lambda\in(0,\infty)$. Hence, \emph{$\mu_\alpha$ is in the domain of attraction of a stable law of index $\alpha$}. Note that if $\alpha\in(1,2)$ then $\lfloor\alpha\rfloor=1$, and if $\alpha=2$ then $\lfloor\alpha\rfloor=2$. Then, keeping the same notation as in Theorem~\ref{scaling_limit_HS_intro}, a theorem of Duquesne~\cite{duquesne_contour_stable} yields that the convergence in law
\begin{equation}
\label{scaling_limit_Catalan}
\big(\tau,\tfrac{a_n}{n}\mathtt{d}_{\mathrm{gr}},\varnothing,\tfrac{1}{n}\sum_{u\in\tau}\delta_u\big)\; \text{ under }\P(\, \cdot\, |\, \#\tau=n+1)\, \xrightarrow[n\rightarrow\infty,n\in\lfloor\alpha\rfloor\bbN]{d} (\sT_{\mathrm{nr}}^\alpha,d_{\mathrm{nr}}^\alpha,\rho_{\mathrm{nr}}^\alpha,\mu_{\mathrm{nr}}^\alpha)
\end{equation}
holds for the rooted Gromov--Hausdorff--Prokhorov distance. Here, $(\sT_{\mathrm{nr}}^\alpha,d_{\mathrm{nr}}^\alpha,\rho_{\mathrm{nr}}^\alpha,\mu_{\mathrm{nr}}^\alpha)$ stands for the \emph{(normalized) $\alpha$-stable tree} describing the genealogical structure of a continuous-state branching process with branching mechanism $\lambda\mapsto\lambda^\alpha$ (following the convention of Duquesne \& Le Gall~\cite{levytree_DLG}; see Section~\ref{more_stable} for a precise definition). When $\alpha=2$, $\sT_{\mathrm{nr}}^\alpha$ corresponds to the celebrated Brownian tree of Aldous~\cite{aldousI,aldous1993}. More precisely, if $(\cT_\be,d_\be,\rho_\be,\mu_\be)$ stands for the real tree coded by the standard Brownian excursion $\be$ in the sense of Le Gall~\cite[Definition 2.2]{legall_trees} (see Section~\ref{real_tree_section}), then $(\cT_\be,\sqrt{2}\,d_\be,\rho_\be,\mu_\be)$ is a $2$-stable tree.
\smallskip

The two limit spaces $\sT^\alpha$ and $\sT_{\mathrm{nr}}^\alpha$ in Theorem~\ref{scaling_limit_HS_intro} and (\ref{scaling_limit_Catalan}) are real trees, i.e.~all pairs of their points are joined by a unique arc that turns out to be a geodesic (see Definition~\ref{real_tree}). Denote by $\bbT_\bbR^{\mathrm{m}}$ the space of (equivalence classes of) rooted (finitely) measured compact real trees, endowed with the rooted Gromov--Hausdorff--Prokhorov distance (see Sections~\ref{topo_function} and \ref{real_tree_section} for details). In light of Theorem~\ref{scaling_limit_HS_intro} and of (\ref{scaling_limit_Catalan}), it seems reasonable to expect that some functional $\Info$ plays the role of a continuum analog of the weighted Horton--Strahler number, so that $\Info(\sT^\alpha)$ is constant. We indeed explicitly introduce such an object in Definition~\ref{info_def}, that we denote by $\Info_\delta:\bbT_\bbR^{\mathrm{m}}\longrightarrow[0,\infty]$ and call the \textit{Strahler dilation with base $\delta=(\tfrac{\alpha}{\alpha-1})^{\alpha-1}$}.

\begin{theorem}
\label{binary_info_intro}
The Strahler dilation $\Info_\delta$ as in Definition~\ref{info_def} is measurable and satisfies
\begin{longlist}
\item[(i)] $\Info_\delta(T,\lambda\, d,\rho,\mu)=\lambda\Info_\delta(T,d,\rho,0)$ for all $\lambda\in(0,\infty)$ and all $(T,d,\rho,\mu)\in\bbT_\bbR^{\mathrm{m}}$,
\item[(ii)] $\Info_\delta(\sT^\alpha)=1$ almost surely, where $\sT^\alpha$ is the limit tree in Theorem~\ref{scaling_limit_HS_intro}.
\end{longlist}
\end{theorem}
We cannot suitably adapt the discrete Definition~\ref{def_tHS_intro_full} or the combinatorial approach by Horton pruning to the continuum setting, but Definition~\ref{info_def} is inspired by the definition of the Horton--Strahler number as the maximal height of an embedded perfect binary tree. A little more precisely, for $T\in \bbT_{\bbR}^{\mathrm{m}}$ and $\cdelta>1$, $\Info_{\cdelta}(T)$ quantifies the largest \emph{infinite $\cdelta$-dyadic tree} --- a self-similar tree obtained by gluing two identical copies of itself, each scaled down by a factor $1/\cdelta$, at the tip of an initial segment --- that can be embedded into $T$ without contraction in average. Furthermore, the Strahler dilations induce a notion of branching dimension for $T$ in the sense that there is at most one value $\cdelta \in (1,\infty)$ for which $\Info_{\cdelta}(T)$ is positive and finite.

The proof of $(ii)$ relies on a self-similar spinal decomposition of $\sT^\alpha$ that highlights its maximal infinite $\delta$-dyadic subtree. We formulate it later as Theorem~\ref{self-similar_alpha}. Noteworthy in itself, it is reminiscent of self-similar fragmentations of the $\alpha$-stable tree: along the path to a uniform leaf (Haas, Pitman \& Winkel~\cite[Corollary~10]{reroot_inv}) or to the highest leaf (Abraham \& Delmas~\cite[Theorem~3.3]{AbrDel09}). In fact, our new decomposition can also be applied to the $\alpha$-stable tree $\sT_{\mathrm{nr}}^\alpha$ thanks to the following result, which relates the laws of $\sT^\alpha$ and $\sT_{\mathrm{nr}}^\alpha$ via the Strahler dilation.

\begin{theorem}
\label{standard_stable_HS_real_thm}
Let $\sT^\alpha$ be the limit tree in Theorem~\ref{scaling_limit_HS_intro}, and let $(\sT_{\mathrm{nr}}^\alpha,d_{\mathrm{nr}}^\alpha,\rho_{\mathrm{nr}}^\alpha,\mu_{\mathrm{nr}}^\alpha)$ be the $\alpha$-stable tree as in (\ref{scaling_limit_Catalan}). Let $\Info_\delta$ be the Strahler dilation with base $\delta$ as in Theorem~\ref{binary_info_intro}. Then for all bounded, measurable functions $F:\bbT_\bbR^{\mathrm{m}}\longrightarrow \bbR$, it holds that
\[\E\big[F(\sT^\alpha)\big]=\frac{\alpha^{\beta}}{\Gamma(1-\frac{1}{\alpha})}\E\Big[F\big(\sT_{\mathrm{nr}}^\alpha \, ,\,  \Info_\delta(\sT_{\mathrm{nr}}^\alpha)^{-1}d_{\mathrm{nr}}^\alpha \, ,\,  \rho_{\mathrm{nr}}^\alpha \, ,\,  \alpha^{-\beta}\Info_\delta(\sT_{\mathrm{nr}}^\alpha)^{-\alpha\beta}\mu_{\mathrm{nr}}^\alpha\big)\Info_\delta(\sT_{\mathrm{nr}}^\alpha)^{\beta}\Big].\]
\end{theorem}
Informally, Theorem~\ref{standard_stable_HS_real_thm} says that the Strahler dilation is to $\sT^\alpha$ what the total mass is to $\sT_{\mathrm{nr}}^\alpha$. We specify this relation and give other formulations later in Theorem~\ref{from_law(Y)_to_N}. Using Theorem~\ref{binary_info_intro}, Theorem~\ref{standard_stable_HS_real_thm} follows from an argument using Bayes' formula, which links (\ref{scaling_limit_Catalan}) and Theorem~\ref{scaling_limit_HS_intro}. Moreover, the same method shows that if $\ftau$ is a $\GWw_\alpha$-weighted tree, then the limit of $\tfrac{a_n}{n}\delta^{\tHS(\ftau)}$ under $\P(\, \cdot\, |\, \#\tau\geq n)$ is the Strahler dilation with base $\delta$ of the scaling limit of the tree $\tau$. The same proof fails for studying the local conditioning $\{\#\tau=n\}$, since this conditioning becomes too degenerate at the limit. To avoid this issue, we exploit a monotonicity property of the Horton--Strahler number by coupling a $\GWw_\alpha$-weighted tree conditioned to have \emph{exactly} $n$ leaves and a $\GWw_\alpha$-weighted tree conditioned to have \emph{at least} $n/2$ leaves, such that one is always embedded into the other. This relies on Marchal's algorithm \cite{marchal} which yields a sequence of nested $\GW_\alpha$-trees. We then reach our goal.

\begin{theorem}
\label{scaling_limit_size_intro}
We keep the notation of Theorem~\ref{scaling_limit_HS_intro}. Let $a_n=\alpha^{-1/\alpha}n^{1/\alpha}$ and let $\sT_{\mathrm{nr}}^\alpha$ be the $\alpha$-stable tree as in (\ref{scaling_limit_Catalan}). Let $\Info_\delta$ be the Strahler dilation with base $\delta$ as in Theorem~\ref{binary_info_intro}. Then jointly with (\ref{scaling_limit_Catalan}), the following convergence in distribution holds on $(0,\infty)$:
\[\frac{a_n}{n}\delta^{\tHS(\ftau)}\ \text{ under }\ \P(\, \cdot\, |\, \#\tau=n+1)\, \xrightarrow[n\rightarrow\infty,n\in \lfloor\alpha\rfloor\bbN]{d} \Info_\delta(\sT_{\mathrm{nr}}^\alpha).\]
\end{theorem}
We can reformulate this result by stating that the following convergence in distribution
\begin{equation}
\label{cvd_tHS}
\tHS(\ftau) -\tfrac{\alpha-1}{\alpha} \log_{\delta} n -\tfrac{1}{\alpha} \log_{\delta} \alpha  \;\text{ under }\P(\, \cdot\, |\, \#\tau=n+1)\, \xrightarrow[n\rightarrow\infty,n\in \lfloor \alpha\rfloor\bbN]{d} \log_{\delta}\Info_\delta(\sT_{\mathrm{nr}}^\alpha)
\end{equation}
holds on $\bbR$. Thanks to (\ref{weighted_vs_classic}), this yields the following convergence in probability
\[\frac{\alpha\gamma}{\ln n}\HS(\tau)\;\text{ under }\P(\, \cdot\, |\, \#\tau=n+1) \xrightarrow[n\rightarrow\infty,n\in \lfloor\alpha\rfloor\bbN]{\P}1,\] 
which is \cite[Theorem 1.2]{companion_1} in the specific case where $\mu= \mu_\alpha$. Furthermore, the variations of the fractional part of $\tfrac{\alpha-1}{\alpha}\log_\delta n$ induce the same periodic phenomenon observed in the asymptotic estimate (\ref{esperance_HS_Catalan}) proved by Flajolet, Raoult \& Vuillemin~\cite{flajolet} and Kemp~\cite{kemp}.

Theorem~\ref{scaling_limit_size_intro} shows that it is possible to asymptotically recover the weighted Horton--Strahler number from the limit metric structure. This is noteworthy because the former depends on the weights of the leaves, which are independent of the tree. While their contribution is limited to the fractional part, this part is of constant-order and is thus non-negligible within the convergence (\ref{cvd_tHS}). The disappearance of this dependence at the limit is explained by the specific choice of the law of weights, introduced to remove the arithmetic interference in the analysis of the classic Horton--Strahler number, which is thus an intrinsic aspect of the stable Galton--Watson tree. Similarly, the index $\Info_\delta(\sT_{\mathrm{nr}}^\alpha)$ is a new metric characteristic of the stable tree, whose properties could provide information about its geometry and complexity.
\smallskip

Our last contribution specifically focuses on $\alpha=2$, namely the binary case. In this setting, Flajolet, Raoult \& Vuillemin~\cite{flajolet} and Kemp~\cite{kemp} have explicitly derived the law of the Horton--Strahler number of Catalan trees. We adapt their computations to determine the law of the Strahler dilation of the Brownian tree. This unveils an unexpected identity.

\begin{theorem}
\label{height_info_CRT_intro}
If $(\cT_\be,d_\be,\rho_\be,\mu_\be)$ is the Brownian tree, then twice its Strahler dilation $2\, \Info_2(\cT_\be)$ with base $2$ has the same law as its height $\mathfrak{h}(\cT_\be)=\sup_{\sigma\in\cT_\be}d_\be(\rho_\be,\sigma)$.
\end{theorem}

The height $\mathfrak{h}(\cT_\be)$ of the Brownian tree $\cT_\be$ is equal to the maximum of the Brownian excursion $\be$. Its cumulative distribution function has been computed by Chung~\cite{chung} and Kennedy~\cite{kennedy}. By computing generating functions, Flajolet, Raoult \& Vuillemin~\cite{flajolet} and Kemp~\cite{kemp} have already found a link between the Horton--Strahler number and the height for discrete trees. Their statements do not exactly coincide because of a miscalculation, but Françon~\cite{francon} found the following correct result with a purely combinatorial method.
\begin{itemize}[leftmargin=20pt,rightmargin=20pt]
\item[] \emph{For all $n,p\in\bbN$, there are as many binary trees $t_2$ with $n$ leaves such that $\HS(t_2)=p$, as there are plane trees $t$ with $n$ vertices whose height $h$ satisfies $\lfloor \log_2(1+h)\rfloor = p$.}
\end{itemize}
Since uniform plane trees with $n$ vertices and uniform binary trees with $n$ leaves share the same scaling limit, namely the Brownian tree, this yields a weaker result than Theorem~\ref{height_info_CRT_intro}.

\paragraph*{Organisation of paper} In Section~\ref{framework}, we precisely set our framework and formally define our objects of interest. In Section~\ref{tools}, we adapt already known results to our setting and we derive classic estimates. We define the weighted Horton--Strahler number and we study its law for stable Galton--Watson weighted trees in Section~\ref{weighted_HS}. Section~\ref{pruning_section} is devoted to the study of weighted Horton pruning. In Section~\ref{scaling_limit_HS_section}, we prove Theorem~\ref{scaling_limit_HS_intro}. A first description of the limit tree in Theorem~\ref{scaling_limit_HS_intro} is given by Section~\ref{geometric_study}. In Section~\ref{binary_information}, we construct the Strahler dilation to show Theorem~\ref{binary_info_intro}. We prove Theorems~\ref{standard_stable_HS_real_thm} and \ref{scaling_limit_size_intro} in Section~\ref{conditioning_at_least}. Finally, Section~\ref{height_info_CRT} consists of the proof of Theorem~\ref{height_info_CRT_intro}. Throughout all this work, we will write the set of nonnegative real numbers, the set of nonnegative integers, and the set of positive integers respectively as
\[\bbR_+=[0,\infty),\quad\bbN=\{0,1,2,3,\ldots\},\quad\text{ and }\quad\bbN^*=\{1,2,3,\ldots\}.\]

\section{Framework, notation, and definitions}
\label{framework}

In this section, which contains no new results, we present the basic objects we will encounter and use throughout the paper.

\subsection{Topological framework}
\label{topo_function}

Let us first present the topologies that we use in this work.

\paragraph*{Càdlàg functions with compact support}  We denote by $\mathbb{D}(\mathbb{R}_+,\mathbb{R})$ the space of all right-continuous with left limits (\textit{càdlàg} for short) functions from $\mathbb{R}_+$ to $\mathbb{R}$. It is equipped with the Skorokhod ($J_1$) topology which makes it \textit{Polish}, i.e.~separable and completely metrizable. We refer to Billingsley~\cite[Chapter 3]{billingsley2013convergence} and Jacod \& Shiryanev~\cite[Chapter VI]{jacod} for background. Let $f\in\bbD(\mathbb{R}_+,\mathbb{R})$, we define its \textit{lifetime} $\zeta(f)$ and its $\eta$-\textit{modulus of continuity} $\omega_\eta(f)$ as
\begin{equation}
\label{lifetime_modulus}
\zeta(f)=\sup\, \{0\}\cup\{s\geq 0\ :\ f(s)\neq 0\}\quad\text{ and }\quad\omega_\eta(f)=\sup_{\substack{s_1,s_2\geq 0\\|s_1-s_2|\leq\eta}}|f(s_1)-f(s_2)|
\end{equation}
for all $\eta>0$. We say that $f$ \textit{has compact support} when $\zeta(f)<\infty$. Since we are also interested in the convergences of lifetimes of càdlàg functions with compact support, it will be useful to work on the following subspaces of the product space $\mathbb{D}(\bbR_+,\bbR)\times\bbR_+$:
\begin{align*}
\mathcal{D}_{\mathrm{K}}&=\left\{(f,\ell)\ :\ \ell\in\bbR_+\text{ and }f:\bbR_+\longrightarrow\mathbb{R}\text{ càdlàg such that }\zeta(f)\leq \ell\right\},\\
\mathcal{C}_{\mathrm{K}}&=\left\{(f,\ell)\ :\ \ell\in\bbR_+\text{ and }f:\bbR_+\longrightarrow\mathbb{R}\text{ continuous such that }\zeta(f)\leq \ell\right\}.
\end{align*}
Next, we define the \emph{uniform distance} $\mathtt{d}_\infty$ and the \emph{Skorokhod distance} $\mathtt{d}_{\mathrm{S}}$ on $\mathcal{D}_{\mathrm{K}}$ by setting
\begin{align}
\label{uniform_distance}
\mathtt{d}_\infty\big((f_1,\ell_1),(f_2,\ell_2)\big)&=|\ell_1-\ell_2|+\sup_{s\geq 0}|f_1(s)-f_2(s)|,\\
\label{skorokhod_distance}
\mathtt{d}_{\mathrm{S}}\big((f_1,\ell_1),(f_2,\ell_2)\big)&=|\ell_1-\ell_2|+\inf_{\psi}\sup_{s\geq 0}\big(\,|\psi(s)-s|+|f_1(\psi(s))-f_2(s)|\,\big)
\end{align}
where the infimum is taken over all increasing and bijective functions $\psi:\bbR_+\longrightarrow\bbR_+$, for all $(f_1,\ell_1),(f_2,\ell_2)\in \mathcal{D}_{\mathrm{K}}$. The proposition below gathers useful properties of $\mathcal{C}_{\mathrm{K}}$ and $\mathcal{D}_{\mathrm{K}}$.

\begin{proposition}
\label{DK_prop}
The following holds true.
\begin{longlist}
\item[(i)] $\mathcal{D}_{\mathrm{K}}$ is a closed subset of $\mathbb{D}(\bbR_+,\bbR)\times\bbR_+$, and $\mathcal{C}_{\mathrm{K}}$ is a closed subset of $\mathcal{D}_{\mathrm{K}}$.
\item[(ii)] The spaces $\mathcal{C}_{\mathrm{K}}$ and $\mathcal{D}_{\mathrm{K}}$ are Polish.
\item[(iii)] The topology of $\mathcal{D}_{\mathrm{K}}$ is induced by the distance $\mathtt{d}_{\mathrm{S}}$.
\item[(iv)] The distances $\mathtt{d}_\infty$ and $\mathtt{d}_{\mathrm{S}}$ are topologically equivalent on $\mathcal{C}_{\mathrm{K}}$.
\end{longlist}
\end{proposition}

\begin{proof}
Let $(f_n,\ell)\in\mathcal{D}_{\mathrm{K}}$ such that $f_n\longrightarrow f$ for the Skorokhod topology and $\ell_n\longrightarrow\ell$. If $\ell<s$ then $f_n(s)=f_n(s-)=0$ for all $n$ large enough, so $f(s)=0$ by \cite[Proposition 2.1, Chapter VI]{jacod}. Thus, $\zeta(f)\leq \ell<\infty$. This same proposition entails that if the $f_n$ are continuous then $f(s)=f(s-)$ for all $s\in\bbR_+$. This proves $(i)$. Finite products and closed subsets of Polish spaces are Polish so $(ii)$ follows from $(i)$. The point $(iii)$ is a consequence of \cite[Theorem 1.14, Chapter VI]{jacod}. Then, $(iv)$ follows from \cite[Proposition 1.17, Chapter VI]{jacod}.
\end{proof}

\begin{notation}
\label{notation_DK}
We identify any càdlàg function with compact support $f$ with the pair $(f,\zeta(f))\in\mathcal{D}_{\mathrm{K}}$. Hence, a sequence $(f_n)$ of càdlàg functions with compact support converges to $(f,\ell)$ on $\mathcal{D}_{\mathrm{K}}$ if and only if $f_n\longrightarrow f$ for the Skorokhod topology and $\zeta(f_n)\longrightarrow\ell$. We point out that $\ell$ and $\zeta(f)$ do not need to be equal a priori. However, we will say that $(f_n)$ converges to $f$ on $\mathcal{D}_{\mathrm{K}}$ when $f_n\longrightarrow f$ for the Skorokhod topology and $\zeta(f_n)\longrightarrow\zeta(f)$.
\end{notation}
\noi
We now provide a variant of the classic tightness criterion for random continuous functions.

\begin{proposition}
\label{criterion_tightness}
A sequence $(\nu_n)$ of distributions on $\mathcal{C}_{\mathrm{K}}$ is tight if and only if it holds
\begin{longlist}
\item[(a)] $0=\lim_{m\rightarrow\infty}\limsup_{n\rightarrow\infty}\nu_n\left(\{(f,\ell)\in\mathcal{C}_{\mathrm{K}}\ :\ |f(0)|\geq m\}\right)$,
\item[(b)] $0=\lim_{m\rightarrow\infty}\limsup_{n\rightarrow\infty}\nu_n\left(\{(f,\ell)\in\mathcal{C}_{\mathrm{K}}\ :\ \ell\geq m\}\right),$
\item[(c)] $0=\lim_{\eta\rightarrow 0^+}\limsup_{n\rightarrow\infty}\nu_n\left(\{(f,\ell)\in\mathcal{C}_{\mathrm{K}}\ :\ \omega_\eta(f)\geq\varepsilon\}\right)$ for all $\varepsilon>0$.
\end{longlist}
\end{proposition}

\begin{proof}
For $L\in\bbR_+$, let $\bbC_L$ be the space of continuous functions from $[0,L]$ to $\bbR$ endowed with the uniform topology and $A$ a compact subset of $\bbC_L$. By $(i)$ and $(iv)$ of Proposition~\ref{DK_prop}, $A_L=\{(f,\ell) : \ell\leq L,f\in A\}$ is compact in $\mathcal{C}_{\mathrm{K}}$. Conversely, any compact subset of $\mathcal{C}_{\mathrm{K}}$ is contained in some $A_L$. The proposition then follows from \cite[Theorem 7.3]{billingsley2013convergence}.
\end{proof}

\noi
Furthermore, we denote by $\rho_{\mathrm{S}}$ the \emph{Prokhorov metric} associated with $\mathtt{d}_{\mathrm{S}}$ on the space $\mathcal{P}(\mathcal{D}_{\mathrm{K}})$ of all Borel probability measures on $\mathcal{D}_{\mathrm{K}}$. Namely, for all $\nu_1,\nu_2\in\mathcal{P}(\mathcal{D}_{\mathrm{K}})$, we define
\begin{equation}
\label{prokhorov_metric}
\rho_{\mathrm{S}}(\nu_1,\nu_2)=\inf\{\varepsilon>0\ :\ \forall A\subset\mathcal{D}_{\mathrm{K}}\ \text{ Borel subset},\ \ \nu_1(A)\leq \nu_2(A^\varepsilon)+\varepsilon\}
\end{equation}
where $A^\varepsilon=\{x\in \mathcal{D}_{\mathrm{K}}\ :\ \exists a\in A,\ \mathtt{d}_{\mathrm{S}}(x,a)<\varepsilon\}$. The space $\mathcal{D}_{\mathrm{K}}$ is Polish so \cite[Theorem 6.8]{billingsley2013convergence} gives us the following result.

\begin{proposition}
\label{prokhorov_prop}
The space $\mathcal{P}(\mathcal{D}_{\mathrm{K}})$ equipped with the topology of weak convergence is Polish. This topology is induced $\rho_{\mathrm{S}}$, which is indeed a distance on $\mathcal{P}(\mathcal{D}_{\mathrm{K}})$.
\end{proposition}

\begin{remark}
\label{lack_complete}
While $\mathcal{D}_{\mathrm{K}}$ is Polish, $\mathtt{d}_{\mathrm{S}}$ (and so $\rho_{\mathrm{S}}$) is not complete, see \cite[Ex.~12.2]{billingsley2013convergence}.
\end{remark}

\paragraph*{Rooted Gromov--Hausdorff--Prokhorov distance}

We define almost the same distances and we follow the same presentation as in \cite[Section 2.1]{GHP_addario}. Let $n\in\bbN$ with $n\geq 1$. We say that $(E,d,\ba)$ is a \textit{$n$-pointed compact metric space} when $(E,d)$ is a compact metric space endowed with a sequence of $n$ distinguished points $\ba=(a_1,\ldots,a_n)$ of $E$. A \textit{$n$-pointed measured compact metric space} $(E,d,\ba,\mu)$ is a $n$-pointed compact metric space $(E,d,\ba)$ equipped with a finite Borel measure $\mu$ on $E$. We say two $n$-pointed compact metric spaces $(E,d,\ba)$ and $(E',d',\ba')$ are \textit{$n$-pointed-isometric} when there exists a bijective isometry $\phi$ from $E$ to $E'$ such that $\phi(a_i)=a_i'$ for all $1\leq i\leq n$. Moreover, two $n$-pointed measured compact metric spaces $(E,d,\ba,\mu)$ and $(E',d',\ba',\mu')$ are said to be \textit{$n$-GHP-isometric} when there exists a bijective isometry $\phi$ from $E$ to $E'$ such that $\phi(a_i)=a_i'$ for all $1\leq i\leq n$ and such that the image measure of $\mu$ by $\phi$ is equal to $\mu'$. We denote by $\mathbb{K}_n$ the space of $n$-pointed-isometry classes of $n$-pointed compact metric spaces, and by $\mathbb{K}_n^{\mathrm{m}}$ the space of $n$-GHP-isometry classes of $n$-pointed measured compact metric spaces.
\begin{notation}
\label{underlying_set}
If no confusion is possible, we denote such a $n$-pointed (measured) compact metric space $(E,d,\ba,\mu)$, as well as its class in $\mathbb{K}_n$ or $\mathbb{K}_n^{\mathrm{m}}$, by its underlying space $E$.
\end{notation}
A \textit{$n$-pointed correspondence} between $E$ and $E'$ is a subset $\mathcal{R}$ of $E\times E'$ with $(a_i,a_i')\in\mathcal{R}$ for all $1\leq i\leq n$ and such that for all $x\in E$ and $y'\in E'$, there are $x'\in E'$ and $y\in E$ such that $(x,x')$ and $(y,y')$ are in $\mathcal{R}$. The \textit{distortion} of a $n$-pointed correspondence $\mathcal{R}$ is given by
\begin{equation}
\label{distortion}
\dis(\mathcal{R})=\sup\left\{\left|d(x,y)-d'(x',y')\right|\ :\ (x,x')\in\mathcal{R}\text{ and }(y,y')\in\mathcal{R}\right\}.
\end{equation}
The \textit{$n$-pointed Gromov--Hausdorff distance} between $E$ and $E'$ is then expressed as
\begin{equation}
\label{n-GH}
\mathtt{d}_{n-\mathrm{GH}}(E,E')=\tfrac{1}{2}\inf_{\mathcal{R}}\mathrm{dis}(\mathcal{R}),
\end{equation}
where the infimum is taken over all $n$-pointed correspondences $\mathcal{R}$ between $E$ and $E'$. We may also restrict the infimum to compact $n$-pointed correspondences without modifying the value. Indeed, the closure of a $n$-pointed correspondence is a compact $n$-pointed correspondence that has the same distortion because $E\times E'$ is compact. For any finite Borel measure $\nu$ on $E\times E'$, the \textit{discrepancy} of $\nu$ with respect to $\mu$ and $\mu'$ is defined by
\[\Dsp(\nu\ ;\ \mu,\mu')=\sup_{\substack{B\subset E\\ \text{Borel subset}}}\big|\nu(B\times E')-\mu(B)\big|+\sup_{\substack{B'\subset E'\\ \text{Borel subset}}}\big|\nu(E\times B')-\mu'(B')\big|.\]
The \textit{$n$-pointed Gromov--Hausdorff--Prokhorov distance} between $E$ and $E'$ is expressed as
\begin{equation}
\label{n-GHP}
\mathtt{d}_{n-\mathrm{GHP}}(E,E')=\inf_{\mathcal{R},\nu}\max\Big(\tfrac{1}{2}\mathrm{dis}(\mathcal{R})\, ,\, \Dsp(\nu\ ;\ \mu,\mu')+\nu\big((E\times E')\backslash \mathcal{R}\big)\Big),
\end{equation}
where the infimum is taken over all finite Borel measures $\nu$ on $E\times E'$ and all compact $n$-pointed correspondences $\mathcal{R}$ between $E$ and $E'$. The objects $\mathtt{d}_{n-\mathrm{GH}}$ and $\mathtt{d}_{n-\mathrm{GHP}}$ are only pseudo-distances but $\mathtt{d}_{n-\mathrm{GH}}(E,E')=0$ if and only if $E$ and $E'$ are $n$-pointed-isometric, and $\mathtt{d}_{n-\mathrm{GHP}}(E,E')=0$ if and only if $E$ and $E'$ are $n$-GHP-isometric. Hence, they respectively define genuine distances on $\mathbb{K}_n$ and $\mathbb{K}_n^{\mathrm{m}}$. The metric spaces $(\mathbb{K}_n,\mathtt{d}_{n-\mathrm{GH}})$ and $(\mathbb{K}_n^{\mathrm{m}},\mathtt{d}_{n-\mathrm{GHP}})$ are separable and complete: see Abraham, Delmas \& Hoscheit~\cite[Theorem 2.5]{GHP_polish}. While they define the distances in terms of isometric embeddings, the expressions (\ref{n-GH}) and (\ref{n-GHP}) in terms of correspondences give the same objects. See e.g.~Khezeli~\cite[Theorem 3.5]{GHP_correspondence}.

\begin{notation}
\label{omit_n}
When $n=1$, we shall omit it in the notation and replace the adjective pointed by \textit{rooted}. The unique distinguished point shall be called the \textit{root} of the space.
\end{notation}
Forgetting the measure of a rooted measured compact metric space yields a rooted compact metric space. This is formalized by the continuous surjection $(E,d,\rho,\mu)\in\mathbb{K}^{\mathrm{m}}\longmapsto(E,d,\rho)\in\mathbb{K}$. Moreover, endowing a rooted compact metric space with its null measure describes an isometric embedding of $\mathbb{K}$ into $\mathbb{K}^{\mathrm{m}}$. Let $(E,d,\rho,\mu)$ be a rooted measured compact metric space and let $\lambda\geq 0$. We will rescale rooted measured compact metric spaces by setting
\begin{equation}
\label{notation_scaling}
\lambda\cdot(E,d,\rho,\mu)=(E,\lambda d,\rho,\mu)\quad\text{ and }\quad\lambda\odot_\alpha (E,d,\rho,\mu)=(E,\lambda d,\rho,\lambda^{\alpha/(\alpha-1)}\mu)
\end{equation}
for all $\alpha\in(1,2]$. These operations are continuous from $\mathbb{K}^{\mathrm{m}}$ to itself and they coincide into one continuous map from $\mathbb{K}$ to itself. Furthermore, we define the \textit{height} $\mathfrak{h}(E)$ and the \textit{mass} $\mathfrak{m}(E)$ of a rooted (measured) compact metric space respectively as
\begin{equation}
\label{height_mass}
\mathfrak{h}(E,d,\rho)=\sup_{x\in E}d(\rho,x)\quad\text{ and }\quad\mathfrak{m}(E,d,\rho,\mu)=\mu(E).
\end{equation}
The functions $\mathfrak{h}:\mathbb{K}\longrightarrow [0,\infty)$ and $\mathfrak{m}:\mathbb{K}^{\mathrm{m}}\longrightarrow [0,\infty)$ are Lipschitz and thus continuous.

\subsection{Discrete trees as sets of words}
\label{words}

\paragraph*{Words and trees}
We recall Ulam's formalism of trees. Let $\mathbb{N}^*=\{1,2,3,\ldots\}$ be the set of positive integers and let $\mathbb{U}$ be the set of finite words
\begin{equation}
\label{set_words}
\mathbb{U}=\bigcup_{n\in\mathbb{N}}(\mathbb{N}^*)^n\quad\text{ with the convention $(\mathbb{N}^*)^0=\{\varnothing\}$.}
\end{equation}
The \textit{lexicographic order}, denoted by $\leq$, is a total order on $\mathbb{U}$. For $u=(u_1,\ldots,u_n)\in\mathbb{U}$ and $v=(v_1,\ldots,v_m)\in\mathbb{U}$, $u*v=(u_1,\ldots,u_n,v_1,\ldots,v_m)\!\in\!\mathbb{U}$ stands for the \textit{concatenation} of $u$ and $v$. Denote by $|u|=n$ the \textit{height} of $u$, and if $n\geq 1$ then denote by $\overleftarrow{u}=(u_1,\ldots,u_{n-1})$ the \textit{parent} of $u$. We then say that $u$ is a \textit{child} of $v$ when $\overleftarrow{u}=v$. The \textit{genealogical order} $\preceq$ is a partial order on $\mathbb{U}$ defined by $u\preceq v\Longleftrightarrow \exists u'\in\mathbb{U},\ v=u*u'$. We say that $u$ is an ancestor of $v$ when $u\preceq v$. When $u\preceq v$ but $u\neq v$, we may write $u\prec v$. Observe that for all $u\in\bbU$, the set $\{v\in\bbU\, :\, v\preceq u\}$ of ancestors of $u$ is totally ordered by $\preceq$. Denote by $u\wedge v\in\mathbb{U}$ the \textit{most recent common ancestor} of $u$ and $v$, that is their common ancestor with maximal height. 
\begin{notation}
\label{parent_root}
Although $\overleftarrow{\varnothing}$ is not defined, we set $\overleftarrow{\varnothing}\! <\! u$ and $\overleftarrow{\varnothing}\! \prec\! u$ for all $u\! \in\! \mathbb{U}$.
\end{notation}

\begin{definition}
\label{tree}
A subset $t$ of $\mathbb{U}$ is a \textit{tree} when the following is verified:
\begin{itemize}
\item[(a)] $t$ is finite and $\varnothing\in t$,
\item[(b)] for all $u\in t$, if $u\neq \varnothing$ then $\overleftarrow{u}\in t$,
\item[(c)] for all $u\in t$, there exists an integer $k_u(t)\in\mathbb{N}$ such that $u*(i)\in t\Longleftrightarrow 1\leq i\leq k_u(t)$.
\end{itemize}
We denote by $\mathbb{T}$ the (countable) space of all trees, endowed with the discrete topology.
\end{definition}
Several times in this work, we will need to embed trees into others, in the following sense.
\begin{definition}
\label{embedding_discrete}
Let $t$ be a tree and let $A\subset\bbU$. Let $\psi:t\longrightarrow A$ be an injective map. 
\begin{longlist}
\item[(i)] We say that $\psi$ is an \emph{embedding} when $\psi(u\wedge v)=\psi(u)\wedge\psi(v)$ for all $u,v\in t$.
\item[(ii)] We say that $\psi$ is \emph{increasing} when $u< v\Longrightarrow \psi(u)<\psi(v)$ for all $u,v\in t$.
\end{longlist}
If $\psi$ is an embedding, then we set $\psi(\overleftarrow{\varnothing})=\overleftarrow{\varnothing}$ in accordance with Notation~\ref{parent_root}. 
\end{definition}
\begin{remark}
\label{canonical_plane_order}
Let $A\subset\bbU$. If $A$ satisfies $(a)$ and $(b)$ in Definition~\ref{tree}, then there exist a unique tree $t$ and a unique increasing embedding $\psi:t\longrightarrow A$ such that $\psi(t)=A$.
\end{remark}
Let $t$ be a tree. The number $\#t$ of its vertices is also called the \emph{size} of $t$. We use the following notation for the \textit{height} of $t$ and its set of \textit{leaves}:
\begin{equation}
\label{height_notation}
|t|=\max_{u\in t}|u|\quad\text{ and }\quad\partial t=\{u\in t\ :\ k_u(t)=0\},
\end{equation} 
and $\#\partial t$ stands for the number of leaves of $t$. For $v\!\in\! t$, the \emph{subtree of $t$ stemming from $v$} is 
\[\theta_v t=\{u\in\mathbb{U}\ :\ v*u\in t\}.\]
Observe that $\theta_v t$ is also a tree. We list all the elements of $t$ in the lexicographic order as 
\begin{equation}
\label{depth-first_exploration}
\varnothing=u_0(t)<\ldots<u_{\#t-1}(t),
\end{equation}
and we call the finite sequence of words $u(t)=\left(u_i(t)\ ;\ 0\leq i\leq\#t-1\right)$ the \textit{depth-first exploration} of $t$. We denote by $H(t):s\in[0,\infty)\longmapsto H_s(t)$ and we call the \textit{height function} of $t$ the affine-by-parts and continuous function defined by
\begin{equation}
\label{height_process_def}
H_i(t)=\I{i<\# t}|u_i(t)|\quad\text{ and }\quad H_s(t)=H_i(t)+(s-i)(H_{i+1}(t)-H_i(t)) 
\end{equation}
for all integers $i\in\bbN$ and for all real numbers $s\in[i,i+1]$. It is folklore (see e.g.~\cite{legall_trees}), and easy to see, that the height function fully characterizes the corresponding tree. The function $H(t)$ is nonnegative and has compact support. In accordance with Notation~\ref{notation_DK}, we shall identify it with the element $\big(H(t),\zeta(H(t))\big)$ of $\mathcal{C}_{\mathrm{K}}$. Let us also observe that 
\begin{equation}
\label{lifetime_height}
H_0(t)=0,\ \quad \sup H(t)=|t|,\quad\text{ and }\quad\zeta(H(t))=\I{\# t\geq 2}\# t.
\end{equation}

\paragraph*{Stable Galton--Watson trees} Recall from (\ref{stable_offspring}) that $\mu_\alpha=(\mu_\alpha(k))_{k\in\bbN}$ is the probability measure on $\bbN$ characterized by $\varphi_\alpha(s):=\sum_{k\geq 0}s^k\mu_\alpha(k)=s+\tfrac{1}{\alpha}(1-s)^\alpha$. We compute that
\begin{equation}
\label{derive_generating}
\varphi_\alpha'(s)=1-(1-s)^{\alpha-1}\quad\text{ and }\quad\varphi_\alpha^{(m)}(s)=\mu_\alpha(m)\cdot m!(1-s)^{\alpha-m},
\end{equation}
for all $s\in[0,1)$ and $m\geq 2$. By induction, the identities (\ref{stable_offspring}) and (\ref{derive_generating}) yield the expressions
\begin{equation}
\label{stable_offspring_explicit}
\mu_\alpha(0)=\tfrac{1}{\alpha},\; \mu_\alpha(1)=0,\; \mu_\alpha(2)=\tfrac{\alpha-1}{2},\; \mu_\alpha(k)=\tfrac{1}{k!}\prod_{i=1}^{k-1}|\alpha-i|=\I{\alpha<2}\tfrac{(\alpha-1)\Gamma(k-\alpha)}{\Gamma(2-\alpha)k!}
\end{equation}
for all $k\geq 3$, where $\Gamma$ stands for the usual Gamma function. In particular, we recover that $\mu_2$ is the critical binary offspring distribution. Let $\mu=(\mu(k))_{k\in\bbN}$ be another probability measure on $\bbN$. We say that $\mu$ is \textit{critical and non-trivial} when
\begin{equation}
\label{critical}
\mu(0)>0\quad\text{ and }\quad\sum_{k\in\bbN}k\mu(k)=1.
\end{equation}
We readily observe from (\ref{derive_generating}) and (\ref{stable_offspring_explicit}) that $\mu_\alpha$ is critical and non-trivial.

\begin{definition}
\label{GWdef}
Let $\mu$ be a probability measure on $\bbN$ which verifies (\ref{critical}). A \textit{Galton--Watson tree with offspring distribution $\mu$} (or \textit{$\GW(\mu)$-tree} for short) is a random (finite) tree $\tau$ that satisfies the following branching property.
\begin{itemize}
\item[(a)] The law of $k_\varnothing(\tau)$ is $\mu$.
\item[(b)] For $k\!\in\!\bbN^*$ such that $\mu(k)\!>\!0$, the random trees $\theta_{(1)}\tau,\ldots,\theta_{(k)}\tau$ under $\P(\, \cdot\, |\, k_\varnothing(\tau)=k)$ are independent with the same law as $\tau$ under $\P$.
\end{itemize}
When $\mu=\mu_\alpha$ as above, we call $\tau$ an \textit{$\alpha$-stable Galton--Watson tree} (\textit{$\GW_\alpha$-tree} for short).
\end{definition}
It is well-known that (\ref{critical}) ensures that a random (finite) tree described by Definition~\ref{GWdef} indeed exists. Moreover, see e.g~Le Gall~\cite[Propositions 1.4, 1.5]{legall_trees}, if $\tau$ is a $\GW(\mu)$-tree then
\begin{equation}
\label{lawGW}
\forall t\in\bbT,\quad\P (\tau = t)= \prod_{u\in t} \mu \big( k_u(t) \big),
\end{equation}
and we can couple $\tau$ with a sequence $(\xi_i)_{i\in\bbN}$ of i.i.d.~random variables of law $\mu$ so that
\begin{equation}
\label{couplage_tree_walk}
\big(k_{u_i(\tau)}(\tau)\ ;\ 0\leq i\leq \#\tau-1\big)=\big(\xi_i\ ;\ 0\leq i\leq \#\tau-1\big)\text{ almost surely}.
\end{equation}

\paragraph*{Trees with edge lengths}

We say that $T=(t,(l_u)_{u\in t})$ is a \textit{tree with edge lengths} when $t$ is a tree, as in Definition~\ref{tree}, and $l_u\in(0,\infty)$ for all $u\in t$. We denote by $\overline{\bbT}$ the space of all trees with edge lengths and we endow it with the product topology, meaning that $(t^n,(l_u^n)_{u\in t^n})\longrightarrow (t,(l_u)_{u\in t})$ if and only if there is $n_0\geq 0$ such that $t^n=t$ for all $n\geq n_0$ and $l_u^n\longrightarrow l_u$ for all $u\in t$. This makes $\overline{\bbT}$ a Polish space. For convenience, we shall write
\begin{equation}
\label{notation_edge_length}
\Ske(T)=t\quad\text{ and }\quad \HS(T)=\HS(t),
\end{equation}
where $\HS$ is the Horton--Strahler number as in Definition~\ref{def_HS_intro_full}. For all $v\in t$ and $\lambda> 0$, we set
\begin{equation}
\label{scaling_edge_length}
\theta_v T=(\theta_v t,(l_{v*u})_{u\in \theta_v t})\quad\text{ and }\quad\lambda\cdot T=(t,(\lambda l_u)_{u\in t}).
\end{equation}
Remark that $\theta_v T$ and $\lambda\cdot T$ are also trees with edge lengths. Recall from (\ref{depth-first_exploration}) that $u(t)$ is the depth-first exploration of $t$. The tree with edge lengths $T$ is described by a nonnegative càdlàg function with compact support $\Hl(T)=(\Hl_s(T))_{s\geq0}$ called its \textit{height function}:
\begin{equation}
\label{height_process_length_def}
\forall s\geq 0,\quad \Hl_s(T)=\sum_{i=0}^{\# t-1}\I{\sum_{j=0}^{i-1} l_{u_j(t)}\leq s<\sum_{j=0}^i l_{u_j(t)}}\Big(\sum_{v\prec u_i(t)}l_v+s-\sum_{j=0}^{i-1} l_{u_j(t)}\Big).
\end{equation}
We stress that $\Hl(T)$ is not continuous and should not be confused with $H(t)$. Written with Notation~\ref{notation_DK}, the map $T\!\in\!\overline{\bbT}\mapsto \Hl(T)\! \in\! \mathcal{D}_{\mathrm{K}}$ is continuous. Note from (\ref{lifetime_modulus}), (\ref{scaling_edge_length}), (\ref{height_process_length_def}) that
\begin{equation}
\label{lifetime_height_length}
\zeta(\Hl(T))=\sum_{u\in t}l_u\quad\text{ and }\quad\Hl_s(\lambda\cdot T)=\lambda \Hl_{s/\lambda}(T).
\end{equation}
One can interpret the height function as follows. Let us draw $T$ into the upper half-plane with its correct edge lengths and let us picture a particle that, starting at $\overleftarrow{\varnothing}$, explores $T$ at unit speed, from left to right, without hitting two times the same point, and backtracking only after reaching a leaf. Then, $\Hl_s(T)$ is equal to the distance of the particle from $\overleftarrow{\varnothing}$ at time $s$. From the point of view of the particle, a single edge or a chain of several edges put back-to-back with the same total length is the same, which translates into the following result.

\begin{proposition}
\label{pruning_height}
For any tree with edge lengths $T=(t,(l_u)_{u\in t})$, there exist a unique tree $t'$ and a unique increasing embedding $\psi:t'\longrightarrow t$ such that $\psi(t')=\{u\in t\, :\, k_u(t)\neq 1\}$. Moreover, if $T'$ is the tree with edge lengths defined below, then it holds that $\Hl(T')=\Hl(T)$:
\[T'=(t',(l_u')_{u\in t'})\quad\text{ where }\quad l_u'=\sum_{\psi(\overleftarrow{u})\prec v\preceq \psi(u)}l_v.\]
\end{proposition}

\begin{proof}
It is straightforward to do the proof by induction, by considering, for example, the $\leq$-minimal vertex $v\in t$ such that $k_v(t)\neq 1$. We leave it as an exercise for the reader.
\end{proof}

\begin{definition}
\label{GWldef}
Let $\mu$ be a probability measure on $\bbN$ which satisfies (\ref{critical}). A \textit{$\GWl(\mu)$-tree with edge lengths} is random tree with edge lengths $\cT=(\tau,(L_u)_{u\in\tau})$ such that $\tau$ is a $\GW(\mu)$-tree and conditionally given $\tau$, the $(L_u)_{u\in\tau}$ are independent with exponential law with mean $1$. When $\mu=\mu_\alpha$ as above, we call $\cT$ a \textit{$\GWl_\alpha$-tree with edge lengths}.
\end{definition}

\paragraph*{Weighted trees}

We say that $\bt=(t,(w_v)_{v\in\partial t})$ is a \textit{weighted tree} when $t$ is a tree and $w_v\in[0,1)$ for all $v\in\partial t$, where $\partial t$ stands for the set of leaves of $t$ as in (\ref{height_notation}). We denote by $\bbT_{\mathrm{w}}$ the space of all weighted trees and we endow it with the product topology, meaning that $\bt^n=(t^n,(w_v^n)_{v\in\partial t_n})\longrightarrow\bt$ if and only if there is $n_0\geq 0$ such that $t^n=t$ for all $n\geq n_0$ and $w_v^n\longrightarrow w_v$ for all $v\in \partial t$. This makes $\mathbb{T}_{\mathrm{w}}$ a Polish space. Recall from Definition~\ref{def_tHS_intro_full} that $\tHS$ stands for the weighted Horton--Strahler number, and observe that the map $\tHS:\bbT_{\mathrm{w}}\longrightarrow \bbR_+$ is continuous. For all $u\in t$, we have $\partial\theta_u t=\{v\in\theta_u t\, :\, u*v\in\partial t\}$, so setting \[\theta_u\bt=(\theta_u t,(w_{u*v})_{v\in\partial\theta_u t})\] defines a weighted tree. Recall from (\ref{FEXP(a)}) and (\ref{alpha_gamma_delta}) that $\mathsf{FExp}(\gamma)$ denotes the law of the fractional part of an exponential random variable with mean $1/\gamma$, where $\gamma=\ln\frac{\alpha}{\alpha-1}$.

\begin{definition}
\label{GWwdef}
An \textit{$\alpha$-stable Galton--Watson weighted tree} (or a \textit{$\GWw_\alpha$-weighted tree} for short) is a random weighted tree $\ftau=\big(\tau,(W_v)_{v\in\partial\tau}\big)$ such that $\tau$ is a $\GW_\alpha$-tree and conditionally given $\tau$, the $(W_v)_{v\in\partial\tau}$ are i.i.d. of law $\mathsf{FExp}(\gamma)$. The law of $\ftau$ satisfies the following.
\begin{itemize}
\item[(a)] The law of $k_\varnothing(\tau)$ is $\mu_\alpha$ as above, i.e.~characterized by (\ref{stable_offspring}).
\item[(b)] Under $\P(\, \cdot\, |\, k_\varnothing(\tau)=0)$, the law of $W_\varnothing$ is $\mathsf{FExp}(\gamma)$.
\item[(c)] For all $k\in\bbN^*$ such that $\mu_\alpha(k)>0$, the random weighted trees $\theta_{(1)}\ftau,\ldots,\theta_{(k)}\ftau$ under $\P(\, \cdot\, |\, k_\varnothing(\tau)=k)$ are independent with the same law as $\ftau$ under $\P$.
\end{itemize}
\end{definition}

\paragraph*{Weighted trees with edge lengths}

We say that $\bT=(t,(l_u)_{u\in t},(w_v)_{v\in \partial t})$ is a \textit{weighted tree with edge lengths} when $T=(t,(l_u)_{u\in t})$ is a tree with edge lengths and $w_v\in[0,1)$ for all $v\in\partial t$. We denote by $\overline{\bbT}_{\mathrm{w}}$ the space of all weighted trees with edge lengths, and we endow it with the product topology, meaning that $\bT_n=(t^n,(l_u^n)_{u\in t^n},(w_v^n)_{v\in\partial t^n})\longrightarrow \bT$ if and only if there is $n_0\geq 0$ such that $t^n=t$ for all $n\geq n_0$, $l_u^n\longrightarrow l_u$ for all $u\in t$, and $w_v^n\longrightarrow w_v$ for all $v\in\partial t$. This makes $\overline{\bbT}_{\mathrm{w}}$ a Polish space. For convenience, we shall write
\begin{equation}
\label{notation_weighted_edge_length}
\Ske(\bT)=(t,(w_v)_{v\in \partial t})\quad\text{ and }\quad\tHS(\bT)=\tHS(\Ske(\bT)).
\end{equation}
Moreover, for all $u'\in t$ and $\lambda>0$, we set
\begin{equation}
\label{scaling_weighted_edge_length}
\theta_{u'} \bT=\big(\theta_{u'} t,(l_{u'*u})_{u\in \theta_{u'} t},(w_{u'*v})_{v\in\partial \theta_{u'} t}\big)\quad\text{ and }\quad \lambda\cdot \bT=(t,(\lambda l_u)_{u\in t},(w_v)_{v\in\partial t}).
\end{equation}
Note that $\Ske(\bT)$ is a weighted tree, and $\theta_{u'} \bT$ and $\lambda\cdot \bT$ are weighted trees with edge lengths. 

\begin{definition}
\label{GWwldef}
A \textit{$\GWwl_\alpha$-weighted tree with edge lengths} is a random weighted tree with edge lengths $\fTau=\big(\tau,(L_u)_{u\in\tau},(W_v)_{v\in\partial\tau}\big)$ such that $\cT=\left(\tau,(L_u)_{u\in\tau}\right)$ is a $\GWl_\alpha$-tree with edge lengths and conditionally given $\cT$, the $(W_v)_{v\in\partial\tau}$ are independent with distribution $\mathsf{FExp}(\gamma)$. Thus, the law of $\fTau$ satisfies the following branching property.
\begin{itemize}
\item[(a)] $k_\varnothing(\tau)$ and $L_\varnothing$ are independent, of respective laws $\mu_\alpha$ and exponential with mean $1$.
\item[(b)] Under $\P(\, \cdot\, |\, k_\varnothing(\tau)=0)$, $W_\varnothing$ is independent from $L_\varnothing$ and has law $\mathsf{FExp}(\gamma)$.
\item[(c)] For all $k\in\bbN^*$ with $\mu_\alpha(k)>0$, under $\P(\,\cdot\, |\, k_\varnothing(\tau)=k)$, $\theta_{(1)}\fTau,\ldots,\theta_{(k)}\fTau$ are independent, jointly independent from $L_\varnothing$, and have the same law as $\fTau$ under $\P$.
\end{itemize}
\end{definition}

\subsection{Real trees}
\label{real_tree_section}

In this work, we are mostly interested in tree-like metric spaces.
\begin{definition}
\label{real_tree}
A metric space $(T,d)$ is a \textit{real tree} when for all $x,y\in T$:
\begin{longlist}
\item[(a)] there is a unique distance-preserving map $g_{x,y}:[0,d(x,y)]\to T$ with $g_{x,y}(0)=x$ and $g_{x,y}(d(x,y))=y$, which is called the \textit{geodesic from $x$ to $y$},
\item[(b)] all injective continuous functions $h:[0,1]\to T$ with $h(0)=x$ and $h(1)=y$ (that we call \textit{arcs from $x$ to $y$}) share the same image, denoted by $\llbracket x,y\rrbracket=h([0,1])=g_{x,y}([0,d(x,y)])$.
\end{longlist}
A rooted (measured) compact metric space $(T,d,\rho,\mu)$ is a \textit{rooted (measured) compact real tree} when $(T,d)$ is a real tree. We denote $\mathbb{T}_{\mathbb{R}}$ (resp.~$\mathbb{T}_{\mathbb{R}}^{\mathrm{m}}$) the set of isometry classes of rooted (measured) compact real trees equipped with $\mathtt{d}_{\mathrm{GH}}$ as in (\ref{n-GH}) (resp.~$\mathtt{d}_{\mathrm{GHP}}$ as in (\ref{n-GHP})).
\end{definition}
The spaces $\mathbb{T}_{\mathbb{R}}$ and $\mathbb{T}_{\mathbb{R}}^{\mathrm{m}}$ are respectively closed subsets of $\mathbb{K}$ and $\mathbb{K}^{\mathrm{m}}$, see e.g.~Evans~\cite[Lemma 4.22]{evans}, so they are separable and complete metric spaces.

\paragraph*{Real trees coded by continuous excursions} We define the set of \textit{continuous excursions} as
\begin{equation}
\label{excursion_space}
\mathcal{E}_{\mathrm{K}}=\{(f,\ell)\in\mathcal{C}_{\mathrm{K}}\ :\ \forall s\geq 0,\ f(s)\geq 0=f(0)\}.
\end{equation}
Endowed with the uniform distance $\mathtt{d}_\infty$ defined by (\ref{uniform_distance}), it is a closed subspace of $\mathcal{C}_{\mathrm{K}}$ and so a Polish space thanks to Proposition~\ref{DK_prop}. For $(f,\ell)\in\mathcal{E}_{\mathrm{K}}$, we set \[d_{f}(s_1,s_2)=d_f(s_2,s_1)=f(s_1)+f(s_2)-2\inf_{[s_1,s_2]}f\]
for all $s_2\geq s_1\geq 0$. The function $d_f$ is a continuous pseudo-distance on $\bbR_+$. Writing $s_1\sim_{f} s_2$ when $d_{f}(s_1,s_2)=0$ defines an equivalence relation on $\bbR_+$. It induces the quotient space $\cT_{f,\ell}=[0,\ell]/\sim_f$ on which $d_{f}$ induces a genuine distance that we denote by $d_{f,\ell}$. By continuity of $d_f$ on $[0,\ell]^2$, the canonical projection $\mathsf{p}_{f,\ell}:[0,\ell]\longrightarrow \cT_{f,\ell}$ is continuous and the metric space $(\cT_{f,\ell},d_{f,\ell})$ is compact. Moreover, $\cT_{f,\ell}$ is a real tree; see e.g.~Le Gall~\cite[Theorem 2.2]{legall_trees}. We extend $\mathsf{p}_{f,\ell}$ to $\bbR_+$ by setting $\mathsf{p}_{f,\ell}(s)=\mathsf{p}_{f,\ell}(\ell)$ when $s\geq\ell$. Finally, we write $\rho_{f,\ell}=\mathsf{p}_{f,\ell}(0)$ and we denote by $\mu_{f,\ell}$ the image measure by $\mathsf{p}_{f,\ell}$ of the Lebesgue measure on $[0,\ell]$. We call
\begin{equation}
\label{real_tree_coded}
\big(\cT_{f,\ell}\, ,\, d_{f,\ell}\, ,\, \rho_{f,\ell}\, ,\, \mu_{f,\ell}\big)
\end{equation}
the \emph{(rooted measured compact) real tree coded by $(f,\ell)$}. Still in accordance with Notation~\ref{notation_DK}, if $\ell=\zeta(f)$ then we shall write $(\cT_{f},d_{f},\rho_{f},\mu_{f})=(\cT_{f,\zeta(f)},d_{f,\zeta(f)},\rho_{f,\zeta(f)},\mu_{f,\zeta(f)})$. Recall the rescaling operation $\odot_\alpha$ from (\ref{notation_scaling}) and the height $\mathfrak{h}$ and mass $\mathfrak{m}$ from (\ref{height_mass}). Note that if $g(s)=\lambda^{1-1/\alpha} f(s/\lambda)$ for all $s\geq 0$ with some $\lambda>0$, then $\cT_g=\lambda^{1-1/\alpha}\odot_\alpha \cT_f$. Observe that $\mathfrak{h}(\cT_{f,\ell})=\sup f$ and $\mathfrak{m}(\cT_{f,\ell})=\ell$. The map $(f,\ell)\in\mathcal{E}_{\mathrm{K}}\longmapsto \cT_{f,\ell}\in\mathbb{T}_\bbR^{\mathrm{m}}$ is continuous since
\begin{equation}
\label{continuite_coded_tree}
\mathtt{d}_{\mathrm{GHP}}(\cT_{f_1,\ell_1},\cT_{f_2,\ell_2})\leq |\ell_1-\ell_2|+2\sup_{s\geq 0}\left|f_1(s)-f_2(s)\right|\leq 2\mathtt{d}_\infty\big((f_1,\ell_1),(f_2,\ell_2)\big).
\end{equation}
To see this, consider the compact rooted correspondence $\cR=\left\{\left(\mathsf{p}_{f_1,\ell_1}(s),\mathsf{p}_{f_2,\ell_2}(s)\right)\ :\ s\geq 0\right\}$ and the image measure by $(\mathsf{p}_{f_1,\ell_1},\mathsf{p}_{f_2,\ell_2})$ of the Lebesgue measure on $[0,\min(\ell_1,\ell_2)]$.

\paragraph*{Scaling limits of discrete trees} 
Any tree $t$ (as in Definition~\ref{tree}) may be rooted at $\varnothing$, equipped with its counting measure $\sum_{u\in t}\delta_u$, and endowed with the graph distance on $\bbU$,
\begin{equation}
\label{graph_distance}
\forall u,v\in\bbU,\quad \mathtt{d}_{\mathrm{gr}}(u,v)=|u|+|v|-2|u\wedge v|,
\end{equation}
 to obtain a rooted measured compact metric space. This describes a continuous injection from the discrete space of trees $\bbT$ into $\mathbb{K}^{\mathrm{m}}$: we thus see $\bbT$ as a closed subset of $\mathbb{K}^{\mathrm{m}}$ with a slight abuse of notation. We stress that (discrete) trees are not real trees because they are not connected. However, linking each vertex of $t$ to its parent by a metric segment of unit length yields a genuine real tree. Also, the latter is encoded by the so-called contour function of $t$. We informally present it below but we refer to Duquesne~\cite{duquesne_contour_stable} for a rigorous definition.

Let us picture a particle that, starting at the root, continuously walks through the tree so that it retraces its steps as little as possible and respects the lexicographic order of the vertices. The particle crosses each edge twice, once upward and once downward, so it takes $2(\#t-1)$ steps to return to the root after having explored the whole tree. We say that the particle follows the \textit{exploration by contour} of the tree $t$ and we denote by $c(t)=(c_i(t)\ ;\ 0\leq i\leq 2(\#t-1))$ the sequence of its positions on $t$. Then, the \textit{contour function} of $t$ is defined by
\[C_i(t)=\I{i\leq 2\#t -2}|c_i(t)|\quad\text{ and }\quad C_s(t)=C_i(t)+(s-i)(C_{i+1}(t)-C_i(t))\]
for all integers $i\in\bbN$ and $s\in[i,i+1]$. The slopes of this piecewise affine function are in $\{1,-1,0\}$. Moreover, it is an element of $\mathcal{E}_{\mathrm{K}}$ such that $\sup C(t)=|t|$ and $\zeta(C(t))=2\#t-2$.

\begin{proposition}
\label{height_contour_tree}
Let $(h,\ell) \in \mathcal{C}_{\mathrm{K}}$. Let $t_n$ be a tree and let $\lambda_n,b_n > 0$ for all $n \in \bbN$. Assume that $\lambda_n\to 0$ and $\lambda_n/b_n \to \infty$. Recall from (\ref{height_process_def}) that $H(t_n)$ is the height function of $t_n$. The two convergences $\left(\lambda_n H_{s/b_n}(t_n)\right)_{s\geq 0}\longrightarrow (h,\ell)$ and $\left(\lambda_n C_{2s/b_n}(t_n)\right)_{s\geq 0}\longrightarrow (h,\ell)$ are equivalent on $\mathcal{C}_{\mathrm{K}}$. If they hold, then $(h,\ell)\in\mathcal{E}_{\mathrm{K}}$ and $\left(t,\lambda_n\mathtt{d}_{\mathrm{gr}},\varnothing,b_n\sum_{u\in t}\delta_u\right)\longrightarrow \cT_{h,\ell}$ for the rooted Gromov--Hausdorff--Prokhorov distance.
\end{proposition}

\begin{proof}
The equivalence of the convergences follows from (\ref{lifetime_height}), Proposition~\ref{DK_prop} $(iv)$, and general deterministic arguments of Duquesne \& Le Gall~\cite[Section 2.5]{levytree_DLG} that show the height and the contour functions are similar once rescaled. The metric of a tree $t$ is given by
\begin{equation}
\label{contour_code}
d_{C(t)}(i,j)=C_i(t)+C_j(t)-2\inf_{[i,j]}C(t)=\mathtt{d}_{\mathrm{gr}}\left(c_i(t),c_j(t)\right)
\end{equation}
for all integers $0\leq i\leq j\leq 2(\#t-1)$; see e.g.~Le Gall~\cite[Section 2]{legall_trees}. For $\lambda,b>0$, we define a continuous excursion $f\in\mathcal{E}_{\mathrm{K}}$ and a rooted measured compact metric space $T\in\mathbb{K}^{\mathrm{m}}$ by setting $f:s\in\bbR_+\longmapsto \lambda C_{2s/b}(t)$ and $T=\left(t,\lambda\mathtt{d}_{\mathrm{gr}},\varnothing,b\sum_{u\in t}\delta_u\right)$. By (\ref{contour_code}), it is easy to check that
$\mathtt{d}_{\mathrm{GH}}(\cT_f,T)\leq \lambda$ and $\mathtt{d}_{\mathrm{GHP}}(\cT_f,T)\leq \lambda+b$. Then, (\ref{continuite_coded_tree}) completes the proof.
\end{proof}

\paragraph*{Genealogical order on a rooted compact real tree} Let $(T,d,\rho)$ be a rooted compact real tree and let $\sigma_1,\sigma_2\in T$. From Definition~\ref{real_tree}, recall that $g_{\sigma_1,\sigma_2}:[0,d(\sigma_1,\sigma_2)]\longrightarrow T$ is the geodesic between $\sigma_1$ and $\sigma_2$ and $\llbracket \sigma_1,\sigma_2\rrbracket =g_{\sigma_1,\sigma_2}([0,d(\sigma_1,\sigma_2)]$ is the unique path between $\sigma_1$ and $\sigma_2$. The equality $d(\sigma_1,\sigma_2)=d(\sigma_1,\sigma)+d(\sigma,\sigma_2)$ thus holds for all $\sigma\in\llbracket \sigma_1,\sigma_2\rrbracket$. The \textit{genealogical order} $\preceq$ on $(T,d,\rho)$ is the partial order defined by
\begin{equation}
\label{genealogy_real_tree}
\sigma_1\preceq \sigma_2\quad\text{ if and only if }\quad\sigma_1\in\llbracket \rho,\sigma_2\rrbracket.
\end{equation}
We then say that $\sigma_1$ is an \textit{ancestor} of $\sigma_2$. We may also write $\sigma_1\prec\sigma_2$ when $\sigma_1\preceq \sigma_2$ and $\sigma_1\neq\sigma_2$. It is immediate that $\rho\preceq\sigma$ for all $\sigma\in T$. While $\preceq$ is only a partial order, it induces a total order on any \textit{ancestral lineage}, i.e., on any $\llbracket\rho,\sigma\rrbracket$ for $\sigma\in T$. Indeed, if $\sigma_1,\sigma_2$ are ancestors of $\sigma$ then $\sigma_i=g_{\rho,\sigma}(d(\rho,\sigma_i))$ for both $i\!\in\!\{1,2\}$. The uniqueness of geodesics implies that
\begin{equation}
\label{ancestral_totally_ordered}
\textit{ if }\:\sigma_1,\sigma_2\preceq\sigma\:\textit{ then }\:\sigma_1\preceq\sigma_2\Longleftrightarrow d(\rho,\sigma_1)\leq d(\rho,\sigma_2)\Longleftrightarrow d(\sigma_2,\sigma)\leq d(\sigma_1,\sigma).
\end{equation}

The geodesics are embeddings of real segments into $T$ so the ancestral lineages are closed in $T$. Thus, the subset $\llbracket\rho,\sigma_1\rrbracket\cap\llbracket\rho,\sigma_2\rrbracket$ admits a unique $\preceq$-maximal element. We call it the \textit{most recent common ancestor} of $\sigma_1$ and $\sigma_2$ and denote it by $\sigma_1\wedge\sigma_2$. It is characterized by
\begin{equation}
\label{common_ancestor_real_tree}
\forall\sigma\in T\text{, }\quad\sigma\preceq \sigma_1\wedge \sigma_2\quad\text{ if and only if }\quad(\sigma\preceq \sigma_1\text{ and }\sigma\preceq\sigma_2).
\end{equation}
In particular, $\sigma_1\preceq\sigma_2$ if and only if $\sigma_1\wedge\sigma_2=\sigma_1$. By (\ref{common_ancestor_real_tree}), we check that $\wedge$ is associative, namely $\sigma_1\wedge(\sigma_2\wedge\sigma_3)=(\sigma_1\wedge\sigma_2)\wedge\sigma_3$ for all $\sigma_3\in T$. Also, Definition~\ref{real_tree} $(b)$ yields that
\begin{equation}
\label{common_ancestor_real_tree_alt}
\{\sigma_1\wedge\sigma_2\}=\llbracket \rho,\sigma_1\rrbracket\cap\llbracket \rho,\sigma_2\rrbracket\cap\llbracket \sigma_1,\sigma_2\rrbracket.
\end{equation}
By definition of real trees, they are locally pathwise connected and it follows that if $\sigma_1\wedge\sigma_2\neq \sigma_1,\sigma_2$ then $\sigma_1$ and $\sigma_2$ are in different connected components of $T\backslash\{\sigma_1\wedge\sigma_2\}$.
\begin{definition}
\label{planted}
A rooted compact real tree $(T,d,\rho)$ is \textit{planted} when $T\backslash\{\rho\}$ is connected. This is equivalent to saying that for all $\sigma_1,\sigma_2\in T$, $\rho=\sigma_1\wedge\sigma_2\Longrightarrow \rho\in\{\sigma_1,\sigma_2\}$.
\end{definition}
Next, we show that the operation $\wedge:T^2\longrightarrow T$ enjoys a fairly strong continuity property. 
\begin{lemma}
\label{tool_wedge_real_tree}
Let $(T,d,\rho)$ be a rooted compact real tree and let $\sigma_1,\sigma_2,\sigma_1',\sigma_2'\in T$. If $d(\sigma_1',\sigma_1)<d(\sigma_1\wedge\sigma_2,\sigma_1)$ and $d(\sigma_2',\sigma_2)<d(\sigma_1\wedge\sigma_2,\sigma_2)$, then $\sigma_1'\wedge\sigma_2'=\sigma_1\wedge\sigma_2$.
\end{lemma}

\begin{proof}
Let first $\sigma_2=\sigma_2'$. We have $d(\sigma_1\wedge\sigma_1',\sigma_1)\leq d(\sigma_1',\sigma_1)$ because $\sigma_1\wedge\sigma_1'\in\llbracket\sigma_1,\sigma_1'\rrbracket$. Also, $\sigma_1\wedge\sigma_1'$ and $\sigma_1\wedge\sigma_2$ are ancestors of $\sigma_1$, so $\sigma_1\wedge\sigma_2\prec \sigma_1\wedge\sigma_1'\preceq\sigma_1'$ by (\ref{ancestral_totally_ordered}). Then,
\begin{equation}
\label{tool_wedge_real_tree_step}
(\sigma_1'\wedge\sigma_1)\wedge(\sigma_1'\wedge\sigma_2)=(\sigma_1\wedge\sigma_2)\wedge(\sigma_1'\wedge\sigma_1')=\sigma_1\wedge\sigma_2\neq
\sigma_1'\wedge\sigma_1.
\end{equation}
The points $\sigma_1'\wedge\sigma_1$ and $\sigma_1'\wedge\sigma_2$ are ancestors of $\sigma_1'$, so (\ref{ancestral_totally_ordered}) and (\ref{tool_wedge_real_tree_step}) imply that $\sigma_1\wedge\sigma_2=\sigma_1'\wedge\sigma_2$. The general result follows after a double application of the case $\sigma_2=\sigma_2'$.
\end{proof}

\noi
We stress that the order $\preceq$ and the operation $\wedge$ depend on $(T,d,\rho)$, and especially on the root $\rho$, but we shall omit this dependence when no confusion is possible according to context.

We end this section with the notion of subtrees of a rooted compact real tree.
\begin{definition}
\label{subtree}
Let $(T_1,d_1,\rho_1)$ and $(T_2,d_2,\rho_2)$ be two rooted compact real trees. We say that $T_1$ is a \textit{subtree} of $T_2$ when there is a distance-preserving map $\phi:T_1\to T_2$ such that $\phi(\rho_1)\preceq \phi(\sigma)$ for all $\sigma\in T_1$. In that case, $\phi(\sigma\wedge\sigma')=\phi(\sigma)\wedge\phi(\sigma')$ for all $\sigma,\sigma'\in T_1$.
\end{definition}

\subsection{Grafting rooted compact metric spaces}
\label{GHP}

Here, we borrow a grafting procedure from Abraham, Delmas \& Hoscheit~\cite{exit_times} in order to later describe the limit real tree in Theorem~\ref{scaling_limit_HS_intro}. Informally, it consists in gluing the roots of some rooted (measured) compact metric spaces $E_i$ onto respective points $a_i$ of the same space $E^o$.

\begin{definition}[Grafting procedure]
\label{graft_def}
Let $(E^o,d^o,\rho^o,\mu^o)$ be a rooted measured compact metric space and let $(E_i,d_i,\rho_i,\mu_i)_{i\in I}$ be a countable family of rooted measured compact metric spaces. Let $(a_i)_{i\in I}$ be a family of points of $E^o$ also indexed by $I$. We denote by $E$ the disjoint union $E^o\sqcup\bigsqcup_{i\in I}E_i\backslash\{\rho_i\}$   and we endow it with a distance $d$ defined as follows:
\begin{longlist}
\item[(a)] if $x,y\in E^o$ then $d(x,y)=d^o(x,y)$,
\item[(b)] if $x,y\in E_i\backslash\{\rho_i\}$ with some $i\in I$ then $d(x,y)=d_i(x,y)$,
\item[(c)] if $x\in E^o$ and $y\in E_i\backslash\{\rho_i\}$ with some $i\in I$ then $d(x,y)=d^o(x,a_i)+d_i(\rho_i,y)$,
\item[(d)] if $x\in E_i\backslash\{\rho_i\}$ and $y\in E_j\backslash\{\rho_j\}$ with some distinct $i,j\in I$, then we set $d(x,y)=d_i(x,\rho_i)+d^o(a_i,a_j)+d_j(\rho_j,y)$.
\end{longlist}
With a slight abuse of notation, we see $E^o$ and the $E_i$ as closed subsets of $E$ by using the identifications $a_i=\rho_i$ for all $i\in I$. In accordance with Notation~\ref{underlying_set}, we denote the rooted metric space $(E,d,\rho^o)$ by $E^o\circledast_{i\in I}(a_i,E_i)$. Note that $E^o\circledast_{i\in I}(a_i,E_i)$ may not be compact.

If the spaces $E^o$ and $E_i$ for all $i\in I$ are resp.~endowed with finite Borel measures $\mu^o$ and $\mu_i$ for all $i\in I$, then we also equip $E$ with a (potentially infinite) Borel measure $\mu$ by setting
\[\mu(B)=\mu^o(B\cap E^o)+\sum_{i\in I}\mu_i(B\cap E_i)\]
for any Borel subset $B$ of $E$. We also denote $(E,d,\rho^o,\mu)$ by $E^o\circledast_{i\in I}(a_i,E_i)$.
\end{definition}
We write a finite number of successive graftings onto the same space as follows:
\[E^o\circledast_{i\in I}(a_i,E_i)\circledast_{j\in J} (a_j,E_j):=\big(E^o\circledast_{i\in I} (a_i,E_i)\big)\circledast_{j\in J}(a_j,E_j)=E^o\circledast_{i\in I\sqcup J} (a_i,E_i).\]

\begin{proposition}
\label{compact_graft}
We keep the notation above. Recall $\mathfrak{h}$ and $\mathfrak{m}$ from (\ref{height_mass}). If it holds
\begin{equation}
\label{graft_cond1}
\forall \varepsilon>0,\quad \{i\in I\ :\ \mathfrak{h}(E_i)\geq\varepsilon\}\text{ is finite},
\end{equation}
then $E^o\circledast_{i\in I}(a_i,E_i)$ is a rooted compact metric space. If in addition to (\ref{graft_cond1}) it holds
\begin{equation}
\label{graft_cond2}
\sum_{i\in I}\mathfrak{m}(E_i)<\infty,
\end{equation}
then $E^o\circledast_{i\in I}(a_i,E_i)$ is a rooted measured compact metric space (i.e.~its measure is finite). 
\end{proposition}

\begin{proof}
Let $(x_j)$ be a sequence of points of $E=E^o\circledast_{i\in I}(a_i,E_i)$, let us show it admits a convergent subsequence. Since $E^o$ and the $E_i$ are compact, we can assume that $x_j\in E_{i_j}$ for all $j\geq 0$, where $(i_j)$ is an injective sequence of elements of $I$. Then, it holds that $\mathfrak{h}(E_{i_j})$ tends to $0$ so we have $d(x_j,a_{i_j})\longrightarrow 0$. The sequence $(a_{i_j})$ stays inside the compact $E^o$ so it admits a subsequential limit, and so does $(x_j)$. The second statement of the proposition follows from the fact that $\mu^o$ is a finite measure and from the identity $\mu(E)=\mu^o(E^o)+\sum_{i\in I}\mathfrak{m}(E_i)$.
\end{proof}

\begin{proposition}
\label{graft_continuous}
Recall from (\ref{n-GHP}) that $\mathbb{K}_n^{\mathrm{m}}$ is endowed with the distance $\mathtt{d}_{n-\mathrm{GHP}}$. For all $n\geq 1$, the following map from $\mathbb{K}_{n+1}^{\mathrm{m}}\times(\mathbb{K}^{\mathrm{m}})^n$ to $\mathbb{K}^{\mathrm{m}}$ is continuous:
\[(E^o,d^o,(\rho^o,a_1,\ldots,a_n),\mu^o)\, ,\, \big((E_i,d_i,\rho_i,\mu_i)\big)_{1\leq i\leq n}\longmapsto (E^o,d^o,\rho^o,\mu^o)\circledast_{i=1}^n(a_i,E_i).\]
\end{proposition}

\begin{proof}
Let $\mathcal{R}^o$ be a compact $(n+1)$-pointed correspondence between $E^o$ and ${E^o}'$, and $\mathcal{R}_i$ be compact rooted correspondences respectively between $E_i$ and $E_i'$. We also take finite Borel measures $\nu^o$ and $\nu_i$ respectively on $E^o\times {E^o}'$ and $E_i\times E_i'$. Then, we construct a compact rooted correspondence between $E^o\circledast_{i=1}^n(a_i,E_i)$ and ${E^o}'\circledast_{i=1}^n(a_i',E_i')$, and a finite Borel measure on the product space $(E^o\circledast_{i=1}^n(a_i,E_i))\times({E^o}'\circledast_{i=1}^n(a_i',E_i'))$ by setting
\[\mathcal{R}=\mathcal{R}^o\sqcup\bigsqcup_{i=1}^n\mathcal{R}_i\backslash\{(\rho_i,\rho_i')\}\quad\text{ and }\quad \nu(\cdot)=\nu^o(\cdot\cap E^o\times {E^o}')+\sum_{i=1}^n\nu_i(\cdot\cap\ E_i\times E_i').\]
Taking the infimum over $(\mathcal{R}^o,\mathcal{R}_i,\nu^o,\nu_i)$ yields the result since we can check that the distortion of $\mathcal{R}$ (resp.~discrepancy of $\nu$) is bounded by the sum of that of $\cR^o,\cR_i$, for $1\leq i\leq n$.
\end{proof}

\noi
We end this section by describing the grafting procedure when the components are real trees.
\begin{proposition}
\label{wedge_grafting}
Let $(T^o,d^o,\rho^o)$ be a rooted compact real tree and let $(a_i)_{i\in I}$ be a countable family of points of $T^o$. Let $(T_i,d_i,\rho_i)_{i\in I}$ be a countable family of rooted compact real trees such that (\ref{graft_cond1}) holds. Then, $T=T^o\circledast_{i\in I}(a_i,T_i)$ is a rooted compact real tree. We specify if the notation of Definition~\ref{real_tree} and  (\ref{common_ancestor_real_tree}) stand for the objects on $T$, $T^o$, or $T_i$ for some $i\in I$ by respectively writing $\llbracket\cdot,\cdot\rrbracket$ and $\wedge$, $\llbracket\cdot,\cdot\rrbracket^o$ and $\wedge^o$, or $\llbracket\cdot,\cdot\rrbracket_i$ and $\wedge_i$. Then, for all $\sigma_1,\sigma_2\in T$ and distinct $i,j\in I$:
\begin{longlist}
\item[(i)] if $\sigma_1,\sigma_2\in T^o$, then $\llbracket\sigma_1,\sigma_2\rrbracket=\llbracket\sigma_1,\sigma_2\rrbracket^o$ and $\sigma_1\wedge\sigma_2=\sigma_1\wedge^o\sigma_2$,
\item[(ii)] if $\sigma_1,\sigma_2\in T_i$ with some $i\in I$, then $\llbracket\sigma_1,\sigma_2\rrbracket=\llbracket\sigma_1,\sigma_2\rrbracket_i$ and $\sigma_1\wedge \sigma_2=\sigma_1\wedge_i\sigma_2$,
\item[(iii)] if $\sigma_1\in T^o$ and $\sigma_2\in T_i$, then $\llbracket\sigma_1,\sigma_2\rrbracket=\llbracket\sigma_1,a_i\rrbracket^o\cup \llbracket\rho_i,\sigma_2\rrbracket_i$ and $\sigma_1\wedge\sigma_2=\sigma_1\wedge^o a_i$,
\item[(iv)] if $\sigma_1\in T_i,\sigma_2\in T_j$, then $\llbracket \sigma_1,\sigma_2\rrbracket\!=\!\llbracket\sigma_1,\rho_i\rrbracket_i \!\cup\! \llbracket a_i,a_j\rrbracket^o \! \cup \! \llbracket \rho_j,\sigma_2\rrbracket_j$ and $\sigma_1\!  \wedge\! \sigma_2=a_i\!\wedge^o\! a_j$.
\end{longlist}
\end{proposition}

\begin{proof}
The proof is straightforward so we only give a sketch here, leaving the details to the reader. From Definition~\ref{graft_def}, we first check that concatenating geodesics on $T^o,T_i,T_j$ as instructed by the desired result yields a geodesic from $\sigma_1$ to $\sigma_2$ on $T$. We observe that $T_j$ and $T_j\backslash\{a_j\}$ are resp.~closed and open in $T$, so any path starting outside and ending inside $T_j$ has to hit $a_j$. By injectivity, an arc from $\sigma_1$ to $\sigma_2$ on $T$ is thus the concatenation of at most three arcs (resp.~on $T^o,T_i,T_j)$ as indicated by the statement, which determines its image.
\end{proof}

\subsection{Stable Lévy processes}
\label{more_stable}

We now present the continuous setting involved in the asymptotics of large stable Galton--Watson trees. We denote by $\rX$ the canonical process on the space $\mathbb{D}(\bbR_+,\mathbb{R})$ of càdlàg functions endowed with the Skorokhod topology, already introduced in Section~\ref{topo_function}. Under the underlying probability measure $\P$, we assume that $\rX=(\rX_s)_{s\geq 0}$ is a spectrally positive stable Lévy process with index $\alpha\in(1,2]$ such that \[\forall \lambda,s\in\bbR_+,\quad \E[\,\exp(-\lambda \rX_s)\,]=\exp(s\lambda^\alpha).\]
We refer to Bertoin~\cite[Chapter VIII]{bertouin} for background and details. If $\alpha=2$, $\frac{1}{\sqrt{2}}\rX$ is a standard Brownian motion. If $\alpha\!\in\!(1,2)$, the Lévy measure of $\rX$ is \[\Pi(\dd r)=\frac{\alpha(\alpha-1)}{\Gamma(2-\alpha)}r^{-1-\alpha}\un_{(0,\infty)}(r)\,\dd r,\]
where $\Gamma$ is the usual Gamma function. Moreover, the process $\rX$ enjoys the following scaling property: for all $\lambda>0$, the laws of $(\lambda^{-1/\alpha}\rX_{\lambda s}\ ;\ s\geq 0)$ and $\rX$ under $\P$ are the same. We write \[\rI_s=\inf_{[0,s]}\rX\quad\text{ and }\quad \rI_s^r=\inf_{[r,s]}\rX\] for all $0\leq r\leq s$. The process $\rI=(\rI_s)_{s\geq 0}$ is continuous because $\rX$ has no negative jumps.

The process $\rX-\rI$ is strong Markov and the point $0$ is regular for itself with respect to $\rX-\rI$ (see \cite[Chapter VI.1]{bertouin}). Moreover, we may and will choose $-\rI$ as the local time of $\rX-\rI$ at level $0$ by \cite[Theorem VII.1]{bertouin}. Let $\{(g_j,d_j),j\in J\}$ be the excursion intervals of $\rX-\rI$ above $0$ and let us set $\omega_s^j=\rX_{\min(g_j+s,d_j)}-\rX_{g_j}$ for all $j\in J$ and $s\geq 0$. Since $\inf \rX=-\infty$ (see \cite[Chapter VIII]{bertouin}), the $\omega^j$ are càdlàg functions with compact support which start at $0$ and stay nonnegative. In accordance with Notation~\ref{notation_DK}, we see them as elements of $\mathcal{D}_{\mathrm{K}}$. Then, the random point measure $\mathcal{N}=\sum_{j\in J} \delta_{(-\rI_{g_j},\omega^j)}$ is a Poisson measure on $\bbR_+\times\mathcal{D}_{\mathrm{K}}$ with intensity measure $\dd t\, \bN_\alpha(\dd \omega)$, where $\bN_\alpha$ is a sigma-finite measure on $\mathcal{D}_{\mathrm{K}}$ called the \emph{excursion measure}. We refer to \cite[Chapter IV]{bertouin} for details and background. Under $\bN_\alpha$, we simply denote by $\zeta=\zeta(\rX)=\sup \{0\}\cup\{s\geq 0\ :\ \rX_s\neq 0\}$ the lifetime of $\rX$ as in (\ref{lifetime_modulus}).

Le Gall \& Le Jan~\cite{lejan_legall} (see also Duquesne \& Le Gall~\cite[Chapter 1]{levytree_DLG}) constructed a continuous process $\rH=(\rH_s)_{s\geq 0}$, measurable with respect to $\rX$, such that the limit 
\begin{equation}
\label{continuous_height_process}
\rH_s=\lim_{\varepsilon\rightarrow 0^+}\frac{1}{\varepsilon}\int_0^s \I{\rX_r<\rI_s^r+\varepsilon}\, \dd r
\end{equation}
holds in $\P$-probability and in $\bN_\alpha$-measure for all $s\in\bbR_+$. When $\alpha=2$, it is known that $\rH$ is equal to $\rX-\rI$ under $\P$ and to $\rX$ under $\bN_2$. The process $\rH$ is called the \textit{height process associated with $\rX$} (and also the \emph{$\alpha$-stable height process} here) because it is the continuous analog of the height function (\ref{height_process_def}) of discrete Galton--Watson trees. Observe that (\ref{continuous_height_process}) and the scaling property of $\rX$ imply that for $\lambda>0$, the laws of $(\lambda^{1/\alpha-1}\rH_{\lambda s}\, ;\, s\geq 0)$ and $\rH$ under $\P$ are equal. Indeed, the former is the height process associated with $(\lambda^{-1/\alpha}\rX_{\lambda s}\, ;\, s\geq 0)$. Moreover, (\ref{continuous_height_process}) entails that $\bN_\alpha$-almost everywhere, $\zeta(\rH)=\zeta$ and $\rH$ is in the space $\mathcal{E}_{\mathrm{K}}$ of continuous excursions, i.e.~the subset of $\mathcal{C}_{\mathrm{K}}$ given by (\ref{excursion_space}). See \cite[Chapter 1]{levytree_DLG} for a proof.

Fluctuation theory (see e.g.~\cite[Chapter VIII.4]{bertouin}) yields that $\bN_\alpha(\zeta\!=\!0)=0$ and
\begin{equation}
\label{mass_stable_tree}
\forall\ell>0,\quad \bN_\alpha(\zeta>\ell)=\frac{1}{\Gamma(1-1/\alpha)}\, \ell^{-1/\alpha}.
\end{equation}
The scaling property of $\rH$ yields that the law of the process $(\zeta^{1/\alpha-1}\rH_{\zeta s}\ ;\ s\geq 0)$ under the probability measure $\bN_\alpha(\, \cdot\, |\, \zeta>\ell)$ does not depend on $\ell>0$. It is called the law of the \textit{normalized excursion of the $\alpha$-stable height process}, and we denote it by $\bN_\alpha(\dd \rH\, |\ \zeta=1)$ here. Informally, it can be understood as the law of an excursion of $\rH$ conditioned to have unit lifetime. In particular, if $\alpha=2$ then $\bN_2(\dd \rH\, |\, \zeta=1)$ is the law of $\sqrt{2}\be$, where $\be$ stands for the standard Brownian excursion. Moreover, (\ref{mass_stable_tree}) entails that for all nonnegative and measurable functions $F:\mathcal{C}_{\mathrm{K}}\longrightarrow\bbR_+$, it holds that
\begin{equation}
\label{decompo_lifetime}
\bN_\alpha\big[F(\rH)\big]=\frac{1}{\alpha\Gamma(1-1/\alpha)}\int_{\bbR_+} \bN_\alpha\Big[F\big(\ell^{1-1/\alpha}\rH_{s/\ell}\, ;\, s\geq 0\big)\ \Big|\ \zeta=1\Big]\, \frac{\dd \ell}{\ell^{1+1/\alpha}}.
\end{equation}
Finally, a random rooted (measured) compact real tree $\sT_{\mathrm{nr}}$ is called an \textit{$\alpha$-stable tree} when it is distributed as the real tree $\cT_\rH$ coded by $\rH$ under $\bN_\alpha(\dd \rH\, |\, \zeta=1)$, as defined by (\ref{real_tree_coded}). Up to the multiplicative constant $\sqrt{2}$, the $2$-stable tree is the Brownian tree introduced by Aldous~\cite{aldousI,aldous1993}, which is the rooted compact real tree coded by the standard Brownian excursion.

Next, recall $\beta=\tfrac{1}{\alpha-1}$ from (\ref{alpha_gamma_delta}) and observe $\bN_\alpha(\sup \rH\! =\! 0)=\bN_\alpha(\zeta\!=\!0)=0$. A consequence of a Ray-Knight theorem (Duquesne \& Le Gall~\cite[Corollary 1.4.2]{levytree_DLG}) is that
\begin{equation}
\label{height_stable_tree}
\forall x>0,\quad\bN_\alpha(\sup \rH>x)=\frac{1}{(\alpha-1)^{\beta}}\, x^{-\beta}.
\end{equation}
Let us write $M=\sup \rH$ to lighten the notation. Just as before, the scaling property ensures that the law of $\left(M^{-1}\rH_{sM^{\alpha\beta}}\ ;\ s\geq 0\right)$ under the probability measure $\bN_\alpha(\, \cdot\, |\, \sup \rH>x)$ is the same for all $x>0$: we denote it by $\bN_\alpha(\dd \rH\, |\, \sup \rH=1)$. Moreover, this identity in law and the formula (\ref{height_stable_tree}) yield that for all nonnegative and measurable functions $F:\mathcal{C}_{\mathrm{K}}\longrightarrow\bbR_+$,
\begin{equation}
\label{decompo_sup}
\bN_\alpha\big[F(\rH)\big]=\frac{1}{(\alpha-1)^{\alpha\beta}}\int_{\bbR_+}\!\! \bN_\alpha\Big[F\big(x\rH_{sx^{-\alpha\beta}}\, ;\, s\geq 0\big)\ \Big|\ \sup \rH=1\Big]\,\frac{\dd x}{x^{1+\beta}}.
\end{equation}
If $\alpha=2$ then $\bN_2(\dd \rH\, |\, \sup \rH=1)$ is equal to the law of $(e_{2s}^*)_{s\geq 0}$, where $e^*$ is the Brownian excursion conditioned to have its maximum equal to $1$. A Brownian excursion with a fixed maximum, such as $e^*$, can be described by William's path decomposition at the maximum, see e.g.~Revuz \& Yor~\cite[Chapter XII]{revuz2004continuous}. This then characterizes the law of its lifetime:
\begin{equation}
\label{max_Brownian_excursion}
\forall\lambda>0,\quad\E\big[\exp(-\lambda\zeta(e^*))\big]=\bN_2\big[\exp(-2\lambda\zeta)\ \big|\ \sup \rH =1\big]=\Big(\tfrac{\sqrt{2\lambda}}{\sinh\sqrt{2\lambda}}\Big)^2.
\end{equation}
See \cite{time_place} for a review of other formulas which can be proved using this method.

\section{Preliminary tools}
\label{tools}

This section focuses on gathering already-known results, and on proving technical estimates via classic methods. Although neither new nor directly related to the Horton--Strahler number, these tools will be useful throughout this paper.

\subsection{Marchal's algorithm}
\label{marchal_section}

Our proof of Theorem~\ref{scaling_limit_size_intro} is based on the recursive algorithm proposed by Marchal~\cite{marchal} to build a sequence of nested $\GW_\alpha$-trees conditioned on their number of leaves that converges, after scaling, towards the $\alpha$-stable tree. This sequence is a Markov chain on the space of \textit{labeled but unrooted trees}: namely connected and acyclic graphs equipped with an exhaustive enumeration of their vertices of degree $1$. Here, we present a slight variation of this construction to produce a sequence of random weighted trees $\fctau_i=(\ctau_i,(\mathtt{w}_v^i)_{v\in\partial \ctau_i}),i\geq 1$ in the sense of the present paper (Section~\ref{words}). See Figure~\ref{ex_marchal} for a better understanding of the construction. Recall from (\ref{FEXP(a)}) and (\ref{alpha_gamma_delta}) that $\mathsf{FExp}(\gamma)$ is the law of the fractional part of an exponential random variable with mean $1/\gamma$, where $\gamma=\ln\frac{\alpha}{\alpha-1}$.

\begin{alg}
\label{marchal_alg}
Let $(W_i)_{i\geq 1}$ be a sequence of i.i.d~RVs with common law $\mathsf{FExp}(\gamma)$. Start by setting $\fctau_1=\left(\{\varnothing\},W_1\right)$. To construct $\fctau_{i+1}$ from $\fctau_i$, randomly choose either a vertex $u\in\ctau_i$ or an edge $\{\overleftarrow{u},u\}$ of $\ctau_i$ --- including the "root" edge $\{\overleftarrow{\varnothing},\varnothing\}$ --- with propability proportional to $p_u\!=\!\I{k_u(\ctau_i)\geq 2}(k_u(\ctau_i)\!-\!\alpha)$ for vertex $u$ and to $p_{\{\overleftarrow{u},u\}}\!=\!\alpha\!-\!1$ for edge $\{\overleftarrow{u},u\}$.
\begin{longlist}
\item[(i)] If we draw an edge $\{\overleftarrow{u},u\}$ then we split it into two edges with a middle vertex, to which we connect a new leaf endowed with the weight $W_{i+1}$. The position of the new leaf compared to its sibling is chosen uniformly. More formally, we take a uniform random variable $J\in\{1,2\}$ and define $\fctau_{i+1}$ to be the unique weighted tree such that:
\begin{longlist}
\item[(a)] for all $v\in\mathbb{U}$ such that $u\wedge v\notin\{u,v\}$, $v\in\ctau_{i+1}$ if and only if $v\in\ctau_i$, in which case it holds $\theta_v\fctau_{i+1}=\theta_v\fctau_i$,
\item[(b)] it holds $u\in\ctau_{i+1}$ and $k_u(\ctau_{i+1})=2$,
\item[(c)] it holds $\theta_{u*(J)}\fctau_{i+1}=\left(\{\varnothing\},W_{i+1}\right)$ and $\theta_{u*(3-J)}\fctau_{i+1}=\theta_u\fctau_i$.
\end{longlist}
\item[(ii)] If we draw a vertex $u$ then we connect a new leaf to it that we endow with the weight $W_{i+1}$. The position of the new leaf among its siblings is chosen uniformly. More formally, we take a uniform random variable $J\in\{1,\ldots,k_u(\ctau_i)+1\}$ and define $\fctau_{i+1}$ to be the unique weighted tree such that:
\begin{longlist}
\item[(a)] for all $v\in\mathbb{U}$ such that $u\wedge v\notin\{u,v\}$, $v\in\ctau_{i+1}$ if and only if $v\in\ctau_i$, in which case it holds $\theta_v\fctau_{i+1}=\theta_v\fctau_i$,
\item[(b)] it holds $u\in\ctau_{i+1}$ and $k_u(\ctau_{i+1})=k_u(\ctau_i)+1$,
\item[(c)] it holds that $\theta_{u*(J)}\fctau_{i+1}=\left(\{\varnothing\},W_{i+1}\right)$ and that $\theta_{u*(j)}\fctau_{i+1}=\theta_{u*(j)}\fctau_i$ for all $1\leq j\leq J-1$, and $\theta_{u*(j)}\fctau_{i+1}=\theta_{u*(j-1)}\fctau_i$ for all $J+1\leq j\leq k_u(\ctau_{i+1})$.
\end{longlist}
\end{longlist}
Throughout this algorithm, all the choices are made independently from the rest.
\end{alg}

\begin{figure}
\begin{center}
\begin{subfigure}{37mm}
\begin{tikzpicture}[line cap=round,line join=round,>=triangle 45,x=1cm,y=1cm,scale=0.3]
\clip(-1,-1.2) rectangle (10,8);
\fill[line width=2pt,color=xdxdff,fill=xdxdff,fill opacity=0.1] (0,0) -- (-1,2) -- (2,2) -- cycle;
\fill[line width=2pt,color=xdxdff,fill=xdxdff,fill opacity=0.1] (5.5,1.5) -- (4.5,3.5) -- (7.5,3.5) -- cycle;
\draw [line width=1.2pt,color=xdxdff] (0,0)-- (-1,2);
\draw [line width=1.2pt,color=xdxdff] (-1,2)-- (2,2);
\draw [line width=1.2pt,color=xdxdff] (2,2)-- (0,0);
\draw [line width=1.2pt,dash pattern=on 1pt off 1.5pt] (0,0)-- (0,-1);
\draw [line width=1.2pt] (5.5,1.5)-- (7,0);
\draw [line width=1.2pt,color=ffqqqq] (8.5,1.5)-- (7,0);
\draw [line width=1.2pt,color=xdxdff] (5.5,1.5)-- (4.5,3.5);
\draw [line width=1.2pt,color=xdxdff] (4.5,3.5)-- (7.5,3.5);
\draw [line width=1.2pt,color=xdxdff] (7.5,3.5)-- (5.5,1.5);
\draw [->,line width=0.5pt] (2,1) -- (5,1);
\begin{scriptsize}
\draw [fill=xdxdff] (0,0) circle (3.5pt);
\draw[color=xdxdff] (0.7,-0.07491982588133289) node {$\varnothing$};
\draw [fill=xdxdff] (7,0) circle (3.5pt);
\draw[color=xdxdff] (7.7,-0.13205747593235065) node {$\varnothing$};
\draw [fill=xdxdff] (5.5,1.5) circle (3.5pt);
\draw [fill=ffqqqq] (8.5,1.5) circle (3.5pt);
\draw[color=ffqqqq] (8.71472835283758,2.3) node {$W_{i+1}$};
\end{scriptsize}
\end{tikzpicture}
\end{subfigure}
\hfill
\begin{subfigure}{47mm}
\begin{tikzpicture}[line cap=round,line join=round,>=triangle 45,x=1cm,y=1cm,scale=0.3]
\clip(-3.2,-1.2) rectangle (12,8);
\fill[line width=2pt,color=xdxdff,fill=xdxdff,fill opacity=0.1] (0,0) -- (-3,2.5) -- (-1.5,2.5) -- (0,2) -- (1.5,2.5) -- (3,2.5) -- cycle;
\fill[line width=2pt,color=xdxdff,fill=xdxdff,fill opacity=0.1] (0,4) -- (-1,6) -- (2,6) -- cycle;
\fill[line width=2pt,color=xdxdff,fill=xdxdff,fill opacity=0.1] (8,0) -- (5,2.5) -- (6.5,2.5) -- (8,2) -- (9.5,2.5) -- (11,2.5) -- cycle;
\fill[line width=2pt,color=xdxdff,fill=xdxdff,fill opacity=0.1] (9.5,5.5) -- (8.5,7.5) -- (11.5,7.5) -- cycle;
\draw [line width=1.2pt,color=xdxdff] (0,0)-- (-3,2.5);
\draw [line width=1.2pt,color=xdxdff] (-3,2.5)-- (-1.5,2.5);
\draw [line width=1.2pt,color=xdxdff] (-1.5,2.5)-- (0,2);
\draw [line width=1.2pt,color=xdxdff] (0,2)-- (1.5,2.5);
\draw [line width=1.2pt,color=xdxdff] (1.5,2.5)-- (3,2.5);
\draw [line width=1.2pt,color=xdxdff] (3,2.5)-- (0,0);
\draw [line width=1.2pt] (0,4)-- (0,2);
\draw [line width=1.2pt,color=xdxdff] (0,4)-- (-1,6);
\draw [line width=1.2pt,color=xdxdff] (-1,6)-- (2,6);
\draw [line width=1.2pt,color=xdxdff] (2,6)-- (0,4);
\draw [line width=1.2pt,color=xdxdff] (8,0)-- (5,2.5);
\draw [line width=1.2pt,color=xdxdff] (5,2.5)-- (6.5,2.5);
\draw [line width=1.2pt,color=xdxdff] (6.5,2.5)-- (8,2);
\draw [line width=1.2pt,color=xdxdff] (8,2)-- (9.5,2.5);
\draw [line width=1.2pt,color=xdxdff] (9.5,2.5)-- (11,2.5);
\draw [line width=1.2pt,color=xdxdff] (11,2.5)-- (8,0);
\draw [line width=1.2pt] (8,2)-- (8,4);
\draw [line width=1.2pt,color=ffqqqq] (6.5,5.5)-- (8,4);
\draw [line width=1.2pt] (9.5,5.5)-- (8,4);
\draw [line width=1.2pt,color=xdxdff] (9.5,5.5)-- (8.5,7.5);
\draw [line width=1.2pt,color=xdxdff] (8.5,7.5)-- (11.5,7.5);
\draw [line width=1.2pt,color=xdxdff] (11.5,7.5)-- (9.5,5.5);
\draw [->,line width=0.5pt] (2.5,4) -- (5.5,4);
\begin{scriptsize}
\draw [fill=xdxdff] (0,2) circle (3.5pt);
\draw[color=xdxdff] (0.7,2.9) node {$\overleftarrow{u}$};
\draw [fill=xdxdff] (0,4) circle (3.5pt);
\draw[color=xdxdff] (-0.65,3.987913417566513) node {$u$};
\draw [fill=xdxdff] (8,2) circle (3.5pt);
\draw[color=xdxdff] (8.7,2.9) node {$\overleftarrow{u}$};
\draw [fill=xdxdff] (8,4) circle (3.5pt);
\draw[color=xdxdff] (7.4,3.9046173404547564) node {$u$};
\draw [fill=ffqqqq] (6.5,5.5) circle (3.5pt);
\draw[color=ffqqqq] (6.752875114950715,6.3) node {$W_{i+1}$};
\draw [fill=xdxdff] (9.5,5.5) circle (3.5pt);
\end{scriptsize}
\end{tikzpicture}
\end{subfigure}
\hfill
\begin{subfigure}{47mm}
\begin{tikzpicture}[line cap=round,line join=round,>=triangle 45,x=1cm,y=1cm,scale=0.3]
\clip(-3.2,-1.2) rectangle (12.3,8);
\fill[line width=2pt,color=xdxdff,fill=xdxdff,fill opacity=0.1] (0,0) -- (-3,2) -- (-1.5,2) -- (-0.5,1.5) -- (0,2) -- (0.5,1.5) -- (1.5,2) -- (3,2) -- cycle;
\fill[line width=2pt,color=xdxdff,fill=xdxdff,fill opacity=0.1] (-1.5,3.5) -- (-3,5) -- (-1,5) -- cycle;
\fill[line width=2pt,color=xdxdff,fill=xdxdff,fill opacity=0.1] (-0.5,4.5) -- (0,3.5) -- (0.5,4.5) -- cycle;
\fill[line width=2pt,color=xdxdff,fill=xdxdff,fill opacity=0.1] (1,5.5) -- (1.5,3.5) -- (2.5,5.5) -- cycle;
\fill[line width=2pt,color=xdxdff,fill=xdxdff,fill opacity=0.1] (8,0) -- (5,2) -- (6.5,2) -- (7.5,1.5) -- (8,2) -- (8.5,1.5) -- (9.5,2) -- (11,2) -- cycle;
\fill[line width=2pt,color=xdxdff,fill=xdxdff,fill opacity=0.1] (4.5,5) -- (6,3.5) -- (6.5,5) -- cycle;
\fill[line width=2pt,color=xdxdff,fill=xdxdff,fill opacity=0.1] (7,4.5) -- (7.5,3.5) -- (8,4.5) -- cycle;
\fill[line width=2pt,color=xdxdff,fill=xdxdff,fill opacity=0.1] (11,3.5) -- (10.5,5.5) -- (12,5.5) -- cycle;
\draw [line width=1.2pt] (-1.5,3.5)-- (0,2);
\draw [line width=1.2pt] (0,3.5)-- (0,2);
\draw [line width=1.2pt] (1.5,3.5)-- (0,2);
\draw [line width=1.2pt,color=xdxdff] (0,0)-- (-3,2);
\draw [line width=1.2pt,color=xdxdff] (-3,2)-- (-1.5,2);
\draw [line width=1.2pt,color=xdxdff] (-1.5,2)-- (-0.5,1.5);
\draw [line width=1.2pt,color=xdxdff] (-0.5,1.5)-- (0,2);
\draw [line width=1.2pt,color=xdxdff] (0,2)-- (0.5,1.5);
\draw [line width=1.2pt,color=xdxdff] (0.5,1.5)-- (1.5,2);
\draw [line width=1.2pt,color=xdxdff] (1.5,2)-- (3,2);
\draw [line width=1.2pt,color=xdxdff] (3,2)-- (0,0);
\draw [line width=1.2pt,color=xdxdff] (-1.5,3.5)-- (-3,5);
\draw [line width=1.2pt,color=xdxdff] (-3,5)-- (-1,5);
\draw [line width=1.2pt,color=xdxdff] (-1,5)-- (-1.5,3.5);
\draw [line width=1.2pt,color=xdxdff] (-0.5,4.5)-- (0,3.5);
\draw [line width=1.2pt,color=xdxdff] (0,3.5)-- (0.5,4.5);
\draw [line width=1.2pt,color=xdxdff] (0.5,4.5)-- (-0.5,4.5);
\draw [line width=1.2pt,color=xdxdff] (1,5.5)-- (1.5,3.5);
\draw [line width=1.2pt,color=xdxdff] (1.5,3.5)-- (2.5,5.5);
\draw [line width=1.2pt,color=xdxdff] (2.5,5.5)-- (1,5.5);
\draw [line width=1.2pt,color=xdxdff] (8,0)-- (5,2);
\draw [line width=1.2pt,color=xdxdff] (5,2)-- (6.5,2);
\draw [line width=1.2pt,color=xdxdff] (6.5,2)-- (7.5,1.5);
\draw [line width=1.2pt,color=xdxdff] (7.5,1.5)-- (8,2);
\draw [line width=1.2pt,color=xdxdff] (8,2)-- (8.5,1.5);
\draw [line width=1.2pt,color=xdxdff] (8.5,1.5)-- (9.5,2);
\draw [line width=1.2pt,color=xdxdff] (9.5,2)-- (11,2);
\draw [line width=1.2pt,color=xdxdff] (11,2)-- (8,0);
\draw [line width=1.2pt] (8,2)-- (6,3.5);
\draw [line width=1.2pt] (8,2)-- (7.5,3.5);
\draw [line width=1.2pt,color=ffqqqq] (8,2)-- (8.5,3.5);
\draw [line width=1.2pt] (8,2)-- (11,3.5);
\draw [line width=1.2pt,color=xdxdff] (4.5,5)-- (6,3.5);
\draw [line width=1.2pt,color=xdxdff] (6,3.5)-- (6.5,5);
\draw [line width=1.2pt,color=xdxdff] (6.5,5)-- (4.5,5);
\draw [line width=1.2pt,color=xdxdff] (7,4.5)-- (7.5,3.5);
\draw [line width=1.2pt,color=xdxdff] (7.5,3.5)-- (8,4.5);
\draw [line width=1.2pt,color=xdxdff] (8,4.5)-- (7,4.5);
\draw [line width=1.2pt,color=xdxdff] (11,3.5)-- (10.5,5.5);
\draw [line width=1.2pt,color=xdxdff] (10.5,5.5)-- (12,5.5);
\draw [line width=1.2pt,color=xdxdff] (12,5.5)-- (11,3.5);
\draw [->,line width=0.5pt] (2.5,2.8) -- (5.5,2.8);
\begin{scriptsize}
\draw [fill=xdxdff] (0,2) circle (3.5pt);
\draw[color=xdxdff] (0.7,2.0253602310989) node {$u$};
\draw [fill=xdxdff] (0,3.5) circle (3.5pt);
\draw [fill=xdxdff] (-1.5,3.5) circle (3.5pt);
\draw [fill=xdxdff] (1.5,3.5) circle (3.5pt);
\draw [fill=xdxdff] (8,2) circle (3.5pt);
\draw[color=xdxdff] (8.7,1.994851541190701) node {$u$};
\draw [fill=xdxdff] (6,3.5) circle (3.5pt);
\draw [fill=xdxdff] (11,3.5) circle (3.5pt);
\draw [fill=xdxdff] (7.5,3.5) circle (3.5pt);
\draw [fill=ffqqqq] (8.5,3.5) circle (3.5pt);
\draw[color=ffqqqq] (9.4,4.2) node {$W_{i+1}$};
\end{scriptsize}
\end{tikzpicture}
\end{subfigure}
\caption{Illustration of the recursive construction of $\ctau_{i+1}$ from $\ctau_i$. \emph{Left :} The case where the conventional parental edge of the root was chosen. \emph{Middle :} The case where a genuine edge was chosen. \emph{Right :} The case where a vertex was chosen.}
\label{ex_marchal}
\end{center}
\end{figure}

\begin{proposition}
\label{marchal_prop}
Let $\ftau$ be a $\GWw_\alpha$-weighted tree as in Definition~\ref{GWwdef}. Let $(\fctau_i)_{i\geq 1}$ be the random sequence constructed by Algorithm~\ref{marchal_alg}. For all $i\geq 1$, the following holds.
\begin{longlist}
\item[(i)] $\#\partial\ctau_i=i$ and $\#\ctau_i+1\leq\#\ctau_{i+1}$.
\item[(ii)] $\tHS(\fctau_i)\leq \tHS(\fctau_{i+1})$.
\item[(iii)] The law of $\fctau_i$ under $\P$ is the same as the law of $\ftau$ under $\P(\, \cdot\, |\, \#\partial\tau=i)$.
\end{longlist}
\end{proposition}

\begin{proof}
The point $(i)$ is clear. In both cases $(i)$ and $(ii)$ of Algorithm~\ref{marchal_alg}, we have $\tHS(\theta_u\fctau_i)\leq\tHS(\theta_u\fctau_{i+1})$ by Definition~\ref{def_tHS_intro_full}. It is then easy to obtain $(ii)$ by a backward induction on the height of the ancestors of $u$. For all $1\leq j\leq i$, let us denote by $V_j$ the unique leaf of $\ctau_i$ of weight $W_j$.
We observe that all the orders of siblings are equiprobable conditionally given the graph structure of $\ctau_i$ (which is the same as that of the labeled but unrooted trees of Marchal~\cite{marchal}). Therefore, we get from Marchal~\cite[Section 2.3]{marchal} and (\ref{stable_offspring_explicit}) that if $t$ is a tree with $i$ leaves enumerated in some order as $v_1,\ldots,v_i$, then $\P(\ctau_i=t;V_1=v_1;\ldots ;V_i=v_i)$ is proportional to $\prod_{u\in t}\mu_\alpha(k_u(t))$. By (\ref{lawGW}) and because there are always $i!$ possible enumerations of the leaves of $t$, the law of $\ctau_i$ under $\P$ is the law of $\tau$ under $\P(\,\cdot\, |\, \#\partial\tau=i)$. The $(W_i)_{i\geq 1}$ are i.i.d~and independent from $(\ctau_i)_{i\geq 1}$, so $(iii)$ follows.
\end{proof}

Finally, recall from Section~\ref{more_stable} that an $\alpha$-stable tree is a random rooted compact real tree distributed as $\cT_\rH$ under $\bN_\alpha(\dd \rH\, |\, \zeta=1)$. Recall from (\ref{graph_distance}) the graph distance $\mathtt{d}_{\mathrm{gr}}$.

\begin{theorem}
\label{marchal_thm}
Let $(\fctau_i)_{i\geq 1}$ be the sequence of random weighted trees constructed by Algorithm~\ref{marchal_alg}. Then, there exists a random $\alpha$-stable tree $(\sT_{\mathrm{nr}},d_{\mathrm{nr}},\rho_{\mathrm{nr}})$ such that
\[(\ctau_i,i^{-1+1/\alpha}\mathtt{d}_{\mathrm{gr}},\varnothing)\longrightarrow (\sT_{\mathrm{nr}},\alpha\, d_{\mathrm{nr}},\rho_{\mathrm{nr}})\]
holds almost surely for the rooted Gromov--Hausdorff distance given by (\ref{n-GH}).
\end{theorem}

\begin{proof}
See Curien \& Haas~\cite[Theorem 5]{nested}.
\end{proof}

\subsection{Limit theorems for stable Galton--Watson trees}

In this section, we gather and apply already known estimates and limit theorems about the asymptotic behavior of $\GW(\mu)$-trees (see Definition~\ref{GWdef}) to the case of stable Galton--Watson trees. Most of these results are stated under the following assumptions on the offspring distribution $\mu$.
\begin{itemize}
\item[(a)] $\mu$ is \emph{critical and non-trivial}, meaning that (\ref{critical}) holds.
\item[(b)] $\mu$ is \emph{in the domain of attraction of a stable law with index $\theta\in(1,2]$}.
\item[(c)] $\mu$ is \emph{aperiodic}, meaning that $\mu$ is not supported by a proper additive subgroup of $\bbZ$.
\end{itemize}
We saw with (\ref{stable_offspring}) that the stable offspring distribution $\mu_\alpha$ satisfies $(a)$. Moreover, it verifies $(b)$ with $\theta=\alpha$ by (\ref{domain_attraction}). This can also be proved directly from (\ref{stable_offspring}) or (\ref{stable_offspring_explicit}): see Bingham, Goldies \& Teugels~\cite[Chapter 8.3]{RegVar} for details. By (\ref{stable_offspring_explicit}), $\mu_\alpha$ satisfies $(c)$ if and only if $\alpha\in(1,2)$. Nonetheless, condition $(c)$ is only required to avoid technical complications, as it conveniently ensures that $\P(\#\tau=n)>0$ when $n$ is large enough for a $\GW(\mu)$-tree $\tau$, but most results we consider here can be extended to the periodic case. Anyway, $\mu_2$ is explicitly the critical binary law and can be studied with simpler arguments.

Recall from (\ref{height_notation}) that we write $\partial t$ for the set of leaves of tree $t$. Motivated by Marchal's algorithm (Section~\ref{marchal_section}), we next gather some results due to Kortchemski~\cite{kortchemski_leaves} about $\GW_\alpha$-trees conditioned on their numbers of leaves.

\begin{proposition}
\label{size_vs_leaves}
Let $\tau$ be a $\GW_\alpha$-tree.
\begin{longlist}
\item[(i)] If $\alpha=2$, then $\#\tau=2\#\partial\tau-1$ almost surely.
\item[(ii)] If $\alpha\in(1,2)$, then the two following convergences hold in probability:
\begin{align}
\label{size_under_leaves}
\tfrac{1}{n}\, \#\tau\;\text{ under }\P(\, \cdot\, |\, \#\partial\tau=n)&\,\longrightarrow \alpha,\\
\label{leaves_under_size}
\tfrac{1}{n}\, \#\partial\tau\;\text{ under }\P(\, \cdot\, |\, \#\tau=n)&\,\longrightarrow \tfrac{1}{\alpha}.
\end{align}
\end{longlist}
\end{proposition}

\begin{proof}
If $\alpha\!=\!2$, $k_u(\tau)\!=\!2$ for all $u\!\in\!\tau\backslash\partial\tau$ by (\ref{stable_offspring_explicit}) and (\ref{lawGW}). Each vertex of $\tau\backslash\{\varnothing\}$ has a single parent so $\#\tau-1\!=\!\sum_{u\in\tau\backslash\partial\tau}k_u(\tau)\!=2\#\tau-2\#\partial\tau$. For (\ref{size_under_leaves}), see Kortchemski~\cite[Corollary 3.3]{kortchemski_leaves} and recall $\mu_\alpha(0)\!=\!\tfrac{1}{\alpha}$ from (\ref{stable_offspring_explicit}). For (\ref{leaves_under_size}), see \cite[Lemma 2.5]{kortchemski_leaves}.
\end{proof}

\begin{lemma}
\label{size_vs_leaves_lemma}
We assume that $\alpha\! \in\!(1,2)$. Let $\tau$ be a $\GW_\alpha$-tree. Let $(U_n)_{n\geq 1}$ be a sequence of nonnegative and uniformly bounded functions on the space $\bbT$ of trees. If $\E[U_n(\tau)\, |\, \#\tau\geq n]$ converges, then $\E[U_n(\tau)\, |\, \#\partial\tau\geq n/\alpha-n^{3/4}]$ converges to the same limit.
\end{lemma}

\begin{proof}
See Kortchemski~\cite[Proposition 4.4]{kortchemski_leaves} and recall that $\mu_\alpha(0)=\tfrac{1}{\alpha}$ from (\ref{stable_offspring_explicit}).
\end{proof}

Next, we gather tail estimates on the height, the size, and the number of leaves of $\tau$. 
\begin{proposition}
Let $\tau$ be a $\GW_\alpha$-tree. Recall from (\ref{alpha_gamma_delta}) that $\beta=\tfrac{1}{\alpha-1}$ and $\gamma=\ln\tfrac{\alpha}{\alpha-1}$.
\begin{longlist}
\item[(i)] Recall from (\ref{height_notation}) that $|\tau|$ stands for the height of $\tau$. It holds that
\begin{align}
\label{equiv_hauteur}
n^{\beta}\P(|\tau|\geq n)&\longrightarrow e^{\gamma\beta},\\
\label{equiv_taille}
n^{1/\alpha}\P(\#\tau\geq n)&\longrightarrow \frac{\alpha^{1/\alpha}}{\Gamma(1-1/\alpha)},\\
\label{equiv_leaves}
n^{1/\alpha}\P(\#\partial\tau\geq n)&\longrightarrow \frac{1}{\Gamma(1-1/\alpha)}.
\end{align}
\item[(ii)] If $\alpha=2$ then it holds that
\begin{equation}
\label{tails_leaves_binary}
n^{1+1/\alpha}\P(\#\tau=2n-1)=n^{1+1/\alpha}\P(\#\partial\tau=n)\longrightarrow\frac{1}{\alpha\Gamma(1-1/\alpha)}.
\end{equation}
\item[(iii)] If $\alpha\in(1,2)$ then it holds that
\begin{equation}
\label{tails_leaves}
 n^{1+1/\alpha}\P(\#\tau\!=\!n)\!\longrightarrow\!\frac{\alpha^{1/\alpha}}{\alpha\Gamma(1-1/\alpha)},  \:\text{ and }\: n^{1+1/\alpha}\P(\#\partial\tau\!=\!n)\!\longrightarrow\!\frac{1}{\alpha\Gamma(1-1/\alpha)}.
\end{equation}
\end{longlist}
\end{proposition}

\begin{proof}
For (\ref{equiv_hauteur}), recall from (\ref{stable_offspring}) that $\sum_{k\geq 0}s^k\mu_\alpha(k)=s+\tfrac{1}{\alpha}(1-s)^\alpha$ for $s\in[0,1]$ and see Slack~\cite[Lemma 2]{slack68}. For (\ref{equiv_taille}), we follow the proof of Kovchegov, Xu \& Zaliapin~\cite[Proposition 3]{kovchegov23}. For all $\lambda>0$, set $f(\lambda)=\E[\exp(-\lambda\#\tau)]$ and get from the branching property of $\GW_\alpha$-trees (see Definition~\ref{GWdef}) that $e^{\lambda}f(\lambda)=f(\lambda)+\tfrac{1}{\alpha}\big(1-f(\lambda)\big)^\alpha$. As $\tau$ is finite,
\begin{equation}
\label{laplace_size}
f(0+)=1\quad\text{ and }\quad\lambda^{-1/\alpha}(1-f(\lambda))\xrightarrow[\lambda\rightarrow 0^+]{} \alpha^{1/\alpha}.
\end{equation}
Moreover, Fubini's theorem allows us to write
\[1-f(\lambda)=\E[1-\exp(-\lambda\#\tau)]=\E\Big[\lambda\int_0^{\#\tau}e^{-\lambda x}\,\dd x\Big]=\lambda \int_0^\infty e^{-\lambda x}\P(\#\tau\geq x)\, \dd x.\]
Since $x\mapsto\P(\#\tau\geq x)$ is monotone and nonnegative, Karamata's Tauberian theorem for Laplace transforms (see Feller~\cite[Chapter XII.5, Theorem 4]{feller1971}) and the convergence in (\ref{laplace_size}) together yield (\ref{equiv_taille}). The last convergence (\ref{equiv_leaves}) of $(i)$ will directly follow from $(ii)$ and $(iii)$.

Let us prove $(ii)$. A famous combinatorial result (that can be easily proved via generating functions) asserts that there are exactly $C_{n-1}=\frac{1}{n}\binom{2n-2}{n-1}$ binary trees with $n$ leaves. By (\ref{stable_offspring_explicit}), (\ref{lawGW}), and Proposition~\ref{size_vs_leaves} $(i)$, we then have $\P(\#\partial\tau\! =\! n)=\P(\#\tau\! =\! 2n-1)=2^{1-2n}C_{n-1}$. Stirling's formula yields (\ref{tails_leaves_binary}) because $\Gamma(\tfrac{1}{2})=\sqrt{\pi}$.

Let us prove $(iii)$. Recall from (\ref{stable_offspring_explicit}) that $\mu_\alpha(0)=\tfrac{1}{\alpha}$ and $\mu_\alpha(k)>0$ for all $k\geq 2$. By (\ref{domain_attraction}), Kortchemski's results~\cite[Lemma 1.11 and Theorem 3.1]{kortchemski_leaves} yield that there is a constant $c_\alpha\in(0,\infty)$ such that $n^{1+1/\alpha}\P(\#\tau\!=\!n)\!\longrightarrow\! c_\alpha$ and $n^{1+1/\alpha}\P(\#\partial\tau\!=\!n)\!\longrightarrow\! c_\alpha \alpha^{-1/\alpha}$. Finally, comparing this with (\ref{equiv_taille}) entails that $c_\alpha\Gamma(1-\tfrac{1}{\alpha})=\alpha^{1/\alpha-1}$, and so (\ref{tails_leaves}) follows.
\end{proof}

Now, we state limit theorems for the height function $H(\tau)$, as in (\ref{height_process_def}), of a $\GW_\alpha$-tree $\tau$ conditioned to be large. Recall from Section~\ref{topo_function} the space $\mathcal{C}_{\mathrm{K}}$ of continuous functions with compact support and endowed with lifetimes. From Section~\ref{more_stable}, recall the excursion measure $\bN_\alpha$, the $\alpha$-stable height process $\rH$, its lifetime $\zeta$, and the law $\bN_\alpha(\dd \rH\, |\, \zeta\!=\!1)$ of its normalized excursion. The work of Duquesne \& Le Gall~\cite{levytree_DLG} entails the following.

\begin{theorem}
\label{duquesne_scaling_atleast}
Recall that $\beta=\tfrac{1}{\alpha-1}$ from (\ref{alpha_gamma_delta}) and that $a_n=\alpha^{-1/\alpha}n^{1/\alpha}$ from (\ref{domain_attraction}). Let $\tau$ be a $\GW_\alpha$-tree. For $\ell\in(0,\infty)$, the following convergences hold in distribution on $\mathcal{C}_{\mathrm{K}}$:
\begin{align}
\label{cv_cond_high}
\big(\tfrac{1}{n}H_{n^{\alpha\beta}s}(\tau)\big)_{s\geq 0}\;\text{ under }\P\big(\ \cdot\ \big|\ |\tau|\geq \ell n\big)&\,\convd \alpha^{1/\alpha}\rH\;\text{ under }\bN_\alpha(\, \cdot\, |\, \sup \rH>\ell),\\
\label{cvd_path_sachant_an}
\big(\tfrac{a_n}{n}H_{ns}(\tau)\big)_{s\geq 0}\;\text{ under }\ \P(\ \cdot\ |\ \#\tau\geq \ell n)&\,\convd \rH\;\text{ under }\bN_\alpha(\, \cdot\, |\, \zeta>\ell).
\end{align}
\end{theorem}

\begin{proof}
By the central limit theorem (\ref{domain_attraction}), Duquesne \& Le Gall~\cite[Theorem 2.3.2 and Proposition 2.5.2]{levytree_DLG} assert that (\ref{cv_cond_high}) holds for the Skorokhod topology on $\mathbb{D}(\bbR_+,\bbR)$. In fact, their proof of \cite[Proposition 2.5.2]{levytree_DLG} also contains the joint convergence of the lifetimes (see the last paragraph, page 66), so (\ref{cv_cond_high}) also holds on $\mathcal{C}_{\mathrm{K}}$. The same methodology yields (\ref{cvd_path_sachant_an}): see the concluding remark of Duquesne \& Le Gall~\cite[Section 2.6]{levytree_DLG}. 
\end{proof}

As explained by Duquesne \& Le Gall~\cite{levytree_DLG}, their method cannot be directly applied to find scaling limits of Galton--Watson trees under degenerate conditionings such as $\{\#\tau=n+1\}$. Nevertheless, Duquesne~\cite{duquesne_contour_stable} addressed this by another method (see also Kortchemski~\cite{kortchemski_simple}).

\begin{theorem}
\label{duquesne_scaling_limit_size=n}
Recall that $a_n=\alpha^{-1/\alpha} n^{1/\alpha}$ from (\ref{domain_attraction}) and let $\tau$ be a $\GW_\alpha$-tree. The following convergence holds in distribution on $\mathcal{C}_{\mathrm{K}}$:
\[\big(\tfrac{a_n}{n}H_{ns}(\tau)\big)_{s\geq 0}\;\text{ under }\P(\, \cdot\, |\, \#\tau= n+1)\,\xrightarrow[n\to\infty,n\in \lfloor\alpha\rfloor\bbN]{d}\rH \;\text{ under }\bN_\alpha(\dd \rH\, |\, \zeta=1).\]
\end{theorem}

\begin{proof}
By (\ref{domain_attraction}), this is a consequence of Duquesne~\cite[Theorem 3.1]{duquesne_contour_stable}. This theorem is stated under the assumption that the offspring distribution is aperiodic, but as argued by Kortchemski~\cite{kortchemski_simple}, it can be easily extended to the periodic case $\alpha=2$ as above.
\end{proof}

\noi
Note that Proposition~\ref{height_contour_tree} then entails the convergence (\ref{scaling_limit_Catalan}) presented in the introduction.

\subsection{An estimate for the height function of a Galton--Watson tree with edge lengths}

Recall respectively from (\ref{height_process_def}) and (\ref{height_process_length_def}) the height function $H(t)$ of a tree $t$ and the height function $\Hl(T)$ of a tree with edge lengths $T$. The goal of this section is to compare them in the case of a $\GWl(\mu)$-tree with edge lengths (see Definition~\ref{GWldef}) by proving the following.

\begin{proposition}
\label{height_process_skeleton}
Recall the Skorokhod distance $\mathtt{d}_{\mathrm{S}}$ from (\ref{skorokhod_distance}). Let $\mu$ be a critical and non-trivial offspring distribution (i.e.~(\ref{critical}) holds), and let $\cT=(\tau,(L_u)_{u\in\tau})$ be a $\GWl(\mu)$-tree with edge lengths. Let $\lambda>0$ and define two random càdlàg functions with compact support $X$ and $Y$ by setting $X_s=\lambda^{1-1/\alpha}\Hl_{s/\lambda}(\cT)$ and $Y_s=\lambda^{1-1/\alpha}H_{s/\lambda}(\tau)$ for all $s\in\bbR_+$. Then, there are two constants $C,c\in(0,\infty)$ that only depend on $\mu$ such that for all $n\geq 1$,
\[\P\big(\#\tau\leq n\, ;\, |\tau|\leq n^{1-1/\alpha}\, ;\, \mathtt{d}_{\mathrm{S}}(X,Y)\geq C(\lambda\sqrt{n})^{1-1/\alpha}\ln n+C\lambda\sqrt{n}\ln n\big)\, \leq\,  C n e^{-c(\ln n)^2}.\]
\end{proposition}
\noi
We prepare the proof of Proposition~\ref{height_process_skeleton} by presenting two standard lemmas.

\begin{lemma}[Chernoff bound]
\label{hoeffding}
Let $(L_i)_{i\geq 0}$ be i.i.d~exponential RVs with mean $1$. Then, there is an universal constant $c_{\mathrm{uni}}\in(0,\infty)$ such that for all $n\in\mathbb{N}^*$ and $x\geq 0$, it holds that
\[\P\bigg(\Big|n-\sum_{i=0}^{n-1} L_i\Big|\geq x\bigg)\leq 2\exp\Big(-c_{\mathsf{uni}}\, x\min\big(1,\tfrac{x}{n}\big)\Big).\]
\end{lemma}

\begin{proof}
For any $\lambda>-1$, we have $\E[e^{-\lambda L_0}]=\tfrac{1}{1+\lambda}$. An elementary inequality asserts there is $\eta\in(0,1)$ such that $\ln(1+x)\geq x-3x^2/4$ for all $x\in[-\eta,\eta]$. The desired result follows from the Chernoff bound $\P(|Z|\geq x)\leq e^{-\lambda x}\E[e^{\lambda Z}+e^{-\lambda Z}]$ with $\lambda=\eta\min(1,\frac{x}{n})$.
\end{proof}

\begin{lemma}
\label{jump_height_lemma}
Let $\mu$ be a critical and non-trivial probability measure on $\bbN$, and let $\tau$ be a $\GW(\mu)$-tree. Let $(K,J)$ have law given by $\P(K\!=\!k \, ;\,  \!J=\!j)\!=\!\I{1\leq j\leq k}\mu(k)$ for all $k,j\!\in\!\bbN$.
\begin{longlist}
\item[(i)] For all $n\in\bbN$ and for all bounded functions $g_0,\ldots,g_{n-1}:\bbN^2\longrightarrow\bbR$, it holds that
\[\E\bigg[\sum_{u\in \tau}\I{|u|=n}\!\! \prod_{\substack{v\in\tau,\, j\geq 1\\ v*(j)\preceq u}}\! g_{|v|}\big(k_v(\tau),j\big)\Bigg]=\prod_{i=0}^{n-1}\E\big[g_i(K,J)\big].\]
\item[(ii)] For all $n,m\geq 2$, $\P\big(\#\tau\leq n\, ;\, \sup_{i\geq0}|H_i(\tau)\! -\! H_{i+1}(\tau)|\geq m\big)\, \leq \, n(1\! -\! \mu(0))^m$.
\end{longlist}
\end{lemma}

\begin{proof}
The point $(i)$ is a simplified version of the so-called \textit{Many-To-One Principle}, which is part of folklore (see e.g.~Duquesne~\cite[Equation (24)]{duquesne09} for a general statement and a proof). To prove $(ii)$, let us first consider two vertices $u_1<u_2$ of $\tau$ that are consecutive in lexicographic order on $\tau$. Since $\overleftarrow{u_2}<u_2$, we have $\overleftarrow{u_2}\leq u_1<u_2$ and so $\overleftarrow{u_2}\preceq u_1$. Thus, it holds $|u_1|-|u_2|\geq -1>-m$. Moreover, for $v\in\tau$ and $j\in\mathbb{N}^*$, if $\overleftarrow{u_2}\prec v$ and $v*(j)\preceq u_1$ then $u_1<v*(j+1)<u_2$ and so $v*(j+1)\notin\tau$ and $j=k_v(\tau)$. Hence, we obtain
\[\P\Big(\#\tau\leq n\, ;\, \sup_{i\geq0}|H_i(\tau)-H_{i+1}(\tau)|\geq m\Big)\leq \E\bigg[\sum_{u\in\tau}\I{m\leq |u|<n}\prod_{\substack{v*(j)\preceq u\\ |v|\geq|u|-m}}\I{k_v(\tau)=j}\bigg].\]
We easily compute $\P(K=J)=\sum_{k\geq 1}\mu(k)=1-\mu(0)$, then we apply $(i)$ to get
\[\P\Big(\#\tau\!\leq\! n\, ;\, \sup_{i\geq0}|H_i(\tau)-H_{i+1}(\tau)|\!\geq\! m\Big) \leq \sum_{h=m}^{n-1}\prod_{i=h-m}^{h-1}\!\!\P(K=J)=(n\!-\!m)(1\!-\!\mu(0))^m.\qedhere\]
\end{proof}

\begin{proof}[Proof of Proposition~\ref{height_process_skeleton}]
Recall from (\ref{depth-first_exploration}) that $(u_0(\tau),\ldots,u_{\#\tau-1}(\tau))$ stand for the vertices of $\tau$ listed in lexicographic order. Fix an increasing and bijective function $\psi:\bbR_+\to\bbR_+$ such that $\psi(0)=0$, $\psi(\lambda i+\lambda)=\psi(\lambda i)+\lambda L_{u_i(\tau)}$ for all $0\!\leq\! i\!\leq\! \#\tau-1$, and $\psi(\lambda \#\tau+s)=\psi(\lambda \#\tau)+s$ for all $s\!\geq\! 0$. For $\lambda i\!\leq\! s\!<\!\lambda(i+1)$, note that
\begin{align*}
|X_{\psi(s)}-Y_s|&\leq |Y_{\lambda i}-Y_{\lambda(i+1)}|+|Y_{\lambda i}-X_{\psi(\lambda i)}|+|X_{\psi(\lambda i)}-X_{\psi(\lambda i+\lambda)}|,\\
|\psi(s)-s|&\leq 2\lambda+2|\lambda i-\psi(\lambda i)|+|\lambda i+\lambda-\psi(\lambda i+\lambda)|.
\end{align*}
Also noting from (\ref{lifetime_height}) and (\ref{lifetime_height_length}) that $\zeta(Y)=\I{\# \tau\geq 2}\lambda\#\tau$  and $\zeta(X)=\lambda \sum_{u\in \tau}L_u$, it is then not hard to check that
\begin{multline}
\label{skeleton_tool1}
\mathtt{d}_{\mathrm{S}}(X,Y)\leq 3\lambda+4\lambda\max_{1\leq i\leq\#\tau}\Big|i-\sum_{j=0}^{i-1} L_{u_j(\tau)}\Big|+\lambda^{1-1/\alpha}\max_{u\in\tau}\Big||u|-\sum_{v\prec u}L_v\Big|\\
+\lambda^{1-1/\alpha}\max_{u\in\tau}L_u+\lambda^{1-1/\alpha}\sup_{i\geq 0}|H_i(\tau)-H_{i+1}(\tau)|.
\end{multline}
We want to bound in probability each one of the terms of the right-hand side under the realization of the event $\{\#\tau\leq n\, ;\, |\tau|\leq n^{1-1/\alpha}\}$. Let $(L_i)_{i\geq 0}$ be a sequence of independent exponential random variables with mean $1$. We begin with the simple union bound
\begin{equation}
\label{skeleton_tool2}
\P\big(\#\tau\leq n\, ;\, \lambda^{1-1/\alpha}\max_{u\in\tau}L_u\geq \lambda^{1-1/\alpha}(\ln n)^2\big)\, \leq\,  n\P\big(L_0\geq (\ln n)^2\big)\, =\, n e^{-(\ln n)^2}.
\end{equation}
Moreover, we use another union bound and Lemma~\ref{hoeffding} to find
\begin{multline}
\label{skeleton_tool3}
\P\Big(\#\tau\leq n\, ;\, |\tau|\leq n^{1-1/\alpha}\, ;\, \lambda^{1-1/\alpha}\max_{u\in\tau}\Big||u|-\sum_{v\prec u}L_v\Big|\geq (\lambda \sqrt{n})^{1-1/\alpha}\ln n\Big)\\
\leq n\max_{0\leq h\leq n^{1-1/\alpha}}\P\bigg(\Big|h-\sum_{j=0}^{h-1} L_j\Big|\geq (\sqrt{n})^{1-1/\alpha}\ln n\bigg)\leq 2n e^{-c_{\mathrm{uni}}(\ln n)^2}
\end{multline} 
for $n$ large enough. The exact same method entails
\begin{equation}
\label{skeleton_tool4}
\P\bigg(\#\tau\leq n\, ;\, \lambda\max_{1\leq i\leq\#\tau}\Big|i-\sum_{j=0}^{i-1} L_{u_j(\tau)}\Big|\geq \lambda\sqrt{n}\ln n\bigg)\, \leq\,  2n e^{-c_{\mathrm{uni}}(\ln n)^2}.
\end{equation}
By (\ref{skeleton_tool1}), (\ref{skeleton_tool2}), (\ref{skeleton_tool3}), (\ref{skeleton_tool4}), an application of Lemma~\ref{jump_height_lemma} $(ii)$ completes the proof.
\end{proof}

\section{The weighted Horton--Strahler number}
\label{weighted_HS}

\subsection{The classic Horton--Strahler number}
\label{preliminary_HS}

In this section, we present two alternative definitions of the (classic) Horton--Strahler number $\HS$, and we apply some general estimates of the companion paper \cite{companion_1} to the case of stable Galton--Watson trees. But first, let us spell a simple observation out. By Definition~\ref{def_HS_intro_full}, for any tree $t$, it holds that
\begin{equation}
\label{nochain}
\textit{if }k_\varnothing(t)\neq 1,\quad\textit{ then }\quad \HS(t)=0\Leftrightarrow t=\{\varnothing\}.
\end{equation}
Indeed, if $k_\varnothing(t)\geq 2$ then $\HS(t)\geq 1$ because the Horton--Strahler numbers are nonnegative.

For $n\in\bbN$, denote by $\bbW_n=\bigcup_{k=0}^n\{1,2\}^k$ the $n$-perfect binary tree, where $\{1,2\}^0=\{\varnothing\}$. Recalling Definition~\ref{embedding_discrete} of embeddings, the Horton--Strahler number of a tree $t$ is given by
\begin{equation}
\label{HS_maximal_height}
\HS(t)=\max\big\{n\in\bbN\ :\ \exists\psi:\bbW_n\longrightarrow t \text{ embedding}\big\}.
\end{equation}
This result seems to be `part of the folklore', but see \cite[Equation (36)]{companion_1} for a brief proof.

Our second alternative definition of $\HS$ involves the action of removing the subtrees with null Horton--Strahler number of a tree $t$. After such reduction, we might obtain some chains of edges put back-to-back without any branching. We wish to see such a chain as a single longer edge. Thus, we work with trees with edge lengths as defined in Section~\ref{words}. The next definition is just a formal rephrasing within our framework of the \textit{Horton pruning} studied by Kovchegov \& Zaliapin, see \cite[Definition 3 and Figure 7]{hortonlaws}. See also Figure~\ref{ex_pruning_weightless}.

\begin{figure}
\begin{center}
\begin{subfigure}{78mm}
\begin{tikzpicture}[line cap=round,line join=round,>=triangle 45,x=0.66cm,y=0.5cm,scale=0.6]
\clip(-11.5,-2) rectangle (8.2,16.5);
\draw [line width=1.2pt,dash pattern=on 3pt off 3pt,color=ffxfqq] (0,4)-- (0,0);
\draw [line width=1.2pt] (5,0)-- (5,1);
\draw [line width=1.2pt] (-5,3)-- (-5,0);
\draw [line width=1.2pt] (5,1)-- (5,3);
\draw [line width=1.2pt,dash pattern=on 3pt off 3pt,color=ffxfqq] (2,1)-- (2,2.5);
\draw [line width=1.2pt,dash pattern=on 3pt off 3pt,color=ffxfqq] (2,2.5)-- (2,5);
\draw [line width=1.2pt,dash pattern=on 3pt off 3pt,color=ffxfqq] (8,6)-- (8,1);
\draw [line width=1.2pt,dash pattern=on 3pt off 3pt,color=ffxfqq] (3.5,7.5)-- (3.5,3);
\draw [line width=1.2pt,dash pattern=on 3pt off 3pt,color=ffxfqq] (6.5,5)-- (6.5,3);
\draw [line width=1.2pt] (-5,0)-- (0,0);
\draw [line width=1.2pt] (0,0)-- (5,0);
\draw [line width=1.2pt,dash pattern=on 3pt off 3pt,color=ffxfqq] (5,1)-- (8,1);
\draw [line width=1.2pt,dash pattern=on 3pt off 3pt,color=ffxfqq] (5,3)-- (6.5,3);
\draw [line width=1.2pt,dash pattern=on 3pt off 3pt,color=ffxfqq] (5,3)-- (3.5,3);
\draw [line width=1.2pt,dash pattern=on 3pt off 3pt,color=ffxfqq] (2,1)-- (5,1);
\draw [line width=1.2pt] (-8,3)-- (-5,3);
\draw [line width=1.2pt] (-5,3)-- (-2,3);
\draw [line width=1.2pt] (-2,3)-- (-2,6);
\draw [line width=1.2pt] (-8,3)-- (-8,5);
\draw [line width=1.2pt] (-2,6)-- (-2,7);
\draw [line width=1.2pt] (-6,5)-- (-6,7);
\draw [line width=1.2pt,dash pattern=on 3pt off 3pt,color=ffxfqq] (-8,5)-- (-8,7);
\draw [line width=1.2pt] (-10,5)-- (-8,5);
\draw [line width=1.2pt] (-8,5)-- (-6,5);
\draw [line width=1.2pt] (-3,7)-- (-2,7);
\draw [line width=1.2pt] (-2,7)-- (-1,7);
\draw [line width=1.2pt] (-10,9)-- (-10,5);
\draw [line width=1.2pt,dash pattern=on 3pt off 3pt,color=ffxfqq] (-10,9)-- (-11,9);
\draw [line width=1.2pt,dash pattern=on 3pt off 3pt,color=ffxfqq] (-10,9)-- (-9,9);
\draw [line width=1.2pt,dash pattern=on 3pt off 3pt,color=ffxfqq] (-9,9)-- (-9,10);
\draw [line width=1.2pt,dash pattern=on 3pt off 3pt,color=ffxfqq] (-11,9)-- (-11,14);
\draw [line width=1.2pt,dash pattern=on 3pt off 3pt,color=ffxfqq] (-6,7)-- (-5,7);
\draw [line width=1.2pt] (-6,7)-- (-7,7);
\draw [line width=1.2pt] (-7,7)-- (-7,8.5);
\draw [line width=1.2pt,dash pattern=on 3pt off 3pt,color=ffxfqq] (-5,7)-- (-5,9.5);
\draw [line width=1.2pt] (-3,7)-- (-3,9.5);
\draw [line width=1.2pt] (-1,7)-- (-1,11);
\draw [line width=1.2pt,dash pattern=on 3pt off 3pt,color=ffxfqq] (-7,8.5)-- (-6.5,8.5);
\draw [line width=1.2pt] (-7,8.5)-- (-7.5,8.5);
\draw [line width=1.2pt,dash pattern=on 3pt off 3pt,color=ffxfqq] (-6.5,11)-- (-6.5,8.5);
\draw [line width=1.2pt] (-7.5,12)-- (-7.5,8.5);
\draw [line width=1.2pt,dash pattern=on 3pt off 3pt,color=ffxfqq] (-7.5,12)-- (-7.8,12);
\draw [line width=1.2pt,dash pattern=on 3pt off 3pt,color=ffxfqq] (-7.5,12)-- (-7.2,12);
\draw [line width=1.2pt,dash pattern=on 3pt off 3pt,color=ffxfqq] (-7.8,12)-- (-7.8,15.5);
\draw [line width=1.2pt,dash pattern=on 3pt off 3pt,color=ffxfqq] (-7.2,12)-- (-7.2,14);
\draw [line width=1.2pt,dash pattern=on 3pt off 3pt,color=ffxfqq] (-1,11)-- (-0.5,11);
\draw [line width=1.2pt,dash pattern=on 3pt off 3pt,color=ffxfqq] (-1,11)-- (-1.5,11);
\draw [line width=1.2pt,dash pattern=on 3pt off 3pt,color=ffxfqq] (-3,9.5)-- (-2.5,9.5);
\draw [line width=1.2pt,dash pattern=on 3pt off 3pt,color=ffxfqq] (-3,9.5)-- (-3.5,9.5);
\draw [line width=1.2pt,dash pattern=on 3pt off 3pt,color=ffxfqq] (-1.5,11)-- (-1.5,12);
\draw [line width=1.2pt,dash pattern=on 3pt off 3pt,color=ffxfqq] (-1,11)-- (-1,15);
\draw [line width=1.2pt,dash pattern=on 3pt off 3pt,color=ffxfqq] (-0.5,11)-- (-0.5,13.5);
\draw [line width=1.2pt,dash pattern=on 3pt off 3pt,color=ffxfqq] (-2.5,12)-- (-2.5,9.5);
\draw [line width=1.2pt,dash pattern=on 3pt off 3pt,color=ffxfqq] (-3.5,11.5)-- (-3.5,9.5);
\draw [line width=1.2pt,dash pattern=on 3pt off 3pt,color=ffxfqq] (-2.5,14)-- (-2.5,12);
\draw [line width=1.2pt] (0,0)-- (0,-1.5);
\begin{scriptsize}
\draw [fill=xdxdff] (0,0) circle (3.5pt);
\draw[color=xdxdff] (0.8,-0.6) node {};
\draw [fill=ffxfqq] (0,4) circle (3.5pt);
\draw[color=ffxfqq] (0.1,4.8) node {};
\draw [color=black] (5,1)-- ++(-3.5pt,-3.5pt) -- ++(7pt,7pt) ++(-7pt,0) -- ++(7pt,-7pt);
\draw[color=black] (5.7435763823729875,0.3) node {};
\draw [fill=xdxdff] (-5,3) circle (3.5pt);
\draw [fill=xdxdff] (5,3) circle (3.5pt);
\draw[color=xdxdff] (5.7435763823729875,2.3) node {};
\draw [fill=ffxfqq] (2,2.5) circle (3.5pt);
\draw[color=ffxfqq] (2.782068495052596,2.4803343389342576) node {};
\draw [fill=ffxfqq] (2,5) circle (3.5pt);
\draw[color=ffxfqq] (2.4,5.8) node {};
\draw [fill=ffxfqq] (8,6) circle (3.5pt);
\draw[color=ffxfqq] (7.4,6.593539737990347) node {};
\draw [fill=ffxfqq] (3.5,7.5) circle (3.5pt);
\draw[color=ffxfqq] (4.333334531268039,7.909765465688295) node {};
\draw [fill=ffxfqq] (6.5,5) circle (3.5pt);
\draw[color=ffxfqq] (5.8,5.6) node {};
\draw [color=black] (-2,6)-- ++(-3.5pt,-3.5pt) -- ++(7pt,7pt) ++(-7pt,0) -- ++(7pt,-7pt);
\draw [fill=xdxdff] (-8,5) circle (3.5pt);
\draw [fill=xdxdff] (-2,7) circle (3.5pt);
\draw [color=black] (-6,7)-- ++(-3.5pt,-3.5pt) -- ++(7pt,7pt) ++(-7pt,0) -- ++(7pt,-7pt);
\draw[color=qqccqq] (-5.3,6.4) node {};
\draw [fill=ffxfqq] (-8,7) circle (3.5pt);
\draw[color=ffxfqq] (-8.9,7) node {};
\draw [fill=xdxdff] (-10,9) circle (3.5pt);
\draw[color=xdxdff] (-10.7,8.3) node {};
\draw [fill=ffxfqq] (-9,10) circle (3.5pt);
\draw[color=ffxfqq] (-9.429354221231744,10.8) node {};
\draw [fill=ffxfqq] (-11,14) circle (3.5pt);
\draw[color=ffxfqq] (-10.286172658519925,14.561406196733286) node {};
\draw [color=black] (-7,8.5)-- ++(-3.5pt,-3.5pt) -- ++(7pt,7pt) ++(-7pt,0) -- ++(7pt,-7pt);
\draw[color=black] (-6.231511598833066,7.9) node {};
\draw [fill=ffxfqq] (-5,9.5) circle (3.5pt);
\draw[color=ffxfqq] (-5.5,10.2) node {};
\draw [fill=xdxdff] (-3,9.5) circle (3.5pt);
\draw [fill=xdxdff] (-1,11) circle (3.5pt);
\draw[color=xdxdff] (-0.2439676082300397,10.4) node {};
\draw [fill=ffxfqq] (-6.5,11) circle (3.5pt);
\draw[color=ffxfqq] (-6.3,11.8) node {};
\draw [fill=xdxdff] (-7.5,12) circle (3.5pt);
\draw[color=xdxdff] (-7.9,11.4) node {};
\draw [fill=ffxfqq] (-7.8,15.5) circle (3.5pt);
\draw[color=ffxfqq] (-7.254152678644286,16.1) node {};
\draw [fill=ffxfqq] (-7.2,14) circle (3.5pt);
\draw[color=ffxfqq] (-6.73,14.53) node {};
\draw [fill=ffxfqq] (-1.5,12) circle (3.5pt);
\draw[color=ffxfqq] (-1.5426731816692456,12.87) node {};
\draw [fill=ffxfqq] (-1,15) circle (3.5pt);
\draw[color=ffxfqq] (-0.20294342311954483,15.4780633999515) node {};
\draw [fill=ffxfqq] (-0.5,13.5) circle (3.5pt);
\draw[color=ffxfqq] (0.29064122476718707,14.114829610550053) node {};
\draw [fill=ffxfqq] (-2.5,12) circle (3.5pt);
\draw[color=qqccqq] (-3.3,12.8) node {};
\draw [fill=ffxfqq] (-3.5,11.5) circle (3.5pt);
\draw[color=ffxfqq] (-4.3,11.1) node {};
\draw [fill=ffxfqq] (-2.5,14) circle (3.5pt);
\draw[color=ffxfqq] (-2.4,14.6) node {};
\end{scriptsize}
\end{tikzpicture}
\end{subfigure}
\hfill
\begin{subfigure}{50mm}
\begin{tikzpicture}[line cap=round,line join=round,>=triangle 45,x=0.4cm,y=0.5cm,scale=0.6]
\clip(-10.33,-2) rectangle (5.5,16.5);
\draw [line width=1.2pt] (5,0)-- (5,1);
\draw [line width=1.2pt] (-5,3)-- (-5,0);
\draw [line width=1.2pt] (5,1)-- (5,3);
\draw [line width=1.2pt] (-5,0)-- (0,0);
\draw [line width=1.2pt] (0,0)-- (5,0);
\draw [line width=1.2pt] (-8,3)-- (-5,3);
\draw [line width=1.2pt] (-5,3)-- (-2,3);
\draw [line width=1.2pt] (-2,3)-- (-2,6);
\draw [line width=1.2pt] (-8,3)-- (-8,5);
\draw [line width=1.2pt] (-2,6)-- (-2,7);
\draw [line width=1.2pt] (-6,5)-- (-6,7);
\draw [line width=1.2pt] (-10,5)-- (-8,5);
\draw [line width=1.2pt] (-8,5)-- (-6,5);
\draw [line width=1.2pt] (-3,7)-- (-2,7);
\draw [line width=1.2pt] (-2,7)-- (-1,7);
\draw [line width=1.2pt] (-10,9)-- (-10,5);
\draw [line width=1.2pt] (-3,7)-- (-3,9.5);
\draw [line width=1.2pt] (-1,7)-- (-1,11);
\draw [line width=1.2pt] (0,0)-- (0,-1.5);
\draw [line width=1.2pt] (-6,7)-- (-6,12);
\begin{scriptsize}
\draw [fill=xdxdff] (0,0) circle (3.5pt);
\draw [fill=xdxdff] (-5,3) circle (3.5pt);
\draw [fill=xdxdff] (-8,5) circle (3.5pt);
\draw [fill=xdxdff] (-2,7) circle (3.5pt);
\draw [fill=xdxdff] (-3,9.5) circle (3.5pt);
\draw [fill=xdxdff] (-1,11) circle (3.5pt);
\draw [fill=xdxdff] (-6,12) circle (3.5pt);
\draw [fill=xdxdff] (-10,9) circle (3.5pt);
\draw [fill=xdxdff] (5,3) circle (3.5pt);
\end{scriptsize}
\end{tikzpicture}
\end{subfigure}
\caption{A tree with edge lengths before and after Horton pruning. Edge length is indicated by height difference from parent vertex. \emph{Left :} Before. The subtrees that will be erased are dashed and orange. The cross marks represent the vertices that will only have a single child left (and will be removed). \emph{Right :} After.}
\label{ex_pruning_weightless}
\end{center}
\end{figure}

\begin{definition}[Horton pruning]
\label{horton-pruning}
Let $T=(t,(l_u)_{u\in t})$ be a tree with edge lengths with $\HS(t)\geq 1$. Remark~\ref{canonical_plane_order} ensures that there are a unique tree $t''$ and a unique increasing embedding $\psi'':t''\to t$ such that $\psi''(t'')=\{u\in t\, :\, \HS(\theta_u T)\geq 1\}$. Proposition~\ref{pruning_height} then yields that there are a unique tree $t'$ and a unique increasing embedding $\psi':t'\to t''$ such that $\psi(t')=\{u\in t''\, :\, k_u(t'')\neq 1\}$. Let $\psi$ be the embedding $\psi''\circ\psi':t'\to t$. For all $u\in t'$, we set
\[l_{u}'=\sum_{\psi(\overleftarrow{u})\prec u'\preceq\psi(u)}l_{u'},\quad\text{ where }\quad\psi(\overleftarrow{\varnothing})=\overleftarrow{\varnothing}\text{ in accordance with Notation~\ref{parent_root}}.\]
We define the \textit{Horton-pruned tree with edge lengths} as $R(T)=(t',(l_{u}')_{u\in t'})$.
\end{definition}
The number of Horton pruning operations needed to entirely erase $T$ is equal to $\HS(T)+1$:
\begin{equation}
\label{HS_pruning}
\textit{ if }\quad \HS(T)\geq 1\quad\textit{ then }\quad \HS(R(T))=\HS(T)-1.
\end{equation}
This identity is also traditional. We will prove a more general version as Proposition~\ref{recursive_pruning}, below.

Next, we gather from \cite{companion_1} and \cite{kovchegov} several tail estimates for the joint law of the Horton--Strahler number $\HS(\tau)$ of a $\GW_\alpha$-tree $\tau$ and either its size $\#\tau$ or height $|\tau|$.

\begin{proposition}
Let $\tau$ be a $\GW_\alpha$-tree. Recall that $\gamma=\ln\tfrac{\alpha}{\alpha-1}$ and $\delta=e^{\gamma(\alpha-1)}$. Then, 
\begin{align}
\label{equiv_HS}
\P\left(\HS(\tau)\geq n\right)&=e^{-\gamma n},\\\label{maj_size_selon_H}
\E\big[\I{\HS(\tau)\leq n}\#\tau\big]&\leq 2e^{\gamma(\alpha-1)n},\\
\label{first_moment_size}
\E\big[\#\tau\ |\ \HS(\tau)=n\big]&\leq 2\alpha e^{\gamma\alpha n},
\end{align}
for all $n\in\bbN$. Moreover, for all $\lambda\in(0,\infty)$, it holds that
\begin{equation}
\label{min_hauteur_selon_H}
\limsup_{n\rightarrow\infty}\P\big(|\tau|\leq \lambda e^{\gamma(\alpha-1)n}\ \big|\ \HS(\tau)=n\big)\leq 1-e^{-\lambda}.
\end{equation}
Furthermore, there is a constant $\lambda_0\in(0,\infty)$ that only depends on $\alpha$ such that for all $\lambda>\lambda_0$,
\begin{align}
\label{maj_hauteur_selon_H_prime}
\limsup_{n\rightarrow\infty}\P\big(\HS(\tau)\leq \log_\delta(n/\lambda)\ \big|\ |\tau|\geq n\big)&\leq e^{-\lambda/20},\\
\label{maj_hauteur_selon_H}
\limsup_{n\rightarrow\infty}\P\big(|\tau|\geq \lambda e^{\gamma(\alpha-1)n}\ \big|\ \HS(\tau)= n\big)&\leq e^{-\lambda/20}.
\end{align}
\end{proposition}

\begin{proof}
Kovchegov \& Zaliapin~\cite[Lemma 10]{kovchegov} yields (\ref{equiv_HS}) (see also \cite[Remark 3.3]{companion_1} and Proposition~\ref{law_HT_stable} later). Recall from (\ref{critical}) and (\ref{domain_attraction}) that the offspring distribution of $\GW_\alpha$-trees is critical, non-trivial, and in the domain of attraction of an $\alpha$-stable law. This allows us to use results from~\cite{companion_1}. These all involve the quantity $1-\varphi_\alpha'(\P(\HS(\tau)\geq n))$ which is equal to $e^{-\gamma(\alpha-1)n}=1/\delta^n$ here by (\ref{stable_offspring}) and (\ref{equiv_HS}). First, \cite[Proposition 3.5]{companion_1} entails (\ref{maj_size_selon_H}). Then, (\ref{first_moment_size}) follows from (\ref{equiv_HS}) and (\ref{maj_size_selon_H}). For (\ref{min_hauteur_selon_H}), see \cite[Proposition 3.6]{companion_1}. By \cite[Corollary~3.8]{companion_1}, the left-hand side of (\ref{maj_hauteur_selon_H_prime}) is bounded by $Ce^{-\lambda/8\delta}$ for all $\lambda>0$, where $C\in(0,\infty)$ is a constant that only depends on $\alpha$. Using (\ref{equiv_hauteur}) to bound $\P(|\tau|\geq \lambda e^{\gamma(\alpha-1)n})$, we then get with Bayes's theorem and (\ref{equiv_HS}) that the left-hand side of (\ref{maj_hauteur_selon_H}) is bounded by $\alpha e^{\gamma\beta}C\lambda^{-\beta} e^{-\lambda/8\delta}$ for all $\lambda\in(0,\infty)$. Since $8\delta<20$, we readily obtain (\ref{maj_hauteur_selon_H_prime}) and (\ref{maj_hauteur_selon_H}).
\end{proof}

\subsection{First properties of the weighted Horton--Strahler number}

Here, we provide basic properties of the weighted Horton--Strahler number $\tHS$ (recall Definition~\ref{def_tHS_intro_full}) which show that it is a good approximation of the (classic) Horton--Strahler number $\HS$ while being a continuous quantity. For all $x\in\bbR_+$, we denote by $\lfloor x\rfloor\in\bbN$ the integer part of $x$ and by $\mathrm{frac}(x)=x-\lfloor x\rfloor\in [0,1)$ the fractional part of $x$.

\begin{proposition}
\label{regularized_HT}
For any weighted tree $\bt=(t,(w_v)_{v\in\partial t})$, it holds that $\HS(t)=\lfloor \tHS(\bt)\rfloor$. Moreover, there exists a leaf $v\in\partial t$ of the tree $t$ such that $\mathrm{frac}(\tHS(\bt))=w_v$.
\end{proposition}

\begin{proof}
We argue by induction on $|t|$. If $|t|=0$ then $t=\partial t=\{\varnothing\}$, so that $\HS(t)=0$ and $\lfloor\tHS(\bt)\rfloor=\lfloor w_\varnothing\rfloor=0$ because $w_\varnothing\in[0,1)$. Moreover, $\mathrm{frac}(\tHS(\bt))=\tHS(\bt)=w_\varnothing$. If $k_\varnothing(t)\geq 1$ then there is $1\leq i\leq k_\varnothing(t)$ such that either $\tHS(\bt)=\tHS(\theta_{(i)}\bt)$ or $\tHS(\bt)=1+\tHS(\theta_{(i)}\bt)$ by Definition~\ref{def_tHS_intro_full}. Either way, it holds $\mathrm{frac}(\tHS(\bt))=\mathrm{frac}(\tHS(\theta_{(i)}\bt))$, so there is $v\in\partial t$ such that $\mathrm{frac}(\tHS(\bt))=w_v$ by induction hypothesis. The induction hypothesis also entails that
\[\lfloor\tHS(\bt)\rfloor=\max_{1\leq i,j\leq k_\varnothing(t)}\max\Big(\HS(\theta_{(i)} t),\HS(\theta_{(j)}t),\I{i\neq j}+\min\big(\HS(\theta_{(i)}t),\HS(\theta_{(j)}t)\big)\Big)\]
since the floor function is non-decreasing. The formula (\ref{def_tHS_intro}) for $\HS(t)$ ends the proof.
\end{proof}

\noi
For any weighted tree $\bt\!=\!(t,(w_v)_{v\in\partial t})$ and $x\!\in\![0,1)$, note from Proposition~\ref{regularized_HT} and (\ref{nochain}) that
\begin{equation}
\label{nochain_weighted}
\textit{if }k_\varnothing(t)\neq 1,\quad\textit{ then }\quad \tHS(\bt)\leq x\Longleftrightarrow \bt=(\{\varnothing\},w_\varnothing)\textit{ with }w_\varnothing\leq x.
\end{equation}

\begin{lemma}
\label{absolute_continuity_HT}
Let $t$ be a tree and let $(W_v)_{v\in\partial t}$ be independent RVs on $[0,1)$. We set $\bt=(t,(W_v)_{v\in \partial t})$. If the law of $W_v$ admits a continuous and positive density on $(0,1)$ for all $v\in\partial t$, then the law of $\tHS(\bt)$ admits a continuous and positive density on $\big(\HS(t),\HS(t)+1\big)$.
\end{lemma}

\begin{proof}
We prove the lemma by induction on $|t|$. If $|t|=0$ then $t=\partial t=\{\varnothing\}$, $\HS(t)=0$, and $\tHS(\bt)=W_\varnothing$ has a continuous and positive density on $(0,1)$ by assumption. Now, we assume $k_\varnothing(t)\geq 1$ and we write $J=\{1\leq j\leq k_\varnothing(t)\ :\ \HS(\theta_{(j)} t)=\HS(t)-1\}$. We set 
\[F(x)=\P\left(\tHS(\bt)\leq \HS(t)+x\right)\quad\text{ and }\quad F_i(x)=\P\left(\tHS(\theta_{(i)}\bt)\leq \HS(\theta_{(i)}t)+x\right)\]
for all $x\in[0,1]$ and $1\leq i\leq k_\varnothing(t)$. We point out that the $\theta_{(i)}\bt$, for $1\leq i\leq k_\varnothing(\tau)$, are independent. We exactly need to show that $F$ is $C^1$ with positive derivative on $(0,1)$. By induction hypothesis, the $F_i$ are $C^1$ with positive derivative on $(0,1)$. There are two cases.

$\bullet$ If there exists $1\leq i\leq k_\varnothing(t)$ such that $\HS(\theta_{(i)} t)=\HS(t)$ then it is unique by Definition~\ref{def_HS_intro_full}. Since $\lfloor \tHS(\bt)\rfloor=\HS(t)$, we deduce that $\tHS(\bt)=\max\big(\tHS(\theta_{(i)}\bt),1+\max_{j\in J}\tHS(\theta_{(j)}\bt)\big)$ by Definition~\ref{def_tHS_intro_full}. It follows that $F(x)=F_i(x)\prod_{j\in J}F_j(x)$ for all $x\in[0,1]$, which yields the result.

$\bullet$ Otherwise, there are at least two elements in $J$ by Definition~\ref{def_tHS_intro_full}. Therefore, we get
\[\tHS(\bt)=1+\max_{\substack{i,j\in J\\ i\neq j}}\, \min\big(\tHS(\theta_{(i)}\bt),\tHS(\theta_{(j)}\bt)\big),\]
which translates into $F(x)=\prod_{j\in J}F_j(x)
+\sum_{i\in J}(1\! -\! F_i(x))\prod_{j\in J\backslash\{i\}}F_j(x)$ for all $x\! \in\! [0,1]$. Thus, $F$ is $C^1$ on $(0,1)$. To obtain the positivity of the derivative, we simply compute that
\[\forall x\in(0,1),\quad F'(x)=\sum_{\substack{i,j\in J\\ i\neq j}}F_i'(x)(1-F_j(x))\prod_{\substack{k\in J \\k\neq i,j}}F_k(x)>0.\qedhere\]
\end{proof}

\subsection{The weighted Horton--Strahler number of stable Galton--Watson weighted trees}

Recall from (\ref{FEXP(a)}) and (\ref{alpha_gamma_delta}) that $\mathsf{FExp}(\gamma)$ stands for the law of the fractional part of an exponential random variable with mean $1/\gamma$, where $\gamma=\ln\frac{\alpha}{\alpha-1}$. We still write $\mathrm{frac}(x)=x-\lfloor x\rfloor$ for all $x\in\bbR_+$. The next result justifies $\mathsf{FExp}(\gamma)$ as the choice of law of weights in Definition~\ref{GWwdef}.

\begin{proposition}
\label{law_HT_stable}
Let $W$ be a random variable on $[0,1)$. Let $\ftau=(\tau,(W_v)_{v\in\partial \tau})$ be a random weighted tree such that $\tau$ is $\GW_\alpha$-tree and conditionally given $\tau$, the $(W_v)_{v\in\partial\tau}$ are independent and distributed as $W$. Then, the following holds.
\begin{longlist}
\item[(i)] The law of $\lfloor\tHS(\ftau)\rfloor=\HS(\tau)$ is geometric with parameter $1/\alpha$.
\item[(ii)] The law of $\mathrm{frac}(\tHS(\ftau))$ is the same as the law of $W$.
\item[(iii)] $\lfloor \tHS(\ftau)\rfloor=\HS(\tau)$ and $\mathrm{frac}(\tHS(\ftau))$ are independent.
\end{longlist}
In particular, if $\ftau$ is $\GWw_\alpha$-weighted tree then the law of $\tHS(\ftau)$ is exponential with mean $1/\gamma$.
\end{proposition}

\begin{proof}
From (\ref{stable_offspring}), recall that $\mu_\alpha$ is the offspring distribution of $\tau$ and $\varphi_\alpha$ is its generating function. Set $f(x)=\P(\tHS(\ftau)>x)$ for all $x\in\bbR_+$. We only need to show that for all $x\in\bbR_+$,
\begin{equation}
\label{goal_law_HT_stable}
f(x)=e^{-\gamma\lfloor x\rfloor}\big(1-\tfrac{1}{\alpha}\P(W\leq x-\lfloor x\rfloor)\big),
\end{equation}
because the integer and fractional parts are measurable. Note that if the law of $W$ is $\mathsf{FExp}(\gamma)$ then (\ref{goal_law_HT_stable}) becomes $f(x)=e^{-\gamma x}$. If $x\in[0,1)$ then by (\ref{nochain_weighted}), and since $\mu_\alpha(1)=0$ by (\ref{stable_offspring_explicit}), we have $\P(\tHS(\ftau)\leq x)=\P(k_\varnothing(\tau)=0\, ;\, W_\varnothing\leq x)=\mu_\alpha(0)\P(W\leq x)$. This is exactly (\ref{goal_law_HT_stable}). 

Now, we assume $x\geq 1$. By Definition~\ref{def_tHS_intro_full} of the weighted Horton--Strahler number, we observe that $\tHS(\ftau)\leq x$ if and only if $\tHS(\theta_u\ftau)\leq x-1$ for all children $u$ of $\varnothing$ in $\tau$ (if any) with the possible exception of one child $v$, which may satisfy $x-1<\tHS(\theta_v\ftau)\leq x$. Hence,
\begin{equation}
\label{decompo_strahler}
\I{\tHS(\ftau)\leq x}=\prod_{j=1}^{k_\varnothing(\tau)}\!\! \un_{\{\tHS(\theta_{(j)}\ftau)\leq x-1\}}+\sum_{i=1}^{k_\varnothing(\tau)}\un_{\{x-1<\tHS(\theta_{(i)}\ftau)\leq x\}}\!\!\!\!\prod_{\substack{1\leq j\leq k_\varnothing(\tau)\\ j\neq i}}\!\!\!\!\un_{\{\tHS(\theta_{(j)}\ftau)\leq x-1\}}\!.
\end{equation}
By properties of $\GWw_\alpha$-weighted trees (see Definition~\ref{GWwdef}), taking the expectation gives \[1-f(x)=\varphi_\alpha\big(1-f(x-1)\big)+\big(f(x-1)-f(x)\big)\varphi_\alpha'\big(1-f(x-1)\big),\]
where recall from (\ref{stable_offspring}) that $\varphi_\alpha(s)=s+\tfrac{1}{\alpha}(1\!-\!s)^\alpha$. It follows that $f(x)=(1-\tfrac{1}{\alpha})f(x\!-\!1)$. We get $f(x)=e^{-\gamma\lfloor x\rfloor}f(x-\lfloor x\rfloor)$ for all $x\!\in\!\bbR_+$ by induction on $\lfloor x\rfloor$, proving~(\ref{goal_law_HT_stable}).
\end{proof}

Let $\ftau$ be a $\GWw_\alpha$-weighted tree. For all $n\in\mathbb{N}$, the event $\{\HS(\tau)=n\}$ has nonzero probability so the law of $\tau$ conditionally given $\HS(\tau)=n$ is straightforwardly defined. Our goal for the rest of the section is to properly define the law of $\tau$ under the degenerate conditioning $\{\tHS(\ftau)=x\}$ in a sufficiently regular and explicit way to be able to carry out calculations and study convergences. Pleasantly, the natural idea to define $\P(\dd\tau\, |\, \tHS(\ftau)=x)$ as the limit of the conditional laws of $\tau$ given $|\tHS(\ftau)-x|<\varepsilon$ when $\varepsilon$ tends to $0^+$ turns out to work just fine.

Recall from (\ref{notation_weighted_edge_length}) that for a weighted tree with edge lengths $\bT=(t,(l_u)_{u\in t},(w_v)_{v\in \partial t})$, we write $\Ske(\bT)=(t,(w_v)_{v\in\partial t})$ and $\tHS(\bT)=\tHS(\Ske(\bT))$.

\begin{proposition}
\label{cond_H=x_prime}
Let $\fTau$ be a $\GWwl_\alpha$-weighted tree with edge lengths as in Definition~\ref{GWwldef}, and let $\cT$ be its underlying $\GWl_\alpha$-tree with edge lengths. For all $x\in\bbR_+\backslash\bbN$, there exists a probability measure $Q_x$ on the space $\overline{\bbT}$ of all trees with edge lengths such that for all bounded measurable functions $g$ on $\overline{\bbT}$,
\begin{equation}
\label{Q_x_cond_H=x}
Q_x[g]=\lim_{\varepsilon\rightarrow0^+}\E\big[g(\cT)\ \big|\ |\tHS(\fTau)-x|<\varepsilon\big].
\end{equation}
Moreover, the function $x\mapsto Q_x[g]$ is continuous on $\bbR_+\backslash\mathbb{N}$, and
\begin{equation}
\label{cond_H=x_identity}
\E\big[h(\tHS(\fTau))g(\cT)\big]=\E\big[h(\tHS(\fTau))Q_{\tHS(\fTau)}[g]\big]=\int_{\bbR_+} h(x)Q_x[g]\gamma e^{-\gamma x}\, \dd x
\end{equation}
for all bounded measurable functions $h$ on $\bbR_+$. If the law of $\cT_x=(\tau_x,(L_u)_{u\in\tau_x})$ is $Q_x$, then conditionally given $\tau_x$, the $(L_u)_{u\in\tau_x}$ are independent with exponential law with mean $1$.
\end{proposition}

\begin{definition}
\label{cond_H=x_def}
Let $x\in\bbR_+\backslash\mathbb{N}$. If $\fTau$ is a $\GWwl_\alpha$-weighted tree with edge lengths, then we denote by $\P(\dd\cT\, |\, \tHS(\fTau)=x)$ the law $Q_x$ defined in Proposition~\ref{cond_H=x_prime}, and we thus write
\[\P(\cT\in A\ |\ \tHS(\fTau)=x)=Q_x(A)\quad\text{ and }\quad\E[g(\cT)\ |\ \tHS(\fTau)=x]=Q_x[g]\]
for all measurable sets $A\subset\overline{\bbT}$ and all bounded measurable functions $g$ on $\overline{\bbT}$. Recall from (\ref{notation_edge_length}) the map $\Ske:(t,(l_u)_{u\in t})\!\in\!\overline{\bbT}\mapsto t\! \in\! \bbT$, where $\bbT$ is the space of trees. If $\ftau$ is a $\GWw_\alpha$-weighted tree, then denote by $\P(\dd\tau\, |\, \tHS(\ftau)=x)$ the law of $\Ske(\cT)$ under $\P(\dd\cT\, |\, \tHS(\fTau)=x)$. Namely, for all sets $A\subset\bbT$ and all bounded functions $g$ on $\bbT$, we write
\[\P(\tau\in A\ |\ \tHS(\ftau)=x)=Q_x(\Ske^{-1}(A))\quad\text{ and }\quad\E[g(\tau)\ |\ \tHS(\ftau)=x]=Q_x[g\circ\Ske].\]
\end{definition}

Before showing Proposition~\ref{cond_H=x_prime}, we state a lemma which gives uniform control of the laws of $\tau$ conditionally given $|\tHS(\ftau)-x|<\varepsilon$. This result is also useful to convert estimates of $\P(\dd\tau\, |\, \HS(\tau)=n)$ into bounds for $\P(\dd\tau\, |\, \tHS(\ftau)=x)$. We endow the countable space $\bbT_{\mathsf{bit}}=\big\{(t,(b_v)_{v\in\partial t})\ :\ t\text{ a tree and }b_v\in\{0,1\}\text{ for all }v\in\partial t\big\}$ with the discrete topology.
\begin{lemma}
\label{cond_domination}
Let $r\in(0,1)$, let $W$ have law $\mathsf{FExp}(\gamma)$, and let $\ftau=(\tau,(W_v)_{v\in\partial\tau})$ be a $\GWw_\alpha$-weighted tree.\footnote{The $W_v$ thus have the same law as $W$, namely $\mathsf{FExp}(\gamma)$.}
We set $\xi=\I{W\geq r}$ and $\xi_v=\I{W_v\geq r}$ for all $v\in\partial\tau$. Let $g:\bbT_{\mathsf{bit}}\to \bbR_+$ be a nonnegative bounded function. For all $x,\varepsilon>0$ such that $\lfloor x\rfloor<x-\varepsilon<x+\varepsilon<\lfloor x\rfloor+1$,
\begin{equation}
\label{cond_domination_prebis}
\E\big[g(\tau,(\xi_v)_{v\in\partial\tau})\, \big|\, |\tHS(\ftau)-x|<\varepsilon\big]\leq c_\alpha(r)\, \E\big[g(\tau,(\xi_v)_{v\in\partial\tau})\#\tau\, \big|\, \HS(\tau)=\lfloor x\rfloor\big]\!<\!\infty,
\end{equation}
where $c_\alpha(r)^{-1}=\min\!\big(\P(W\!<\!r),\P(W\! \geq\! r)\big)$ only depends on $r$ and $\alpha$. Moreover, for $n\!\in\!\bbN$,
\begin{equation}
\label{cond_domination_bis}
\E\big[\I{\#\tau\leq n}g(\tau,(\xi_v)_{v\in\partial\tau})\, \big|\, |\tHS(\ftau)\! -\! x|\! <\! \varepsilon\big]\! \leq\!  c_\alpha(r)\alpha n e^{\gamma x} \E\big[\I{\#\tau\leq n}g(\tau,(\xi_v)_{v\in\partial\tau})\big]\!.
\end{equation}
\end{lemma}

\begin{proof}[Proof of Lemma~\ref{cond_domination}]
The second inequality in (\ref{cond_domination_prebis}) follows from (\ref{first_moment_size}). The space $\bbT_{\mathsf{bit}}$ is countable, so we can assume there is $(t^0,(b_v^0)_{v\in\partial t^0})\in \bbT_{\mathsf{bit}}$ such that $g(t,(b_v)_{v\in\partial t})=\I{t=t^0}\prod_{v\in\partial t^0}\I{b_v=b_v^0}$. Thanks to the assumption on $x$ and $\varepsilon$, Proposition~\ref{regularized_HT} entails that $|\tHS(\ftau)-x|<\varepsilon$ if and only if $\HS(\tau)=\lfloor x\rfloor$ and $|\mathrm{frac}(\tHS(\ftau))-\mathrm{frac}(x)|<\varepsilon$. Thus, we get
\begin{equation}
\label{tool_cond_domination}
\P(|\tHS(\ftau)-x|<\varepsilon)=\P(\HS(\tau)=\lfloor x\rfloor)\P(|W-\mathrm{frac}(x)|<\varepsilon)
\end{equation}
by Proposition~\ref{law_HT_stable}. Moreover, Proposition~\ref{regularized_HT} also ensures that
\[\I{|\tHS(\ftau)-x|<\varepsilon}g(\tau,(\xi_v)_{v\in\partial\tau})\leq \I{\HS(\tau)=\lfloor x\rfloor}\I{\tau=t^0}\sum_{u\in\partial t^0}\I{|W_u-\mathrm{frac}(x)|<\varepsilon}\prod_{v\in(\partial t^0)\backslash\{u\}}\!\!\!\!\I{\xi_v=b_v^0}.\]
Note that $\#\partial t^0\leq \# t^0$ and $\P(\xi=b_v^0)\geq c_\alpha(r)^{-1}$ for all $v\in\partial t^0$. Since the $(W_v)_{v\in\partial\tau}$ are independent and distributed as $W$ conditionally given $\tau$, taking the expectation yields that
\[\E\big[\!\I{|\tHS(\ftau)-x|<\varepsilon}g(\tau,(\xi_v)_{v\in\partial\tau})\big]\!\leq\! c_\alpha(r)\P(|W\!-\!\mathrm{frac}(x)|\!<\!\varepsilon)\E\big[\!\I{\HS(\tau)=\lfloor x\rfloor}g(\tau,(\xi_v)_{v\in\partial\tau})\#\tau\big]\! .\]
Dividing this bound by $\P(|\tHS(\ftau)-x|<\varepsilon)$ and using (\ref{tool_cond_domination}) proves (\ref{cond_domination_prebis}) in full. For (\ref{cond_domination_bis}), write $\E[\I{\#\tau\leq n}g\#\tau\, |\, \HS(\tau)=\lfloor x\rfloor]\leq n\P(\HS(\tau)=\lfloor x\rfloor)^{-1}\E[\I{\#\tau\leq n}g]$ and use (\ref{equiv_HS}).
\end{proof}

\begin{proof}[Proof of Proposition~\ref{cond_H=x_prime}]
Write $\ftau=\Ske(\fTau)$ which is a $\GWw_\alpha$-weighted tree. Let $t$ be a tree and let $(W_v)_{v\in\partial t}$ be independent random variables with law $\mathsf{FExp}(\gamma)$. Set $\bt=(t,(W_v)_{v\in\partial t})$. Denote by $f_\bt$ the density of $\mathrm{frac}(\tHS(\bt))$ and by $f$ the density of $\mathrm{frac}(\tHS(\ftau))$. These densities are positive and continuous on $(0,1)$ by Lemma~\ref{absolute_continuity_HT} and Proposition~\ref{law_HT_stable}. As $x\notin\bbN$, it holds $\lfloor x\rfloor\!<\!x-\varepsilon\!<\!x+\varepsilon\!<\!\lfloor x\rfloor+1$ for all $\varepsilon\!>\!0$ small enough. For such $\varepsilon$, we find
\begin{align*}
\P(|\tHS(\ftau)-x|<\varepsilon)&=\P(\HS(\tau)=\lfloor x\rfloor)\int_{\mathrm{frac}(x)-\varepsilon}^{\mathrm{frac}(x)+\varepsilon}f(y)\, \dd y,\\
\P(\tau=t\, ;\, |\tHS(\ftau)-x|<\varepsilon)&=\P(\tau=t\, ;\, \HS(\tau)=\lfloor x\rfloor)\int_{\mathrm{frac}(x)-\varepsilon}^{\mathrm{frac}(x)+\varepsilon}f_\bt(y)\, \dd y,
\end{align*}
by respectively applying Proposition~\ref{law_HT_stable} $(iii)$ and conditioning on $\{\tau=t\}$. The existence of
\begin{equation}
\label{q_x(t)}
q_x(t):=\P\big(\tau=t\ \big|\ \HS(\tau)=\lfloor x\rfloor\big)\tfrac{f_\bt(\mathrm{frac}(x))}{f(\mathrm{frac}(x))}=\lim_{\varepsilon\rightarrow0^+}\P\big(\tau=t\ \big|\ |\tHS(\ftau)-x|<\varepsilon\big)
\end{equation}
and the continuity of $x\longmapsto q_x(t)$ on $\bbR_+\backslash\mathbb{N}$ follow from the continuity of $f_\bt$ and $f$ on $(0,1)$. Now, set $F(x)=\P(\tau\!=\!t\, ;\, \tHS(\ftau)\!\leq\! x)$ for all $x\in\bbR$ (note that $F(0)=0$). With the same argument used to get (\ref{q_x(t)}), we obtain that $F$ is $C^1$ on $\bbR_+\backslash\mathbb{N}$ with $F'(x)=q_x(t)\gamma e^{-\gamma x}$. If $h:\bbR_+\to \mathbb{R}$ is $C^1$ with compact support, then we write $h(x)=-\int_x^\infty h'(x)\, \dd x$ and we obtain
\begin{equation}
\label{proto_cond_H=x_identity}
\E\big[h(\tHS(\ftau))\I{\tau=t}\big]=\int_0^\infty h(x)q_x(t)\gamma e^{-\gamma x}\, \dd x
\end{equation}
after an application of Fubini's theorem and an integration by parts. The identity (\ref{proto_cond_H=x_identity}) holds in fact for any bounded and measurable function $h$ by the functional monotone class theorem.

Finally define the measure $Q_x$ by setting $Q_x[g]=\sum_{t\in\bbT}q_x(t)\E[g(\cT)\ |\ \tau=t]$ for all bounded measurable functions $g$ on $\overline{\bbT}$. Taking $\varepsilon\rightarrow0^+$ in (\ref{cond_domination_prebis}) with $g(\tau,(\xi))=\I{\tau=t}$ leads to $q_x(t)\leq c_\alpha(1/2)\E[\I{\tau=t}\#\tau\ |\ \HS(\tau)=\lfloor x\rfloor]$ for all trees $t$ and $x\in\bbR_+\backslash\bbN$. This ensures that the function $x\mapsto Q_x[g]$ is continuous on $\bbR_+\backslash\mathbb{N}$. By Definition~\ref{GWwldef},
\[\E\big[g(\cT)\ \big|\ |\tHS(\ftau)-x|<\varepsilon\big]=\sum_{t\in\bbT}\P\big(\tau=t\ \big|\ |\tHS(\ftau)-x|<\varepsilon\big)\E\big[g(\cT)\ \big|\ \tau=t\big].\]By (\ref{q_x(t)}) and Lemma~\ref{cond_domination}, the dominated convergence theorem yields (\ref{Q_x_cond_H=x}). In particular, $Q_x$ is a probability measure. Similarly, (\ref{cond_H=x_identity}) follows from Fubini's theorem and (\ref{proto_cond_H=x_identity}).
\end{proof}

\section{The weighted Horton pruning of trees with edge lengths}
\label{pruning_section}

\subsection{Definition and invariance}

In this section, we extensively use the notation from (\ref{notation_edge_length}), (\ref{scaling_edge_length}), (\ref{notation_weighted_edge_length}), and (\ref{scaling_weighted_edge_length}) to manipulate (weighted) trees with edge lengths (see Section~\ref{words}). For a tree with edge lengths $T=(t,(l_u)_{u\in t})$, we also write
\[K(T)=k_\varnothing(t)\quad\text{ and }\quad L(T)=l_\varnothing.\]

We recall from (\ref{alpha_gamma_delta}) that $\delta=(\frac{\alpha}{\alpha-1})^{\alpha-1}=e^{\gamma(\alpha-1)}$. As discussed in the introduction, Kovchegov \& Zaliapin~\cite[Proposition 4]{kovchegov23} showed that stable Galton--Watson trees with edge lengths are invariant in law under the Horton pruning $R$ (see Definition~\ref{horton-pruning}). More precisely, if $\cT$ is a $\GWl_\alpha$-tree with edge lengths, then the law of $R(\cT)$ under $\P(\, \cdot\, |\, \HS(\cT)\geq 1)$ is the same as the law of $\delta\cdot\cT$ under $\P$. In what follows, we seek to obtain a similar result for the weighted Horton--Strahler number and $\GWwl_\alpha$-weighted trees with edge lengths.

First, we adapt the Horton pruning of Definition~\ref{horton-pruning} to the framework of weighted trees with edge lengths. The \textit{$r$-Horton pruning} consists in erasing the subtrees with weighted Horton--Strahler numbers smaller than a threshold $r$, and then removing the vertices with only one child left. The weight of a new leaf is given by the maximal Horton--Strahler number previously achieved on the new parental edge, minus $r$. See Figure~\ref{ex_pruning} for an example. The formal definition is given below.

\begin{definition}[Weighted Horton pruning]
\label{pruning}
Let $r\in\bbR_+$ and $\bT=(t,(l_u)_{u\in t},(w_v)_{v\in\partial t})$ be a weighted tree with edge lengths such that $\tHS(\bT)\geq r$. Remark~\ref{canonical_plane_order} ensures that there exist a unique tree $t''$ and a unique increasing embedding $\psi'':t''\to t$ such that $\psi''(t'')=\{u\in t\, :\, \tHS(\theta_u \bT)\geq r\}$. Proposition~\ref{pruning_height} then states that there exist a unique tree $t'$ and a unique increasing embedding $\psi':t'\to t''$ such that $\psi(t')=\{u\in t''\, :\, k_u(t'')\neq 1\}$. We denote by $\psi$ the embedding $\psi''\circ\psi':t'\to t$ and for all $u\in t'$ and for all $v\in\partial t'$, we set
\[l_u'=\sum_{\psi(\overleftarrow{u})\prec u'\preceq\psi(u)}l_{u'}\quad\text{ and }\quad w_v'=\max_{\psi(\overleftarrow{v})\prec v'\preceq \psi(v)}\tHS(\theta_{v'}\bT)-r.\]
We define $R_r(\bT)=(t',(l_u')_{u\in t'})$ and $\bR_r(\bT)=(t',(l_u')_{u\in t'},(w_v')_{v\in\partial t'})$.
\end{definition}

\begin{figure}
\begin{center}
\begin{subfigure}{78mm}
\begin{tikzpicture}[line cap=round,line join=round,>=triangle 45,x=0.66cm,y=0.5cm,scale=0.6]
\clip(-11.5,-2) rectangle (8.2,16.5);
\draw [line width=1.2pt,dash pattern=on 3pt off 3pt,color=ffxfqq] (0,4)-- (0,0);
\draw [line width=1.2pt] (5,0)-- (5,1);
\draw [line width=1.2pt] (-5,3)-- (-5,0);
\draw [line width=1.2pt] (5,1)-- (5,3);
\draw [line width=1.2pt,dash pattern=on 3pt off 3pt,color=ffxfqq] (2,1)-- (2,2.5);
\draw [line width=1.2pt,dash pattern=on 3pt off 3pt,color=ffxfqq] (2,2.5)-- (2,5);
\draw [line width=1.2pt,dash pattern=on 3pt off 3pt,color=ffxfqq] (8,6)-- (8,1);
\draw [line width=1.2pt,dash pattern=on 3pt off 3pt,color=ffxfqq] (3.5,7.5)-- (3.5,3);
\draw [line width=1.2pt] (6.5,5)-- (6.5,3);
\draw [line width=1.2pt] (-5,0)-- (0,0);
\draw [line width=1.2pt] (0,0)-- (5,0);
\draw [line width=1.2pt,dash pattern=on 3pt off 3pt,color=ffxfqq] (5,1)-- (8,1);
\draw [line width=1.2pt] (5,3)-- (6.5,3);
\draw [line width=1.2pt,dash pattern=on 3pt off 3pt,color=ffxfqq] (5,3)-- (3.5,3);
\draw [line width=1.2pt,dash pattern=on 3pt off 3pt,color=ffxfqq] (2,1)-- (5,1);
\draw [line width=1.2pt] (-8,3)-- (-5,3);
\draw [line width=1.2pt] (-5,3)-- (-2,3);
\draw [line width=1.2pt] (-2,3)-- (-2,6);
\draw [line width=1.2pt] (-8,3)-- (-8,5);
\draw [line width=1.2pt] (-2,6)-- (-2,7);
\draw [line width=1.2pt] (-6,5)-- (-6,7);
\draw [line width=1.2pt] (-8,5)-- (-8,7);
\draw [line width=1.2pt] (-10,5)-- (-8,5);
\draw [line width=1.2pt] (-8,5)-- (-6,5);
\draw [line width=1.2pt] (-3,7)-- (-2,7);
\draw [line width=1.2pt] (-2,7)-- (-1,7);
\draw [line width=1.2pt] (-10,9)-- (-10,5);
\draw [line width=1.2pt,dash pattern=on 3pt off 3pt,color=ffxfqq] (-10,9)-- (-11,9);
\draw [line width=1.2pt] (-10,9)-- (-9,9);
\draw [line width=1.2pt] (-9,9)-- (-9,10);
\draw [line width=1.2pt,dash pattern=on 3pt off 3pt,color=ffxfqq] (-11,9)-- (-11,14);
\draw [line width=1.2pt,dash pattern=on 3pt off 3pt,color=ffxfqq] (-6,7)-- (-5,7);
\draw [line width=1.2pt] (-6,7)-- (-7,7);
\draw [line width=1.2pt] (-7,7)-- (-7,8.5);
\draw [line width=1.2pt,dash pattern=on 3pt off 3pt,color=ffxfqq] (-5,7)-- (-5,9.5);
\draw [line width=1.2pt] (-3,7)-- (-3,9.5);
\draw [line width=1.2pt] (-1,7)-- (-1,11);
\draw [line width=1.2pt,dash pattern=on 3pt off 3pt,color=ffxfqq] (-7,8.5)-- (-6.5,8.5);
\draw [line width=1.2pt] (-7,8.5)-- (-7.5,8.5);
\draw [line width=1.2pt,dash pattern=on 3pt off 3pt,color=ffxfqq] (-6.5,11)-- (-6.5,8.5);
\draw [line width=1.2pt] (-7.5,12)-- (-7.5,8.5);
\draw [line width=1.2pt,dash pattern=on 3pt off 3pt,color=ffxfqq] (-7.5,12)-- (-7.8,12);
\draw [line width=1.2pt,dash pattern=on 3pt off 3pt,color=ffxfqq] (-7.5,12)-- (-7.2,12);
\draw [line width=1.2pt,dash pattern=on 3pt off 3pt,color=ffxfqq] (-7.8,12)-- (-7.8,15.5);
\draw [line width=1.2pt,dash pattern=on 3pt off 3pt,color=ffxfqq] (-7.2,12)-- (-7.2,14);
\draw [line width=1.2pt,dash pattern=on 3pt off 3pt,color=ffxfqq] (-1,11)-- (-0.5,11);
\draw [line width=1.2pt,dash pattern=on 3pt off 3pt,color=ffxfqq] (-1,11)-- (-1.5,11);
\draw [line width=1.2pt] (-3,9.5)-- (-2.5,9.5);
\draw [line width=1.2pt] (-3,9.5)-- (-3.5,9.5);
\draw [line width=1.2pt,dash pattern=on 3pt off 3pt,color=ffxfqq] (-1.5,11)-- (-1.5,12);
\draw [line width=1.2pt,dash pattern=on 3pt off 3pt,color=ffxfqq] (-1,11)-- (-1,15);
\draw [line width=1.2pt,dash pattern=on 3pt off 3pt,color=ffxfqq] (-0.5,11)-- (-0.5,13.5);
\draw [line width=1.2pt] (-2.5,12)-- (-2.5,9.5);
\draw [line width=1.2pt] (-3.5,11.5)-- (-3.5,9.5);
\draw [line width=1.2pt] (-2.5,14)-- (-2.5,12);
\draw [line width=1.2pt] (0,0)-- (0,-1.5);
\begin{scriptsize}
\draw [fill=xdxdff] (0,0) circle (3.5pt);
\draw[color=xdxdff] (0.8,-0.6) node {$3.2$};
\draw [fill=ffxfqq] (0,4) circle (3.5pt);
\draw[color=ffxfqq] (0.1,4.8) node {$0.4$};
\draw [color=qqccqq] (5,1)-- ++(-3.5pt,-3.5pt) -- ++(7pt,7pt) ++(-7pt,0) -- ++(7pt,-7pt);
\draw[color=qqccqq] (5.7435763823729875,0.3) node {$1.3$};
\draw [fill=xdxdff] (-5,3) circle (3.5pt);
\draw [color=black] (5,3)-- ++(-3.5pt,-3.5pt) -- ++(7pt,7pt) ++(-7pt,0) -- ++(7pt,-7pt);
\draw[color=black] (5.7435763823729875,2.3) node {$1.2$};
\draw [fill=ffxfqq] (2,2.5) circle (3.5pt);
\draw[color=ffxfqq] (2.782068495052596,2.4803343389342576) node {$0.3$};
\draw [fill=ffxfqq] (2,5) circle (3.5pt);
\draw[color=ffxfqq] (2.4,5.8) node {$0.3$};
\draw [fill=ffxfqq] (8,6) circle (3.5pt);
\draw[color=ffxfqq] (7.4,6.593539737990347) node {$0.1$};
\draw [fill=ffxfqq] (3.5,7.5) circle (3.5pt);
\draw[color=ffxfqq] (4.333334531268039,7.909765465688295) node {$0.2$};
\draw [fill=xdxdff] (6.5,5) circle (3.5pt);
\draw[color=xdxdff] (5.8,5.6) node {$0.6$};
\draw [color=black] (-2,6)-- ++(-3.5pt,-3.5pt) -- ++(7pt,7pt) ++(-7pt,0) -- ++(7pt,-7pt);
\draw [fill=xdxdff] (-8,5) circle (3.5pt);
\draw [fill=xdxdff] (-2,7) circle (3.5pt);
\draw [color=qqccqq] (-6,7)-- ++(-3.5pt,-3.5pt) -- ++(7pt,7pt) ++(-7pt,0) -- ++(7pt,-7pt);
\draw[color=qqccqq] (-5.3,6.4) node {$1.4$};
\draw [fill=qqccqq] (-8,7) circle (3.5pt);
\draw[color=qqccqq] (-8.9,7) node {$0.7$};
\draw [color=qqccqq] (-10,9)-- ++(-3.5pt,-3.5pt) -- ++(7pt,7pt) ++(-7pt,0) -- ++(7pt,-7pt);
\draw[color=qqccqq] (-10.7,8.3) node {$1.3$};
\draw [fill=xdxdff] (-9,10) circle (3.5pt);
\draw[color=xdxdff] (-9.429354221231744,10.8) node {$0.9$};
\draw [fill=ffxfqq] (-11,14) circle (3.5pt);
\draw[color=ffxfqq] (-10.286172658519925,14.561406196733286) node {$0.3$};
\draw [color=black] (-7,8.5)-- ++(-3.5pt,-3.5pt) -- ++(7pt,7pt) ++(-7pt,0) -- ++(7pt,-7pt);
\draw[color=black] (-6.231511598833066,7.9) node {$1.4$};
\draw [fill=ffxfqq] (-5,9.5) circle (3.5pt);
\draw[color=ffxfqq] (-5.5,10.2) node {$0.3$};
\draw [fill=xdxdff] (-3,9.5) circle (3.5pt);
\draw [fill=qqccqq] (-1,11) circle (3.5pt);
\draw[color=qqccqq] (-0.2439676082300397,10.4) node {$1.2$};
\draw [fill=ffxfqq] (-6.5,11) circle (3.5pt);
\draw[color=ffxfqq] (-6.3,11.8) node {$0.4$};
\draw [fill=xdxdff] (-7.5,12) circle (3.5pt);
\draw[color=xdxdff] (-7.9,11.4) node {$1$};
\draw [fill=ffxfqq] (-7.8,15.5) circle (3.5pt);
\draw[color=ffxfqq] (-7.254152678644286,16.1) node {$0.1$};
\draw [fill=ffxfqq] (-7.2,14) circle (3.5pt);
\draw[color=ffxfqq] (-6.73,14.53) node {$0$};
\draw [fill=ffxfqq] (-1.5,12) circle (3.5pt);
\draw[color=ffxfqq] (-1.5426731816692456,12.87) node {$0$};
\draw [fill=ffxfqq] (-1,15) circle (3.5pt);
\draw[color=ffxfqq] (-0.20294342311954483,15.4780633999515) node {$0.2$};
\draw [fill=ffxfqq] (-0.5,13.5) circle (3.5pt);
\draw[color=ffxfqq] (0.29064122476718707,14.114829610550053) node {$0.2$};
\draw [color=qqccqq] (-2.5,12)-- ++(-3.5pt,-3.5pt) -- ++(7pt,7pt) ++(-7pt,0) -- ++(7pt,-7pt);
\draw[color=qqccqq] (-3.3,12.8) node {$0.5$};
\draw [fill=qqccqq] (-3.5,11.5) circle (3.5pt);
\draw[color=qqccqq] (-4.3,11.1) node {$0.8$};
\draw [fill=xdxdff] (-2.5,14) circle (3.5pt);
\draw[color=xdxdff] (-2.4,14.6) node {$0.5$};
\end{scriptsize}
\end{tikzpicture}
\end{subfigure}
\hfill
\begin{subfigure}{50mm}
\begin{tikzpicture}[line cap=round,line join=round,>=triangle 45,x=0.4cm,y=0.5cm,scale=0.6]
\clip(-10.33,-2) rectangle (5.5,16.5);
\draw [line width=1.2pt] (5,0)-- (5,1);
\draw [line width=1.2pt] (-5,3)-- (-5,0);
\draw [line width=1.2pt] (5,1)-- (5,3);
\draw [line width=1.2pt] (-5,0)-- (0,0);
\draw [line width=1.2pt] (0,0)-- (5,0);
\draw [line width=1.2pt] (-8,3)-- (-5,3);
\draw [line width=1.2pt] (-5,3)-- (-2,3);
\draw [line width=1.2pt] (-2,3)-- (-2,6);
\draw [line width=1.2pt] (-8,3)-- (-8,5);
\draw [line width=1.2pt] (-2,6)-- (-2,7);
\draw [line width=1.2pt] (-6,5)-- (-6,7);
\draw [line width=1.2pt] (-8,5)-- (-8,7);
\draw [line width=1.2pt] (-10,5)-- (-8,5);
\draw [line width=1.2pt] (-8,5)-- (-6,5);
\draw [line width=1.2pt] (-3,7)-- (-2,7);
\draw [line width=1.2pt] (-2,7)-- (-1,7);
\draw [line width=1.2pt] (-10,9)-- (-10,5);
\draw [line width=1.2pt] (-3,7)-- (-3,9.5);
\draw [line width=1.2pt] (-1,7)-- (-1,11);
\draw [line width=1.2pt] (-3,9.5)-- (-2.5,9.5);
\draw [line width=1.2pt] (-3,9.5)-- (-3.5,9.5);
\draw [line width=1.2pt] (-2.5,12)-- (-2.5,9.5);
\draw [line width=1.2pt] (-3.5,11.5)-- (-3.5,9.5);
\draw [line width=1.2pt] (-2.5,14)-- (-2.5,12);
\draw [line width=1.2pt] (0,0)-- (0,-1.5);
\draw [line width=1.2pt] (-6,7)-- (-6,12);
\draw [line width=1.2pt] (-10,9)-- (-10,10);
\draw [line width=1.2pt] (5,3)-- (5,5);
\begin{scriptsize}
\draw [fill=xdxdff] (0,0) circle (3.5pt);
\draw [fill=xdxdff] (-5,3) circle (3.5pt);
\draw [fill=xdxdff] (-8,5) circle (3.5pt);
\draw [fill=xdxdff] (-2,7) circle (3.5pt);
\draw [fill=qqccqq] (-8,7) circle (3.5pt);
\draw[color=qqccqq] (-7.709255679081878,7.8) node {$0.2$};
\draw [fill=xdxdff] (-3,9.5) circle (3.5pt);
\draw [fill=qqccqq] (-1,11) circle (3.5pt);
\draw[color=qqccqq] (-0.1,11.4) node {$0.7$};
\draw [fill=qqccqq] (-3.5,11.5) circle (3.5pt);
\draw[color=qqccqq] (-3.983157872334351,12.25) node {$0.3$};
\draw [fill=qqccqq] (-2.5,14) circle (3.5pt);
\draw[color=qqccqq] (-2.1082424154359134,14.65) node {$0$};
\draw [fill=qqccqq] (-6,12) circle (3.5pt);
\draw[color=qqccqq] (-6.5,12.7) node {$0.9$};
\draw [fill=qqccqq] (-10,10) circle (3.5pt);
\draw[color=qqccqq] (-9.323106958437494,10.7) node {$0.8$};
\draw [fill=qqccqq] (5,5) circle (3.5pt);
\draw[color=qqccqq] (4.19994980379935,5.506923195830028) node {$0.8$};
\end{scriptsize}
\end{tikzpicture}
\end{subfigure}
\caption{A weighted tree with edge lengths before and after $r$-Horton pruning with $r=0.5$. Edge length is indicated by height difference from parent vertex. \emph{Left :} Before. Each number indicates the weighted Horton--Strahler number of the subtree stemming from the respective vertex. The subtrees that will be erased are dashed and orange. The cross marks represent the vertices that will only have a single child left (and will be removed). Subtracting $r$ from the numbers in green gives the weights assigned to the new leaves. \emph{Right :} After.}
\label{ex_pruning}
\end{center}
\end{figure}

\begin{proposition}
\label{recursive_pruning}
Let $r\in\bbR_+$. Let $\bT=(t,(l_u)_{u\in t},(w_v)_{v\in\partial t})$ be a weighted tree with edge lengths and set $T=(t,(l_u)_{u\in t})$. Set $K^r(\bT)=\#\big\{1\leq i\leq K(T)\, :\, \tHS(\theta_{(i)}\bT)\geq r\big\}$ and denote by $i_1<\ldots<i_{K^r(\bT)}$ the integers such that $\tHS(\theta_{(i)}\bT)\geq r$. If $\tHS(\bT)\geq r$ then $\bR_r(\bT)$ is a weighted tree with edge lengths and can be recursively described as follows.
\begin{enumerate}
\item[(a)] If $K^r(\bT)=0$, then $\bR_r(\bT)=(\{\varnothing\},L(T),\tHS(\bT)-r)$.
\item[(b)] If $K^r(\bT)\!=\!1$, then $K(R_r(\bT))\!=\!K(R_r(\theta_{(i_1)}\!\bT))$, $L(R_r(\bT))\!=\!L(T)\!+\!L(R_r(\theta_{(i_1)}\!\bT))$ and:
\begin{enumerate}
\item[(b1)] if $K(R_r(\bT))=0$ then $\bR_r(\bT)=(\{\varnothing\},L(R_r(\bT)),\tHS(\bT)-r)$;
\item[(b2)] if $K(R_r(\bT))\geq 1$ then $\theta_{(j)}\bR_r(\bT)=\theta_{(j)}\bR_r(\theta_{(i_1)}\bT)$ for all $1\leq j\leq K(R_r(\bT))$.
\end{enumerate}
\item[(c)] If $K^r(\bT)\geq 2$, then $K(R_r(\bT))=K^r(\bT)$ and $L(R_r(\bT))=L(T)$, and it holds that $\theta_{(j)}\bR_r(\bT)=\bR_r(\theta_{(i_j)}\bT)$ for all $1\leq j\leq K^r(\bT)$.
\end{enumerate}
Furthermore, it holds that $K(R_r(\bT))\neq 1$ and $\tHS(\bR_r(\bT))=\tHS(\bT)-r$.
\end{proposition}

\begin{proof}
Thanks to the uniqueness provided by Remark~\ref{canonical_plane_order} and Proposition~\ref{pruning_height}, verifying the points $(a),(b),(c)$ is easy and straightforward; we leave it to the reader. Using these points, a quick induction on the height $|t|$ shows that $K(R_r(\bT))\neq 1$.

Recall from Section~\ref{words} that the definition of weighted trees with edge lengths require that their weights belong to $[0,1)$, so we still need to check that this is the case for $\bR_r(\bT)$. By induction on $|t|$, $(b2)$ and $(c)$ allow us to only consider the case $K(R_r(\bT))=0$, where then the unique weight of $\bR_r(\bT)$ is $w_\varnothing'=\tHS(\bT)-r\geq 0$ by assumption. To sum up, it remains to show that if $K(R_r(\bT))=0$ then $\tHS(\bT)<1+r$, and if $K(R_r(\bT))\geq 2$ then $\tHS(\bR_r(\bT))=\tHS(\bT)-r$. We do so by induction on $|t|$. When $|t|=0$, we have $K(R_r(\bT))=0$ and $\tHS(\bT)=w_\varnothing<1+r$. When $|t|\geq 1$, we begin by defining
\begin{align*}
x&=\max_{1\leq j,j'\leq K^r(\bT)}\max\Big(\tHS(\theta_{(i_j)}\bT),\I{j\neq j'}+\min\big(\tHS(\theta_{(i_j)}\bT),\tHS(\theta_{(i_{j'})}\bT)\big)\Big),\\
y&=\max_{\substack{1\leq i,i'\leq K(T) \\ \tHS(\theta_{(i)}\bT)<r}}\max\Big(\tHS(\theta_{(i)}\bT),\I{i\neq i'}+\min\big(\tHS(\theta_{(i)}\bT),\tHS(\theta_{(i')}\bT)\big)\Big).
\end{align*}
We see that $\tHS(\bT)=\max(x,y)$ and $y<1+r$. We complete the proof by separating the cases.
\begin{itemize}
\item If $K(R_r(\bT))=K^r(\bT)=0$: then $\tHS(\bT)=y<1+r$.
\item If $K(R_r(\bT))=0$ and $K^r(\bT)=1$: then $x=\tHS(\theta_{(i_1)} \bT)=\tHS(\bR_r(\theta_{(i_1)}\bT))+r$ by induction hypothesis. Thanks to $(b)$, $K(R_r(\theta_{(i_1)}\bT))=0$ so $x<1+r$ by (\ref{nochain_weighted}), and $\tHS(\bT)<1+r$.
\item If $K(R_r(\bT))\geq 2$ and $K^r(\bT)=1$: then $x=\tHS(\bR_r(\theta_{(i_1)}\bT))+r$ again, and $(b2)$ yields $x=\tHS(\bR_r(\bT))+r$ by Definition~\ref{def_tHS_intro_full}. By (\ref{nochain_weighted}), $x\geq 1+r$ so $\tHS(\bT)=x=\tHS(\bR_r(\bT))+r$.
\item If $K(R_r(\bT))=K^r(\bT)\geq 2$: then $x=\tHS(\bR_r(\bT))+r$ by $(c)$, Definition~\ref{def_tHS_intro_full}, and induction hypothesis. As in the previous case, we get $x\geq 1+r$ and $\tHS(\bT)=\tHS(\bR_r(\bT))+r$.
\end{itemize}
Since $K(R_r(\bT))\neq 1$, the points $(a),(b),(c)$ ensure that we treated all cases.
\end{proof}

\noi
Recall that $R$ stands for the (classic) Horton pruning as in Definition~\ref{horton-pruning}. Let $r\! \in\! \bbR_+$, $\lambda\! >\! 0$, and $T$ (resp.~$\bT$) be a tree (resp.~weighted tree) with edge lengths. By definition, it is clear that
\begin{equation}
\label{pruning_scaling}
R(\lambda\cdot T)=\lambda \cdot R(T)\quad\text{ and }\quad\bR_r(\lambda\cdot\bT)=\lambda\cdot \bR_r(\bT).
\end{equation}
Observe that $R\!:\!\{T\!\in\! \overline{\bbT} :\HS(T)\geq 1\}\!\to\!\overline{\bbT}$ and $\bR_r\!:\!\{\bT\!\in\!\overline{\bbT}_{\mathrm{w}} : \tHS(\bT)\geq r\}\!\to\! \overline{\bbT}_{\mathrm{w}}$ are measurable, where $\overline{\bbT}$ (resp.~$\overline{\bbT}_{\mathrm{w}}$) is the space of trees (resp.~weighted trees) with edge lengths. Additionally, one does not need all the information in $\bT$ to determine a particular component of $\bR_r(\bT)$.
\begin{remark}
\label{measurability_pruning}
Let $\bT=(t,(l_u)_{u\in t},(w_v)_{v\in\partial t})\in \overline{\bbT}_{\mathrm{w}}$ and set $T=(t,(l_u)_{u\in t})\in \overline{\bbT}$.
\begin{longlist}
\item[(i)] The Horton pruning is a specific case of a weighted Horton pruning as $R_1(\bT)=R(T)$. Indeed, for any $u\in t$, $\tHS(\theta_u\bT)\geq 1$ if and only if $\HS(\theta_u T)\geq 1$ by Proposition~\ref{regularized_HT}. In particular, (\ref{HS_pruning}) follows from $\tHS(\bR_1(\bT))=\tHS(\bT)-1$, as given by Proposition~\ref{recursive_pruning}.
\item[(ii)] We can write $\Ske(\bR_r(\bT))=\Ske\circ\bR_r(t,(1)_{u\in t},(w_v)_{v\in\partial t})$, which is a measurable function of $\Ske(\bT)$. Indeed, the lengths of $\bT$ are only used to compute those of $\bR_r(\bT)$.
\item[(iii)] If $r\in[0,1]$ then $R_r(\bT)$ is a measurable function of $(t,(l_u)_{u\in t},(\I{w_v\geq r})_{v\in\partial t})$. Indeed, we do not need to compute the weights to know $R_r(\bT)$, only to observe which vertices $u\in t$ are such that $\tHS(\theta_u\bT)\geq r$. More precisely, since $r\leq1$, we can check using (\ref{nochain_weighted}) that $R_r(\bT)=R_r(t,(l_u)_{u\in t},(r\I{w_v\geq r})_{v\in\partial t})$.
\item[(iv)] The weighted tree with edge lengths $\bR_1(\bT)$ can be expressed as a measurable function of $(t,(l_u)_{u\in t\backslash\partial t},(w_v)_{v\in \partial t})$. Indeed, $\tHS(\theta_v\bT)<1$ for all $v\in\partial t$, so the parental edges of leaves are erased and  $\bR_1(\bT)=\bR_1(t,(l_u^*)_{u\in t},(w_v)_{v\in\partial t})$ where $l_u^*=l_u\I{u\notin\partial t}+\I{u\in\partial t}$.
\end{longlist}
\end{remark}
We are now ready to present the desired invariance of a $\GWwl_\alpha$-weighted tree with edge lengths $\fTau=(\tau,(L_u)_{u\in\tau},(W_v)_{v\in\partial\tau})$ (see Definition~\ref{GWwldef}) by weighted Horton pruning.
\begin{theorem}
\label{invariance}
Let $r\in \bbR_+$ and let $\fTau$ be as above. Set $\cT=(\tau,(L_u)_{u\in\tau})$.
\begin{longlist}
\item[(i)] The law of $R_r(\fTau)$ under $\P(\, \cdot\, |\, \tHS(\fTau)\geq r)$ is the same as the law of $\delta^r\cdot \cT$ under $\P$.
\item[(ii)] The law of $\Ske(\bR_r(\fTau))$ under $\P(\, \cdot\, |\, \tHS(\fTau)\geq r)$ is equal to that of $\Ske(\fTau)$ under $\P$.
\item[(iii)] The law of $\bR_1\circ \bR_r(\fTau)$ under $\P(\, \cdot\, |\, \tHS(\fTau)\geq r+1)$ is the same as the law of $\delta^r\cdot \bR_1(\fTau)$ under $\P(\, \cdot\, |\, \tHS(\fTau)\geq 1)$.
\end{longlist}
\end{theorem}

Before the proof, let us explain why the theorem cannot state that the law of $\bR_r(\fTau)$ under $\P(\, \cdot\, |\, \tHS(\fTau)\geq r)$ is the same as that of $\delta^r\cdot\fTau$ under $\P$. If a leaf of $\bR_r(\fTau)$ has a long parental edge then there is a higher probability that this edge comes from the fusion of many edges of $\bT$, which means its weight is more likely to be large. Thus, the weight and the length of the parental edge of such a leaf are not independent. We have solved this issue in $(iii)$ by erasing the leaves with the map $\bR_1$ before comparing $\bR_r(\fTau)$ and $\delta^r\cdot\fTau$.

\begin{proof}
By Remark~\ref{measurability_pruning} $(iv)$ and Proposition~\ref{recursive_pruning}, we only need to show that the law of $F_i(\bR_r(\fTau))$ under $\P(\, \cdot\, |\, \tHS(\fTau)\! \geq\! r)$ is equal to that of $F_i(\delta^r\cdot\fTau)$ under $\P$, for $i\!\in\!\{1,2\}$, where \[F_1(\bT)=(t,(l_u)_{u\in t})\quad\text{ and }\quad F_2(\bT)=(t,(l_u)_{u\in t\backslash\partial t},(w_v)_{v\in\partial t}).\] Note that $K(\cT)\neq 1$ by (\ref{stable_offspring_explicit}), and $K(R_r(\fTau))\neq 1$ by Proposition~\ref{recursive_pruning}. Let $k\in\bbN$ with $k\geq 2$, $\lambda>0$, $x\in(0,1)$, and let $g_1,\ldots,g_k:\overline{\bbT}_{\mathrm{w}}\longrightarrow\mathbb{R}$ be bounded measurable functions. Let us set $G(\bT)=\I{K(T)=k}e^{-\lambda L(T)}\prod_{j=1}^k g_j(\theta_{(j)} \bT)$ for all $\bT\in\overline{\bbT}_{\mathrm{w}}$. Then, by induction on the height and since the laws of nonnegative random variables are determined by
their Laplace transforms, to reach our goal we only need to check the three following identities:
\begin{align}
\label{inv_1}
\P(K(\cT)=0\, ;\, \tHS(\fTau)\leq x)&=\P\big(K(R_r(\fTau))=0\, ;\, \tHS(\bR_r(\fTau))\leq x\, \big|\, \tHS(\fTau)\geq r\big),\\
\label{inv_2}
\E\big[\un_{\{K(\cT)=0\}}e^{-\lambda 
\delta^r L(\cT)}\big]&=\E\big[\un_{\{K(R_r(\fTau))=0\}}e^{-\lambda L(R_r(\fTau))}\, \big|\, \tHS(\fTau)\geq r\big],\\
\label{inv_3}
\E\big[G(\bR_r(\fTau))\, \big|\, \tHS(\fTau)\geq r\big]&=\E\big[\un_{\{K(\cT)=k\}}e^{-\lambda \delta^r L(\cT)}\big]\!\!\prod_{j=1}^k\! \E\big[g_j(\bR_r(\fTau))\, \big|\, \tHS(\fTau)\geq r\big]\! .
\end{align}

Since $x<1$, we have $\P(K(\cT)=0\,;\, \tHS(\fTau)\leq x)=\P(\tHS(\fTau)\leq x)$ by (\ref{nochain_weighted}). Again by (\ref{nochain_weighted}) and since $\tHS(\bR_r(\fTau))=\tHS(\fTau)-r$ when $\tHS(\fTau)\geq r$ by Proposition~\ref{recursive_pruning}, we find
\[\P(\tHS(\fTau)\geq r\, ;\, K(R_r(\fTau))=0\, ;\, \tHS(\bR_r(\fTau))\leq x)=\P(r\leq \tHS(\fTau)\leq x+r).\]
But Proposition~\ref{law_HT_stable} states that law of $\tHS(\fTau)$ is exponential, so the identity (\ref{inv_1}) follows.

Next set $\chi_2(\bT)=\un_{\{\tHS(\bT)\geq r\}}\un_{\{K(R_r(\bT))=0\}}e^{-\lambda L(R_r(\bT))}$. Proposition~\ref{recursive_pruning} entails that
\[\chi_2(\fTau)=e^{-\lambda L(\cT)}\bigg(\sum_{i=1}^{K(\cT)}\chi_2(\theta_{(i)}\fTau)\!\!\!\! \prod_{\substack{1\leq j\leq K(\cT) \\ j\neq i}}\!\!\!\!\!\! \un_{\{\tHS(\theta_{(j)}\fTau)<r\}}+\prod_{i=1}^{K(\cT)}\!\!\un_{\{\tHS(\theta_{(i)}\fTau)< r\}}-\I{\tHS(\fTau)< r}\bigg)\, ;\]
the first term of the righ-hand side corresponds to the case $(b1)$ in Proposition~\ref{recursive_pruning} where exactly one child $(i)$ has $\tHS(\theta_{(i)}\fTau)\geq r$, and the other terms correspond to the case $(a)$ where $\tHS(\fTau)\geq r$ but every child $(j)$ has $\tHS(\theta_{(j)}\fTau)\geq r$. Recall that $\varphi_\alpha$ stands for the generating function of the law $\mu_\alpha$ of $K(\cT)$. Using the above identity, $\E[e^{-\lambda L(\cT)}]\!=\!\tfrac{1}{1+\lambda}$, and the branching property of $\GWwl_\alpha$-weighted trees with edge lengths, we compute that
\begin{equation}
\label{inv_2_step}\E[\chi_2(\fTau)]=\frac{\varphi_\alpha\left(\P(\tHS(\fTau)<r)\right)-\P(\tHS(\fTau)< r)}{1+\lambda-\varphi_\alpha'\left(\P(\tHS(\fTau)< r)\right)}.
\end{equation}
We know from Proposition~\ref{law_HT_stable} that $\P(\tHS(\fTau)<r)=1-e^{-\gamma r}$, and from (\ref{stable_offspring}) that $\varphi_\alpha(s)=s+\tfrac{1}{\alpha}(1-s)^\alpha$ for $s\in[0,1]$. Since $\mu_\alpha(0)=1/\alpha$, (\ref{inv_2}) follows from (\ref{inv_2_step}) and from
\begin{equation}
\label{step_inv_2and3}
\frac{\P(\tHS(\fTau)\geq r)^\alpha}{1+\lambda-\varphi_\alpha'\left(\P(\tHS(\fTau)< r)\right)}=\frac{e^{-\gamma \alpha r}}{\lambda+e^{-\gamma(\alpha-1)r}}=\P(\tHS(\fTau)\geq r)\E\big[e^{-\lambda 
\delta^r L(\cT)}\big].
\end{equation}

Finally, set $\chi_3(\bT)\!=\!\un_{\{\tHS(\bT)\geq r\}}G(\bR_r(\bT))$ and $\chi_j^*(\bT)\!=\!\un_{\{\tHS(\bT)\geq r\}}g_j(\bR_r(\bT))$ for $1\!\leq\! j\!\leq\! k$. Once again, Proposition~\ref{recursive_pruning} yields that $\chi_3(\fTau)e^{\lambda L(\cT)}$ is equal to
\[\sum_{i=1}^{K(\cT)} \chi_3(\theta_{(i)}\fTau) \!\!\!\!\! \!\!\!\!\! \prod_{\substack{\ \quad 1\leq j\leq K(\cT) \\ j\neq i}}\!\!\!\!\! \!\!\!\!\! \un_{\{\tHS(\theta_{(j)}\fTau)<r\}} \ + \!\!\!\!  \sum_{\substack{ 1\leq i_1<\ldots\\
\ldots <i_k\leq K(\cT)}} \!\!\!\!\! \Big(\prod_{j=1}^k \chi_j^*(\theta_{(i_j)}\fTau)\Big)\!\!\!\!\! \!\!\!\!\!\prod_{\substack{\ \quad 1\leq i\leq K(\cT) \\ \ \quad i\neq i_1,\ldots,i_k}}\!\!\!\!\! \!\!\!\!\! \un_{\{\tHS(\theta_{(i)}\fTau)<r\}}\, ;\]
the first and second terms respectively correspond to the cases $(b2)$ and $(c)$. From (\ref{derive_generating}), we have $\varphi_\alpha^{(k)}(s)/k!=\E\big[\I{K(\cT)\geq k}\binom{K(\cT)}{k}s^{K(\cT)-k}\big]=\mu_\alpha(k)(1-s)^{\alpha-k}$ for all $s\!\in\![0,1)$ . Similarly to how we derived (\ref{inv_2_step}), we then obtain
\[\E[\chi_3(\fTau)]=\frac{\P(\tHS(\fTau)\geq r)^\alpha}{1+\lambda-\varphi_\alpha'\left(\P(\tHS(\fTau)< r)\right)} \mu_\alpha(k)\prod_{j=1}^k\E\big[g_j(\bR_r(\fTau))\ \big|\ \tHS(\fTau)\geq r\big],\]
which implies the identity (\ref{inv_3}) thanks to (\ref{step_inv_2and3}).
\end{proof}

\subsection{Some applications}

Here, we give some useful consequences of Theorem~\ref{invariance}. Recall Definition~\ref{embedding_discrete} of embeddings between two trees.

\begin{corollary}
\label{croissance_selon_H}
Let $A:\bbT\longrightarrow\bbR$ have the following monotonicity property: for all trees $t,t'\in\bbT$, if there is an embedding $\psi:t'\to t$ then $A(t')\leq A(t)$. Let $\ftau=(\tau,(W_v)_{v\in\partial\tau})$ be a $\GWw_\alpha$-weighted tree. For all $a\in\bbR$ and for all $x,y\in \bbR_+\backslash\mathbb{N}$ such that $x\leq y$,
\begin{align}
\label{croissance_selon_H_1}
\P\big(A(\tau)\geq a\ \big|\ \tHS(\ftau)=x\big)&\leq \P\big(A(\tau)\geq a\ \big|\ \tHS(\ftau)=y\big),\\
\label{croissance_selon_H_2}
\P\big(A(\tau)\geq a\ \big|\ \tHS(\ftau)=x\big)&\leq \P\big(A(\tau)\geq a\ \big|\ \HS(\tau)=\lfloor x\rfloor+1\big),\\
\label{croissance_selon_H_3}
\P\big(A(\tau)\geq a\ \big|\ \HS(\tau)=\lfloor y\rfloor\big)&\leq \P\big(A(\tau)\geq a\ \big|\ \tHS(\ftau)=y+1\big).
\end{align}
In particular, the result applies for the size $A(t)=\#t$ and for the height $A(t)=|t|$.
\end{corollary}

\begin{proof}
We assume that $\ftau\!=\!\Ske(\fTau)$, where $\fTau$ is a $\GWwl_\alpha$-weighted tree with edge lengths. Let $r\!\in\![0,1]$ and $x\!>\!0$ be such that $x,x+r\!\notin\!\mathbb{N}$. By monotonicity and Definition~\ref{pruning},
\[\P\big(A(\Ske(R_r(\fTau)))\geq a\, ;\, |\tHS(\fTau)-(x+r)|<\varepsilon\big)\, \leq\, \P\big(A(\tau)\geq a\, ;\, |\tHS(\fTau)-(x+r)|<\varepsilon\big)\]
for all $\varepsilon\in (0,x)$. As $\tHS(\Ske(\bR_r(\fTau)))=\tHS(\fTau)-r$ when $\tHS(\fTau)\geq r$, Theorem~\ref{invariance} $(ii)$ entails
\[\P(\tHS(\fTau)\geq r)\P\big(A(\tau)\geq a\, ;\, |\tHS(\fTau)-x|<\varepsilon\big)\, \leq\, \P\big(A(\tau)\geq a\, ;\, |\tHS(\fTau)-(x+r)|<\varepsilon\big).\]
Moreover, the law of $\tHS(\fTau)$ is exponential by Proposition~\ref{law_HT_stable} so dividing the above inequality by $\P(|\tHS(\fTau)-(x+\varepsilon)|<\varepsilon)$ and letting $\varepsilon\rightarrow0^+$ gives (\ref{croissance_selon_H_1}) with $y=x+r$, by Definition~\ref{cond_H=x_def}. We generalize that for any $x,y$ by induction on $\lfloor y-x\rfloor$. Integrating (\ref{croissance_selon_H_1}) with respect to $y$ against the conditional law of $\tHS(\ftau)$ given $\lfloor x\rfloor+1\leq\tHS(\ftau)<\lfloor x\rfloor+2$ yields (\ref{croissance_selon_H_2}), by Proposition~\ref{cond_H=x_prime}. We get (\ref{croissance_selon_H_3}) similarly by integrating (\ref{croissance_selon_H_1}) with respect to $x$.
\end{proof}

Like Lemma~\ref{cond_domination}, Corollary~\ref{croissance_selon_H} allows us to control $\P(\dd \tau\, |\, \tHS(\ftau)=x)$ using estimates of the nondegenerate conditional law $\P(\dd\tau\, |\, \HS(\tau)=n)$. Now, recall from (\ref{height_process_length_def}) the height function $\Hl(T)$ of a tree with edge lengths $T$. The following proposition shows that the height function of a $\GWwl_\alpha$-weighted tree with edge lengths does not change much after $r$-Horton pruning. Indeed, if $r\in[0,1]$ then only some leaves are erased, and in a uniform manner.

\begin{proposition}
\label{estimee_pruning}
Recall the Skorokhod distance $\mathtt{d}_{\mathrm{S}}$ from (\ref{skorokhod_distance}). Let $r\in[0,1]$ and $\lambda>0$. Let $\fTau=(\tau,(L_u)_{u\in\tau},(W_v)_{v\in\partial\tau})$ be a $\GWwl_\alpha$-weighted tree with edge lengths and set $\cT=(\tau,(L_u)_{u\in\tau})$. Define two random càdlàg functions with compact support $X$ and $Y$ by setting
\[X_s=\lambda^{1-1/\alpha}\Hl_{s/\lambda}(\cT)\quad\text{ and }\quad Y_s=\I{\tHS(\fTau)\geq r}\lambda^{1-1/\alpha}\Hl_{e^{-\gamma r}s/\lambda}(R_r(\fTau))\quad\text{ for all $s\in\bbR_+$.}\]
Then, there are two constants $C,c\in(0,\infty)$ that only depend on $\alpha$ such that for all $n\geq 1$,
\[\P\big(\#\tau\leq n\, ;\, |\tau|\leq n^{1-1/\alpha}\, ;\, \mathtt{d}_{\mathrm{S}}(X,Y)\geq C(\lambda\sqrt{n})^{1-1/\alpha}\ln n+C\lambda\sqrt{n}\ln n\big)\, \leq\,  Cne^{-c(\ln n)^2}.\]
\end{proposition}

\begin{proof}
Let $(\xi_i,L_i,W_i)_{i\in\bbN}$ be independent RVs such that for all $i\!\in\!\bbN$, the law of $\xi_i+1$ is $\mu_\alpha$, the law of $L_i$ is exponential with mean $1$, and the law of $W_i$ is $\mathsf{FExp}(\gamma)$. Recall from (\ref{depth-first_exploration}) the depth-first exploration $u(\tau)$. By (\ref{couplage_tree_walk}) and Definition~\ref{GWwldef}, we can assume that \[\xi_i=k_{u_{i}(\tau)}(\tau)-1,\ \quad L_i=L_{u_{i}(\tau)},\quad\text{ and }\quad(\xi_i=-1)\Longrightarrow (W_i=W_{u_{i}(\tau)})\] 
for all $0\! \leq\! i\leq\! \#\tau\!-\!1$. We now define a sequence of stopping times $(\mathtt{J}_i)_{i\in\bbN}$ and a random integer $\mathtt{N}$ by setting $\mathtt{J}_{-1}=-1$ and for all $i\in\bbN$,
\begin{equation}
\label{stopping_times_J}
\mathtt{J}_{i}=\inf\big\{j> \mathtt{J}_{i-1}\ :\ \xi_j\neq -1\text{ or }W_j\geq r\big\},\quad\text{ and }\quad\mathtt{N}=\inf\{j\geq 0\ :\ \mathtt{J}_j\geq \#\tau\}.
\end{equation}
Since $r\leq 1$ and $\mu_\alpha(1)=0$, we observe by (\ref{nochain_weighted}) that $u_{\mathtt{J}_0}(\tau),u_{\mathtt{J}_1}(\tau),\ldots,u_{\mathtt{J}_{\mathtt{N}-1}}(\tau)$ are exactly the vertices $u\in\tau$ such that $\tHS(\theta_u\ftau)\geq r$, listed in lexicographic order. Then, for $j\in\bbN$, set
\[s_j(X)=\lambda\sum_{\iota=0}^{\mathtt{J}_{j-1}}L_\iota=\lambda\sum_{i=0}^{j-1}\sum_{\iota=\mathtt{J}_{i-1}+1}^{\mathtt{J}_i}L_\iota\quad\text{ and }\quad s_j(Y)=\lambda e^{\gamma r}\sum_{i=0}^{j-1} L_{\mathtt{J}_i}.\]
Note that $s_0(X)=s_0(Y)=0$. Thanks to Proposition~\ref{pruning_height} and (\ref{lifetime_height_length}), we see that $\zeta(Y)=s_{\mathtt{N}}(Y)$, $s_{\mathtt{N}}(X)\leq \zeta(X)< s_{\mathtt{N}+1}(X)$, and that if $s\in\left[s_j(Y),s_{j+1}(Y)\right)$ with $0\leq j\leq\mathtt{N}-1$ then
\begin{equation}
\label{estimee_pruning_tool1}
\lambda^{1-1/\alpha}\sum_{v\prec u_{\mathtt{J}_j}(\tau)}L_v=Y_{s_j(Y)}\, \leq\,  Y_s\, \leq\,  Y_{s_{j+1}(Y)-}=\lambda^{1-1/\alpha}\sum_{v\preceq u_{\mathtt{J}_j}(\tau)}L_v.
\end{equation}
We write $\Delta(f)=\sup_{s\geq 0}|f(s)-f(s-)|$ for any càdlàg function $f$. We claim that 
\begin{equation}
\label{estimee_pruning_tool3}
\mathtt{d}_{\mathrm{S}}(X,Y)\leq 6\!\max_{1\leq j\leq\mathtt{N}+1}\!\!|s_j(X)-s_j(Y)|+3e^{\gamma}\lambda\max_{0\leq j\leq \mathtt{N}}L_{\mathtt{J}_j}+2\lambda^{1-1/\alpha}\max_{u\in\tau}L_u+\Delta(Y).
\end{equation}

\emph{Proof of (\ref{estimee_pruning_tool3}).} Let $1\leq j\leq\mathtt{N}-1$ and let $u\in\tau$ such that $u_{\mathtt{J}_{j-1}}(\tau)<u\leq u_{\mathtt{J}_j}(\tau)$. By definition of the lexicographic order, we remark that $u_{\mathtt{J}_{j-1}}(\tau)\wedge u_{\mathtt{J}_j}(\tau)$ is an ancestor of $u_{\mathtt{J}_{j-1}}(\tau)\wedge u$. Moreover, it either holds $u=u_{\mathtt{J}_j}(\tau)$, or $\tHS(\theta_u\ftau)<r$ and $u$ is a leaf of $\tau$. The contrapositive of this argument entails that $\overleftarrow{u}\leq u_{\mathtt{J}_{j-1}}(\tau)$, and it further follows that $\overleftarrow{u}=u\wedge u_{\mathtt{J}_{j-1}}(\tau)$. In particular, the parent of $u_{\mathtt{J}_j}(\tau)$ is the most recent common ancestor of $u_{\mathtt{J}_{j-1}}(\tau)$ and $u_{\mathtt{J}_j}(\tau)$. Hence, we obtain $\overleftarrow{u_{\mathtt{J}_j}(\tau)}\preceq \overleftarrow{u}\preceq u_{\mathtt{J}_{j-1}}(\tau)$, and (\ref{estimee_pruning_tool1}) then yields that
\begin{equation}
\label{estimee_pruning_tool2}
Y_{s_j(Y)}\leq\lambda^{1-1/\alpha}\sum_{v\prec u}L_v\leq Y_{s_j(Y)-}.
\end{equation}
Now, we choose an increasing and bijective $\psi:\bbR_+\longrightarrow\bbR_+$ such that $\psi(s_j(Y))=s_j(X)$ for all $0\! \leq\! j\!\leq\! \mathtt{N}$ and $\psi(s_{\mathtt{N}}(Y)+s)=s_{\mathtt{N}}(X)+s$ for all $s\!\geq\! 0$. The facts preceding (\ref{estimee_pruning_tool1}) give that \[|\zeta(X)-\zeta(Y)|\leq 2|s_{\mathtt{N}}(X)-s_{\mathtt{N}}(Y)|+|s_{\mathtt{N}+1}(X)-s_{\mathtt{N}+1}(Y)|+\lambda e^{\gamma r}L_{\mathtt{J}_{\mathtt{N}}}.\] If $s\in\left[s_j(Y),s_{j+1}(Y)\right)$ with $1\leq j< \mathtt{N}$, then $\psi(s)\in\left[s_j(X),s_{j+1}(X)\right)$ so 
\[|\psi(s)-s|\leq 2|s_j(X)-s_j(Y)|+|s_{j+1}(X)-s_{j+1}(Y)|+2\lambda e^{\gamma r}L_{\mathtt{J}_j}.\]
Also, there is $u\in\tau$ such that $u_{\mathtt{J}_{j-1}}(\tau)\!<\!u\!\leq\! u_{\mathtt{J}_j}(\tau)$ and $\sum_{v\prec u}L_v\leq \lambda^{1/\alpha-1}X_{\psi(s)}<\sum_{v\preceq u}L_v$, by definition (\ref{height_process_length_def}) of the height function, and (\ref{estimee_pruning_tool1}) and (\ref{estimee_pruning_tool2}) then entail that
\[|X_{\psi(s)}-Y_s|\leq \lambda^{1-1/\alpha}L_u+\lambda^{1-1/\alpha}L_{u_{\mathtt{I}_j}(\tau)}+|Y_{s_j(Y)}-Y_{s_j(Y)-}|.\] We treat the cases where $s\in\left[0,s_1(Y)\right)$ or $s\geq s_{\mathtt{N}}(Y)$ similarly. This completes the proof.\cq
\smallskip

From (\ref{stopping_times_J}), we observe that the sequence $(\mathtt{J}_j)_{j\in\bbN}$ is independent of $(L_i)_{i\in\bbN}$, which is a sequence of independent and identically distributed random variables, thus $(L_{\mathtt{J}_j})_{j\in\bbN}$ has the same law as $(L_j)_{j\in\bbN}$. Moreover, the random variables $(\mathtt{J}_{j}-\mathtt{J}_{j-1}-1)_{j\in\bbN}$ are independent and geometric with parameter $\P(\xi_0\!\neq\! -1\text{ or }W_0\!\geq\! r)=\P(\tHS(\fTau)\geq r)=e^{-\gamma r}$, and jointly independent from $(L_i)_{i\in\bbN}$. A well-known fact then asserts that the random sequence $\big(e^{-\gamma r}\sum_{\mathtt{J}_{j-1}<\iota\leq \mathtt{J}_j}L_{\iota}\big)_{j\in\bbN}$ has also the same law as $(L_j)_{j\in\bbN}$. For all $j\in\bbN$, we thus have
\[\P\big(|s_j(X)-s_j(Y)|\geq e^{\gamma r}\lambda\sqrt{n}\ln n\big)\leq 2\, \P\bigg(\Big|i-\sum_{i=0}^{j-1} L_i\Big|\geq \sqrt{n}\ln n\bigg).\]
Note from (\ref{stopping_times_J}) that $\mathtt{N}\leq\#\tau$ almost surely, so a Chernoff bound (Lemma~\ref{hoeffding}) entails that
\begin{equation}
\label{estimee_pruning_tool4}
\P\big(\#\tau\leq n\, ;\, \max_{1\leq j\leq \mathtt{N}+1}|s_j(X)-s_j(Y)|\geq e^{\gamma r}\lambda\sqrt{n}\ln n\big)\, \leq\,  8ne^{-c_{\mathrm{uni}}(\ln n)^2}
\end{equation}
for all $n$ large enough. Then, a simple union bound together with $\mathtt{N}\leq\#\tau$ gives us
\begin{equation}
\label{estimee_pruning_tool5}
\P\big(\#\tau\leq n ; \lambda\max_{0\leq j\leq\mathtt{N}}L_{\mathtt{I}_j}+\lambda^{1-1/\alpha}\max_{u\in\tau}L_u\geq \lambda(\ln n)^2\!+\!\lambda^{1-1/\alpha}(\ln n)^2\big) \leq 3n e^{-(\ln n)^2}\!\!.
\end{equation}
By Definition~\ref{pruning} of weighted Horton pruning, we see that if $\tHS(\fTau)\geq r$, so that $R_r(\fTau)$ is defined, then $\#R_r(\fTau)\leq \#\tau$ and $|R_r(\fTau)|\leq |\tau|$. Hence,
\begin{align*}
\chi&:=\P\big(\#R_r(\fTau)\!\leq\! n\, ;\, |R_r(\fTau)|\!\leq\! n^{1-1/\alpha}\, ;\, \delta^{-r}\Delta(Y)\geq  2C(\lambda\sqrt{n})^{1-1/\alpha}\ln n\! +\! 2C\lambda\sqrt{n}\ln n\big)\\
&\geq\P\big(\#\tau\leq n\, ;\, |\tau|\leq n^{1-1/\alpha}\, ;\, \delta^{-r}\Delta(Y)\geq  2C(\lambda\sqrt{n})^{1-1/\alpha}\ln n+2C\lambda\sqrt{n}\ln n\big)
\end{align*}
for any $C>0$. Also, if $\tHS(\fTau)<r$ then $\Delta(Y)=0$. By observing from (\ref{lifetime_height_length}) that we have $\Delta(\Hl(\delta^r\cdot\cT))=\delta^r\Delta(\Hl(\cT))$, Theorem~\ref{invariance} $(i)$ then yields
\[\chi=\P(\tHS(\fTau)\geq r)\, \P\big(\#\tau\leq n\, ;\, |\tau|\leq n^{1-1/\alpha}\, ;\, \Delta(X)\geq 2C(\lambda\sqrt{n})^{1-1/\alpha}\ln n+2C\lambda\sqrt{n}\ln n\big).\]
Recalling the expression (\ref{skorokhod_distance}) of $\mathtt{d}_{\mathrm{S}}$, we see that if $f$ and $g$ are respectively càdlàg and continuous with compact support, then $\Delta(f)\leq \Delta(g)+ 2\mathtt{d}_{\mathrm{S}}(f,g)=2\mathtt{d}_{\mathrm{S}}(f,g)$. Therefore, according to Proposition~\ref{height_process_skeleton}, there are $C,c\in(0,\infty)$ that only depend on $\alpha$ such that
\[\P\big(\#\tau\leq n\, ;\, |\tau|\leq n^{1-1/\alpha}\, ;\, \Delta(Y)\geq  2C(\lambda\sqrt{n})^{1-1/\alpha}\ln n+2C\lambda\sqrt{n}\ln n\big)\, \leq\,  Cne^{-c(\ln n)^2}\!\!\]
because the height function (\ref{height_process_def}) of a tree (without edge lengths) is continuous. Combining this last estimate with the inequalities (\ref{estimee_pruning_tool3}), (\ref{estimee_pruning_tool4}), and (\ref{estimee_pruning_tool5}) completes the proof.
\end{proof}

\section{Proof of Theorem~\ref{scaling_limit_HS_intro}}
\label{scaling_limit_HS_section}

In this section, we extensively use the topological tools and the notation presented in Section~\ref{topo_function}, so recall the spaces $\mathcal{C}_{\mathrm{K}}$ and $\mathcal{D}_{\mathrm{K}}$, the Skorokhod distance $\mathtt{d}_{\mathrm{S}}$ from (\ref{skorokhod_distance}), the Prokhorov metric $\rho_{\mathrm{S}}$ associated with $\mathtt{d}_{\mathrm{S}}$ from (\ref{prokhorov_metric}), and Notation~\ref{notation_DK} identifying any càdlàg function $f$ with finite lifetime $\zeta(f)$ with the pair $(f,\zeta(f))\in \mathcal{D}_{\mathrm{K}}$. Furthermore, recall the height function $H(t)$ of a tree $t$ from (\ref{height_process_def}), and the law of a $\GW_\alpha$-tree conditioned on its weighted Horton--Strahler number from Definition~\ref{cond_H=x_def}. Given a $\GWw_\alpha$-weighted tree $\ftau$, we shall denote by $\tau$ its underlying tree. Our goal is to prove the following.
\begin{theorem}
\label{scaling_limit_HT}
Let $\ftau$ be a $\GWw_\alpha$-weighted tree. The law of $\big(e^{-\gamma(\alpha-1)x}H_{e^{\gamma\alpha x}s}(\tau)\big)_{s\geq 0}$ under $\P(\dd\tau\, |\, \tHS(\ftau)=x)$ weakly converges on $\mathcal{C}_{\mathrm{K}}$, when $x\in\bbR_+\backslash\bbN$ tends to $\infty$.
\end{theorem}
Theorem~\ref{scaling_limit_HT} implies Theorem~\ref{scaling_limit_HS_intro} by Proposition~\ref{height_contour_tree} and Skorokhod's representation theorem. Moreover, the estimates (\ref{min_hauteur_selon_H}) and (\ref{maj_hauteur_selon_H}), together with Corollary~\ref{croissance_selon_H}, entail that the limit law is not degenerate: see Section~\ref{geometric_study} for details. We even give a rate of convergence for Theorem~\ref{scaling_limit_HT} via the following theorem, whose proof we postpone.

\begin{theorem}
\label{cauchy_criterion}
Let $\ftau$ be a $\GWw_\alpha$-weighted tree. For all $x\in \bbR_+\backslash\mathbb{N}$, we denote by $\nu_x$ the law on $\mathcal{C}_{\mathrm{K}}$ of $\big(e^{-\gamma(\alpha-1)x}H_{e^{\gamma\alpha x}s}(\tau)\big)_{s\geq 0}$ under $\P(\dd\tau\, |\, \tHS(\ftau)=x)$. There are two constants $C,c\in(0,\infty)$ that only depend on $\alpha$ such that $\rho_{\mathrm{S}}(\nu_{x},\nu_{y})\leq C e^{-c\min(x,y)}$ for all $x,y\in \bbR_+\backslash\mathbb{N}$.
\end{theorem}
One might be tempted to see Theorem~\ref{scaling_limit_HT} as a direct consequence of Theorem~\ref{cauchy_criterion}, arguing the function $x\longmapsto\nu_x$ is Cauchy so it must converge. This argument fails since $\rho_{\mathrm{S}}$ is not a complete metric (see Remark~\ref{lack_complete}); but this is not difficult to overcome.

\begin{proof}[Proof of Theorem~\ref{scaling_limit_HT} from Theorem~\ref{cauchy_criterion}]
We keep the notation of Theorem~\ref{cauchy_criterion}. We claim that we only need to show that the family $(\nu_x)_{x\in\bbR_+\backslash\mathbb{N}}$ is tight on $\mathcal{C}_{\mathrm{K}}$. Indeed, there would then exist an increasing sequence of points $x_n\in\bbR_+\backslash\mathbb{N}$ that tends to $\infty$ such that $(\nu_{x_n})$ weakly converges on $\mathcal{C}_{\mathrm{K}}$ by Prokhorov's theorem. Together with Proposition~\ref{prokhorov_prop}, Theorem~\ref{cauchy_criterion} would then yield that $(\nu_x)_{x\in\bbR_+\backslash\mathbb{N}}$ converges with respect to $\rho_{\mathrm{S}}$ when $x\to\infty$. As $\nu_x$ are laws on $\mathcal{C}_{\mathrm{K}}$, they would weakly converge on $\mathcal{C}_{\mathrm{K}}$ by Proposition~\ref{DK_prop}.

We thus aim to apply Proposition~\ref{criterion_tightness}. Firstly, we have Proposition~\ref{criterion_tightness} $(a)$ because $H_0(\tau)=0$.
Secondly, we know that $\zeta(H(\tau))=\I{\#\tau\geq 2}\#\tau$ from (\ref{lifetime_height}). Next, we apply Corollary~\ref{croissance_selon_H} followed by Markov's inequality to find that for all $m\geq 1$ and $x\in\bbR_+\backslash\mathbb{N}$,
\[\nu_x(\zeta\geq m)\leq \P\big(\#\tau\geq e^{\gamma\alpha x}m\ \big|\ \HS(\tau)=\lfloor x\rfloor+1\big)\leq \tfrac{1}{m} e^{-\gamma\alpha x}\E\big[\#\tau\ \big|\ \HS(\tau)=\lfloor x\rfloor+1\big].\]
Therefore, (\ref{first_moment_size}) yields Proposition~\ref{criterion_tightness} $(b)$. It only remains to show that for all $\varepsilon>0$,
\begin{equation}
\label{tight_step_goal}
\lim_{\eta\rightarrow 0^+}\limsup_{x\rightarrow\infty,x\notin\bbN}\nu_x(\omega_\eta\geq 2\varepsilon)=0.
\end{equation}

As a step towards this, we first work under the discrete conditioning $\{\HS(\tau)=n\}$. For $n\in\mathbb{N}$, denote by $\nu_n'$ the law on $\mathcal{C}_{\mathrm{K}}$ of $\big(e^{-\gamma(\alpha-1)n}H_{e^{\gamma\alpha n}s}(\tau)\big)_{s\geq 0}$ under $\P(\, \cdot\, |\, \HS(\tau)=n)$. Set $N_n=e^{\gamma(\alpha-1)n}$ and note that $N_n^{\alpha\beta}=e^{\gamma\alpha n}$, where recall from (\ref{alpha_gamma_delta}) that $\beta(\alpha\!-\!1)\!=\!1$. For $h\!>\!0$ and $E_h=\{(f,\ell)\in\mathcal{C}_{\mathrm{K}}:\omega_\eta(f)\!\geq\!\varepsilon,\sup f\!\geq\! h\}$, using $\sup H(\tau)=|\tau|$ from (\ref{lifetime_height}), we have
\[\nu_n'(E_h)\, \leq\,  \frac{\P(|\tau|\geq h e^{\gamma(\alpha-1)n})}{\P(\HS(\tau)=n)}\, \P\Big(\omega_\eta(\tfrac{1}{N_n}H_{N_n^{\alpha\beta}s}(\tau)\,;\, s\geq 0)\geq \varepsilon\ \Big|\ |\tau|\geq h N_n\Big).\]
We apply (\ref{cv_cond_high}) in Theorem~\ref{duquesne_scaling_atleast}, in conjunction with (\ref{equiv_hauteur}) and (\ref{equiv_HS}), and we eventually obtain
\[\lim_{\eta\rightarrow 0+}\limsup_{n\rightarrow\infty}\nu_n'(\omega_\eta\geq\varepsilon)\leq \limsup_{n\rightarrow\infty}\P\big(|\tau|\leq h e^{\gamma(\alpha-1)n}\ \big|\ \HS(\tau)=n\big)\]
for all $\varepsilon,h>0$. By taking $h\to 0^+$, the estimate (\ref{min_hauteur_selon_H}) entails that for all $\varepsilon>0$, it holds that
\begin{equation}
\label{tight_step_preliminar}
\lim_{\eta\rightarrow 0^+}\limsup_{n\rightarrow\infty}\nu_n'(\omega_\eta\geq\varepsilon)=0.
\end{equation}
Now let $\varepsilon,\eta>0$ and let $n\in\mathbb{N}$ be large enough that $2Ce^{-cn}\leq \max(\varepsilon,\eta)$, where $C$ and $c$ are as in Theorem~\ref{cauchy_criterion}. For all $x,y\in(n,n+1)$, we get $\nu_x(\omega_\eta\geq 2\varepsilon)\leq \nu_y(\omega_{2\eta}\geq\varepsilon)+Ce^{-cn}$ by Theorem~\ref{cauchy_criterion}, and then $\nu_x(\omega_\eta\geq 2\varepsilon)\leq \nu_y\big(e^{\gamma(\alpha-1)(y-n)}\omega_{e^{\gamma\alpha(n-y)}\cdot 2e^{\gamma\alpha}\eta}\geq \varepsilon\big)+Ce^{-cn}$ because $0\leq y-n\leq 1$. By Proposition~\ref{cond_H=x_prime}, integrating this inequality with respect to $y$ against the law of $\tHS(\ftau)$ under $\P(\, \cdot\, |\, \HS(\tau)=n)$ yields that $\nu_x(\omega_\eta\geq 2\varepsilon)\leq\nu_n'\big(\omega_{2e^{\gamma\alpha}\eta}\geq \varepsilon\big)+Ce^{-cn}$ since 
\[e^{\gamma(\alpha-1)(y-n)}\omega_{e^{\gamma\alpha(n-y)}\xi}\big(e^{-\gamma(\alpha-1)y}H_{e^{\gamma\alpha y}s}(\tau)\, ;\, s\geq 0\big)=\omega_{\xi}\big(e^{-\gamma(\alpha-1)n}H_{e^{\gamma\alpha n}s}(\tau)\, ;\, s\geq 0\big)\]
for any $\xi>0$. Finally, (\ref{tight_step_goal}) follows from (\ref{tight_step_preliminar}), which concludes the proof.
\end{proof}

The rest of the section is devoted to the proof of Theorem~\ref{cauchy_criterion}. We begin by applying the invariance by weighted Horton pruning stated in Theorem~\ref{invariance} (together with Proposition~\ref{estimee_pruning}) to bound the Prokhorov distance between the conditional laws given $\tHS(\fTau)$ of the rescaled height function (as defined by (\ref{height_process_length_def})) of a $\GWwl_\alpha$-weighted tree with edge lengths $\fTau=(\tau,(L_u)_{u\in\tau},(W_v)_{v\in\partial\tau})$ for different values of $\tHS(\fTau)$.

\begin{lemma}
\label{pseudo_cauchy_criterion}
Let $\fTau$ be a $\GWwl_\alpha$-weighted tree with edge lengths as above and set $\cT=(\tau,(L_u)_{u\in\tau})$. For all $x\in\bbR_+\backslash \mathbb{N}$, denote by $\bar{\nu}_x$ the law on $\mathcal{D}_{\mathrm{K}}$ of $\big(e^{-\gamma(\alpha-1)x}\Hl_{e^{\gamma\alpha x}s}(\cT)\big)_{s\geq 0}$ under $\P(\dd\cT\, |\, \tHS(\fTau)=x)$. Then, there are two constants $C,c\in(0,\infty)$ that only depend on $\alpha$ such that $\rho_{\mathrm{S}}(\bar{\nu}_x,\bar{\nu}_{x-r})\leq C e^{-cx}$, for all $r\in[1/10,9/10]$ and all $x>r$ with $x,x-r\notin\mathbb{N}$.
\end{lemma}

\begin{proof}
The Prokhorov distance between two probability measures is always bounded by $1$, so we can assume $x\geq 2>r+1$ without loss of generality. We define five random càdlàg functions with compact support $X,Y,Z,X',Y'$ by setting for all $s\in\bbR_+$, 
\begin{align*}
X_s&=e^{-\gamma(\alpha-1)x}\Hl_{e^{\gamma\alpha x}s}(\cT)\, ,\:\:  & Y_s&=\I{\tHS(\fTau)\geq r}e^{-\gamma(\alpha-1)x}\Hl_{e^{-\gamma r}e^{\gamma\alpha x}s}(R_r(\fTau))\, ,\\
Z_s&=\mathrlap{\I{\tHS(\fTau)\geq r,\tHS(\bR_r(\fTau))\geq 1}e^{-\gamma(\alpha-1)x}\Hl_{e^{-\gamma(1+r)}e^{\gamma\alpha x}s}(R_1\circ \bR_r(\fTau))\, ,}\\
X_s'&=e^{-\gamma(\alpha-1)(x-r)}\Hl_{e^{\gamma\alpha(x-r)}s}(\cT)\, ,\:\: & Y_s'&=\I{\tHS(\fTau)\geq 1}e^{-\gamma(\alpha-1)(x-r)}\Hl_{e^{-\gamma}e^{\gamma\alpha(x-r)}s}(R_1(\fTau))\, .
\end{align*}
The law of $X$ under $\P(\dd\cT\, |\, \tHS(\fTau)=x)$ is $\bar{\nu}_x$, and the law of $X'$ under $\P(\dd \cT\, |\, \tHS(\fTau)=x-r)$ is $\bar{\nu}_{x-r}$. The first step of the proof is to show that $X$ and $X'$ are respectively close to $Z$ and $Y'$ with high probability by using Proposition~\ref{estimee_pruning}. Then, we prove using Theorem~\ref{invariance} that the "conditional law of $Z$ given $\tHS(\fTau)=x$" is the same as that of $Y'$ under $\P(\dd\cT\, |\, \tHS(\fTau)\!=\!x\!-\!r)$. However, $Z$ is not measurable with respect to $\cT$ since we need the weights to determine $R_r(\fTau)$, so Definition~\ref{cond_H=x_def} does not define the law of $Z$ under $\P(\dd\cT\, |\, \tHS(\fTau)\!=\!x)$. We avoid this issue by working conditionally given $|\tHS(\fTau)\!-\!x|\!<\!\varepsilon$, with a small $\varepsilon\!>\!0$.

Let us set $n=e^{3\gamma\alpha x/2}$. Applying Proposition~\ref{estimee_pruning} with $\lambda=e^{-\gamma\alpha x}$ and with $\lambda=e^{-\gamma\alpha(x-r)}$, we deduce that there exist two constants $C_0,c_0\in(0,\infty)$ only depending on $\alpha$ such that
\begin{align}
\label{estimee_XtoY}
\P\big(\#\tau\leq n\, ;\, |\tau|\leq n^{1-1/\alpha}\, ;\, \mathtt{d}_{\mathrm{S}}(X,Y)\geq C_0 xe^{-\gamma(\alpha-1)x/4}\big)\, \leq\,  C_0 e^{3\gamma\alpha x/2-c_0x^2},\\
\label{estimee_X'toY'}
\P\big(\#\tau\leq n\, ;\, |\tau|\leq n^{1-1/\alpha}\, ;\, \mathtt{d}_{\mathrm{S}}(X',Y')\geq C_0 xe^{-\gamma(\alpha-1)x/4}\big)\, \leq\,  C_0 e^{3\gamma\alpha x/2-c_0x^2}.
\end{align}
Then, Remark~\ref{measurability_pruning} $(i)$ ensures that the pair $(Y,Z)$ can be expressed as a measurable function of $R_r(\fTau)$ when $\tHS(\fTau)\geq r$. Hence, by the scaling relations in (\ref{lifetime_height_length}) and (\ref{pruning_scaling}), Theorem~\ref{invariance} $(i)$ yields that the law of $(Y,Z)$ under $\P(\, \cdot\, |\, \tHS(\fTau)\geq r)$ is the same as the law of $(X',Y')$ under $\P$. Since it always holds that $\#R_r(\fTau)\leq \#\tau$ and $|R_r(\fTau)|\leq |\tau|$, (\ref{estimee_X'toY'}) still holds even if we replace $(X',Y')$ by $(Y,Z)$. Combining this new version of (\ref{estimee_X'toY'}) with (\ref{estimee_XtoY}) entails that
\begin{equation}
\label{estimee_XtoZ}
\P\big(\#\tau\leq n\, ;\, |\tau|\leq n^{1-1/\alpha}\,;\, \mathtt{d}_{\mathrm{S}}(X,Z)\geq 2C_0 xe^{-\gamma(\alpha-1)x/4}\big)\, \leq\,  2C_0 e^{3\gamma\alpha x/2-c_0x^2}.
\end{equation}

By Remark~\ref{measurability_pruning} $(i)$ and $(iii)$, the only information about the weights of $\fTau$ we need to determine $X',Y',X,Z$ is whether they are smaller than $r$. Thus, we can use (\ref{cond_domination_bis}) in Lemma~\ref{cond_domination} to deduce from (\ref{estimee_X'toY'}) and (\ref{estimee_XtoZ}) that by setting $\eta=ne^{\gamma x}\cdot e^{3\gamma\alpha x/2-c_0 x^2}=e^{\gamma(3\alpha+1) x-c_0x^2}$, there is a constant $C_1\in(0,\infty)$, that only depends on $\alpha$, such that
\begin{align}
\label{estimee_X'toY'_cond}
\P\big(\#\tau\!\leq\! n ;|\tau|\!\leq\! n^{1-1/\alpha} ;\mathtt{d}_{\mathrm{S}}(X',Y')\!\geq\! C_1 xe^{-\frac{1}{4}\gamma(\alpha-1)x}\, \big|\, |\tHS(\fTau)-x\!+\!r|\!<\!\varepsilon\big)&\leq C_1\eta,\\
\label{estimee_XtoZ_cond}
\P\big(\#\tau\leq n\, ;|\tau|\leq n^{1-1/\alpha}\, ;\mathtt{d}_{\mathrm{S}}(X,Z)\geq C_1 xe^{-\frac{1}{4}\gamma(\alpha-1)x}\, \big|\, |\tHS(\fTau)-x|\!<\!\varepsilon\big)&\leq C_1 \eta,
\end{align}
for all $\varepsilon>0$ small enough (according to $x$ and $r$). We stress that $C_1$ does not depend on $r$ because we restrict ourselves to the case $r\in[1/10,9/10]$, which ensures $c_\alpha(r)\leq c_\alpha(1/10)$ with the notation of Lemma~\ref{cond_domination}. It remains to control the conditional probabilities that $\#\tau\geq n$ or $|\tau|\geq n$. We successively use (\ref{Q_x_cond_H=x}) and Corollary~\ref{croissance_selon_H}, then a rough bound, and finally (\ref{equiv_HS}) with (\ref{equiv_taille}) (resp.~with (\ref{equiv_hauteur})) to write that 
\begin{align}
\label{tails_size_tHS_equal}
\limsup_{\varepsilon\to 0^+}\P\big(\#\tau\geq n\ \big|\ |\tHS(\fTau)-x|<\varepsilon\big)& \leq  \tfrac{\P(\#\tau\geq n)}{\P(\HS(\cT)=\lfloor x\rfloor +1)}\leq C_2 e^{-\gamma x/2},\\
\label{tails_height_tHS_equal}
\limsup_{\varepsilon\to 0^+}\P\big(|\tau|\geq n^{1-1/\alpha}\ \big|\ |\tHS(\fTau)-x|<\varepsilon\big)&\leq \tfrac{\P(|\tau|\geq n^{1-1/\alpha})}{\P(\HS(\cT)=\lfloor x\rfloor +1)}\leq C_2 e^{-\gamma x/2},
\end{align}
where $C_2\in(0,\infty)$ is a constant that only depends on $\alpha$. The same bounds hold for the conditional probabilities given $|\tHS(\fTau)-(x\!-\!r)|<\varepsilon$. Using (\ref{estimee_X'toY'_cond}) and (\ref{estimee_XtoZ_cond}), we conclude that there is a constant $C_3\in(0,\infty)$ only depending on $\alpha$ such that for all $\varepsilon>0$ small enough,
\begin{align}
\label{estimee_X'toY'_final}
\P\big(\mathtt{d}_{\mathrm{S}}(X',Y')\geq C_3 xe^{-\gamma(\alpha-1)x/4}\ \big|\ |\tHS(\fTau)-(x\!-\!r)|<\varepsilon\big)&\, \leq\, C_3 e^{-\gamma x/2},\\
\label{estimee_XtoZ_final}
\P\big(\mathtt{d}_{\mathrm{S}}(X,Z)\geq C_3 xe^{-\gamma(\alpha-1)x/4}\ \big|\ |\tHS(\fTau)-x|<\varepsilon\big)&\, \leq\,  C_3 e^{-\gamma x/2}.
\end{align}

Let us prove that the conditional law of $Y'$ given $|\tHS(\fTau)-(x-r)|<\varepsilon$ is equal to the conditional law of $Z$ given $|\tHS(\fTau)-x|<\varepsilon$, for all small enough $\varepsilon>0$. More precisely, let $\varepsilon\in(0,x-r-1)$ and $A$ be a measurable subset of $\mathcal{D}_{\mathrm{K}}$. By the choice of $\varepsilon$, if $\left|\tHS(\fTau)-x\right|<\varepsilon$ then $\tHS(\fTau)\geq r+1$. Moreover, if $\tHS(\fTau)\geq r+1$ then we know from Proposition~\ref{recursive_pruning} that $\tHS(\bR_1\circ\bR_r(\fTau))=\tHS(\fTau)-r-1$. Therefore, it holds that
\[\P\big(|\tHS(\fTau)-x|<\varepsilon\, ;\, Z\in A\big)=\P\big(\tHS(\fTau)\geq r+1\,;\, |\tHS(\bR_1\circ\bR_r(\fTau))-(x\!-\!r\!-\!1)|<\varepsilon\, ;\, Z\in A\big).\]
Similarly, (\ref{lifetime_height_length}), (\ref{pruning_scaling}), and Theorem~\ref{invariance} $(iii)$ together yield that
\[\P(\tHS(\fTau)\!\geq\! 1)\P\big(|\tHS(\fTau)-x|<\varepsilon\, ;\, Z\!\in\! A\big)=\P(\tHS(\fTau)\!\geq\! r\!+\!1)\P\big(|\tHS(\fTau)-(x\!-\!r)|<\varepsilon\, ;\, Y'\!\in\! A).\]
Furthermore, we know from Proposition~\ref{law_HT_stable} that the law of $\tHS(\fTau)$ is exponential, so we get
\begin{equation}
\label{identite_en_loi}
\P\big(Z\in A\ \big|\ |\tHS(\fTau)-x|<\varepsilon\big)\, =\,\P\big(Y'\in A\ \big|\ |\tHS(\fTau)-(x\!-\!r)|<\varepsilon\big).
\end{equation}
Thanks to (\ref{Q_x_cond_H=x}) and (\ref{identite_en_loi}), taking $\varepsilon\!\to\!0^+$ in (\ref{estimee_X'toY'_final}) and (\ref{estimee_XtoZ_final}) completes the proof.
\end{proof}

\begin{proof}[Proof of Theorem~\ref{cauchy_criterion}]
Let $\fTau=(\tau,(L_u)_{u\in\tau},(W_v)_{v\in\partial\tau})$ be a $\GWwl_\alpha$-weighted tree with edge lengths and set $\ftau=(\tau,(W_v)_{v\in\partial\tau})$, which is a $\GWw_\alpha$-weighted tree. For $x\in\bbR_+\backslash\bbN$, define two random càdlàg functions with compact support $X,X^*$ by setting for all $s\in\bbR_+$,
\[X_s=e^{-\gamma(\alpha-1)x}\Hl_{e^{\gamma\alpha x}s}(\cT)\quad\text{ and }\quad X^*_s=e^{-\gamma(\alpha-1)x}H_{e^{\gamma\alpha x}s}(\tau).\]
As in the statements of Lemma~\ref{pseudo_cauchy_criterion} and of Theorem~\ref{cauchy_criterion}, we call the laws of $X$ and $X^*$ under $\P(\dd\cT\, |\, \tHS(\fTau)=x)$ respectively $\bar{\nu}_x$ and $\nu_x$. By Proposition~\ref{height_process_skeleton} with $n=e^{3\gamma\alpha x/2}$ and $\lambda=e^{-\gamma\alpha x}$, there are $C_0,c_0\in(0,\infty)$ only depending on $\alpha$ such that
\begin{equation}
\label{estimee_ske_prelim}
\P\big(\#\tau\leq n\, ;\, |\tau|\leq n^{1-1/\alpha}\, ;\, \mathtt{d}_{\mathrm{S}}(X,X^*)\geq C_0 x e^{-\gamma(\alpha-1)x/4}\big)\, \leq\,  C_0  e^{3\gamma\alpha x/2-c_0x^2}.
\end{equation}
Just as we deduced (\ref{estimee_XtoZ_final}) from (\ref{estimee_XtoZ}), we combine (\ref{estimee_ske_prelim}) with (\ref{cond_domination_bis}), (\ref{tails_size_tHS_equal}) and (\ref{tails_height_tHS_equal}) to get
\begin{equation}
\label{estimee_ske_final}
\P\big(\mathtt{d}_{\mathrm{S}}(X,X^*)\geq C_3 x e^{-\gamma(\alpha-1)x/4}\ \big|\ \tHS(\fTau)=x\big)\leq C_3 e^{-\gamma x/2},
\end{equation}
where $C_3\in(0,\infty)$ only depends on $\alpha$. Then, Lemma~\ref{pseudo_cauchy_criterion} and (\ref{estimee_ske_final}) ensure the existence of two constants $C_4,c_4\in(0,\infty)$ only depending on $\alpha$ such that $\rho_{\mathrm{S}}(\nu_x,\bar{\nu}_x)\leq C_4 e^{-c_4 x}$ and $\rho_{\mathrm{S}}(\bar{\nu}_x,\bar{\nu}_y)\leq C_4e^{-c_4\min(x,y)}$ for all $x,y\in\bbR_+\backslash\mathbb{N}$ with $1/10\leq |x-y|\leq 9/10$. By the triangle inequality (see Proposition~\ref{prokhorov_prop}), it follows that $\rho_{\mathrm{S}}(\nu_x,\nu_y)\leq 3C_4e^{-c_4 \min(x,y)}$ for all $x,y\in\bbR_+\backslash\mathbb{N}$ with $1/10\leq |x-y|\leq 9/10$. For all $x,y\in\bbR_+\backslash\mathbb{N}$, we can find a finite sequence $(x_i)_{0\leq i\leq p}$ of positive non-integers such that $1/2\leq x_i-x_{i-1}\leq 9/10$ for all $1\leq i\leq p$, $x_0=\min(x,y)$, and $1/10\leq x_p-\max(x,y)\leq 9/10$. We complete the proof by writing \[\rho_{\mathrm{S}}(\nu_x,\nu_y)\leq \rho_{\mathrm{S}}(\nu_{x_p},\nu_{\max(x,y)})+\sum_{i=0}^{p-1}\rho_{\mathrm{S}}(\nu_{x_{i+1}},\nu_{x_i})\leq e^{-c_4\min(x,y)}\cdot 6C_4\sum_{i\in\bbN}e^{-c_4 i/2}.\qedhere\]
\end{proof}

\section{A first description of the limit tree in Theorem~\ref{scaling_limit_HS_intro}}
\label{geometric_study}

In this section, we begin to study the limit objects of Theorems~\ref{scaling_limit_HS_intro} and \ref{scaling_limit_HT}, i.e.~the scaling limit of a $\GWw_\alpha$-weighted tree conditioned to have large weighted Horton--Strahler number. In a nutshell, we control its height and mass (\ref{height_mass}), give some facts about its root, and present an important self-similar spinal decomposition (as announced in the introduction). We will build on these initial properties to prove Theorems~\ref{binary_info_intro}, \ref{standard_stable_HS_real_thm} and \ref{scaling_limit_size_intro}  in the following sections. Recall from Section~\ref{topo_function} that $\mathbb{K}^{\mathrm{m}}$ stands for the space of (GHP-isometry classes of) rooted measured compact metric spaces, endowed with the rooted Gromov--Hausdorff--Prokhorov distance given by (\ref{n-GHP}).
\begin{definition}
\label{stable_height_excursion_def}
A random variable $(\tilde{H},\tilde{\zeta})$ on $\mathcal{C}_{\mathrm{K}}$ is a \emph{$\aHS_\alpha$-excursion} when it has the limit law in Theorem~\ref{scaling_limit_HT}. A random rooted measured compact metric space $\sT$ (i.e.~a random variable on $\mathbb{K}^{\mathrm{m}}$) is a \emph{$\aHS_\alpha$-real tree} when it has the limit law in Theorem~\ref{scaling_limit_HS_intro}.
\end{definition}
By Proposition~\ref{height_contour_tree} and Skorokhod's representation theorem, if $(\tilde{H},\tilde{\zeta})$ is an $\aHS_\alpha$-excursion, then it is in $\mathcal{E}_{\mathrm{K}}$ and the real tree $\,\cT_{\tilde{H},\tilde{\zeta}}\, $ it codes, as defined by (\ref{real_tree_coded}), is an $\aHS_\alpha$-real tree. In particular, an $\aHS_\alpha$-real tree is almost surely a real tree in the sense of Definition~\ref{real_tree}.

We next state some estimates for the mass and height of $\aHS_\alpha$-real trees. The one on the height in particular ensures that the law of $\mathfrak{h}(\sT)$ is nondegenerate, so the same holds for $\sT$. 

\begin{proposition}
\label{maj_mass_limi_tree}
Let $\sT$ be an $\aHS_\alpha$-real tree. Then, $\E[\mathfrak{m}(\mathscr{T})]<\infty$ and there exists a constant $y_0\in(0,\infty)$ that only depends on $\alpha$ such that for all $x\geq 0$ and $y\geq y_0$,
\begin{equation}
\label{maj_hauteur_limit_tree}
\P(\mathfrak{h}(\mathscr{T})<x)\leq 1-e^{-4x}\quad\text{ and }\quad\P(\mathfrak{h}(\mathscr{T})>y)\leq e^{-y/40}.
\end{equation}
\end{proposition}

\begin{proof}
For $n\!\in\!\bbN$, let $\tau_n$ have law $\P(\dd\tau\, |\, \tHS(\ftau)\!=\!x_n)$, where $x_n=n+\frac{1}{2}$ and $\ftau$ is $\GWw_\alpha$-weighted tree, as in Definition~\ref{cond_H=x_def}. We can assume that $e^{-\gamma\alpha x_n}\#\tau_n\longrightarrow \mathfrak{m}(\sT)$ a.s.~by continuity of $\mathfrak{m}$. Fatou's lemma, Corollary~\ref{croissance_selon_H}, and (\ref{first_moment_size}) then imply that $\E[\mathfrak{m}(\mathscr{T})]<\infty$. Next, note that $e^{-\gamma(\alpha-1)x_n}|\tau_n|\convd\mathfrak{h}(\sT)$ by continuity of $\mathfrak{h}$. Then, the Portmanteau theorem, Corollary~\ref{croissance_selon_H}, (\ref{min_hauteur_selon_H}), and (\ref{maj_hauteur_selon_H}) yield (\ref{maj_hauteur_limit_tree}) --- because $\delta\sqrt{\delta} \leq  4$ and $20\sqrt{\delta}\leq 40$ as $\delta\in(1,2]$.
\end{proof}

Next, we show how to construct an $\aHS_\alpha$-real tree by rescaling (recall from (\ref{notation_scaling}) the notation $\odot_\alpha$) and grafting (recall from Definition~\ref{graft_def} the notation $\circledast$) a countable number of independent $\aHS_\alpha$-trees on a segment. This self-similarity is our main tool to prove Theorem~\ref{binary_info_intro} $(ii)$. Because of the number of different random variables involved in the construction, the theorem is both technical to state and to prove. However, it is neither surprising nor difficult to understand. Indeed, it is just based on the idea to focus on the subtrees with large weighted Horton--Strahler numbers, to apply Definition~\ref{GWwdef} of $\GWw_\alpha$-weighted trees, and to take the scaling limit. The formal proof, that we postpone to the end of the section, gives little intuition and may be skipped at first reading. Recall from (\ref{alpha_gamma_delta}) that $\beta=\tfrac{1}{\alpha-1}$, $\gamma=\ln \tfrac{\alpha}{\alpha-1}$ and $\delta=e^{\gamma(\alpha-1)}$. We work on the product topological space $\Omega=\bbR_+\times [1,\infty)\times\mathbb{K}^{\mathrm{m}}\times\prod_{i\geq 2}([0,1]\times\mathbb{K}^{\mathrm{m}})$ and we write each of its element $\omega\in\Omega$ as $\omega=\big(\, \ell(\omega)\, ,\, \xi(\omega)\, ,\, T_1(\omega)\, ,\ (\lambda_i(\omega),T_i(\omega))_{i\geq 2}\, \big)$. Also set $\bar{\lambda}_i(\omega)=\prod_{j=2}^i\lambda_j(\omega)$ for all $i\geq 1$.

\begin{theorem}
\label{self-similar_alpha}
Let $\sT_\rg,\sT_\rd$ be two $\aHS_\alpha$-real trees. Let $L$ have exponential law with mean $\delta^{-1}$ and let $U$ have law $\un_{[1,\delta]}(s)s^{\beta-1}\dd s$. Let $\left(\mathscr{T}_i\right)_{i\geq 1}$ be a sequence of independent $\aHS_\alpha$-real trees. Let $(\Lambda_i)_{i\geq 2}$ be a sequence of independent random variables on $[0,1]$ with laws respectively given by $\Lambda_i=0$ almost surely when $\alpha=2$, or by
\[\P(\Lambda_i\in \dd\lambda_i)=\un_{[0,1]}(\lambda_i)\beta(i-\alpha)\lambda_i^{\beta(i-\alpha)-1}\dd\lambda_i\quad \text{ when } \alpha\in(1,2).\]
Set $\bar{\Lambda}_i=\prod_{j=2}^i\Lambda_j$ for all $i\geq 1$. Let $\mathcal{N}$ be a Poisson point process on $\Omega$ with intensity measure \[\delta\dd \ell \times \un_{[1,\infty)}(\xi)\dd\xi\times\P(\mathscr{T}_\rg\in \dd T_1)\times\prod_{i\geq 2}\big(\P(\Lambda_i\in\dd\lambda_i)\times\P(\mathscr{T}_\rg\in \dd T_i)\big).\] Assume that $\sT_\rg,\sT_\rd,L,U,\left(\mathscr{T}_i\right)_{i\geq 2},\left(\Lambda_i\right)_{i\geq 2},\mathcal{N}$ are jointly independent. Denote by $\mathtt{L}\in\mathbb{K}^{\mathrm{m}}$ the real segment $[0,L]$ rooted at $0$ and endowed with the null measure. Then, there exists an $\aHS_\alpha$-real tree $\sT^*$ such that almost surely
\begin{equation}
\label{sT_etoile}
\mathscr{T}^*=\mathtt{L}\circledast\Big(\! L,\frac{1}{U}\odot_\alpha\sT_\rg\! \Big)\circledast\Big(\! L,\frac{1}{\delta}\odot_\alpha\sT_\rd\! \Big)\Circledast_{i\geq 2}\bigg(\! L,\frac{\bar{\Lambda}_i}{\delta} \odot_\alpha\sT_i\! \bigg)\! \Circledast_{\substack{i\geq 1,\omega\in\mathcal{N}\\ \ell(\omega)\leq L}}\! \bigg(\! \ell(\omega),\frac{\bar{\lambda}_i(\omega)}{\delta\xi(\omega)}\odot_\alpha T_i(\omega)\! \bigg)\! .
\end{equation}
\end{theorem}

\begin{remark}
If $\alpha\! \in\! (1,2)$, we have $\E[\Lambda_i^q]=1 - \tfrac{q}{q+\beta(i-\alpha)}$ for all $q\! >\! 0$ and $i\! \geq\! 2$. We get 
\begin{equation}
\label{moments_Lambda}
(\alpha=2\:\text{ or }\:q>\beta)\Longrightarrow\E\Big[\sum_{i\geq 1}(\bar{\Lambda}_i)^q\Big]=\sum_{i\geq 1}\prod_{j=2}^i\E\big[\Lambda_j^q\big]<\infty,
\end{equation}
after an elementary asymptotic study. This will be useful for proving Theorem~\ref{self-similar_alpha}.
\end{remark}
In the specific case where $\alpha=2$, Theorem~\ref{self-similar_2} takes the following simpler form.

\begin{corollary}
\label{self-similar_2}
Let $\sT_\rg,\sT_\rd$ be two $\aHS_2$-real trees. Let $L$ have exponential law with mean $\frac{1}{2}$ and let $U$ have uniform law on $[1,2]$. Let $\mathcal{N}'$ be a Poisson point process on $[0,\infty)\times [1,\infty)\times\mathbb{K}^{\mathrm{m}}$ with intensity measure $2\dd\ell \times \un_{[1,\infty)}(\xi)\dd\xi\times\P(\sT_\rg\in \dd T)$. Assume that $\mathscr{T}_\rg,\mathscr{T}_\rd,L,U,\mathcal{N}'$ are jointly independent. Denote by $\mathtt{L}\in\mathbb{K}^{\mathrm{m}}$ the real segment $[0,L]$ rooted at $0$ and endowed with the null measure. Then, there exists an $\aHS_2$-real tree $\sT^*$ such that almost surely $\mathscr{T}^*=\mathtt{L}\circledast \big(L,\tfrac{1}{U}\odot_2 \mathscr{T}_\rg\big)\circledast \big(L,\tfrac{1}{2}\odot_2 \mathscr{T}_\rd\big)\Circledast_{\substack{(\ell,\xi,T)\in\mathcal{N}'\\ \ell\leq L}}\big(\ell,\tfrac{1}{2\xi}\odot_2 T\big)$.
\end{corollary}

The following result gathers some other properties of $\aHS_\alpha$-excursions. In particular, it allows saying that $\tilde{H}$ is an $\aHS_\alpha$-excursion instead of $(\tilde{H},\tilde{\zeta})$ in accordance with Notation~\ref{notation_DK}.

\begin{proposition}
\label{natural_lifetime_scaling_limit}
Recall from (\ref{lifetime_modulus}) the lifetime $\zeta(f)$ of a function $f:\bbR_+\to\bbR$. Recall from Definition~\ref{planted} that we say that a rooted compact real tree is planted when removing its root does not disconnect it. The following holds.
\begin{longlist}
\item[(i)] If $(\sT,d,\rho,\mu)$ is $\aHS_\alpha$-real tree, then $\mu(\{\rho\})=0$ almost surely.
\item[(ii)] If $(\tilde{H},\tilde{\zeta})$ is an $\aHS_\alpha$-excursion, then $\tilde{\zeta}=\zeta(\tilde{H})$ almost surely.
\item[(iii)] If $\sT$ is $\aHS_\alpha$-real tree, then $\sT$ is almost surely planted.
\end{longlist}
\end{proposition}

\begin{proof}
We use the notation of Theorem~\ref{self-similar_alpha}. We begin by observing that almost surely,
\begin{equation}
\label{natural_lifetime_scaling_tool}
L>0\quad\text{ and }\quad\mathcal{N}\cap\{\omega\in\Omega\ :\ \ell(\omega)=0\}=\emptyset.
\end{equation}
Thus, the root of $\sT^*$ is almost surely not an atom of the measure of $\sT^*$ by Definition~\ref{graft_def}, which yields $(i)$. Since $(\tilde{H},\tilde{\zeta})$ is a random variable on $\mathcal{C}_{\mathrm{K}}$, we have $\zeta(\tilde{H})\leq\tilde{\zeta}$ and $\tilde{H}_s=0$ for all $s\geq \zeta(\tilde{H})$. By definition (\ref{real_tree_coded}) of $\cT_{\tilde{H},\tilde{\zeta}}$, it follows that $\tilde{\zeta}-\zeta(\tilde{H})\leq \mu_{\tilde{H},\tilde{\zeta}}(\{\rho_{\tilde{H},\tilde{\zeta}}\})$. Then, $(i)$ implies $(ii)$. Proposition~\ref{wedge_grafting} and (\ref{natural_lifetime_scaling_tool}) entail that $\sT^*$ is a.s.~planted, giving $(iii)$.
\end{proof}

\subsection{Proof of Theorem~\ref{self-similar_alpha}}

It is well-known that the first moment measure of the Poisson Point process is its intensity measure. We check that the two series $\sum_{i\geq 1}(\bar{\Lambda}_i)^{\alpha\beta}$ and $\sum_{\omega\in\mathcal{N}}\sum_{i\geq 1}\I{\ell(\omega)\leq L}(\bar{\lambda}_i(\omega)/\xi(\omega))^{\alpha\beta}$ are integrable with (\ref{moments_Lambda}). For all $n\in\bbN^*$, it follows that there exists a random variable $\sT_n^*$ on $\mathbb{K}^{\mathrm{m}}$ which is almost surely equal to
\[\mathtt{L}\circledast\big(L,\tfrac{1}{U}\odot_\alpha\sT_\rg\big)\circledast\big(L,\tfrac{1}{\delta}\odot_\alpha\sT_\rd\big)\!\!\! \Circledast_{\substack{i\geq 2\\ \bar{\Lambda}_i\geq 1/n}}\!\!\!\Big(L,\tfrac{1}{\delta} \bar{\Lambda}_i\odot_\alpha\sT_i\Big)\!\!\!\! \!\!\!\Circledast_{\substack{\omega\in\mathcal{N},i\geq 1\\ \ell(\omega)\leq L \\ \bar{\lambda}_i(\omega)\geq \xi(\omega)/n}}\!\!\!\!\!\!\Big(\ell(\omega),\tfrac{1}{\delta\xi(\omega)}\bar{\lambda}_i(\omega)\odot_\alpha T_i(\omega)\Big).\]
Indeed, the above grafting procedure involves an almost surely finite number of random variables on $\mathbb{K}^{\mathrm{m}}$, so the measurability comes from Proposition~\ref{graft_continuous}. As $\frac{\alpha}{\alpha-1}=\alpha\beta>\beta$, we similarly check with (\ref{moments_Lambda}) and Proposition~\ref{maj_mass_limi_tree} that these two series are integrable:
\begin{align*}
\sum_{i\geq 2}\mathfrak{m}\big(\delta^{-1}\bar{\Lambda}_i\, \odot_\alpha\sT_i\big)\, +\, &\sum_{\omega\in\mathcal{N}}\sum_{i\geq 1}\I{\ell(\omega)\leq L}\mathfrak{m}\big(\delta^{-1}\xi(\omega)^{-1}\bar{\lambda}_i(\omega)\odot_\alpha T_i(\omega)\big),\\
\sum_{i\geq 2}\mathfrak{h}\big(\delta^{-1}\bar{\Lambda}_i\, \odot_\alpha\sT_i\big)^{\alpha\beta}\, +\, &\sum_{\omega\in\mathcal{N}}\sum_{i\geq 1}\I{\ell(\omega)\leq L}\mathfrak{h}\big(\delta^{-1}\xi(\omega)^{-1}\bar{\lambda}_i(\omega)\odot_\alpha T_i(\omega)\big)^{\alpha\beta}.
\end{align*}
It follows that the conditions (\ref{graft_cond1}) and (\ref{graft_cond2}) are a.s.~satisfied, so the right-hand side of (\ref{sT_etoile}) is a.s.~an element of $\mathbb{K}^{\mathrm{m}}$ by Proposition~\ref{compact_graft}. Let us denote it by $(\sT',d',\rho',\mu')$. Furthermore, we a.s.~get $\sup_{\sigma'\in \sT',\sigma_n\in\sT_n^*}d'(\sigma',\sigma_n)\to 0$ and $\mu'(\sT'\backslash\sT_n^*)\to 0$ thanks to the dominated convergence theorem. Recalling the expression (\ref{n-GHP}) of the Gromov--Hausdorff--Prokhorov distance, we thus obtain that $\sT_n^*\longrightarrow(\sT',d',\rho',\mu')$ a.s.~on $\mathbb{K}^{\mathrm{m}}$. For the rest of the proof, we choose a random variable $\sT^*$ on $\mathbb{K}^{\mathrm{m}}$ such that 
\begin{equation}
\label{sT_etoile_as_limit}
\sT^*=\lim_{n\rightarrow\infty} \sT_n^*\: \; \text{ almost surely}.
\end{equation}
This is possible as a limit of measurable functions is measurable. By construction, we know that $\sT^*$ verifies (\ref{sT_etoile}) almost surely. It only remains to show that $\sT^*$ is an $\aHS_\alpha$-real tree.
\medskip

Recall from Notation~\ref{underlying_set} that we sometimes denote an element $(E,d,\rho,\mu)$ of $\mathbb{K}^{\mathrm{m}}$ by its underlying space $E$. We work with a countable and closed subset of $\mathbb{K}^{\mathrm{m}}$ defined as follows: \[\mathbb{K}_{\mathrm{dsc}}^{\mathrm{m}}=\big\{(E,d,\rho,\mu)\in\mathbb{K}^{\mathrm{m}} : E\text{ finite such that }d(x,y)\in\bbN\text{ and }\mu(\{x\})\in\bbN\text{ for all }x,y\in E\big\}.\]
For any $n\in\bbN^*$ and any weighted tree $\bt$, we define an element of $\mathbb{K}_{\mathrm{dsc}}^{\mathrm{m}}$ by setting
\[D_n(\bt)=\big\{u\in t\ :\ \tHS(\theta_u\bt)=\tHS(\bt)\text{ or }\exists v\in t\text{ with }v\preceq u,0<\tHS(\bt)-\tHS(\theta_{v}\bt)\leq 1+\log_\delta n\big\},\]
which we see as a metric subspace of $\tau$, also endowed with the graph distance given by (\ref{graph_distance}), rooted at $\varnothing$, and equipped with its counting measure. Let $\ftau=(\tau,(W_v)_{v\in\partial\tau})$ be a $\GWw_\alpha$-weighted tree. Our proof requires us to show the three following points.
\begin{longlist}
\item[(I)] For $n\!\in\!\bbN^*,x\!\in\!\bbR_+\backslash\bbN$ with $x\! >\! 1\! +\! \log_\delta n$, there is a random variable $D_n^x$ on $\mathbb{K}_{\mathrm{dsc}}^{\mathrm{m}}$ such that $D_n(\ftau)$ under $\P\big(\cdot \big|\, |\tHS(\ftau)\!-\!x|\!<\!\varepsilon\big)$ converges in law to $D_n^x$ under $\P$, on $\mathbb{K}^{\mathrm{m}}$, as $\varepsilon\to 0^+$.
\item[(II)] For all $n\in\bbN^*$, $\delta^{-x}\odot_\alpha D_n^x$ converges in law to $\sT_n^*$ on $\mathbb{K}^{\mathrm{m}}$ as $x\in\bbR_+\backslash\bbN$ tends to $\infty$.
\item[(III)] Recall the distance $\mathtt{d}_{\mathrm{GHP}}$ as in (\ref{n-GHP}), then for all $\eta>0$,
\[0=\limsup_{n\rightarrow\infty}\limsup_{x\rightarrow\infty,x\notin\bbN}\limsup_{\varepsilon\rightarrow 0^+}\P\big(\mathtt{d}_{\mathrm{GHP}}(\delta^{-x}\odot_\alpha\tau,\delta^{-x}\odot_\alpha D_n(\ftau))\geq\eta\ \big|\ |\tHS(\ftau)-x|<\varepsilon\big).\]
\end{longlist}
These three assertions and (\ref{sT_etoile_as_limit}) entail that $\delta^{-x}\odot_\alpha\tau$ under $\P(\dd\tau\, |\, \tHS(\ftau)=x)$ converges in law to $\sT^*$ under $\P$. By Definition~\ref{stable_height_excursion_def}, Theorem~\ref{self-similar_alpha} eventually follows. Before getting down to the proof of (I), (II), and (III), let us describe the law of $D_n^x$ via the following lemma. Recall from (\ref{stable_offspring}) that $\varphi_\alpha$ stands for the generating function of the offspring law of $\GW_\alpha$-trees.

\begin{lemma}
\label{ikea}
We fix $x\in\bbR_+\backslash\bbN$ and $n\in\bbN^*$ such that $x>1+\log_\delta n$. Let $L^x$ be a geometric RV with parameter $\delta^{1-x}$ and let $\mathtt{L}^x\in\mathbb{K}_{\mathrm{dsc}}^{\mathrm{m}}$ denote the metric space $\{0,\ldots,L^x\}$ rooted at $0$ and endowed with its counting measure. Let $\tau_\rg^x$ be distributed as $\tau$ under $\P(\, \cdot\, |\, x-1<\tHS(\ftau)<x)$. Let $\tau_\rd^x$ be distributed as $\tau$ under $\P(\dd\tau\, |\, \tHS(\ftau)=x-1)$. Let $(\tilde{T}_i^x)_{2\leq i\leq \tilde{J}^x}$ be a random finite sequence with random length $\tilde{J}^x-1\geq 0$ (that may be $0$) such that for all $m\geq 2$, the quantities $\delta^{1-x}e^{\gamma (x-1)}\frac{e^{-\gamma}}{1-e^{-\gamma}}\E\Big[\I{\tilde{J}^x\geq m}\prod_{i=2}^m f_i(\tilde{T}_i^x)\Big]$ are equal to
\[\int_{\bbR_+^{m-1}}\!\!\I{x-1>y_2>\ldots >y_m>x-\log_\delta(\delta n)}\varphi_\alpha^{(m+1)}(1-e^{-\gamma y_m})\prod_{i=2}^m \E[f_i(\tau)\, |\, \tHS(\ftau)=y_i]\gamma e^{-\gamma y_i}\, \dd y_i\]
for all bounded measurable $f_2,\ldots ,f_m:\mathbb{K}^{\mathrm{m}}\to \bbR$. Let $(N_p^x)_{p\geq 0}$ be a RW on $\bbZ$ started at $N_0^x=-1$ such that $N_{p+1}^x-N_p^x-1$ is geometric with parameter $\delta(n-1)/(\delta^x-\delta)$ for all $p\geq 0$. Let $\big((T_{i,p}^x)_{1\leq i\leq J_{p}^x}\big)_{p\geq 1}$ be a sequence of i.i.d~random finite sequences with random lengths $J_{p}^x\geq 1$ such that for all $m\geq 1$, $\delta(n-1)\delta^{-x}\E\Big[\I{J_{p}^x\geq m}\prod_{i=1}^m f_i(T_{i,p}^x)\Big]$ are equal to
\[\int_{\bbR_+^{m}}\I{x-1>y_1>\ldots >y_m>x-\log_\delta(\delta n)}\varphi_\alpha^{(m+1)}(1-e^{-\gamma y_m})\prod_{i=1}^m \E[f_i(\tau)\, |\, \tHS(\ftau)=y_i]\gamma e^{-\gamma y_i}\, \dd y_i\]
for all bounded measurable $f_1,\ldots ,f_m:\mathbb{K}^{\mathrm{m}}\to \bbR$. We assume all these RVs are independent. For all $E\in\mathbb{K}_{\mathrm{dsc}}^{\mathrm{m}}$, denote by $\varpi\circledast E$ the rooted measured compact metric space $\{0,1\}\circledast(1,E)$ where $\{0,1\}\in\mathbb{K}_{\mathrm{dsc}}^{\mathrm{m}}$ is rooted at $0$ and endowed with the null measure. We define
\[D_n^x=\mathtt{L}^x\circledast(L^x,\varpi\circledast \tau_\rg^x)\circledast(L^x,\varpi\circledast \tau_\rd^x)\circledast_{2\leq i\leq \tilde{J}^x}(L^x,\varpi\circledast\tilde{T}_i^x)\Circledast_{\substack{p,i\geq 1\\ N_p^x<L^x,\, i\leq J_{p}^x}}(N_p^x,\varpi\circledast T_{i,p}^x).\]
Let $\{o\}\in\mathbb{K}_{\mathrm{dsc}}^{\mathrm{m}}$ stands for the compact metric space reduced to a single point $o$ and equipped with its Dirac measure. If $D$ is an independent random variable on $\mathbb{K}_{\mathrm{dsc}}^{\mathrm{m}}$ such that it holds
\begin{multline}
\label{ikea_identity}
\E\big[F(D)\big]=\delta^{1-x}\E\big[F(\{o\}\circledast(o,\varpi\circledast\tau_\rg^x)\circledast(o,\varpi\circledast\tau_\rd^x)\circledast_{2\leq i\leq \tilde{J}^x}(o,\varpi\circledast \tilde{T}_i^x))\big]\\
+(n-1)\delta^{1-x}\E\big[F(\{o\}\circledast(o,\varpi\circledast D)\circledast_{1\leq i\leq J^x_{1}}(o,\varpi\circledast T_{i,1}^x))\big]\\
+(1-n\delta^{1-x})\E\big[F(\{o\}\circledast(o,\varpi\circledast D))]
\end{multline}
for all bounded and measurable $F:\mathbb{K}_{\mathrm{dsc}}^{\mathrm{m}}\to\bbR$, then $D$ has the same law as $D_n^x$.
\end{lemma}

\begin{proof}[Proof of Lemma~\ref{ikea}]
It is easy to check that an independent copy of $D_n^x$ satisfies (\ref{ikea_identity}): the three terms of the right-hand side resp.~correspond to the events $\{L^x=0\}$, $\{N_1^x=0<L^x\}$, and $\{1\leq N_1^x,L^x\}$. Using (\ref{ikea_identity}), we show that $\P(D=E)=\P(D_n^x=E)$ for all $E\in\mathbb{K}_{\mathrm{dsc}}^{\mathrm{m}}$ by induction on the size of $E$. This ends the proof since $\mathbb{K}_{\mathrm{dsc}}^{\mathrm{m}}$ is countable.
\end{proof}

\begin{proof}[Proof of (I)]
We fix $n\in\bbN^*$ and $x\in\bbR_+\backslash\bbN$ with $x>1+\log_\delta n$. With the notation of Lemma~\ref{ikea}, we show that $D_n(\ftau)$ under $\P\big( \cdot \big|\, |\tHS(\ftau)-x|<\varepsilon\big)$ converges in law to $D_n^x$ under $\P$. By Proposition~\ref{cond_H=x_prime}, the laws of $D_n(\ftau)$ under $\P(\, \cdot \, \big|\, |\tHS(\ftau)-x|<\varepsilon)$ are tight on $\mathbb{K}_{\mathrm{dsc}}^{\mathrm{m}}$ as $\varepsilon\to 0^+$ because $D_n(\ftau)$ is a subset of $\tau\in\mathbb{K}_{\mathrm{dsc}}^{\mathrm{m}}$. We can thus assume that $(\tau,D_n(\ftau))$ under $\P\big( \cdot\, \big|\, |\tHS(\ftau)-x|<\varepsilon\big)$ converges in law to a random variable $(\ctau,D)\in(\mathbb{K}_{\mathrm{dsc}}^{\mathrm{m}})^2$. Then, we only need to prove that $D$ satisfies (\ref{ikea_identity}) by Lemma~\ref{ikea}. Before we start, we obtain that
\begin{equation}
\label{precision_eps_unique}
\forall y,z>0,\: \E\bigg[\sum_{\substack{1\leq i,j\leq k_\varnothing(\tau)\\ i\neq j}}\I{|\tHS(\theta_{(i)}\ftau)-y|<\varepsilon}\I{|\tHS(\theta_{(j)}\ftau)-z|<\varepsilon}\ \bigg|\ |\tHS(\ftau)-x|<\varepsilon\bigg]\underset{\varepsilon\rightarrow 0^+}{\longrightarrow} 0,
\end{equation}
because the expectation is smaller than $4\gamma^2\varepsilon^2\P(|\tHS(\ftau)-x|<\varepsilon)^{-1}\varphi_\alpha''\left(\P(\tHS(\ftau)\leq x+1)\right)$, which goes to $0$ when $\varepsilon\rightarrow0^+$ by Proposition~\ref{law_HT_stable}. We extensively use  (\ref{precision_eps_unique}) in this proof.
\smallskip

Let us denote by $\ftau_0,\ftau_1,\ldots ,\ftau_{k_\varnothing(\tau)-1}$ the weighted subtrees of $\ftau$ stemming from the children of $\varnothing$ listed in the decreasing order of their weighted Horton--Strahler numbers. Namely, 
\[\tHS(\ftau_0)>\ldots >\tHS(\ftau_{k_\varnothing(\tau)-1})\quad\text{and}\quad\big\{\ftau_i\ :\ 0\leq i\!\leq k_\varnothing(\tau)-1\}=\{\theta_{(j)}\ftau\ :\ 1\leq j\leq k_\varnothing(\tau)\big\}.\]
We also define $K_n(\ftau)=-1+\#\big\{1\leq i\leq k_\varnothing(\tau)\ :\ \tHS(\theta_{(i)}\ftau)\geq \tHS(\ftau)-1-\log_\delta n\big\}$.
From Definition~\ref{def_tHS_intro_full}, we observe that if $k_\varnothing(\tau)\geq 1$ then $K_n(\ftau)\geq 0$, and if $\tHS(\ftau_0)<\tHS(\ftau)$ then $K_n(\ftau)\geq 1$. In particular, the conditional probability that $K_n(\ftau)<0$ given $|\tHS(\ftau)-x|<\varepsilon$ goes to $0$ as $\varepsilon\to 0^+$, since $x>1$. Moreover, we observe that $D_n(\ftau)$ is expressed as follows.
\begin{enumerate}
\item[(a)] If $K_n(\ftau)=0$, then $D_n(\ftau)=\{o\}\circledast(o,\varpi\circledast D_n(\ftau_0))$.
\item[(b)] If $\tHS(\ftau_0)\!=\!\tHS(\ftau),K_n(\ftau)\!\geq\! 1$, then $D_n(\ftau)\!=\!\{o\}\!\circledast\!(o,\varpi\!\circledast\! D_n(\ftau_0))\!\circledast_{1\leq i\leq K_n(\ftau)}\!(o,\varpi\!\circledast\! \tau_i)$.
\item[(c)] If $\tHS(\ftau_0)<\tHS(\ftau)$, then $D_n(\ftau)=\{o\}\circledast(o,\varpi\circledast \tau_0)\circledast_{1\leq i\leq K_n(\ftau)}(o,\varpi\circledast \tau_i)$.
\end{enumerate}
\smallskip

\noi
We deduce from Definition~\ref{def_tHS_intro_full} and (\ref{precision_eps_unique}) (with $y=x$ and $z=x-1-\log_\delta n$) that outside an event of negligible probability compared to $\P(|\tHS(\ftau)-x|<\varepsilon)$ as $\varepsilon\to 0^+$, it holds that
\[\I{|\tHS(\ftau)-x|<\varepsilon\, ;\, K_n(\ftau)=0}=\sum_{i=1}^{k_\varnothing(\tau)}\I{|\tHS(\theta_{(i)}\ftau)-x|<\varepsilon}\prod_{\substack{1\leq j\leq k_\varnothing(\tau)\\ j\neq i}}\I{\tHS(\theta_{(j)}\ftau)<x-1-\log_\delta n)}.\]
We compute $\varphi_\alpha'(\P(\tHS(\ftau)<x-1-\log_\delta n))=1-n\delta^{1-x}$ by (\ref{derive_generating}) and Proposition~\ref{law_HT_stable}. Let $g_1,g_2:\mathbb{K}_{\mathrm{dsc}}^{\mathrm{m}}\longrightarrow \bbR$ be bounded, it follows from Definition~\ref{GWwdef} of the law of $\ftau$ that
\begin{equation}
\label{R_A_distribution_1}
\E\big[\I{K_n(\ftau)=0}g_1(D_n(\ftau_0))g_2(\tau_0)\ \big|\ |\tHS(\ftau)-x|<\varepsilon\big]\underset{\varepsilon\rightarrow 0^+}{\longrightarrow}(1-n\delta^{1-x})\E\big[g_1(D)g_2(\ctau)\big].
\end{equation}
Let $m\in\bbN^*$. Thanks to (\ref{precision_eps_unique}), we find that outside an event of negligible probability compared to $\P(|\tHS(\ftau)-x|<\varepsilon)$ as $\varepsilon\to 0^+$, the RV $\I{|\tHS(\ftau)-x|<\varepsilon\, ;\, \tHS(\ftau_0)=\tHS(\ftau)\, ;\,  K_n(\ftau)\geq m}$ is equal to
\[\!\!\! \!\!\!\!\! \!\!\!\!\! \!\!\!\!\! \!\!\!\!\!\sum_{\substack{\quad\quad\quad 1\leq i_0,\ldots,i_m\leq k_\varnothing(\tau)\\ \quad\quad\text{distinct}}}\!\! \!\!\!\!\!\!\!\!\!\!\! \!\!\!\!\! \!\!\!\!\! \!\!\I{|\tHS(\theta_{(i_0)}\ftau)-x|<\varepsilon}\I{x-1>\tHS(\theta_{(i_1)}\ftau)>\ldots>\tHS(\theta_{(i_m)}\ftau)>x-\log_\delta(\delta n)}\!\!\!\!\! \!\!\!\! \!\!\!\!\! \prod_{\substack{\quad\quad 1\leq j\leq k_\varnothing(\tau)\\ \quad\quad j\neq i_0,\ldots,i_m}}\!\!\!\!\! \!\!\!\!\! \!\!\!\!\! \I{\tHS(\theta_{(j)}\ftau)<\tHS(\theta_{(i_m)}\ftau)}.\]
Let $g_1,g_2,f_1,\ldots ,f_m:\mathbb{K}_{\mathrm{dsc}}^{\mathrm{m}}\longrightarrow\bbR$ be bounded. Thanks to Definition~\ref{GWwdef} and the identity (\ref{cond_H=x_identity}) in Proposition~\ref{cond_H=x_prime}, we recognize that
\begin{multline}
\label{R_A_distribution_2}
\E\Big[\I{\tHS(\ftau_0)=\tHS(\ftau)\, ;\, K_n(\ftau)\geq m}g_1(D_n(\ftau_0))g_2(\tau_0)\prod_{j=1}^m f_j(\tau_j)\ \Big|\ |\tHS(\ftau)-x|<\varepsilon\Big]\\
\underset{\varepsilon\rightarrow 0^+}{\longrightarrow}(n-1)\delta^{1-x}\E\big[g_1(D)g_2(\ctau)\big]\E\Big[\I{J_1^x\geq m}\prod_{j=1}^m f_j(T_{j,1}^x)\Big].
\end{multline}
Note from Proposition~\ref{law_HT_stable} that $\P(|\tHS(\ftau)-x|<\varepsilon)=e^{-\gamma}\P(|\tHS(\ftau)-x+1|<\varepsilon)$ and also that $\P(x-1<\tHS(\ftau)<x)=e^{-\gamma (x-1)}(1-e^{-\gamma})$. The same method used to find (\ref{R_A_distribution_2}) yields that
\begin{multline}
\label{R_A_distribution_3}
\E\Big[\I{\tHS(\ftau_0)<\tHS(\ftau)\, ;\, K_n(\ftau)\geq m}g_2(\tau_0)\prod_{j=1}^{m} f_j(\tau_{j})\ \Big|\ |\tHS(\ftau)-x|<\varepsilon\Big]\\
\underset{\varepsilon\rightarrow0^+}{\longrightarrow}\delta^{1-x}\E\big[g_2(\tau)\ \big|\ x-1<\tHS(\ftau)<x\big]\E\big[f_1(\tau)\ \big|\ \tHS(\ftau)=x-1\big]\E\Big[\I{\tilde{J}^x\geq m}\prod_{j=2}^m f_j(\tilde{T}_j^x)\Big].
\end{multline}
The convergences (\ref{R_A_distribution_1}), (\ref{R_A_distribution_2}), and (\ref{R_A_distribution_3}) together with the observations $(a)$, $(b)$, and $(c)$ entail that $D$ satisfies (\ref{ikea_identity}), which completes the proof of (I) as announced.
\end{proof}

\begin{proof}[Proof of (II)]
We keep the notation of Lemma~\ref{ikea}. We begin by showing that each component involved in the construction of $D_n^x$ converges in law after scaling to its continuum counterpart in Theorem~\ref{self-similar_alpha}. First, it holds $\delta^{-x}\odot_\alpha\tau_\rd^x\convd\delta^{-1}\odot_\alpha\sT_\rd$ on $\mathbb{K}^{\mathrm{m}}$ by Definition~\ref{stable_height_excursion_def}. We have $\P(x-1<\tHS(\ftau)<x)=\frac{1}{\alpha} e^{-\gamma(x-1)}$ by Proposition~\ref{law_HT_stable}. Let $f_0:\mathbb{K}^{\mathrm{m}}\to\bbR$ be continuous and bounded, we apply Proposition~\ref{cond_H=x_prime} and set $y=x-\log_\delta s$ to get
\begin{align*}
\E\big[f_0(\delta^{-x}\odot_\alpha\tau_\rg^x)\big]&=\alpha e^{\gamma(x-1)}\int_{x-1}^x\E\big[f_0(\delta^{-x}\odot_\alpha\tau)\ \big|\ \tHS(\ftau)=y\big]\gamma e^{-\gamma y}\, \dd y\\
&=\int_1^\delta s^{\beta-1}\E\big[f_0(s^{-1}\delta^{-x+\log_\delta s}\odot_\alpha\tau)\ \big|\ \tHS(\ftau)=x-\log_\delta s\big]\dd s.
\end{align*}
Thus, we have $\delta^{-x}\odot_\alpha\tau_\rg^x\convd U^{-1}\odot_\alpha\sT_\rg$ on $\mathbb{K}^{\mathrm{m}}$ by dominated convergence. It is clear that $\delta^{-x}L^x\convd L$ and this yields $\delta^{-x}\odot_\alpha\mathtt{L}^x\convd\mathtt{L}$ on $\mathbb{K}^{\mathrm{m}}$. For all $p\geq 1$, $(\delta^{-x}N_1^x,\ldots ,\delta^{-x}N_p^x)$ similarly converges jointly in distribution to the first $p$ points of a Poisson process with intensity $\delta(n-1)$, such as $\{\ell(\omega)\ :\ \omega\in\cN,1\leq \xi(\omega)\leq n\}$. Now, let us consider $m\geq 1$ and let $f_1,\ldots ,f_m:\mathbb{K}^{\mathrm{m}}\longrightarrow\bbR$ be continuous and bounded. Recall from (\ref{derive_generating}) and (\ref{stable_offspring_explicit}) that
$\varphi^{(m+1)}(s)=(1-s)^{\alpha-m-1}(\alpha-1)\prod_{i=2}^m(i-\alpha)$ for all $s\in[0,1)$. The change of variables $\xi=\delta^{x-1-y_1}$ and $\lambda_i=\delta^{y_i-y_{i-1}}$ for all $2\leq i\leq m$ reveal after some computations that
\[\E\Big[\I{J_{1}^x\geq m}\!\prod_{i=1}^m f_i(\delta^{-x}\odot_\alpha T_{i,1}^x)\Big]\!\longrightarrow\tfrac{1}{n-1}\!\!\int_1^n \!\! \E\Big[\I{\prod_{i=2}^{m}\Lambda_i\geq \xi/n}\prod_{i=1}^{m}f_i\Big(\!\Big(\tfrac{1}{\delta\xi}\prod_{j=2}^i\!\Lambda_j\Big)\odot_\alpha\sT_{i}\Big)\Big]\dd\xi.\]
Setting $\lambda_2=\delta^{y_2-x+1}$ and $\lambda_i=\delta^{y_i-y_{i-1}}$ for all $3\leq i\leq m$ similarly leads to
\[\E\Big[\I{\tilde{J}^x\geq m}\prod_{i=2}^m f_i(\delta^{-x}\odot_\alpha \tilde{T}_{i}^x)\Big]\longrightarrow\E\Big[\I{\prod_{i=2}^{m}\Lambda_i\geq 1/n}\prod_{i=2}^{m}f_i\Big(\!\Big(\tfrac{1}{\delta}\prod_{j=2}^i\!\Lambda_j\Big)\odot_\alpha\sT_{i}\Big)\Big].\]

By independence, all those convergences in distribution happen jointly. In particular, the numbers of components involved in the construction of the $D_n^x$ by grafting are tight as $x\in\bbR_+\backslash\bbN$ tends to $\infty$, because it converges in law to the number of components involved in the construction of $\sT_n^*$ by grafting. Hence, Proposition~\ref{graft_continuous} completes the proof of (II).
\end{proof}

\begin{proof}[Proof of (III)]
Let $\eta>0$, we observe that we only need to show
\begin{align}
\label{ikea_est_mass}
0&=\limsup_{n\rightarrow\infty}\limsup_{x\rightarrow\infty,x\notin\bbN}\limsup_{\varepsilon\rightarrow 0^+}\E\Big[\min\big(e^{-\gamma\alpha x}\#(\tau\backslash D_n(\ftau)),1\big)\, \big|\, |\tHS(\ftau)-x|\! <\! \varepsilon\Big],\\
\label{ikea_est_dist}
0&=\limsup_{n\rightarrow\infty}\limsup_{x\rightarrow\infty,x\notin\bbN}\limsup_{\varepsilon\rightarrow 0^+}\P\Big(\max_{u\in\tau, v\in D_n(\ftau)}\!\!\!\! \mathtt{d}_{\mathrm{gr}}(u,v)\geq \eta e^{\gamma(\alpha-1)x}\,  \Big|\, |\tHS(\ftau)-x|\! <\!\varepsilon\Big),
\end{align}
where we recall from (\ref{graph_distance}) that $\mathtt{d}_{\mathrm{gr}}$ is the graph distance on the set $\bbU$ of words. Let $x\in\bbR_+\backslash\bbN$ and $n\in\bbN^*$ with $x>1+\log_\delta n$. Let $\bt$ be a weighted tree whose underlying tree is denoted by $t$. To lighten notation, set 
\begin{align*}
F_{x,n}^{\mathfrak{m}}(\bt)&=\min\big(e^{-\gamma\alpha x}\#(t\backslash D_n(\bt)),1\big)\;& \text{and} & \; & F_{x,n}^{\mathfrak{h}}(\bt)&=\I{\exists u\in t, v\in D_n(\bt)\, :\, \mathtt{d}_{\mathrm{gr}}(u,v)\geq \eta e^{\gamma(\alpha-1)x}},\\
f_{x}^{\mathfrak{m}}(t)&=e^{-\gamma\alpha x}\#t\;& \text{and} &\; & f_{x}^{\mathfrak{h}}(t) &=\I{|t|+1\geq \eta e^{\gamma(\alpha-1)x}},
\end{align*}
where recall from (\ref{height_notation}) that $|t|$ is the height of $t$. As in the proof of (I), denote by $\ftau_0$ the weighted subtree of $\ftau$ stemming from a child of $\varnothing$ with maximal weighted Horton--Strahler number. From the definition of $D_n(\ftau)$, we deterministically observe that for both $\mathfrak{a}\in\{\mathfrak{m},\mathfrak{h}\}$,
\[F_{x,n}^{\mathfrak{a}}(\ftau)\leq \I{\tHS(\ftau_0)=\tHS(\ftau)}F_{x,n}^{\mathfrak{a}}(\ftau_0)+\sum_{i=1}^{k_\varnothing(\tau)}\I{\tHS(\theta_{(i)}\ftau)<\tHS(\ftau)-1-\log_\delta n}f_x^{\mathfrak{a}}(\theta_{(i)}\tau).\]
Let us set $E_{x,n}^{\mathfrak{a}}=\limsup_{\varepsilon\to 0+}\E\big[F_{x,n}^{\mathfrak{a}}(\ftau)\, \big|\, |\tHS(\ftau)-x|<\varepsilon\big]$. We obviously have $E_{x,n}^{\mathfrak{a}}\leq 1$. By tightness, an elementary argument based on the convergences (\ref{R_A_distribution_1}) and (\ref{R_A_distribution_2}) yields that
\[\limsup_{\varepsilon\rightarrow 0^+}\E\big[\I{\tHS(\ftau_0)=\tHS(\ftau)}F_{x,n}^{\mathfrak{a}}(\ftau_0)\ \big|\ |\tHS(\ftau)-x|<\varepsilon\big]\leq (1-\delta^{1-x})E_{x,n}^{\mathfrak{a}}.\]
Let $\varepsilon\in(0,1)$. If $|\tHS(\ftau)-x|<\varepsilon$ then there is $1\leq j\leq k_\varnothing(\tau)$ such that $|\tHS(\theta_{(j)}\ftau)-x|<\varepsilon$ or $|\tHS(\theta_{(j)}\ftau)-x+1|<\varepsilon$. Moreover, it would also hold that $\tHS(\theta_{(i)}\ftau)\leq x+1$ for all other $1\leq i\leq k_\varnothing(\tau)$. As $\HS(\tau)\leq\tHS(\ftau)$ and the law of $\tHS(\ftau)$ is exponential with mean $1/\gamma$, we get
\[E_{x,n}^{\mathfrak{a}}\leq (1-\delta^{1-x})E_{x,n}^{\mathfrak{a}}+(e^\gamma+1)\E\big[\I{\HS(\tau)\leq x-\log_\delta n}f_x^{\mathfrak{a}}(\tau)\big]\varphi_\alpha''(\P(\tHS(\ftau)\leq x+1)).\]
We deduce from (\ref{derive_generating}) that there is a constant $c\in(0,\infty)$ that only depends on $\alpha$ such that
\[E_{x,n}^{\mathfrak{a}}\leq ce^{\gamma x}\E\big[\I{\HS(\tau)\leq x-\log_\delta n}f_x^{\mathfrak{a}}(\tau)\big].\]
For $\mathfrak{a}=\mathfrak{m}$, (\ref{ikea_est_mass}) follows from (\ref{maj_size_selon_H}). For $\mathfrak{a}=\mathfrak{h}$, (\ref{equiv_hauteur}) and (\ref{maj_hauteur_selon_H_prime}) entail (\ref{ikea_est_dist}).
\end{proof}

\section{The Strahler dilation of rooted compact real trees}
\label{binary_information}

In this section, we prove Theorem~\ref{binary_info_intro}. Therefore, we construct and study the Strahler dilation, which corresponds to the weighted Horton--Strahler number for rooted compact real trees. 

\subsection{Definition and properties of the Strahler dilation}

Recall the notation needed to work on the set of words $\bbU$ from Section~\ref{words}. Here, we focus on the subset of $\mathbb{U}$ of all the words written in the alphabet $\{1,2\}$. For all $n\in\mathbb{N}$, we denote by \[\mathbb{W}=\bigcup_{k\in\mathbb{N}}\{1,2\}^k\quad\text{ and }\quad\mathbb{W}_n=\{u\in\mathbb{W}\ :\ |u|\leq n\}\] the perfect binary trees rooted at $\varnothing$ respectively of infinite height and of height equal to $n$. Moreover, recall from Section~\ref{real_tree_section} the genealogical order on a rooted compact real tree.

\begin{definition}
\label{embedding_continuum}
Let $(T,d,\rho)$ be a rooted compact real tree. An \textit{embedding} $\varphi$ of $\mathbb{W}$ (resp.~$\mathbb{W}_n$) into $T$ is an injective map from $\mathbb{W}$ (resp.~$\mathbb{W}_n$) to $T$ such that $\varphi(u\!\wedge\! v)\!=\!\varphi(u)\wedge \varphi(v)$ for all $u,v\in\mathbb{W}$ (resp.~$\mathbb{W}_n$). In that case, we write $\varphi:\mathbb{W}\hookrightarrow T$ (resp.~$\varphi:\mathbb{W}_n\hookrightarrow T$).
\end{definition}

\begin{notation}
\label{embedded_parent_root}
If $\varphi$ is an embedding of $\bbW$ or $\bbW_n$ into a rooted compact real tree $(T,d,\rho)$, then we set $\varphi(\overleftarrow{\varnothing})=\rho$. We stress that we still allow $\varphi(\varnothing)=\rho$ nonetheless.
\end{notation}

\begin{definition}
\label{info_def}
We set $\sup\emptyset=0$ and $\inf\emptyset=\infty$. For all $\cdelta\in(1,\infty)$ and all rooted compact real trees $(T,d,\rho)$, we define the \textit{Strahler dilation with base $\cdelta$} of $T$ as
\[\Info_\cdelta(T)=\lim_{k\rightarrow\infty}\; \lim_{n\rightarrow\infty}\; \sup_{\varphi_n:\mathbb{W}_n\hookrightarrow T}\; \inf_{\substack{u\in\mathbb{W}_n\\ k\leq |u|}}\; \frac{1}{|u|+1}\sum_{v\preceq u}\cdelta^{|v|+1}d\big(\varphi_n(\overleftarrow{v}),\varphi_n(v)\big).\]
Moreover, we define for all $k,n\in\bbN$,
\[\Info^{k,n}_\cdelta(T)=\sup_{\varphi_n:\mathbb{W}_n\hookrightarrow T}\; \inf_{\substack{u\in\mathbb{W}_n\\ k\leq |u|}}\; \frac{1}{|u|+1}\sum_{v\preceq u}\cdelta^{|v|+1}d\big(\varphi_n(\overleftarrow{v}),\varphi_n(v)\big).\]
For all $k\in\bbN$, we also write $\Info^k_\cdelta(T)=\lim_{n\rightarrow\infty}\Info^{k,n}_\cdelta(T)$, so that $\Info_\cdelta(T)=\lim_{k\rightarrow\infty}\Info^k_\cdelta(T)$.
\end{definition}

We fix $\cdelta>1$ throughout this section. Let us first prove that the Strahler dilation is well-defined. Let $(T,d,\rho)$ be a rooted compact real tree. As a supremum of infima of nonnegative numbers, $\Info^{k,n}_\cdelta(T)$ is well-defined in $[0,\infty]$. Moreover, if $\varphi_n:\bbW_n\hookrightarrow T$ and $v\in\bbW_n$, then $\varphi_n(\overleftarrow{v})\preceq\varphi_n(v)$ and so $d\big(\varphi_n(\overleftarrow{v}),\varphi_n(v)\big)\leq \mathfrak{h}(T)$, where $\mathfrak{h}$ is the height as in (\ref{height_mass}). Hence,
\begin{equation}
\label{info_first_bound}
\textit{if }\ k\leq n\quad\textit{ then }\quad\Info^{k,n}_\cdelta(T)\leq \tfrac{\cdelta^{k+2}-\cdelta}{(k+1)(\cdelta-1)}\mathfrak{h}(T)<\infty.
\end{equation}
As an embedding of $\mathbb{W}_{n+1}$ into $T$ induces an embedding of $\mathbb{W}_n$ into $T$ by restriction, we see
\begin{equation}
\label{monotony_info_k,n}
\forall k,n\in\bbN,\quad \Info^{k,n+1}_\cdelta(T)\leq \Info^{k,n}_\cdelta(T).
\end{equation}
Hence, $\Info_\cdelta^k(T)$ is well-defined in $[0,\infty)$ as the limit of a non-increasing sequence of finite terms. Moreover, clearly $\Info^{k,n}_\cdelta(T)\leq \Info^{k+1,n}_\cdelta(T)$ for all $k,n\in\bbN$, so taking $n$ to $\infty$ yields that
\begin{equation}
\label{monotony_info_k}
\forall k\in\bbN,\quad \Info^{k}_\cdelta(T)\leq\Info^{k+1}_\cdelta(T).
\end{equation}
Therefore, $\Info_\cdelta(T)$ is well-defined in $[0,\infty]$ as the limit of a non-decreasing sequence.

At the moment, it might be unclear why the Strahler dilation does enjoy the desired properties in Theorem~\ref{binary_info_intro} or why it is indeed analogous to the Horton--Strahler number. We will later give another expression for $\Info_\cdelta$ that may be more insightful. The main benefit of the first definition is that it only involves a countable number of operations. Thus, we will use it to show the measurability of $\Info_\cdelta$ and to later upper-bound the Strahler dilation. 

\begin{proposition}
\label{measurability_info}
Recall from Definition~\ref{real_tree} that $(\bbT_\bbR,\mathtt{d}_{\mathrm{GH}})$ is the metric space of (rooted-isometry classes of) rooted compact real trees. For all $k,n\in\bbN$, the map $\Info^{k,n}_\cdelta:\mathbb{T}_{\mathbb{R}}\to [0,\infty]$ is lower semicontinuous, and the maps $\Info^{k,n}_\cdelta,\Info^k_\cdelta,\Info_\cdelta:\mathbb{T}_{\mathbb{R}}\to [0,\infty]$ are measurable.
\end{proposition}

\begin{proof}
Fix $k,n\in\mathbb{N}$. Let $(T_0,d_0,\rho_0)$ be a rooted compact real tree and let $\varphi_0$ be an embedding of $\mathbb{W}_n$ into $T_0$. Set $\varepsilon=\min_{\substack{u,v\in\mathbb{W}_n \\ u\neq v}}d_0\big(\varphi_0(u),\varphi_0(v)\big)>0$. Let $(T,d,\rho)$ be a rooted compact real tree at rooted Gromov--Hausdorff distance less than $\varepsilon/10$ from $T_0$. There is a rooted correspondence $\cR$ between $T$ and $T_0$ such that $\dis(\cR)<\varepsilon/3$, as in (\ref{distortion}). For each $u\in\bbW_n$ with $|u|=n$, we choose a point of $T$ denoted by $\varphi(u)$ such that $\left(\varphi(u),\varphi_0(u)\right)\in\cR$. We are going to show that $\varphi$ extends into an embedding from $\mathbb{W}_n$ into $T$.

We begin by proving the following statement.
\begin{equation}
\label{measurability_info_1}
\forall u,v_1,v_2\in\mathbb{W}_n\backslash\mathbb{W}_{n-1},\ \textit{ if }\ u\wedge v_1\prec u\wedge v_2\ \textit{ then }\ \varphi(u)\wedge\varphi(v_1)\prec \varphi(u)\wedge\varphi(v_2).
\end{equation}
Thanks to (\ref{ancestral_totally_ordered}), we only need to show that $d(\rho,\varphi(u)\wedge\varphi(v_1))< d(\rho,\varphi(u)\wedge\varphi(v_2))$ since $\varphi(u)\wedge \varphi(v_1)$ and $\varphi(u)\wedge \varphi(v_2)$ are ancestors of $\varphi(u)$. To do this, we check that
\[2d(\rho,\sigma_1\wedge\sigma_2)=d(\rho,\sigma_1)+d(\rho,\sigma_2)-d(\sigma_1,\sigma_2)\]
for all $\sigma_1,\sigma_2\in T$ with (\ref{common_ancestor_real_tree_alt}) and we use the rooted correspondence $\mathcal{R}$ to obtain
\begin{equation}
\label{measurability_info_3}
\big|d(\rho,\varphi(u)\wedge\varphi(v_i))-d_0(\rho_0,\varphi_0(u)\wedge\varphi_0(v_i))\big|\leq \tfrac{3}{2}\dis(\mathcal{R})
\end{equation}
for all $i\in\{1,2\}$. Also, (\ref{ancestral_totally_ordered}) implies that $d_0(\rho_0,\varphi_0(u)\wedge\varphi_0(v_1))< d_0(\rho_0,\varphi_0(u)\wedge\varphi_0(v_2))$ because we have $\varphi_0(u)\wedge\varphi_0(v_1)\prec\varphi_0(u)\wedge \varphi_0(v_2)$ by Definition~\ref{embedding_continuum} of embeddings. By choice of $\varepsilon$, we even have $d_0(\rho_0,\varphi_0(u)\wedge\varphi_0(v_1))+\varepsilon\leq d_0(\rho_0,\varphi_0(u)\wedge\varphi_0(v_2))$.
The desired statement (\ref{measurability_info_1}) follows from (\ref{measurability_info_3}) because $3\dis(\mathcal{R})<\varepsilon$.
\smallskip

Next, we complete (\ref{measurability_info_1}) by showing the following implication.
\begin{equation}
\label{measurability_info_2}
\forall u,v_1,v_2\in\mathbb{W}_n\backslash\bbW_{n-1},\ \textit{ if }\ u\wedge v_1=u\wedge v_2\ \textit{ then }\ \varphi(u)\wedge\varphi(v_1)=\varphi(u)\wedge\varphi(v_2).
\end{equation}
This is obvious when $u,v_1,v_2$ are not distinct, so let us assume they are. As such, we have $u\wedge v_1\in\bbW_{n-1}$. In particular, $u\wedge v_1=u\wedge v_2$ has exactly two children in $\bbW_n$: one is an ancestor of $u$, the other is an ancestor of $v_1$ and $v_2$. Therefore, $u\wedge v_1=u\wedge v_2\prec v_1\wedge v_2$ so two applications of (\ref{measurability_info_1}) entail that $\varphi(u)\wedge\varphi(v_1)\preceq \varphi(v_2)$ and $\varphi(u)\wedge\varphi(v_2)\preceq \varphi(v_1)$. We compute $\varphi(u)\wedge\varphi(v_1)\wedge \varphi(v_2)$ twice with the associative property of $\wedge$ to obtain (\ref{measurability_info_2}).
\smallskip

Let $u_1,v_1,u_2,v_2\in\bbW_n\backslash\bbW_{n-1}$ with $u_1\wedge v_1=u_2\wedge v_2$. The ancestral lineage of $u_1$ is totally ordered by $\preceq$ so we may assume $u_1\wedge v_2\preceq u_1\wedge u_2$ without loss of generality. This yields that $u_1\wedge v_1=u_2\wedge v_2=u_1\wedge u_2\wedge v_2=u_1\wedge v_2$. We use (\ref{measurability_info_2}) twice to obtain that $\varphi(u_1)\wedge\varphi(v_1)=\varphi(u_2)\wedge\varphi(v_2)$. Hence, we have justified that there exists a unique extension $\varphi:\bbW_n\longrightarrow T$ such that $\varphi(u\wedge v)=\varphi(u)\wedge\varphi(v)$ for all $u,v\in\bbW_n\backslash\bbW_{n-1}$. Now, let $u,v\in\bbW_n$ be arbitrary and let us choose $u_1,u_2,v_1,v_2\in\bbW_n\backslash\bbW_{n-1}$ such that $u=u_1\wedge u_2$ and $v=v_1\wedge v_2$. The set of ancestors of $u_1$ is totally ordered by $\preceq $ so there is $u'\in\{u_2,v_1,v_2\}$ such that $u_1\wedge u'\preceq u_1\wedge u_2,u_1\wedge v_1, u_1\wedge v_2$, and so $u_1\wedge u'=u_1\wedge u_2\wedge v_1\wedge v_2=u \wedge v$. It then follows from (\ref{measurability_info_1}) and (\ref{measurability_info_2}) that 
\begin{equation}
\label{measurability_info_4}
\forall u,v\in\bbW_n,\quad \varphi(u)\wedge\varphi(v)=\varphi(u\wedge v),
\end{equation}
because they are both equal to $\varphi(u_1)\wedge\varphi(u')$. Combining (\ref{measurability_info_1}) and (\ref{measurability_info_4}) shows that if $u\prec v$ then $\varphi(u)\neq \varphi(v)$. Finally, if $\varphi(u)=\varphi(v)$ then $\varphi(u\wedge v)=\varphi(u)=\varphi(v)$ by (\ref{measurability_info_4}) and $u=u\wedge v=v$. Hence, $\varphi$ is indeed an embedding from $\bbW_n$ into $T$ as initially claimed.
\smallskip

Definition~\ref{embedding_continuum} gives that $d(\varphi(\overleftarrow{v}),\varphi(v))=d(\rho,\varphi(v))-d(\rho,\varphi(\overleftarrow{v}))$ for all $v\in\bbW_n$. This formula together with the bound (\ref{measurability_info_3}) yields that for all $v\in\bbW_n$, we have
\[\big|d(\varphi(\overleftarrow{v}),\varphi(v))-d_0(\varphi_0(\overleftarrow{v}),\varphi_0(v))\big|\leq 3\dis(\cR).\]
As $\varphi$ is an embedding from $\bbW_n$ into $T$, this inequality leads to
\[\inf_{\substack{u\in\mathbb{W}_n\\ k\leq |u|}}\; \tfrac{1}{|u|+1}\sum_{v\preceq u}\cdelta^{|v|+1}d_0\big(\varphi_0(\overleftarrow{v}),\varphi_0(v)\big)\leq 3\cdelta\tfrac{\cdelta^{n+1}-1}{(k+1)(\cdelta-1)}\dis(\cR)+\Info^{k,n}_\cdelta(T).\]
Take the infimum over all rooted correspondences $\cR$ between $T$ and $T_0$, then let $T$ tend to $T_0$ for the rooted Gromov--Hausdorff distance, and finish by taking the supremum over all embeddings $\varphi_0$ of $\bbW_n$ into $T_0$. We get $\Info^{k,n}_\cdelta(T_0)\leq \liminf_{T\rightarrow T_0}\Info^{k,n}_\cdelta(T)$, which is the lower semicontinuity of $\Info^{k,n}_\cdelta$. The measurability of $\Info^{k,n}_\cdelta,\Info^k_\cdelta,\Info_\cdelta$ then follows by classic results.
\end{proof}

Let $(T,d,\rho)$ be a rooted compact real tree. If $\varphi$ is an embedding of $\bbW$ into $T$, we define
\begin{equation}
\label{alt_info_step}
\Info_\cdelta(T,\varphi)=\liminf_{n\rightarrow\infty}\; \frac{1}{n+1}\inf_{\substack{u\in\mathbb{W}\\ |u|=n}}\; \sum_{v\preceq u}\cdelta^{|v|+1}d\big(\varphi(\overleftarrow{v}),\varphi(v)\big).
\end{equation}
Then, the Strahler dilation of $T$ can be expressed as follows:
\begin{equation}
\label{alt_info}
\Info_\cdelta(T)=\sup_{\varphi:\mathbb{W}\hookrightarrow T}\Info_\cdelta(T,\varphi).
\end{equation}
The formula (\ref{alt_info}) is similar in spirit to the expression (\ref{HS_maximal_height}) of the Horton--Strahler number. Indeed, let us explain why we may understand $\log_\cdelta \Info_\cdelta(T,\varphi)$ as a continuum counterpart for the height of an embedded perfect binary tree. If $t$ is a perfect binary tree, then so are $\theta_{(1)}t$ and $\theta_{(2)}t$, and $|t|=1+|\theta_{(1)}t|=1+|\theta_{(2)}t|$. Given an embedding $\varphi$ of $\bbW$ into $T$, set $\varphi_i(u)=\varphi((i)*u)$ for all $u\in\bbW$ and for $i\in\{1,2\}$, and note that $\varphi_1$ and $\varphi_2$ are two embeddings of $\bbW$ into $T$ such that $\Info_\cdelta(T,\varphi)=\cdelta\min(\Info_\cdelta(T,\varphi_1),\Info_\cdelta(T,\varphi_2))$; this follows readily from (\ref{alt_info_step}). Of course, it is not reasonable to require that $\Info_\cdelta(T,\varphi_1)=\Info_\cdelta(T,\varphi_2)$, so that the embedding $\varphi$ would be truly perfect, but the involved $\min$ penalizes unbalanced embeddings of $\bbW$ into $T$. Informally, $\Info_\cdelta(T)$ measures the largest scaling we can apply to the \emph{$\cdelta$-dyadic tree} with edge lengths $(\bbW,(\cdelta^{-|u|-1})_{u\in\bbW})$ to essentially embed it into $T$ without contracting its distances. We use limits of averages both to only focus on the boundary structure of the real tree and to obtain a deterministic result for the $\aHS_\alpha$-real tree thanks to a law of large numbers.

\begin{proof}[Proof of the identity (\ref{alt_info})]
An embedding $\varphi:\mathbb{W}\hookrightarrow T$ induces embeddings of $\mathbb{W}_n$ into $T$ by restriction, for all $n\in\mathbb{N}$. For all $k,n\in\mathbb{N}$ with $k\leq n$, it thus holds that
\[\inf_{j\geq k}\; \frac{1}{j+1}\inf_{\substack{u\in\mathbb{W}\\ |u|=j}}\; \sum_{v\preceq u}\cdelta^{|v|+1}d\big(\varphi(\overleftarrow{v}),\varphi(v)\big)\leq \Info^{k,n}_\cdelta(T).\]
We first let $n\rightarrow\infty$ then $k\rightarrow\infty$ to find that $\Info_\cdelta(T,\varphi)\leq \Info_\cdelta(T)$ holds for any embedding $\varphi$ of $\mathbb{W}$ into $T$, and so $\Info_\cdelta(T)\geq \sup_{\varphi:\mathbb{W}\hookrightarrow T}\Info_\cdelta(T,\varphi)$. Now, we assume $\Info_\cdelta(T)>0$, because the result is obvious otherwise, and we fix $k\in\mathbb{N}$ such that $\Info^k_\cdelta(T)>0$. The fact (\ref{monotony_info_k,n}) ensures that for all $n\geq k$, we can choose an embedding $\varphi_n^k:\mathbb{W}_n\hookrightarrow T$ such that
\[\inf_{\substack{u\in\mathbb{W}_{n}\\ k\leq |u|}}\; \tfrac{1}{|u|+1}\sum_{v\preceq u}\cdelta^{|v|+1}d\big(\varphi_n^k(\overleftarrow{v}),\varphi_n^k(v)\big)\geq \big(1-\tfrac{1}{n+2}\big)\Info^{k,n}_\cdelta(T)>0.\]
By compactness of $T$, we can assume there are $\varphi^k(u)\in T$ such that $\varphi_n^k(u)\to \varphi^k(u)$ for all $u\in\mathbb{W}$, by using a diagonal extraction argument. By taking $n$ to $\infty$, it is then clear that
\begin{equation}
\label{alt_info_tool}
0<\Info^k_\cdelta(T)\leq\tfrac{1}{|u|+1}\sum_{v\preceq u}\cdelta^{|v|+1}d\big(\varphi^k(\overleftarrow{v}),\varphi^k(v)\big)
\end{equation}
for all $u\in\mathbb{W}$ with $|u|\geq k$. That uniform lower bound implies that any $u\in\bbW$ has a descendant $u_+=u*(1,\ldots ,1)$ such that $\varphi^k(\overleftarrow{u_+})\neq \varphi^k(u_+)$. Also, it holds that $d(\varphi^k(\overleftarrow{u}),\varphi^k(u_+))\geq d(\varphi^k(\overleftarrow{u_+}),\varphi^k(u_+))>0$ because the $\varphi_n^k$ are embeddings. Applying this observation to $u*(1)$ and $u*(2)$ yields the existence of some $u_1,u_2\in\bbW$ with $u=u_1\wedge u_2$ such that $\varphi^k(u)\neq\varphi^k(u_i)$ for both $i\in\{1,2\}$. It follows that when $n$ is large enough, the point $\varphi_n^k(u_i)$ is closer to $\varphi^k(u_i)$ than to $\varphi_n^k(u)=\varphi_n^k(u_1)\wedge\varphi_n^k(u_2)$ for all $i\in\{1,2\}$. Lemma~\ref{tool_wedge_real_tree} then yields that the convergent sequence $(\varphi_n^k(u))_{n\geq k}$ is eventually constant equal to $\varphi^k(u_1)\wedge\varphi^k(u_2)$. Hence, we have $\varphi_n^k(u)=\varphi^k(u)$ for large enough $n$, for all $u\in\bbW$. It is then straightforward to show the map $\varphi^k:u\in\bbW\longmapsto\varphi^k(u)\in T$ is an embedding. Finally, we deduce from (\ref{alt_info_tool}) that $\Info^k_\cdelta(T)\leq\sup_{\varphi:\mathbb{W}\hookrightarrow T}\Info_\cdelta(T,\varphi)$ for all large enough $k\in\mathbb{N}$, then we let $k$ tend to $\infty$.
\end{proof}

Although measurable, the Strahler dilation is not continuous. Nonetheless, it enjoys some other regularity properties. Recall Definition~\ref{subtree} of subtrees of a rooted compact real tree $(T,d,\rho)$. We remind from (\ref{notation_scaling}) that if $\lambda\geq 0$, then $\lambda\cdot T=(T,\lambda d,\rho)$. 
\begin{proposition}
\label{monotone_homogene}
The function $\Info_\cdelta$ is monotone and homogeneous, namely:
\begin{longlist}
\item[(i)] if $T_1$ is a subtree of a rooted compact real tree $T_2$, then $\Info_\cdelta(T_1)\leq\Info_\cdelta(T_2)$,
\item[(ii)] if $T$ is a rooted compact real tree and if $\lambda\geq 0$, then $\Info_\cdelta(\lambda\cdot T)=\lambda\Info_\cdelta(T)$.
\end{longlist}
\end{proposition}
This proposition is clear, either by Definition~\ref{info_def} or by using the formula (\ref{alt_info}). Monotony is not surprising as it is shared with the Horton Strahler number: if $\psi:t_1\to t_2$ is an embedding between two (discrete) trees, then we observe $\HS(t_1)\leq \HS(t_2)$ by (\ref{HS_maximal_height}). Homogeneity was one of the needed properties and is indicative of the metric nature of the Strahler dilation. In particular, Propositions~\ref{measurability_info} and \ref{monotone_homogene} $(ii)$ entails point $(i)$ of Theorem~\ref{binary_info_intro} (setting $\Info_\cdelta(T,d,\rho,\mu)=\Info_\cdelta(T,d,\rho)$ for all rooted measured compact real trees).

The following result further relates the Strahler dilation with the Horton--Strahler number.

\begin{proposition}
\label{branching_property}
Let $\ell\geq 0$ and let $\mathtt{L}\in\bbT_\bbR$ be the real segment $[0,\ell]$ rooted at $0$. Let $(T_i,d_i,\rho_i)_{i\in I}$ be a countable family of rooted compact real trees such that $\{i\in I:\mathfrak{h}(T_i)\geq\varepsilon\}$ is finite for all $\varepsilon>0$. Let $i_0\in I$ and let $\ell_j\in[0,\ell]$ for all $j\in I\backslash\{i_0\}\! =:\! J$. If we denote $(T,d,\rho)=\mathtt{L}\circledast(\ell,T_{i_0})\circledast_{j\in J}(\ell_j,T_j)$, then
\begin{align}
\label{branching_property_<_step}
(n+1)\Info_\cdelta^{n,n}(T)&\leq \cdelta\ell+\max\big((n+1)\Info_\cdelta^{n,n}(T_{i_0}),\cdelta\sup_{j\in J}n\Info_\cdelta^{n-1,n-1}(T_j)\big),\\
\label{branching_property_<_limit}
\Info_\cdelta(T)&\leq \max\big(\Info_\cdelta(T_{i_0}),\cdelta\sup_{j\in J}\Info_\cdelta(T_j)\big),
\end{align}
for all $n\in\bbN^*$. Furthermore, if $T_i$ is planted for all $i\in I$, as in Definition~\ref{planted}, then
\begin{equation}
\label{branching_property_=}
\Info_\cdelta\big(\mathtt{L}\circledast_{i\in I}(\ell,T_i)\big)=\sup_{i,j\in I}\max\Big(\Info_\cdelta(T_i),\Info_\cdelta(T_j),\cdelta^{\I{i\neq j}}\min\big(\Info_\cdelta(T_i),\Info_\cdelta(T_j)\big)\Big).
\end{equation}
\end{proposition}

\begin{proof}
We begin with gathering some useful facts. We use the convention $\bbW=\bbW_\infty$ and we write $\ell_{i_0}=\ell$. Let $n\in\bbN^*\!\cup\!\{\infty\}$, $w\in\bbW_n$, $j\in I$, and let $\varphi:\bbW_n\hookrightarrow T$. When $\varphi(w*u)\in T_j$ for all $u\in\bbW_{n-|w|}$, Proposition~\ref{wedge_grafting} yields that $\varphi_{j}^w:u\in\bbW_{n-|w|}\mapsto\varphi(w*u)\in T_j$ is an embedding of $\bbW_{n-|w|}$ into $T_j$. We also deduce from Proposition~\ref{wedge_grafting} the following.
\begin{equation}
\label{branching_property_1}
\textit{If }\quad \varphi(w)\in T_j\backslash\{\rho_j\},\quad \textit{ then }\quad \varphi(w*u)\in T_j\: \textit{ for all }u\in\bbW_{n-|w|}.
\end{equation}
When $\varphi(\bbW_n)\subset T_j$, then $\varphi(u)=\varphi_j^\varnothing(u)$ for all $u\in\bbW_n$, but we stress that $\varphi(\overleftarrow{\varnothing})\! =\! \rho\! =\! 0$ and $\varphi_j^\varnothing(\overleftarrow{\varnothing})=\rho_j=\ell_j$ by Notation~\ref{embedded_parent_root}. Still, if $\varphi(\bbW_n)\subset T_j$ then $\ell_j\preceq\varphi(\varnothing)$, so we compute that
\begin{equation}
\label{branching_property_2}
\forall u\in\bbW_n,\quad \sum_{v\preceq u}\cdelta^{|v|+1}d\big(\varphi(\overleftarrow{v}),\varphi(v)\big)=\cdelta\ell_j+\sum_{v\preceq u}\cdelta^{|v|+1}d_j\big(\varphi_j^{\varnothing}(\overleftarrow{v}),\varphi_j^{\varnothing}(v)\big).
\end{equation}
Here, we assume that $|w|=1$. If $\varphi(w)\in T_j\backslash\{\rho_j\}$ \textbf{\emph{and}} $\varphi(\varnothing)=\ell_j$, then we readily see that
\begin{equation}
\label{branching_property_3}
\forall u\in\bbW_{n-1},\quad \sum_{v\preceq w*u}\cdelta^{|v|+1}d\big(\varphi(\overleftarrow{v}),\varphi(v)\big)=\cdelta\ell_j+\cdelta\sum_{v\preceq u}\cdelta^{|v|+1}d_j\big(\varphi_j^{w}(\overleftarrow{v}),\varphi_j^{w}(v)\big).
\end{equation}
We then claim that at least one of the following assertions holds true:
\begin{itemize}
\item[(a)] there is $j\in I$ such that $\varphi(\bbW_n)\subset T_j$,
\item[(b)] there are $j\in I\backslash\{i_0\}$ and $w\in\bbW$ with $|w|=1$ such that $\varphi(w)\in T_j\backslash\{\rho_j\}$ and $\varphi(\varnothing)=\ell_j$.
\end{itemize}
Indeed, let us assume that $(b)$ does not hold. Since $\mathtt{L}=\llbracket \rho,\ell\rrbracket$ is totally ordered by $\preceq$ and $(\varphi(1)\wedge\ell)\wedge(\varphi(2)\wedge\ell)=\varphi(\varnothing)\wedge\ell$, we have $w\in\bbW$ with $|w|=1$ with $\varphi(w)\wedge\ell=\varphi(\varnothing)\wedge\ell$. In particular, $\varphi(w)\notin\mathtt{L}$ because it would imply $\varphi(w)\preceq \varphi(\varnothing)$ otherwise. Thus, there is $j\in I$ such that $\varphi(w)\in T_j\backslash\{\rho_j\}$. Then, Proposition~\ref{wedge_grafting} entails that either $\varphi(\varnothing)\in T_j\backslash\{\rho_j\}$ or $\varphi(\varnothing)\preceq\ell_j=\varphi(w)\wedge\ell$. In the first case, (\ref{branching_property_1}) yields that $(a)$ is satisfied. In the second case, we get $\varphi(\varnothing)=\varphi(\varnothing)\wedge\ell=\ell_j$. Thus, $j=i_0$, and so $\varphi(\varnothing)=\ell$, because $(b)$ would hold otherwise. Since $\ell\prec\varphi(1),\varphi(2)$, we get $\varphi(1),\varphi(2)\notin\mathtt{L}$ and $\varphi(1)\wedge\ell=\varphi(2)\wedge \ell=\ell=\varphi(\varnothing)$. As before, this implies that $\varphi(1),\varphi(2)\in T_{i_0}\backslash\{\rho_{i_0}\}$. Finally, $(a)$ is verified by (\ref{branching_property_1}).
\smallskip

We are now ready to prove the desired relations. First, recall that $\cdelta>1$ and $\ell_j\leq\ell$ for all $j\in I$. Then, combining the previous alternative with the identities (\ref{branching_property_2}) and (\ref{branching_property_3}) readily gives (\ref{branching_property_<_step}). Relying on (\ref{alt_info}), the same method with $n=\infty$ yields (\ref{branching_property_<_limit}). To prove (\ref{branching_property_=}), we focus on the case where $n=\infty$ and $\ell_j=\ell$ for all $j\in I$. When $(a)$ is satisfied, (\ref{branching_property_2}) yields $\Info_\cdelta(T,\varphi)=\Info_\cdelta(T_j,\varphi_j^\varnothing)$. Conversely, if $\varphi_j$ is an embedding of $\bbW$ into $T_j$ then we can also see it as an embedding $\varphi:\bbW\hookrightarrow T$ such that $\varphi_j^\varnothing=\varphi_j$. When $(b)$ is satisfied, $\varphi(\varnothing)=\ell$, so $\varphi(1)\in T_i\backslash\{\rho_i\}$ and $\varphi(2)\in T_j\backslash\{\rho_j\}$ for some $i,j\in I$. In fact, $i\neq j$ by Definition~\ref{planted} of planted real trees. Then, (\ref{branching_property_3}) entails that $\Info_\cdelta(T,\varphi)=\cdelta\min\big(\Info_\cdelta(T_i,\varphi_i^{(1)}),\Info_\cdelta(T_j,\varphi_j^{(2)})\big)$. Conversely, if $\varphi_i,\varphi_j$ are two embeddings of $\bbW$ resp.~into $T_i,T_j$, we check with Proposition~\ref{wedge_grafting} that we can construct $\varphi:\bbW\hookrightarrow T$ such that $\varphi(\varnothing)=\ell$, $\varphi_i^{(1)}=\varphi_i$ and $\varphi_j^{(2)}=\varphi_j$. Note that $\varphi$ is injective because $T_i$ and $T_j$ are planted. Taking the supremum over all $\varphi$ completes the proof.
\end{proof}

We have constructed a family of functions $(\Info_\cdelta)_{\cdelta>1}$ parametrized by $(1,\infty)$. However, a rooted compact real tree $T$ admits at most one parameter $\cdelta > 1$ such that its Strahler dilation with base $\cdelta$ is not trivial. We define this critical parameter as its \textit{Strahler base}:
\begin{equation}
\label{dim_info_def}
\base_{\Info} T=\inf\, \{\cdelta>1\ :\ \Info_\cdelta(T)=\infty\}.
\end{equation}

\begin{proposition}
\label{dim_info_prop}
Let $T$ be a rooted compact real tree. If $1<\cdelta <\base_{\Info} T$, then $\Info_\cdelta(T)\!=\!0$. If $\cdelta\! >\! \base_{\Info} T$, then $\Info_\cdelta(T)\!=\!\infty$. Thus, $\base_{\Info} T=\sup\, \{1\}\cup\{\cdelta>1\, :\, \Info_\cdelta(T)\!=\!0\}$.
\end{proposition}

\begin{proof}
Let $1<\cdelta_1<\cdelta_2$ and let $m\in\mathbb{N}$. We can write the easy inequality 
\[\Info^{k,n}_{\cdelta_1}(T)\leq \big(\tfrac{\cdelta_1}{\cdelta_2}\big)^{m+1}\Info^{k,n}_{\cdelta_2}(T)+\tfrac{1}{k+1}\sum_{i=0}^{m-1}\cdelta_1^{i+1}\mathfrak{h}(T)\]
for all $k,n\in\mathbb{N}$. In that precise order, we take $n,k,m$ to $\infty$. We thus obtain the inequality $\Info_{\cdelta_1}(T)\leq \I{\Info_{\cdelta_2}(T)=\infty}\Info_{\cdelta_2}(T)\leq \Info_{\cdelta_2}(T)$, which yields the result.
\end{proof}

\subsection{Proof of Theorem~\ref{binary_info_intro} $(ii)$}

Recall $\delta=e^{\gamma(\alpha-1)}=(\tfrac{\alpha}{\alpha-1})^{\alpha-1}$ and $\delta\in(1,2]$ from (\ref{alpha_gamma_delta}). Here, our goal is to prove Theorem~\ref{binary_info_intro} $(ii)$, namely that the Strahler dilation with base $\delta$ of an $\aHS_\alpha$-real tree $\sT$ is almost surely equal to $1$. Recall their Definitions~\ref{stable_height_excursion_def} and \ref{info_def}. We divide the proof into three steps: showing that $\Info_\delta(\sT)$ is almost surely constant, showing that its mean is not smaller than $1$, and showing that its mean is not bigger than $1$.

\begin{proposition}
Let $\sT$ be an $\aHS_\alpha$-real tree. Then $\Info_\delta(\sT)$ is a.s.~constant.
\end{proposition}

\begin{proof}
Let $F(x)=\P(\Info_\delta(\sT)\!\leq\! x)$ for all $x\in\bbR$ and let $a=\sup\{x\in\bbR : F(x)=0\}$. We have $a\geq 0$. If $a=\infty$ then $\Info_\delta(\sT)=\infty$ almost surely. Thus, we now assume that $a<\infty$ and we only need to show that $F(a)=1$ to conclude that $\Info_\delta(\sT)=a$ almost surely. In this proof, we use the notation of Theorem~\ref{self-similar_alpha}. We define $G(x)=\P(\delta\, \Info_\delta(\sT_\rg)\leq xU)$ for all $x\geq 0$. Note that $G(x)=\E[F(xU/\delta)]$ by independence between $\sT_\rg$ and $U$. By Proposition~\ref{natural_lifetime_scaling_limit} $(iii)$, we can apply (\ref{branching_property_=}) in Proposition~\ref{branching_property} together with Proposition~\ref{monotone_homogene} to learn that
\[\Info_\delta(\sT^*)\geq \max\Big(\tfrac{1}{U}\Info_\delta(\sT_\rg)\, ,\, \tfrac{1}{\delta}\Info_\delta(\sT_\rd)\, ,\, \min\big(\tfrac{\delta}{U}\Info_\delta(\sT_\rg),\Info_\delta(\sT_\rd)\big)\Big)\]
almost surely. Since $\sT^*,\sT_\rg,\sT_\rd$ are $\aHS_\alpha$-real trees, this leads to the inequality
\begin{equation}
\label{functional_inequation_I(T)_prelim}
F(x)\leq F(x)G(x)+ G(x)\big(F(\delta x)-F(x)\big)+F(x)\big(G(\delta x)-G(x)\big)
\end{equation}
for all $x\geq 0$. As cumulative distribution functions, $F$ and $G$ are nonnegative, bounded by $1$, non-decreasing, and càdlàg. Moreover, we have $F(y)=0<F(x)$ when $y<a<x$ by definition of $a$. So $G(x)=\E[\I{xU\geq \delta a}F(xU/\delta)]$ which implies $G(x)\leq \P(xU\geq\delta a)F(x)$ because $U\leq\delta$ a.s. Inserting this inequality into (\ref{functional_inequation_I(T)_prelim}) and dividing by $F(x)$ get us
\begin{equation}
\label{functional_inequation_I(T)}
1\leq G(\delta x)+\P\left(xU\geq \delta a\right)\left(F(\delta x)-F(x)\right)
\end{equation}
for all $x>a$. If $a=0$ then $F(\delta x)-F(x)\longrightarrow F(0)-F(0)=0$ as $x$ tends to $a+$. If $a>0$ then $\P(xU\geq\delta a)\longrightarrow \P(U\geq\delta)=0$ as $x\to a^+$. We thus obtain that $G(\delta a)=1$ by letting $x\to a^+$ in (\ref{functional_inequation_I(T)}), whatever the case. It follows that $F(aU)=1$ almost surely because $F\leq 1$. Since $U$ admits a positive density on $[1,\delta]$, so we get $F(a) = F(a+) = 1$. 
\end{proof}

\begin{proposition}
Let $\sT$ be an $\aHS_\alpha$-real tree. Then $\E[\Info_\delta(\sT)]\geq 1$.
\end{proposition}

\begin{proof}
Let $(L_u,U_u)_{u\in\mathbb{W}}$ be independent RVs such that for all $u\in\bbW$, the law of $L_u$ is exponential with mean $\delta^{-1}$ and the law of $U_u$ is $\un_{[1,\delta]}(s)s^{\beta-1}\dd s$. Denote by $u_1^n,\ldots,u_{2^{n+1}-1}^n$ the vertices of $\bbW_n$ in lexicographic order. We claim that for all $n\in\bbN$, there is a random $2^{n+1}$-pointed compact metric space $(\sT_n,d_n,\rho_n,\varphi_n(u_1^n),\ldots,\varphi_n(u_{2^{n+1}-1}^n))$ such that
\begin{itemize}
\item[(a)]$(\sT_n,d_n,\rho_n)$ is an $\aHS_\alpha$-real tree,
\item[(b)] $\varphi_n:u\in\bbW_n\mapsto\varphi_n(u)\in\sT_n$ is almost surely an embedding,
\item[(c)] $(\delta^{|u|+1}d_n(\varphi_n(\overleftarrow{u}),\varphi_n(u)))_{u\in\bbW_n}$ has the same law as $(\delta L_u\prod_{v*(1)\preceq u} \delta U_v^{-1})_{u\in\bbW_n}$.
\end{itemize}
We prove this using the notation of Theorem~\ref{self-similar_alpha}. For $n=0$, we set $\sT_0=\sT^*$ and $\varphi_0(\varnothing)=L$. By induction, we can assume that $\sT_\rg$ (resp.~$\sT_\rd$) is endowed with an a.s.~embedding $\varphi_{\rg}$ (resp.~$\varphi_\rd$) from $\bbW_n$ into $\sT_\rg$ (resp.~$\sT_\rd)$ such that $(c)$ is satisfied. Then, we construct $\sT_{n+1}$ by endowing $\sT^*$ with points $(\varphi_{n+1}(u))_{u\in\bbW_{n+1}}$ defined by $\varphi_{n+1}(\varnothing)=L$, and $\varphi_{n+1}((1)*u)=\varphi_{\rg}(u)$ and $\varphi_{n+1}((2)*u)=\varphi_{\rd}(u)$ for all $u\in\bbW_n$. We easily check the properties $(a$-$c)$ thanks to Theorem~\ref{self-similar_alpha} and Proposition~\ref{wedge_grafting}, which establishes the claim. 
\smallskip

Then, our claim yields that for all $n,k\in\bbN$,
\begin{equation}
\label{minoration_E[I]_analitc_step} 
\E\big[\Info^{k,n}_\delta(\sT)\big]\geq \E\bigg[ \inf_{\substack{u\in\bbW\\ k\leq |u|\leq n}}\; \frac{\delta}{|u|+1}\sum_{v\preceq u}L_v\prod_{w*(1)\preceq v} \frac{\delta}{U_w} \bigg].
\end{equation}
Next, we find that $\E[\Info^{k,k}_\delta(\sT)]\! <\! \infty$ by (\ref{info_first_bound}) and (\ref{maj_hauteur_limit_tree}). On both sides of (\ref{minoration_E[I]_analitc_step}), we may first apply the dominated convergence theorem as $n\to \infty$, since $\Info_\delta^{k,n}(\sT)\leq\Info_\delta^{k,k}(\sT)$ for $n\geq k$ by (\ref{monotony_info_k,n}), and then the monotone convergence theorem as $k\to\infty$ thanks to (\ref{monotony_info_k}). Thus,
\begin{equation}
\label{minoration_E[I]_analitc}
\E\big[\Info_\delta(\sT)\big]\geq \E\bigg[\liminf_{n\rightarrow\infty}\; \frac{\delta}{n+1}\inf_{\substack{u\in\bbW\\ |u|=n}}\; \sum_{v\preceq u}L_v\prod_{w*(1)\preceq v} \frac{\delta}{U_w}\bigg].
\end{equation}
It remains to study the right-hand side of (\ref{minoration_E[I]_analitc}). For all $n\in\mathbb{N}$ and $s\in\bbR_+$, we set
\[f_n(s)=\E\bigg[\sum_{u\in\bbW,|u|=n}\exp\Big(-s\delta\sum_{v\preceq u}L_v\prod_{w*(1)\preceq v} \frac{\delta}{U_w}\Big)\bigg]\:\text{ and }\: g(s)=\limsup_{n\rightarrow\infty}\; \frac{1}{n+1}\ln f_n(s).\]
The function $f_n$ is positive, non-increasing, and  bounded by $2^n$, so $g\leq \ln 2$. From now, we write $U=U_\varnothing$ to lighten the notation. The two independent families $(L_{(1)*u},U_{(1)*u})_{u\in\bbW}$ and $(L_{(2)*u},U_{(2)*u})_{u\in\bbW}$ have the same law as $(L_u,U_u)_{u\in\bbW}$, and they are jointly independent from $(L_\varnothing,U)$. As $L_\varnothing$ and $U$ are independent, we compute using $\E[e^{-s\delta L_\varnothing}]=(1+s)^{-1}$ that
\begin{equation}
\label{induction_f_n_E[I]_minoration}
\forall n\in\bbN,\ \forall s\in\bbR_+,\quad f_{n+1}(s)=\frac{f_n(s)+\E[f_n(s\delta/U)]}{1+s}.
\end{equation}
Now fix $r\in(1,\delta)$. Then, we use the monotonicity of $f_n$ and that $U\leq\delta$ with (\ref{induction_f_n_E[I]_minoration}) to get
\begin{equation}
\label{induction_f_n_variante}
f_{n+1}(s)\leq a(s)f_n(s)+f_n(rs),\quad \text{ where }\quad a(s)=\tfrac{1+\P(rU>\delta)}{1+s}.
\end{equation}
Let $b>\exp(g(rs))$. By definition of $g$, there exists $c\in\bbR_+$ such that $f_n(rs)\leq c b^n$ for all $n\in\bbN$. If $A>\max(a(s),b)$, then we can choose $C\in(1,\infty)$ such that $Ca(s)+c\leq CA$, and thus (\ref{induction_f_n_variante}) entails by induction that $f_n(s)\leq CA^n$ for all $n\in\bbN$. Then, $g(s)\leq \ln A$, and taking the infimum over $b$ and $A$ gives that $g(s)\leq \max(g(rs),\ln a(s))$ for all $s\in\bbR_+$. Note that $a(rs)\leq a(s)$ because $r>1$, so we obtain by induction that
\begin{equation}
\label{induction_g_n_E[I]_minoration}
\forall m\in\bbN,\ \forall s\in\bbR_+,\quad g(s)\leq\max\Big(g(r^ms)\, ,\, \ln \tfrac{1+\P(rU>\delta)}{1+s}\, \Big).
\end{equation}
Since $U\leq \delta$ and $f_n$ is non-increasing, (\ref{induction_f_n_E[I]_minoration}) yields that $f_{n+1}(s)\leq \frac{2}{1+s}f_n(s)$ for all $n\in\bbN$ and $s\in\bbR_+$. This implies that $g(s)\leq \ln(2)-\ln(1+s)$, so $\lim_{s\to \infty} g(s)=-\infty$. Therefore, taking $m$ to $\infty$ in (\ref{induction_g_n_E[I]_minoration}) entails that $g(s)\leq\ln(1+\P(rU>\delta))-\ln(1+s)$ for all $s>0$ and $r\in(1,\delta)$. It follows that $g(s)\leq-\ln(1+s)$ for all $s>0$ by taking $r$ to $1$, because $U<\delta$ a.s.

Finally, let $\varepsilon>0$. Our upper bound for $g$ ensures that there is $s_0>0$ such that $g(s_0)<-(1-\varepsilon)s_0$. It then holds $s_0(1-2\varepsilon)(n+1)+\ln f_n(s_0)\leq -s_0\varepsilon(n+1)$ for all large enough $n$, by definition of $g$. For large $n$, a Chernoff bound then entails that
\[\P\bigg(\, \inf_{\substack{u\in\bbW\\ |u|=n}}\; \delta\sum_{v\preceq u}L_v\prod_{w*(1)\preceq v}\frac{\delta}{U_w}\leq (1-2\varepsilon)(n+1)\bigg)\leq e^{-s_0\varepsilon(n+1)}.\]
By the Borel-Cantelli lemma and (\ref{minoration_E[I]_analitc}), we get $\E[\Info_\delta(\sT)]\geq 1-2\varepsilon$ for all $\varepsilon>0$.
\end{proof}

To show $\E[\Info_\delta(\sT)]\leq 1$ when $\sT$ is an $\aHS_\alpha$-real tree, we need a technical lemma.
\begin{lemma}
\label{technical_X<YX'+Z}
Let $X,Y,Z$ be three almost surely finite nonnegative RVs. We assume that $X$ and $Y$ are independent, $\E[Y]<1$, and $Z$ is integrable. If there is a random variable $X'$ distributed as $X$ such that $X'\leq YX+Z$ a.s., then $X$ is integrable and $\E[X]\leq \frac{\E[Z]}{1-\E[Y]}$.
\end{lemma}

\begin{proof}
We choose $a\in(\E[Y],1)$ and let $(Y_n)_{n\in\bbN^*}$ be a sequence of independent RVs distributed as $Y$ and jointly independent from $(X',X,Y,Z)$. The almost sure inequality $X'\leq YX+Z$ implies that $\P(X>x)\leq \P(XY_n>ax)+\P(Z>(1-a)x)$ for all $x>0$ and all $n\in\bbN^*$. By conditioning on $(Y_1,\ldots,Y_{n-1})$, we prove by induction that
\[\P(X>x)\leq \P\Big(X\prod_{i=1}^n Y_i>a^n x\Big)+\sum_{i=1}^n\P\Big(Z\prod_{j=1}^{i-1} Y_j> (1-a)a^{i-1}x\Big)\]
for all $x>0$ and $n\in\mathbb{N}$. The $Y_i$ are independent and $\E[Y_i/a]=\E[Y/a]<1$, so the $\prod_{i=1}^n (Y_i/a)$ converge in mean to $0$. As $X$ is a.s.~finite, it follows that $X\prod_{i=1}^n (Y_i/a)\convP 0$, which yields $\P(X>x)\leq \sum_{i\geq 1}\P\big(\tfrac{1}{1-a}Z\prod_{j=1}^{i-1} \tfrac{1}{a}Y_j>x\big)$ for all $x>0$ by taking $n$ to $\infty$.  Now, we integrate over $x\in(0,\infty)$ and use the independence of $Z$ and $(Y_n)_{n\in\bbN^*}$ to find
\[\E[X]\leq \sum_{i\geq 1}\E\Big[\tfrac{1}{1-a}Z\prod_{j=1}^{i-1} \tfrac{1}{a}Y_j\Big]=\frac{a\E[Z]}{(1-a)(a-\E[Y])}<\infty.\]
We conclude by rearranging $\E[X]\leq \E[YX+Z]=\E[Y]\E[X]+\E[Z]$.
\end{proof}

\begin{proposition}
Let $\sT$ be an $\aHS_\alpha$-real tree. Then $\E[\Info_\delta(\sT)]\leq 1$.
\end{proposition}

\begin{proof}
With the notation of Definition~\ref{info_def}, set $S_n(T)=(n+1)\Info^{n,n}_\delta(T)$ for any rooted compact real tree $(T,d,\rho)$ and all $n\in\mathbb{N}$.
Note that $S_0(T)=\sup_{x\in T}\delta d(\rho,x)=\delta\mathfrak{h}(T)$. 
By (\ref{monotony_info_k,n}), $\Info_\delta(T)\leq \liminf \tfrac{1}{n+1}S_n(T)$. Thus, Fatou's lemma and Jensen's inequality entail that
\begin{equation}
\label{maj_E[I]_idee_stable}
\forall a>0,\quad \E[\Info_\delta(\sT)]\leq \liminf_{n\rightarrow\infty}\tfrac{1}{a(n+1)}\ln\E\big[e^{a S_n(\sT)}\big].
\end{equation}
Moreover, the estimate (\ref{maj_hauteur_limit_tree}) ensures that $\E[e^{a S_0(\sT)}]=\E[e^{a \delta\mathfrak{h}(\sT)}]<\infty$ as soon as $a>0$ is small enough. Let us fix for now $a\in(0,1)$ and $n\in\mathbb{N}$ such that $e^{a S_n(\sT)}$ is integrable. Using the notation of Theorem~\ref{self-similar_alpha}, we define a set $\mathcal{E}$ of rooted compact real trees as
\[\Big\{\tfrac{1}{\delta}\odot_\alpha \sT_\rd\Big\}\cup\Big\{\tfrac{1}{\delta}\bar{\Lambda}_i  \odot_\alpha\sT_i \, :\,  i\! \geq\! 2\Big\}\cup\Big\{\tfrac{1}{\delta\xi(\omega)} \bar{\lambda}_i(\omega) \odot_\alpha T_i(\omega)\, :\,  i\!\geq\! 1,\omega\! \in\! \mathcal{N},\ell(\omega)\leq L\Big\}.\]
Then, (\ref{branching_property_<_step}) in Proposition~\ref{branching_property} ensures that the following inequality holds almost surely:
\begin{equation}
\label{functional_maj_I}
S_{n+1}(\sT^*)\leq \delta L+\max \big(S_{n+1}(U^{-1}\odot_\alpha \sT_\rg)\, ,\, \sup_{T\in \mathcal{E}}\delta S_n(T)\big).
\end{equation}

Before proceeding, let us fix some parameters. We choose $\varepsilon\in(0,1)$ such that $2\varepsilon<1-a$. We can pick $\eta\in(0,\frac{\varepsilon}{1+\varepsilon})$ such that $\P((1-\eta) U\leq 1)\leq\varepsilon$ and $\P(\Lambda_2\geq 1-\eta)\leq\varepsilon$ because $\Lambda_2<1< U$ almost surely. Recall from (\ref{alpha_gamma_delta}) that $\beta(\alpha-1)=1$. There is $c\in(e^a,\infty)$ that does not depend on $n$ (but on $\alpha,a,\varepsilon,\eta$) such that $(1+(1+x)x^{\beta+1})e^{a(1-\eta) x}\leq \varepsilon e^{a x}+c$ for all $x\geq 0$. Moreover, for all $x\geq 0$ and all $\lambda\in[0,1]$, we see that $e^{a\lambda x}\leq \I{\lambda\geq 1-\eta}e^{a x}+ \I{\lambda x\geq 1}e^{a(1-\eta)x}+c$. Applying these inequalities to bound each term in (\ref{functional_maj_I}), we obtain that
\begin{equation}
\label{function_maj_E[I]}
e^{aS_{n+1}(\sT^*)}\leq (\I{1\geq (1-\eta)U}+\varepsilon)e^{a\delta L}\times e^{a S_{n+1}(\sT_\rg)}+e^{a\delta L}(2c+e^{a S_n(\sT_\rd)}+Z_+ +Z_-)
\end{equation}
almost surely, where we have denoted
\begin{align*}
Z_{+}  &=\sum_{i\geq 1}\Big(\I{\bar{\Lambda}_{i+1}\geq 1-\eta}e^{a S_n(\sT_{i+1})}+\sum_{\substack{\omega\in\cN\\ \ell(\omega)\leq L}}\I{\bar{\lambda}_i(\omega)\geq(1-\eta)\xi(\omega)}e^{a S_n(T_i(\omega))}\Big),\\
Z_-  &=\sum_{i\geq 1}\!\bigg( \I{S_n(\sT_{i+1})\bar{\Lambda}_{i+1}\geq 1}e^{a(1-\eta) S_n(\sT_{i+1})}+ \!\!\sum_{\substack{\omega\in\cN\\ \ell(\omega)\leq L}}\!\! \I{S_n(T_i(\omega))\bar{\lambda}_i(\omega)\geq\xi(\omega)}e^{a(1-\eta)S_n(T_i(\omega))}\!\bigg).
\end{align*}
Now, we want to use Lemma~\ref{technical_X<YX'+Z} to bound $\E[e^{a S_{n+1}(\sT)}]$ by an affine function of $\E[e^{a S_n(\sT)}]$, which was already assumed to be finite. Indeed, $S_{n+1}(\sT_\rg)$ has the same law as $S_{n+1}(\sT^*)$, is independent of $(L,U)$, and is almost surely finite by (\ref{info_first_bound}). Moreover, $\delta L$ has exponential law with mean $1$ and is independent from $U$, so the choice of $\eta$ and $\varepsilon$ ensure that
\begin{equation}
\label{maj_E[I]_1}
\E\big[(\I{(1-\eta)U\leq 1}+\varepsilon)e^{a\delta L}\big]=\frac{\P((1-\eta)U\leq 1)+\varepsilon}{1-a}\leq \frac{2\varepsilon}{1-a}<1.
\end{equation}
Since $L$ is independent of $\sT_\rd$, it only remains to control the means of $e^{a\delta L}Z_+$ and $e^{a\delta L}Z_-$. Using the independence, we average over the Poisson point process $\cN$ to compute that
\begin{align*}
\E\big[e^{a\delta L}Z_+\big]&=\sum_{i\geq 1}\E\bigg[e^{a S_n(\sT_\rg)}e^{a\delta L}\Big(\I{\prod_{j=2}^{i+1}\Lambda_j\geq 1-\eta}+\delta L \int_1^{\frac{1}{1-\eta}\prod_{j=2}^i \Lambda_j}\, \dd \xi\Big)\bigg],\\
\E\big[e^{a\delta L}Z_-\big]&=\sum_{i\geq 1}\E\bigg[e^{a(1-\eta) S_n(\sT_\rg)}e^{a\delta L}\Big(\I{S_n(\sT_\rg)\prod_{j=2}^{i+1}\Lambda_j\geq 1}+\delta L \int_1^{S_n(\sT_\rg)\prod_{j=2}^i \Lambda_j}\, \dd \xi\Big)\bigg].
\end{align*}
The $\Lambda_j$ are all smaller than $1$ and $\eta\! \leq \! \frac{\varepsilon}{1+\varepsilon}\! \leq\! 1/2$, so if $\prod_{j=2}^{i+1}\Lambda_j\! \geq\! 1\! -\! \eta$ then $\Lambda_2\! \geq 1\! -\! \eta$ and $\prod_{j=3}^i\Lambda_j\! \geq\! 1/2$. Also, for any $x\geq 0$, we have $x\prod_{j=2}^i\Lambda_j\! \leq\!  1\! +\! (x\!-\!1)\I{x\prod_{j=2}^i\Lambda_j\geq 1}$. Hence,
\begin{align*}
\E\big[e^{a\delta L}Z_+\big]&\leq\sum_{i\geq 1}\E\bigg[\big(\I{\Lambda_2\geq 1-\eta}+\varepsilon\big)(1+\delta L)e^{a\delta L}e^{a S_n(\sT_\rg)}\I{\prod_{j=3}^{i}\Lambda_j\geq 1/2}\bigg],\\
\E\big[e^{a\delta L}Z_-\big]&\leq \sum_{i\geq 1}\E\bigg[(1+\delta L)e^{a\delta L}(1+S_n(\sT_\rg))e^{a(1-\eta) S_n(\sT_\rg)}\I{S_n(\sT_\rg)\prod_{j=2}^{i}\Lambda_j\geq 1}\bigg].
\end{align*}
Now, we use Markov's inequality on $\prod_{j=3}^i\Lambda_j^{\beta+1}$ and apply the joint independence to get
\begin{align*}
\E\big[e^{a\delta L}Z_+\big]&\leq\big(\P(\Lambda_2\geq 1-\eta)+\varepsilon\big)\E\big[(1+\delta L)e^{a\delta L}\big]\E\big[e^{a S_n(\sT)}\big] 2^{\beta+1}\sum_{i\geq 1}\prod_{j=3}^{i}\E\big[\Lambda_j^{\beta+1}\big],\\
\E\big[e^{a\delta L}Z_-\big]&\leq \E\big[(1+\delta L)e^{a\delta L}\big]\E\big[(1+S_n(\sT))S_n(\sT)^{\beta+1}e^{a(1-\eta) S_n(\sT)}\big]\sum_{i\geq 1}\prod_{j=2}^{i}\E\big[\Lambda_j^{\beta+1}\big].
\end{align*}
We easily compute $\E[(1+\delta L)e^{a\delta L}]=\frac{2-a}{(1-a)^2}$, which together with the preceding two inequalities, (\ref{moments_Lambda}), and the choice of $\eta$ and $c$ yield that
\begin{align}
\label{maj_E[I]_2}
\E\big[e^{a\delta L}Z_+\big]&\leq \frac{2^{\beta+3}}{(1-a)^2}\sum_{i\geq 1}\prod_{j=3}^i\E\big[\Lambda_j^{\beta+1}\big]\times\varepsilon\E\big[e^{a S_n(\sT)}\big]<\infty,\\
\label{maj_E[I]_3}
\E\big[e^{a\delta L}Z_-\big]&\leq\frac{2}{(1-a)^2}\sum_{i\geq 1}\prod_{j=2}^i\E\big[\Lambda_j^{\beta+1}\big]\times\big(\varepsilon\E\big[e^{a S_n(\sT)}\big]+c\big)<\infty.
\end{align}

Finally, using (\ref{function_maj_E[I]}), together with Lemma~\ref{technical_X<YX'+Z}, (\ref{maj_E[I]_1}), (\ref{maj_E[I]_2}), and (\ref{maj_E[I]_3}), we deduce that
\[\E\big[e^{a S_{n+1}(\sT)}\big]\leq \frac{1+\frac{C}{1-a}\varepsilon}{1-a-2\varepsilon}\E\big[e^{a S_n(\sT)}\big]+c'<\infty,\]
where $C\in(0,\infty)$ is a constant only depending on $\alpha$, and $c'\in(0,\infty)$ does not depend on $n$ (but on $\alpha,a,\varepsilon,\eta$). An elementary study yields that for all $a\in(0,1)$ small enough, we have
\[\forall \varepsilon\in(0,\tfrac{1-a}{2}),\quad \liminf_{n\rightarrow\infty}\tfrac{1}{a(n+1)}\ln\E\big[e^{a S_n(\sT)}\big]\leq \tfrac{1}{a}\ln \Big(\tfrac{1+\frac{C}{1-a}\varepsilon}{1-a-2\varepsilon}\Big).\]
Thanks to (\ref{maj_E[I]_idee_stable}), we conclude the proof by taking $\varepsilon\to0^+$ and then $a\to 0^+$ .
\end{proof}

\section{Proofs of Theorems~\ref{standard_stable_HS_real_thm} and \ref{scaling_limit_size_intro}}
\label{conditioning_at_least}

In all this section, we denote by $\ftau=(\tau,(W_v)_{v\in\partial\tau})$ a $\GWw_\alpha$-weighted tree. Our goal here is to prove Theorems~\ref{standard_stable_HS_real_thm} and \ref{scaling_limit_size_intro} by using the Strahler dilation $\Info_\delta$ with base $\delta$ (see Definition~\ref{info_def}) to relate the weighted Horton--Strahler number $\tHS(\ftau)$ of $\ftau$, when its size $\#\tau$ is conditioned to be large, with the scaling limit of $\tau$ when $\tHS(\ftau)$ is conditioned to be large. Before we start, let us recall some notation. 

Recall from (\ref{height_process_def}) that $H(\tau)$ stands for the height function of $\tau$ (which continuously interpolates the heights of the vertices listed in depth-first order). Recall from Section~\ref{topo_function} the space $\mathcal{C}_{\mathrm{K}}$ of continuous functions with compact support and endowed with lifetimes. As explained in Notation~\ref{notation_DK}, we identify any continuous function with compact support $f$ with $(f,\zeta(f))$, where $\zeta(f)=\sup\{0\}\cup \{s\geq 0:f(s)\neq 0\}$, and thus may write $f\in\mathcal{C}_{\mathrm{K}}$. Recall from (\ref{real_tree_coded}) that for $f$ a continuous excursion, i.e.~$f\in\mathcal{E}_{\mathrm{K}}$ as in (\ref{excursion_space}), $\cT_f$ stands for the real tree coded by $f$. Recall from Section~\ref{more_stable} that $\bN_\alpha(\dd \rH)$ denotes the excursion measure of the $\alpha$-stable height process $\rH$, and that $\bN_\alpha(\dd \rH\, |\, \zeta=1)$ denotes the law of the normalized (i.e.~with unit lifetime) excursion of $\rH$. As introduced in (\ref{alpha_gamma_delta}) and (\ref{domain_attraction}), we write\[\beta=\tfrac{1}{\alpha-1},\quad \gamma=\ln\tfrac{\alpha}{\alpha-1},\quad \delta=e^{\gamma(\alpha-1)},\quad \text{and}\quad a_n=\alpha^{-1/\alpha}n^{1/\alpha}\:\text{ for all }n\in\bbN.\]

\subsection{Proof of Theorem~\ref{standard_stable_HS_real_thm}}

Recall from Section~\ref{geometric_study}, and particularly from Definition~\ref{stable_height_excursion_def} and Proposition~\ref{natural_lifetime_scaling_limit} $(ii)$, that we say that $\tilde{H}$ is an $\aHS_\alpha$-excursion when it is the scaling limit of $H(\tau)$ under $\P(\dd\tau\, |\, \tHS(\ftau)=x)$ as $x\to\infty$ (see Theorem~\ref{scaling_limit_HT}), and that $\cT_{\tilde{H}}$ is then called an $\aHS_\alpha$-real tree and is the scaling limit of $\tau$ under $\P(\dd\tau\, |\, \tHS(\ftau)=x)$ as $x\to\infty$ (see Theorem~\ref{scaling_limit_HS_intro}). To show Theorem~\ref{standard_stable_HS_real_thm}, we describe the law of an $\aHS_\alpha$-excursion in terms of the Strahler dilation and $\bN_\alpha$. In the course of the proof, we obtain a result similar to the desired Theorem~\ref{scaling_limit_size_intro} but under the tail conditioning $\P(\, \cdot\, |\, \#\tau\geq n)$.

\begin{theorem}
\label{from_law(Y)_to_N}
Recall from (\ref{dim_info_def}) that for a rooted compact real tree $T$, $\base_{\Info}\! T=\inf \{\cdelta>1 : \Info_\cdelta(T)=\infty\}$. Let $\tilde{H}$ be an $\aHS_\alpha$-excursion\footnote{We emphasize that $H(\tau)$, $\rH$, and $\tilde{H}$ denote three different random continuous functions.}. Then, $\bN_\alpha(\base_{\Info}\! \cT_\rH\neq \delta)=0$ and
\begin{equation}
\label{relation4}
\forall x>0,\quad \bN_\alpha(\Info_\delta(\cT_\rH)>x)=(\alpha x)^{-\beta}\quad\text{ and }\quad \bN_\alpha(\Info_\delta(\cT_\rH)\in \{0,\infty\})=0.
\end{equation}
Furthermore, if $F:\mathcal{C}_{\mathrm{K}}\to\bbR_+$ and $g:[0,\infty]\to\bbR_+$ are measurable and nonnegative, then 
\begin{align}
\label{relation1}
\bN_\alpha\big[F(\rH)g(\Info_\delta(\cT_\rH))\big]&=\frac{\alpha^{-\beta}}{\alpha-1}\int_{\bbR_+} \E\Big[F\big(x\, \tilde{H}_{s(\alpha x^\alpha)^{-\beta}}\, ;\, s\geq 0\big) \Big]g(x) \frac{\dd x}{x^{1+\beta}},\\
\label{relation2}
\E\big[F(\tilde{H})\big]\bN_\alpha\big[g(\Info_\delta(\cT_\rH))\big]
&=\bN_\alpha\Big[F\big(\Info_\delta(\cT_\rH)^{-1}\,  \rH_{\alpha^\beta\Info_\delta(\cT_\rH)^{\alpha\beta} s}\, ;\, s\geq 0\big)g(\Info_\delta(\cT_\rH))\Big].
\end{align}
\end{theorem}

\begin{proof}
Let $\ell,z>0$ and let $F:\mathcal{C}_{\mathrm{K}}\to\bbR_+$ be nonnegative, bounded, and continuous. For all $y>0$ and $n\in\bbN^*$, we set $ x_n(y)=\tfrac{1}{\gamma \alpha}\ln(ny)$, $f_n(y)=0$ when $x_n(y)\in\bbN$, and
\begin{multline*}
f_n(y)= \E\Big[\un_{\{y\#\tau\geq \ell e^{\gamma\alpha x_n(y)}\}}F\big(\alpha^{-1/\alpha}y^{1-1/\alpha}e^{-\gamma(\alpha-1)x_n(y)} H_{e^{\gamma\alpha x_n(y)}s/y}(\tau)\, ;\, s\geq 0\big)\\
\, \Big|\, \tHS(\ftau)\!=\!x_n(y)\Big]
\end{multline*}
when $x_n(y)\notin\bbN$. The functions $f_n$ are uniformly bounded, nonnegative, and measurable by Proposition~\ref{cond_H=x_prime}. We apply (\ref{cond_H=x_identity}) and we make the change of variable $\gamma\alpha x =\ln(ny)$ to get
\[\E\Big[\I{\#\tau\geq \ell n\,;\,\delta^{\tHS(\ftau)}>zn/a_n}F\big(\tfrac{a_n}{n}H_{ns}(\tau)\, ;\, s\geq 0\big)\Big]=\frac{1}{\alpha n^{1/\alpha}}\int_{(\alpha z^\alpha)^{\beta}}^\infty\frac{1}{y^{1+1/\alpha}}f_n(y)\, \dd y.\]
There are at most a countable number of values $y>0$ for which $\P(y\zeta(\widetilde{H})=\ell)>0$ or for which $x_n(y)\in\bbN$ for some $n\in\mathbb{N}^*$. For all other $y>0$, Theorem~\ref{scaling_limit_HT} and Proposition~\ref{natural_lifetime_scaling_limit} $(ii)$ entail that $f_n(y)\longrightarrow \E\big[\I{y\zeta(\tilde{H})> \ell}F\big(\alpha^{-1/\alpha}y^{1-1/\alpha} \tilde{H}_{s/y}\, ;\, s\geq 0\big)\big]$ since $\I{\# \tau\geq 2}\#\tau=\zeta(H(\tau))$ by (\ref{lifetime_height}). Then, by (\ref{mass_stable_tree}) and (\ref{equiv_taille}), the dominated convergence theorem yields that
\begin{multline}
\label{cvd_H_et_path_sachant_an}
\lim_{n\rightarrow\infty}\E\Big[\I{\delta^{\tHS(\ftau)}>zn/a_n}F\big(\tfrac{a_n}{n}H_{ns}(\tau)\, ;\, s\geq 0\big)\ \big|\ \#\tau\geq \ell n\Big]\\
=\frac{1}{\alpha^{1+1/\alpha}\bN_\alpha(\zeta>\ell)}\int_{(\alpha z^\alpha)^{\beta}}^\infty\frac{1}{y^{1+1/\alpha}}\E\Big[\I{y\zeta(\tilde{H})> \ell}F\big(\alpha^{-1/\alpha}y^{1-1/\alpha} \tilde{H}_{s/y}\, ;\, s\geq 0\big) \Big]\, \dd y
\end{multline}
for all $z\in(0,\infty)$. We next claim that (\ref{cvd_H_et_path_sachant_an}) still holds when $z=0$. Indeed, by the monotone convergence theorem, the limit of the right-hand side of (\ref{cvd_H_et_path_sachant_an}) as $z\!\to\! 0^+$ is obtained by simply replacing $z$ with $0$. Thus, we only need to check that
\begin{equation}
\label{still_fine_z=0}
\lim_{z\rightarrow0^+}\limsup_{n\rightarrow\infty}\P\big(\delta^{\tHS(\ftau)}\leq z\tfrac{n}{a_n}\ \big|\ \#\tau\geq \ell n\big)=0.
\end{equation}
To do this, we recall that $|\tau|=\sup H(\tau)$ stands for the height of $\tau$ from (\ref{height_notation}) and (\ref{lifetime_height}), and that $\HS(\tau)\leq\tHS(\ftau)$ from Proposition~\ref{regularized_HT}. For all $\eta>0$, we then use the bound
\[\P\big(\delta^{\tHS(\ftau)}\leq z\tfrac{n}{a_n}\, ;\, |\tau|\geq \eta \tfrac{n}{a_n}\ \big|\ \#\tau\geq \ell n\big)\leq \tfrac{\P(|\tau|\geq \eta n/a_n)}{\P(\#\tau\geq\ell n)}\P\big(\HS(\tau)\leq \log_\delta(z\tfrac{n}{a_n})\ \big|\ |\tau|\geq \eta \tfrac{n}{a_n}\big).\]By the estimates (\ref{equiv_hauteur}), (\ref{equiv_taille}), and (\ref{maj_hauteur_selon_H_prime}), it then follows that for all $\eta>0$,
\[\lim_{z\rightarrow0^+}\limsup_{n\rightarrow\infty}\P\big(\delta^{\tHS(\ftau)}\leq z\tfrac{n}{a_n}\ \big|\ \#\tau\geq \ell  n\big)\leq \limsup_{n\rightarrow\infty}\P\big(|\tau|<\eta\tfrac{n}{a_n}\ \big|\ \#\tau\geq \ell  n\big).\]
The right-hand side is bounded by $\bN_\alpha(\sup \rH<\eta\ |\ \zeta>\ell)$ by Theorem~\ref{duquesne_scaling_atleast}. If $\zeta>\ell$ then $\sup \rH>0$, so taking $\eta\!\to\!0^+$ gives (\ref{still_fine_z=0}), and so (\ref{cvd_H_et_path_sachant_an}) with $z=0$.
\smallskip

Next, observe that if $z=0$ then the limit on the left-hand side of (\ref{cvd_H_et_path_sachant_an}) is also described as $\bN_\alpha[F(\rH)\, |\, \zeta>\ell]$ by Theorem~\ref{duquesne_scaling_atleast}. Using this, we obtain that
\[\bN_\alpha[F(\rH)\, |\, \zeta>\ell]=\frac{\alpha^{-1-1/\alpha}}{\bN_\alpha(\zeta>\ell)}\int_{0}^\infty\!\! \frac{1}{y^{1+1/\alpha}}\E\bigg[\I{y\zeta(\tilde{H})> \ell}F\Big(\tfrac{y^{1-1/\alpha}}{\alpha^{1/\alpha}} \tilde{H}_{s/y}\, ;\, s\geq 0\Big) \bigg]\, \dd y.\]
We multiply both sides of the previous identity by $\bN_\alpha(\zeta>\ell)$ and we let $\ell$ tend to $0^+$. Recall from Section~\ref{more_stable} that $\bN_\alpha(\zeta\!=\!0) =0$. We also have $\P(\zeta(\tilde{H})\!=\!0)=\P(\sup\tilde{H}\!=\!0)=0$ by (\ref{maj_hauteur_limit_tree}). Thus, the monotone convergence theorem yields that
\begin{equation}
\label{from_law(Y)_to_N_alt}
\bN_\alpha[F(\rH)]=\frac{1}{\alpha^{1+1/\alpha}}\int_{\bbR_+}\frac{1}{y^{1+1/\alpha}}\E\Big[F\big(\alpha^{-1/\alpha}y^{1-1/\alpha} \tilde{H}_{s/y}\, ;\, s\geq 0\big) \Big]\, \dd y,
\end{equation}
for all nonnegative, bounded, and continuous functions $F:\mathcal{C}_{\mathrm{K}}\to\bbR_+$. Therefore, (\ref{from_law(Y)_to_N_alt}) also holds for all nonnegative and measurable functions. Theorem~\ref{binary_info_intro} ensure that the measurable function $\Upsilon:\omega\in\mathcal{C}_{\mathrm{K}}\longmapsto \I{\omega\in\mathcal{E}_{\mathrm{K}}}\Info_\delta(\cT_\omega)\in[0,\infty]$ satisfies that $\Upsilon(\lambda \widetilde{H}_{s/y}\, ;\, s\geq0)=\lambda$ almost surely for all $\lambda,y>0$. Applying (\ref{from_law(Y)_to_N_alt}) with the function $F\cdot g\circ\Upsilon$ and making the change of variable $x=\alpha^{-1/\alpha}y^{1-1/\alpha}$ then yields (\ref{relation1}). Using again that $\Upsilon(\lambda \widetilde{H}_{s/y}\, ;\, s\geq0)=\lambda$, (\ref{relation2}) is a consequence of (\ref{relation1}). Next, (\ref{relation4}) also follows from (\ref{relation1}) with $F\equiv 1$ and with $g=\un_{(x,\infty]}$ and $g=\un_{\{0,\infty\}}$. By Proposition~\ref{dim_info_prop}, it then follows that $\bN_\alpha(\base_{\Info}\!\cT_\rH\neq \delta)=0$.
\end{proof}

We complete the proof of Theorem~\ref{standard_stable_HS_real_thm} by applying (\ref{relation2}), then (\ref{decompo_lifetime}), then Fubini's theorem, and finally (\ref{relation1}). Theorem~\ref{from_law(Y)_to_N} justifies the intuition that an $\aHS_\alpha$-excursion is (a multiple of) an excursion of the $\alpha$-stable height process conditioned on its coded real tree having fixed Strahler dilation. Let us also highlight that the identities (\ref{relation4}) and (\ref{height_stable_tree}) entail that there is equality between the two image measures of $\bN_\alpha$ by the functions $\sup \rH=\mathfrak{h}(\cT_\rH)$ and $\alpha\beta\, \Info_\delta(\cT_\rH)$ (which is reminiscent of Theorem~\ref{height_info_CRT_intro} when $\alpha=2$).
\smallskip

Thanks to the identity (\ref{relation1}) or (\ref{from_law(Y)_to_N_alt}), we see that the right-hand side of (\ref{cvd_H_et_path_sachant_an}) is equal to $\bN_\alpha[\I{\Info_\delta(\cT_\rH)>z}F(\rH)\ |\ \zeta>\ell]$, which yields the following result.

\begin{theorem}
\label{conditioning_at_least_thm}
Let $\ftau$ be a $\GWw_\alpha$-weighted tree and let $a_n=\alpha^{-1/\alpha}n^{1/\alpha}$ as in (\ref{domain_attraction}). For all $\ell>0$, the following joint convergence holds in distribution on $\mathcal{C}_{\mathrm{K}}\times\bbR_+$:
\[\big(\big(\tfrac{a_n}{n}H_{ns}(\tau)\big)_{s\geq 0}\; ,\;  \tfrac{a_n}{n}\delta^{\tHS(\ftau)}\big)\;\text{ under }\ \P(\, \cdot\, |\, \#\tau\geq \ell n)\,\convd (\rH,\Info_\delta(\cT_\rH))\;\text{ under }\bN_\alpha(\, \cdot\, |\, \zeta>\ell).\]
\end{theorem}

\subsection{Proof of Theorem~\ref{scaling_limit_size_intro}}

Our goal is to deduce from Theorem~\ref{conditioning_at_least_thm} the following result.

\begin{theorem}
\label{conditioning_equal_thm}
Let $\ftau$ be a $\GWw_\alpha$-weighted tree and let $a_n=\alpha^{-1/\alpha}n^{1/\alpha}$ as in (\ref{domain_attraction}). The following joint convergence holds in distribution on $\mathcal{C}_{\mathrm{K}}\times \bbR_+$:
\begin{multline*}
\big(\big(\tfrac{a_n}{n}H_{ns}(\tau)\big)_{s\geq 0}\; ,\;  \tfrac{a_n}{n}\delta^{\tHS(\ftau)}\big)\;\text{ under }\ \P(\, \cdot\, |\, \#\tau=n+1)\,\\
\xrightarrow[n\rightarrow\infty,n\in \lfloor\alpha\rfloor\bbN]{d} (\rH,\Info_\delta(\cT_\rH))\;\text{ under }\bN_\alpha(\dd\rH\, |\, \zeta=1).
\end{multline*}
\end{theorem}

Observe that Theorem~\ref{scaling_limit_size_intro} follows readily from Theorem~\ref{conditioning_equal_thm}, thanks to Proposition~\ref{height_contour_tree}, so we now focus on proving Theorem~\ref{conditioning_equal_thm}. The fact that $\bN_\alpha(\zeta=1)=0$ makes the argument less straightforward than that for Theorem~\ref{conditioning_at_least_thm}. To address this issue, we take advantage of the monotonicity of the weighted Horton--Strahler number. Let us give an overview of our strategy for the case $\alpha=2$. It is well-known from Rémy's algorithm \cite{remy} that removing the parental edge of a random leaf from a uniform $n$-Catalan tree $\tau_n$ results in a uniform $(n\!-\!1)$-Catalan tree $\tau_{n-1}$. In the same spirit, we can embed a weighted tree $\ftau_n$ with the same law as $\ftau$ under $\P(\, \cdot\, |\, \#\tau=2n\!-\!1)$ into a weighted tree $\ftau_n'$ with the same law as $\ftau$ under $\P(\, \cdot\, |\, \#\tau\geq 2n\!-\!1)$, so that $\tHS(\ftau_n)\leq\tHS(\ftau_n')$. Then, Theorem~\ref{conditioning_at_least_thm} yields an asymptotic upper bound for $2^{\tHS(\ftau_n)}$. However, this proof might not work without alteration when $\alpha\in(1,2)$. Indeed, Janson~\cite{janson} exhibited an example of a critical offspring distribution such that there are no embeddings $\tau_n\hookrightarrow\tau_{n+1}$ for some $n\geq 1$.

Nevertheless, as presented in Section~\ref{marchal_section}, Marchal~\cite{marchal} has constructed a sequence of nested $\GW_\alpha$-trees conditioned on their number of leaves, instead of their number of vertices. To use it, we first transform Theorem~\ref{conditioning_at_least_thm} into a result about the tail conditioning on the number of leaves $\#\partial\tau$. Recall from Section~\ref{topo_function} that $\mathbb{K}$ stands for the space of (rooted-isometry classes of) rooted compact metric spaces, endowed with the rooted Gromov--Hausdorff distance.
As discussed in Section~\ref{real_tree_section}, we identify any tree $t$ with $(t,\mathtt{d}_{\mathrm{gr}},\varnothing)\in\mathbb{K}$, where $\mathtt{d}_{\mathrm{gr}}$ is the graph distance given by (\ref{graph_distance}). As in (\ref{notation_scaling}), we write $\lambda\cdot t=(t,\lambda\mathtt{d}_{\mathrm{gr}},\varnothing)$ for all $\lambda>0$.

\begin{corollary}
\label{scaling_limit_leaves>n_HS}
Let $\ftau$ be a $\GWw_\alpha$-weighted tree and let $a_n=\alpha^{-1/\alpha}n^{1/\alpha}$ as in (\ref{domain_attraction}). For all $\ell\in(0,1]$, the following joint convergence holds in distribution on $\mathbb{K}\times\bbR_+$:
\[\big(\tfrac{a_n}{n}\cdot\tau\; ,\;  \tfrac{a_n}{n}\delta^{\tHS(\ftau)}\big)\;\text{ under }\ \P(\, \cdot\, |\, \#\partial\tau\geq \ell n/\alpha)\,\convd (\cT_\rH,\Info_\delta(\cT_\rH))\;\text{ under }\bN_\alpha(\, \cdot\, |\, \zeta>\ell).\]
\end{corollary}

\begin{proof}
If $\alpha=2$, Proposition~\ref{size_vs_leaves} $(i)$ yields that the result directly follows from Theorem~\ref{conditioning_at_least_thm} and Proposition~\ref{height_contour_tree}. If $\alpha\in(1,2)$, observe that $m_n=\lceil \ell n/\alpha-(\ell n)^{3/4}\rceil$ takes all integer values large enough, because $\ell\!<\!\alpha$, and that $a_{\lceil \ell n/\alpha\rceil} m_n/(\lceil \ell n/\alpha\rceil a_{m_n})\longrightarrow 1$. Lemma~\ref{size_vs_leaves_lemma}, Theorem~\ref{conditioning_at_least_thm}, and Proposition~\ref{height_contour_tree} then yield the result.
\end{proof}

\begin{proof}[Proof of Theorem~\ref{conditioning_equal_thm}]
We claim that we only need to show that
\begin{equation}
\label{goal_size=n_HS}
\E\Big[g\big(\tfrac{a_n}{n}\cdot\tau\big)h\big(\tfrac{a_n}{n}\delta^{\tHS(\ftau)}\big)\, \big|\, \#\tau=n+1\Big]\xrightarrow[n\rightarrow\infty,n\in \lfloor\alpha\rfloor\bbN]{}\bN_\alpha\big[g(\cT_\rH)h(\Info_\delta(\cT_\rH))\ \big|\ \zeta=1\big]
\end{equation}
for any bounded and continuous functions $g:\mathbb{K}\to\bbR$ and $h:\bbR_+\to\bbR$. Indeed, Theorem~\ref{duquesne_scaling_limit_size=n} and (\ref{goal_size=n_HS}) imply that the laws of $\big(\big(\tfrac{a_n}{n}H_{ns}(\tau)\big)_{s\geq 0} ,  \tfrac{a_n}{n}\cdot\tau ,   \tfrac{a_n}{n}\delta^{\tHS(\ftau)}\big)$ under $\P(\, \cdot\, |\, \#\tau=n+1)$ are tight on $\mathcal{C}_{\mathrm{K}}\!\times\!\mathbb{K}\!\times\!\bbR_+$. Moreover, if $(H^*,\sT^*,S^*)$ is one of their subsequential limit in law, then the law of $H^*$ is $\bN_\alpha(\dd\rH\, |\, \zeta=1)$ by Theorem~\ref{duquesne_scaling_limit_size=n}, $\sT^*=\cT_{H^*}$ a.s.~by Proposition~\ref{height_contour_tree}, and $S^*=\Info_\delta(\sT^*)$ a.s.~by (\ref{goal_size=n_HS}). The desired theorem then follows as claimed.
\smallskip

Thus, the rest of the proof is devoted to showing (\ref{goal_size=n_HS}). To lighten notation, we write $p(n)=\P(\#\tau=n+1)$ and $q(n)=\P(\#\partial\tau=n)$ for all $n\in\bbN$. Let $(\fctau_i)_{i\geq 1}$ be the sequence of random weighted trees constructed by Algorithm~\ref{marchal_alg}. For all $n\in \lfloor\alpha\rfloor\bbN$, let
\[\mathtt{J}_n:=\inf\{i\geq 1\ :\ \#\ctau_i\geq n+1\}\text{, }\:\bT_n=(T_n,(\mathrm{W}_{v}^n)_{v\in\partial T_n}):=\fctau_{\mathtt{J}_n}\text{, }\:\text{ and }\:\xi_n:=\I{\# T_n=n+1}.\]
The random integer $\mathtt{J}_n$ is finite and the tree $T_n$ has exactly $\mathtt{J}_n$ leaves by Proposition~\ref{marchal_prop} $(i)$. Let $t$ be a tree with $n\!+\!1$ vertices. The sequence $(\#\ctau_i)_{i\geq 1}$ is strictly increasing so $T_n=t$ if and only if $t\in\{\ctau_i : i\geq 1\}$. But the trees $t$ and $\ctau_i$ have the same number of leaves only when $i=\#\partial t$. Thus, $T_n=t$ if and only if $\ctau_{\#\partial t}=t$, in which case $\mathtt{J}_n=\#\partial t$ and $\bT_n=\fctau_{\#\partial t}$. By Proposition~\ref{marchal_prop} $(iii)$, for any bounded measurable function $f$ on the space of weighted trees,
\begin{equation}
\label{law_pseudo_size=n}
\E\big[\xi_n f(\bT_n)\big]=\E\Big[\tfrac{p(n)}{q(\#\partial \tau)}f(\ftau)\ \Big|\ \#\tau=n+1\Big].
\end{equation}
When $\alpha=2$, $p(n)=q(\#\partial\tau)$ under $\P(\, \cdot\, |\, \#\tau=n+1)$ by Proposition~\ref{size_vs_leaves} $(i)$, so $\xi_n=1$ a.s. However, $T_n$ is not a $\GW_\alpha$-tree conditioned to have $n$ vertices when $\alpha\in(1,2)$, because it is biased by its number of leaves. Nonetheless, we claim that this bias disappears as $n\to\infty$:
\begin{equation}
\label{leaves_bias_vanishing}
\E\Big[\big|\tfrac{p(n)}{q(\#\partial \tau)}-\chi\big|\, \Big|\, \#\tau=n+1\Big]\xrightarrow[n\rightarrow\infty,n\in \lfloor\alpha\rfloor\bbN]{}0,\quad\text{ where }\chi=\begin{cases}
	\alpha^{-1}&\text{ if }\alpha\in(1,2),\\
	1&\text{ if }\alpha=2.
 \end{cases}
\end{equation}
Indeed, by Proposition~\ref{marchal_prop} $(i)$, it holds $\mathtt{J}_n\leq xn$ if and only if $\#\ctau_{\lfloor xn\rfloor}\geq n+1$, for all $x\geq 1/n$. Thanks to Proposition~\ref{size_vs_leaves}, we get that $\#\ctau_i/i\convP\alpha$ as $i\to\infty$, and this also implies that
\begin{equation}
\label{LLN_leaves_pseudo_size=n}
\tfrac{1}{n}\mathtt{J}_n \xrightarrow[n\rightarrow\infty,n\in \lfloor\alpha\rfloor\bbN]{\P}\alpha^{-1}.
\end{equation}
Then, the estimate (\ref{tails_leaves}) or (\ref{tails_leaves_binary}) on $p$ and $q$ yields that $p(n)/q(\#\mathtt{J}_n)\convP\chi$. Next, we use (\ref{law_pseudo_size=n}) to obtain the following uniform integrability:
\begin{equation}
\label{uniform_inte}
\E\big[\tfrac{p(n)}{q(\#\partial \tau)}\I{p(n)\geq 10\chi \, q(\#\partial\tau)}\ \big|\ \#\tau=n+1\big]\leq \P\big(p(n)\geq 10\chi q(\mathtt{J}_n)\big)\xrightarrow[n\rightarrow\infty,n\in \lfloor\alpha\rfloor\bbN]{} 0.
\end{equation}
Moreover, Proposition~\ref{size_vs_leaves}, together with (\ref{tails_leaves}) or (\ref{tails_leaves_binary}), yields that $p(n)/q(\#\partial\tau)$ under $\P(\, \cdot\, |\, \#\tau=n+1)$ converges to $\chi$ in probability, so (\ref{uniform_inte}) completes the proof of (\ref{leaves_bias_vanishing}).
\smallskip

From Theorem~\ref{marchal_thm}, we know there is an $\alpha$-stable tree $\sT_{\mathrm{nr}}$ such that $i^{1/\alpha-1}\cdot\ctau_i\longrightarrow\alpha\cdot\sT_{\mathrm{nr}}$ almost surely on $\mathbb{K}$. Since $\mathtt{J}_n$ tends to $\infty$, it follows that $\mathtt{J}_n^{1/\alpha-1}\cdot T_n\longrightarrow\alpha\cdot \sT_{\mathrm{nr}}$ almost surely on $\mathbb{K}$. Then, we apply the convergence in probability (\ref{LLN_leaves_pseudo_size=n}) and we write $\alpha a_n/n=(\alpha\mathtt{J}_n/n)^{1-1/\alpha}\cdot \mathtt{J}_n^{1/\alpha-1}$ to deduce that on $\mathbb{K}$,
\begin{equation}
\label{cv_P_Tn}
\tfrac{a_n}{n}\cdot T_n \xrightarrow[n\rightarrow\infty,n\in \lfloor\alpha\rfloor\bbN]{\P} \sT_{\mathrm{nr}}.
\end{equation}
For all $n\in \lfloor\alpha\rfloor\bbN$, let $\mathtt{\Lambda}_n$ be a RV independent of $(\fctau_i)_{i\geq 1}$ and $\sT_{\mathrm{nr}}$, and distributed as $\#\partial\tau$ under $\P(\, \cdot\, |\, \#\partial\tau\geq n/2\alpha)$. Thanks to the asymptotic equivalent (\ref{equiv_leaves}) of the tails of $\#\partial\tau$ and by Kolmogorov's representation theorem, we can then assume that there is a RV $\Lambda\geq 0$, independent of $(\fctau_i)_{i\geq 1}$ and $\sT_{\mathrm{nr}}$, such that
\begin{equation}
\label{cv_Lambdan}
\big(\tfrac{\alpha}{n}\mathtt{\Lambda}_n\big)^{1-1/\alpha}\!\xrightarrow[n\rightarrow\infty,n\in \lfloor\alpha\rfloor\bbN]{\text{a.s.}}\Lambda,\quad\text{and}\quad\P\left(\Lambda>\lambda\right)=\min\big(1,2^{-1/\alpha}\lambda^{-\beta}\big)\text{ for all } \lambda\! >\! 0.
\end{equation}
We now set $\ftau_n=(\tau_n,(W_v^n)_{v\in\partial\tau_n}):=\fctau_{\mathtt{\Lambda}_n}$ for all $n\in \lfloor\alpha\rfloor\bbN$. By Proposition~\ref{marchal_prop} $(iii)$, it is clear that $\ftau_n$ is distributed as $\ftau$ under $\P(\, \cdot\, |\, \#\partial\tau\geq n/2\alpha)$, so we can apply Corollary~\ref{scaling_limit_leaves>n_HS} to it. Moreover, the a.s.~convergence $i^{1/\alpha-1}\cdot\ctau_i\longrightarrow\alpha\cdot\sT_{\mathrm{nr}}$ and (\ref{cv_Lambdan}) entail that on $\mathbb{K}$,
\begin{equation}
\label{cv_taun_prime}
\tfrac{a_n}{n}\cdot\tau_n\xrightarrow[n\rightarrow\infty,n\in \lfloor\alpha\rfloor\bbN]{\text{a.s.}} \Lambda\cdot \sT_{\mathrm{nr}}.
\end{equation}

We now claim that $(a_n\delta^{\tHS(\bT_n)}/n)_{n\in \lfloor\alpha\rfloor\bbN}$ is tight. Indeed, we begin by using the independence between $\bT_n=\fctau_{\mathtt{J}_n}$ and $\mathtt{\Lambda}_n$ to obtain that for all $x>0$,
\[\P\big(\delta^{\tHS(\bT_n)}\geq x \tfrac{n}{a_n}\, ;\, 2n>\alpha\mathtt{J}_n\big)\leq\frac{1}{\P(\alpha\mathtt{\Lambda}_n>2n)}\P\big(\delta^{\tHS(\bT_n)}\geq x\tfrac{n}{a_n}\, ;\, \mathtt{\Lambda}_n>\mathtt{J}_n\big).\]
Since $\ftau_n=\fctau_{\mathtt{\Lambda}_n}$, Proposition~\ref{marchal_prop} $(ii)$ then implies that
\[\P\big(\delta^{\tHS(\bT_n)}\geq x \tfrac{n}{a_n}\, ;\, 2n>\alpha\mathtt{J}_n\big)\leq \frac{1}{\P(\alpha\mathtt{\Lambda}_n/n>2)}\P\big(\delta^{\tHS(\ftau_n)}\geq x\tfrac{n}{a_n}\big)\]
for all $x>0$. While letting $n\to\infty$, we use the convergence (\ref{LLN_leaves_pseudo_size=n}) on the left-hand side, and we apply Corollary~\ref{scaling_limit_leaves>n_HS} together with (\ref{cv_Lambdan}) on the right-hand side. As such, we find that
\[\limsup_{n\rightarrow\infty,n\in \lfloor\alpha\rfloor\bbN}\P\big(\delta^{\tHS(\bT_n)}\geq x\tfrac{n}{a_n}\big)\leq\frac{1}{\P(\Lambda>2^{1-1/\alpha})}\bN_\alpha(\Info_\delta\big(\cT_\rH)\geq x\ \big|\ \zeta>1/2\big)\]
for all $x>0$. By letting $x\to\infty$, the claimed tightness follows from (\ref{relation4}) in Theorem~\ref{from_law(Y)_to_N}.
\smallskip

Next, thanks to the tightness we just established, (\ref{cv_Lambdan}), (\ref{cv_P_Tn}), and Corollary~\ref{scaling_limit_leaves>n_HS}, the family $(\Theta_n)_{n\in \lfloor\alpha\rfloor\bbN}$ of random variables defined by
\[\Theta_n=\Big(\big(\tfrac{\alpha}{n}\mathtt{\Lambda}_n\big)^{1-1/\alpha}\, ,\, \xi_n\, ,\, \tfrac{a_n}{n}\cdot T_n\, ,\, \tfrac{a_n}{n}\delta^{\tHS(\bT_n)}\, ,\, \tfrac{a_n}{n}\cdot\tau_n\, ,\,\tfrac{a_n}{n}\delta^{\tHS(\ftau_n)}\Big)\]
is therefore tight on the product space $\bbR_+\times\{0,1\}\times\mathbb{K}\times\bbR_+\times\mathbb{K}\times\bbR_+$. Let us consider a subsequence $(\Theta_{n_k})$, with $n_k\in \lfloor \alpha\rfloor\bbN$ for all $k\in\bbN$, that converges in distribution. By (\ref{cv_Lambdan}), (\ref{cv_P_Tn}), and (\ref{cv_taun_prime}), there are random variables $\xi\in\{0,1\}$ and $U,V\geq 0$ such that
\[\Theta_{n_k}\convd(\Lambda,\xi,\sT_{\mathrm{nr}},U,\Lambda\cdot\sT_{\mathrm{nr}},V)\]
in distribution. Corollary~\ref{scaling_limit_leaves>n_HS} yields that $(\Lambda\cdot\sT_{\mathrm{nr}},V)$ has the same law as $(\cT_\rH,\Info_\delta(\cT_\rH))$ under $\bN_\alpha(\, \cdot\, |\, \zeta>1/2)$. Thus, it holds $V=\Lambda\Info_\delta(\sT_{\mathrm{nr}})$ a.s., thanks to Proposition~\ref{monotone_homogene}. Moreover, $\mathtt{\Lambda}_n$ is independent of $\bT_n$ for all $n\in \lfloor\alpha\rfloor\bbN$, so $\Lambda$ is independent of $(\xi,\sT_{\mathrm{nr}},U)$. Hence, we have
\[\P\big(U>(1+2\varepsilon)\Info_\delta(\sT_{\mathrm{nr}})\big)\leq\frac{1}{\P(1+\varepsilon<\Lambda<1+2\varepsilon)}\P\big(U>\Lambda\Info_\delta(\sT_{\mathrm{nr}})\, ;\, 1+\varepsilon<\Lambda\big)\]
for all $\varepsilon>0$. Now, it follows from (\ref{cv_Lambdan}) that $\Lambda$ is absolutely continuous, and that its density is positive and continuous at $\lambda=1$. The Portmanteau theorem entails that
\[\P\big(U>\Lambda\Info_\delta(\sT_{\mathrm{nr}})\, ;\, 1+\varepsilon<\Lambda\big)\leq \limsup_{k\rightarrow\infty}\P\big(\tHS(\bT_{n_k})>\tHS(\ftau_{n_k})\, ;\, (1+\varepsilon)^{\alpha\beta}n_k<\alpha\mathtt{\Lambda}_{n_k}\big).\]
By Proposition~\ref{marchal_prop} $(ii)$, if $\tHS(\bT_{n})>\tHS(\ftau_{n})$ and $(1+\varepsilon)^{\alpha\beta}n<\alpha\mathtt{\Lambda}_{n}$ then $(1+\varepsilon)^{\alpha\beta}n<\alpha\mathtt{J}_n$. Thanks to the convergence in probability (\ref{LLN_leaves_pseudo_size=n}), we get that $U\leq (1+2\varepsilon)\Info_\delta(\sT_{\mathrm{nr}})$ almost surely. We show that $U\geq(1-2\varepsilon)\Info_\delta(\sT_{\mathrm{nr}})$ for all $\varepsilon>0$ in the same manner. Therefore, we have $U=\Info_\delta(\sT_{\mathrm{nr}})$ almost surely.

Let us now focus on the conditional law of $\xi$ given $\sT_{\mathrm{nr}}$. Let $g:\mathbb{K}\to\mathbb{R}$ be a bounded and continuous function. We apply Theorem~\ref{duquesne_scaling_limit_size=n}, together with Proposition~\ref{height_contour_tree}, and we use the estimate (\ref{leaves_bias_vanishing}) on the right-hand side of (\ref{law_pseudo_size=n}) to get
\[\E\big[\xi g(\sT_{\mathrm{nr}})\big]=\chi\bN_\alpha\big[g(\cT_\rH)\ \big|\ \zeta=1\big]=\chi\E\big[g(\sT_{\mathrm{nr}})\big].\]
This means that $\xi$ is independent of $\sT_{\mathrm{nr}}$, and has Bernoulli law with success probability $\chi$. Hence, we have proved that $(\Theta_n)_{n\in \lfloor\alpha\rfloor\bbN}$ converges to $(\Lambda,\xi,\sT_{\mathrm{nr}},\Info_\delta(\sT_{\mathrm{nr}}),\Lambda\cdot\sT_{\mathrm{nr}},\Lambda\Info_\delta(\sT_{\mathrm{nr}}))$ in law, where $\Lambda,\xi,\sT_{\mathrm{nr}}$ are independent and $\xi$ is a Bernoulli RV with success probability $\chi$. In particular, the following convergence holds in law on the product space $\{0,1\}\times\mathbb{K}\times\bbR_+$:
\[\Big(\xi_n,\tfrac{a_n}{n}\cdot T_n,\tfrac{a_n}{n}\delta^{\tHS(\bT_n)}\Big)\xrightarrow[n\rightarrow\infty,n\in \lfloor\alpha\rfloor\bbN]{d} (\xi,\sT_{\mathrm{nr}},\Info_\delta(\sT_{\mathrm{nr}}))\]
The convergence (\ref{goal_size=n_HS}) then follows from (\ref{law_pseudo_size=n}) and (\ref{leaves_bias_vanishing}). This completes the proof.
\end{proof}

\section{The Strahler dilation of the Brownian tree: proof of Theorem~\ref{height_info_CRT_intro}}
\label{height_info_CRT}

We restrict ourselves to the case $\alpha=2$ in all this section. Then, (\ref{alpha_gamma_delta}) and (\ref{stable_offspring_explicit}) become \[\beta=1,\quad \gamma=\ln 2,\quad \delta=2,\quad\text{ and }\quad\mu_2=\tfrac{1}{2}\delta_0+\tfrac{1}{2}\delta_2.\] Denote by $\ftau=(\tau,(W_v)_{v\in\partial\tau})$ a $\GWw_2$-weighted tree. By (\ref{lawGW}), $\tau$ can be equal to a tree $t$ with positive probability if and only if $t$ is \textit{binary}, meaning that all its vertices have either $0$ or $2$ children. We are interested in the joint law of the size $\#\tau$ and the weighted Horton--Strahler $\tHS(\ftau)$ of $\ftau$. Flajolet, Raoult \& Vuillemin~\cite{flajolet} and Kemp~\cite{kemp} have independently characterized the law of $(\#\tau,\HS(\tau))$. For all $p,n\in\bbN$ and $z\geq 0$, set
\[R_{p,n}=\#\{t\text{ binary tree}\ :\ \HS(t)=p\text{ and }\#t=2n+1\}\quad\text{ and }\quad R_p(z)=\sum_{n\geq 0}R_{p,n}z^n.\]
We stress that if $t$ is a binary tree then $\mu_2(k_u(t))=\tfrac{1}{2}$ for all $u\in t$, so $\P(\tau=t)=2^{-\# t}$. Since any binary tree has a odd number of vertices, we obtain that $\E[\I{\HS(\tau)=p}s^{\#\tau}]= \frac{s}{2}R_p(s^2/4)$ for all $s\geq 0$. Flajolet, Raoult \& Vuillemin~\cite{flajolet} found a recurrence relation for $R_p$. Moreover, they solved it exactly via a trigonometric change of variable:
\[R_p(z)=\frac{\sin\phi}{\sin 2^p\phi}\quad\text{ where }\quad\cos^2 \tfrac{\phi}{2}=\tfrac{1}{4z}.\]
Recalling from (\ref{law_HS_discrete}) that $\HS(\tau)$ is geometric with parameter $1/2$, and using $\cos(ix)=\cosh(x)$, $\sin(ix)=i\sinh(x)$ and $\sinh(2x)=2\cosh(x)\sinh(x)$ for $x\in\bbR$, it follows that
\begin{equation}
\label{size_HS}
\E\big[s^{\#\tau}\ \big|\ \HS(\tau)=p\big]=2^{p+1}\frac{\sinh(a(s))}{\sinh(2^{p+1}a(s))},\quad\text{ where }\quad \cosh a(s)=\tfrac{1}{s},
\end{equation}
for all $p\in\bbN$ and all $s\in(0,1)$. A similar phenomenon happens during the computation of the generating function of $\#\tau$ under the law $\P(\dd\tau\, |\, \tHS(\ftau)=x)$ given by Definition~\ref{cond_H=x_def}.

\begin{proposition}
\label{size_tHS_equal}
For all $p\in\bbN$, $y\in(0,1)$, and $s\in(0,1)$, it holds that
\[\E\big[s^{\#\tau}\ \big|\ \tHS(\ftau)=p+y\big]=4^p s\frac{\sinh^2\big(a_{y}(s)\big)}{\sinh^2\big(2^{p}a_{y}(s)\big)},\quad\text{ where }\quad\coth a_y(s)=\frac{1-s^2+s^2 2^{-y}}{\sqrt{1-s^2}}.\]
\end{proposition}

\begin{proof}
We set $F_x(s)=\E[\I{\tHS(\ftau)\leq x}s^{\#\tau}]$ for all $x\in\bbR_+$ and $s\in(0,1)$. We fix $s\in(0,1)$ for now. Using the assertion (\ref{nochain_weighted}), Proposition~\ref{law_HT_stable} implies that for all $y\in[0,1)$, it holds
\begin{equation}
\label{size_tHS_init}
F_y(s)=\P(\tHS(\ftau)\leq y)s=(1-2^{-y})s.
\end{equation}
Let $x\geq 1$. Recall from (\ref{stable_offspring}) that $\varphi_2(r)=r\! +\! \frac{1}{2}(1\! -\! r)^2$ is the generating function of the offspring law $\mu_2$. By Definition~\ref{GWwdef} of $\GWw_\alpha$-weighted trees, the decompositions (\ref{decompo_strahler}) and $\#\tau=1 + \sum_{i=1}^{k_\varnothing(\tau)}\! \#\theta_{(i)}\tau$ lead to 
\[F_x(s)=s\, \varphi_2(F_{x-1}(s))+s(F_x(s)-F_{x-1}(s))\,\varphi_2'(F_{x-1}(s)).\]
After some manipulations, this identity becomes $F_x(s)=\tfrac{s}{2}\big(1-F_{x-1}(s)^2\big)+sF_x(s)F_{x-1}(s)$. Now, the trick is to set $G_x=\frac{1}{\sqrt{1-s^2}}(1-sF_x(s))$ for all $x\in\bbR_+$ and to check that
\[\forall x\geq 1,\quad G_x=\frac{G_{x-1}^2+1}{2G_{x-1}}.\]
This allows us to recognize the double-angle formula for the hyperbolic cotangent:
\[\forall z\in\bbR,\quad \coth 2z=\frac{\coth^2 z+1}{2\coth z}.\]
Furthermore, we can verify via (\ref{size_tHS_init}) and an elementary analysis that if $y\in[0,1)$ then $G_y>1$ because $s>0$. In particular, there is a unique $a_y(s)>0$ such that $G_y=\coth a_y(s)$. We obtain by induction on $\lfloor x\rfloor$ that $G_x=\coth\big(2^{\lfloor x\rfloor}a_{x-\lfloor x\rfloor}(s)\big)$ for all $x\! \in\! \bbR_+$. Thus,
\begin{equation}
\label{size_tHS_less}
sF_{p+y}(s)=1-\sqrt{1\! -\! s^2}\coth\big(2^{p}a_{y}(s)\big)\quad\text{where}\quad\coth a_y(s)=\frac{1\! -\! s^2 \! +\! s^2 2^{-y}}{\sqrt{1-s^2}},
\end{equation}
for all $s\in(0,1)$, all $p\in\bbN$, and all $y\in[0,1)$. For all $s\in(0,1)$, Proposition~\ref{cond_H=x_prime} entails that the function $y\in(0,1)\longmapsto F_{p+y}(s)$ is $C^1$ with derivative $(\ln 2)2^{-p-y}\E[s^{\#\tau}\ |\ \tHS(\ftau)=p+y]$. As $\coth'=\sinh^{-2}$, differentiating (\ref{size_tHS_less}) with respect to $y$ completes the proof.
\end{proof}

Taking $p\to\infty$ in Proposition~\ref{size_tHS_equal} leads to a description of the law of the mass of the $\aHS_2$-real tree, and then of the Strahler dilation of the $2$-stable tree. This results in a proof of Theorem~\ref{height_info_CRT_intro}. Recall $\mathfrak{h}$ and $\mathfrak{m}$ from (\ref{height_mass}). We recall from (\ref{real_tree_coded}) that if $f$ is a continuous excursion then $\cT_f$ stands for the real tree coded by $f$, so that $\mathfrak{h}(\cT_f)=\sup f$ and $\mathfrak{m}(\cT_f)=\zeta(f)$. Furthermore, recall from Section~\ref{more_stable} that if $\be$ is a standard Brownian excursion then $\cT_\be$, which is a Brownian tree, has the same law as $(\cT_\rH,\tfrac{1}{\sqrt{2}}d_\rH,\rho_\rH,\mu_\rH)$ under $\bN_2(\dd \rH\, |\, \zeta=1)$.

\begin{proof}[Proof of Theorem~\ref{height_info_CRT_intro}]
Let $\sT$ be an $\aHS_2$-real tree, i.e.~the limit tree in Theorem~\ref{scaling_limit_HS_intro}. Fix $\lambda>0$, and set $x_n=n+1/2$ and $s_n=e^{-2\lambda/4^n}$ for all $n\in\mathbb{N}$. Theorem~\ref{scaling_limit_HS_intro} yields that
\[\E\big[s_n^{\#\tau}\ \big|\ \tHS(\ftau)=n+1/2\big]=\E\big[\exp(-4\lambda \cdot 4^{-x_n}\#\tau)\ \big|\ \tHS(\ftau)=x_n\big]\longrightarrow \E\big[e^{-4\lambda \mathfrak{m}(\mathscr{T})}\big].\]
Using this, an easy asymptotic study of the formula given by Proposition~\ref{size_tHS_equal} entails that
\[\forall\lambda>0,\quad \E\big[\exp(-\lambda \mathfrak{m}(\sT))\big]=\Big(\tfrac{\sqrt{2\lambda}}{\sinh \sqrt{2\lambda}}\Big)^2.\]
By (\ref{max_Brownian_excursion}), we then see that
$\E\big[\exp(-\lambda\mathfrak{m}(\sT))\big]=\bN_2\big[\exp(-2\lambda\mathfrak{m}(\cT_\rH))\, \big|\, \sup \rH=1\big]$ for all $\lambda>0$. We claim that this implies that $\bN_2\big[e^{-\lambda\zeta}g(\, 2\Info_2(\cT_\rH)\, )\big]=\bN_2\big[e^{-\lambda\zeta}g(\, \mathfrak{h}(\cT_\rH)\, )\big]$ for any $\lambda>0$ and any measurable and bounded $g:[0,\infty]\to\bbR$. Indeed, applying (\ref{relation1}) yields 
\begin{align*}
\bN_2\big[e^{-\lambda\zeta}g(\, 2\Info_2(\cT_\rH)\, )\big]&=\frac{1}{2} \int_{\bbR_+} g(2x)x^{-2}\cdot \E\big[\exp(-2x^{2}\lambda\mathfrak{m}(\sT))\big] \, \dd x\\
\intertext{one the one hand, and doing the change of variable $x=2y$ in the identity (\ref{decompo_sup}) gives}
\bN_2\big[e^{-\lambda\zeta}g(\, \mathfrak{h}(\cT_\rH)\, )\big]&=\frac{1}{2}\int_{\bbR_+} g(2y)y^{-2}\cdot \bN_2\big[\exp(-2\cdot 2y^{2}\lambda\mathfrak{m}(\cT_\rH))\ \big|\ \sup \rH=1\big]\, \dd y 
\end{align*}
on the other hand. Hence, the image measures of $\bN_2$ by $(\zeta,\mathfrak{h}(\cT_\rH))$ and by $(\zeta,2\, \Info_2(\cT_\rH))$ are equal. Since $\bN_2(\dd \rH\, |\ \zeta=1)$ is the law of $(\zeta^{1/\alpha-1}\rH_{\zeta s}\ ;\ s\geq 0)$ under the probability measure $\bN_2(\, \cdot\, |\, \zeta\!>\!1)$, the desired result follows by the homogeneity from Proposition~\ref{monotone_homogene} $(ii)$.
\end{proof}

%
%

\begin{acks}[Acknowledgments]
I am very much grateful to my Ph.D.~advisor Thomas Duquesne for introducing me to the Horton-Strahler number, for sharing his intuition regarding the asymptotic behavior, and for many engaging conversations. I warmly thank Nicolas Broutin for useful feedback and observations about the link between the height and the Horton-Strahler number. He also pointed me out that such a link was found in the discrete setting in \cite{flajolet}. I am indebted to dedicated reviewers for their numerous, careful, and helpful comments.
\end{acks}
\begin{funding}
This research has been supported by the Natural Sciences and Engineering Research Council of Canada (NSERC) via a Banting postdoctoral fellowship [BPF-198443].
\noi
Cette recherche a été financée par le Conseil de recherches en sciences naturelles et en génie du Canada (CRSNG) via une bourse postdoctorale Banting [BPF-198443].
\end{funding}



\bibliographystyle{imsart-number} 
\bibliography{ht_gw_bibli}       

\begin{thebibliography}{63}

\bibitem{AbrDel09}
\begin{barticle}[author]
\bauthor{\bsnm{Abraham},~\bfnm{Romain}\binits{R.}} \AND
  \bauthor{\bsnm{Delmas},~\bfnm{Jean-François}\binits{J.-F.}}
(\byear{2009}).
\btitle{Williams’ decomposition of the Lévy continuum random tree and
  simultaneous extinction probability for populations with neutral mutations}.
\bjournal{Stochastic Processes and their Applications}
\bvolume{119}
\bpages{1124 -- 1143}.
\bdoi{https://doi.org/10.1016/j.spa.2008.06.001}
\end{barticle}
\endbibitem

\bibitem{AD2014}
\begin{barticle}[author]
\bauthor{\bsnm{Abraham},~\bfnm{Romain}\binits{R.}} \AND
  \bauthor{\bsnm{Delmas},~\bfnm{Jean-Fran{\c{c}}ois}\binits{J.-F.}}
(\byear{2014}).
\btitle{{Local limits of conditioned Galton-Watson trees: the infinite spine
  case}}.
\bjournal{Electronic Journal of Probability}
\bvolume{19}
\bpages{1 -- 19}.
\bdoi{10.1214/EJP.v19-2747}
\end{barticle}
\endbibitem

\bibitem{GHP_polish}
\begin{barticle}[author]
\bauthor{\bsnm{Abraham},~\bfnm{Romain}\binits{R.}},
  \bauthor{\bsnm{Delmas},~\bfnm{Jean-François}\binits{J.-F.}} \AND
  \bauthor{\bsnm{Hoscheit},~\bfnm{Patrick}\binits{P.}}
(\byear{2013}).
\btitle{A note on the {G}romov-{H}ausdorff-{P}rokhorov distance between
  (locally) compact metric measure spaces}.
\bjournal{Electronic Journal of Probability}
\bvolume{18}
\bpages{1 -- 21}.
\bdoi{10.1214/EJP.v18-2116}
\end{barticle}
\endbibitem

\bibitem{coalescent}
\begin{barticle}[author]
\bauthor{\bsnm{Abraham},~\bfnm{Romain}\binits{R.}} \AND
  \bauthor{\bsnm{Delmas},~\bfnm{Jean-Fran{\c c}ois}\binits{J.-F.}}
(\byear{2015}).
\btitle{{$\beta$-coalescents and stable Galton-Watson trees}}.
\bjournal{{ALEA : Latin American Journal of Probability and Mathematical
  Statistics}}
\bvolume{12}
\bpages{451 -- 476}.
\end{barticle}
\endbibitem

\bibitem{exit_times}
\begin{barticle}[author]
\bauthor{\bsnm{Abraham},~\bfnm{Romain}\binits{R.}},
  \bauthor{\bsnm{Delmas},~\bfnm{Jean-Fran{\c c}ois}\binits{J.-F.}} \AND
  \bauthor{\bsnm{Hoscheit},~\bfnm{Patrick}\binits{P.}}
(\byear{2014}).
\btitle{{Exit times for an increasing L{\'e}vy tree-valued process}}.
\bjournal{{Probability Theory and Related Fields}}
\bvolume{159}
\bpages{357 -- 403}.
\end{barticle}
\endbibitem

\bibitem{GHP_addario}
\begin{barticle}[author]
\bauthor{\bsnm{Addario-Berry},~\bfnm{Louigi}\binits{L.}},
  \bauthor{\bsnm{Broutin},~\bfnm{Nicolas}\binits{N.}},
  \bauthor{\bsnm{Goldschmidt},~\bfnm{Christina}\binits{C.}} \AND
  \bauthor{\bsnm{Miermont},~\bfnm{Gr{\'e}gory}\binits{G.}}
(\byear{2017}).
\btitle{{The scaling limit of the minimum spanning tree of the complete
  graph}}.
\bjournal{The Annals of Probability}
\bvolume{45}
\bpages{3075 -- 3144}.
\bdoi{10.1214/16-AOP1132}
\end{barticle}
\endbibitem

\bibitem{aldousI}
\begin{barticle}[author]
\bauthor{\bsnm{Aldous},~\bfnm{David}\binits{D.}}
(\byear{1991}).
\btitle{The {C}ontinuum {R}andom {T}ree {I}}.
\bjournal{The Annals of Probability}
\bvolume{19}
\bpages{1 -- 28}.
\bmrnumber{1085326}
\end{barticle}
\endbibitem

\bibitem{aldous1993}
\begin{barticle}[author]
\bauthor{\bsnm{Aldous},~\bfnm{David}\binits{D.}}
(\byear{1993}).
\btitle{The {C}ontinuum {R}andom {T}ree {III}}.
\bjournal{The Annals of Probability}
\bvolume{21}
\bpages{248 -- 289}.
\bdoi{10.1214/aop/1176989404}
\end{barticle}
\endbibitem

\bibitem{aldouspitman98}
\begin{barticle}[author]
\bauthor{\bsnm{Aldous},~\bfnm{David}\binits{D.}} \AND
  \bauthor{\bsnm{Pitman},~\bfnm{Jim}\binits{J.}}
(\byear{1998}).
\btitle{{Tree-valued Markov chains derived from Galton-Watson processes}}.
\bjournal{Annales de l'Institut Henri Poincaré (B) Probability and Statistics}
\bvolume{34}
\bpages{637 -- 686}.
\bdoi{https://doi.org/10.1016/S0246-0203(98)80003-4}
\end{barticle}
\endbibitem

\bibitem{order_stats}
\begin{bbook}[author]
\bauthor{\bsnm{Arnold},~\bfnm{Barry~C.}\binits{B.~C.}},
  \bauthor{\bsnm{Balakrishnan},~\bfnm{Narayanaswamy}\binits{N.}} \AND
  \bauthor{\bsnm{Nagaraja},~\bfnm{H.~N.}\binits{H.~N.}}
(\byear{2008}).
\btitle{{A First Course in Order Statistics}},
\bedition{complete reprinted} ed.
\bseries{Classics in Applied Mathematics}.
\bpublisher{Society for Industrial and Applied Mathematics}.
\bdoi{10.1137/1.9780898719062}
\end{bbook}
\endbibitem

\bibitem{bamufleh}
\begin{barticle}[author]
\bauthor{\bsnm{Bamufleh},~\bfnm{Sameer}\binits{S.}},
  \bauthor{\bsnm{Al-Wagdany},~\bfnm{Abdullah}\binits{A.}},
  \bauthor{\bsnm{Elfeki},~\bfnm{Amro}\binits{A.}} \AND
  \bauthor{\bsnm{Chaabani},~\bfnm{Anis}\binits{A.}}
(\byear{2020}).
\btitle{{Developing a geomorphological instantaneous unit hydrograph (GIUH)
  using equivalent Horton-Strahler ratios for flash flood predictions in arid
  regions}}.
\bjournal{Geomatics, Natural Hazards and Risk}
\bvolume{11}
\bpages{1697 -- 1723}.
\bdoi{10.1080/19475705.2020.1811404}
\end{barticle}
\endbibitem

\bibitem{bertouin}
\begin{bbook}[author]
\bauthor{\bsnm{Bertoin},~\bfnm{Jean}\binits{J.}}
(\byear{1996}).
\btitle{L{\'e}vy Processes}.
\bseries{Volume 121 of Cambridge Tracts in Mathematics}.
\bpublisher{Cambridge University Press}.
\end{bbook}
\endbibitem

\bibitem{billingsley2013convergence}
\begin{bbook}[author]
\bauthor{\bsnm{Billingsley},~\bfnm{Patrick}\binits{P.}}
(\byear{1999}).
\btitle{Convergence of probability measures},
\bedition{second} ed.
\bseries{Wiley Series in Probability and Statistics: Probability and
  Statistics}.
\bpublisher{John Wiley \& Sons Inc.}, \baddress{New York}.
\bnote{A Wiley-Interscience Publication}.
\bmrnumber{MR1700749 (2000e:60008)}
\end{bbook}
\endbibitem

\bibitem{RegVar}
\begin{bbook}[author]
\bauthor{\bsnm{Bingham},~\bfnm{Nicholas~H.}\binits{N.~H.}},
  \bauthor{\bsnm{Goldie},~\bfnm{Charles~M.}\binits{C.~M.}} \AND
  \bauthor{\bsnm{Teugels},~\bfnm{Jozef~L.}\binits{J.~L.}}
(\byear{1989}).
\btitle{Regular Variation}.
\bseries{Encyclopedia of Mathematics and its Applications}
\bvolume{27}.
\bpublisher{Cambridge University Press}.
\end{bbook}
\endbibitem

\bibitem{brandenberger}
\begin{barticle}[author]
\bauthor{\bsnm{Brandenberger},~\bfnm{Anna}\binits{A.}},
  \bauthor{\bsnm{Devroye},~\bfnm{Luc}\binits{L.}} \AND
  \bauthor{\bsnm{Reddad},~\bfnm{Tommy}\binits{T.}}
(\byear{2021}).
\btitle{{The Horton–Strahler number of conditioned Galton–Watson trees}}.
\bjournal{Electronic Journal of Probability}
\bvolume{26}
\bpages{1 -- 29}.
\bdoi{10.1214/21-EJP678}
\end{barticle}
\endbibitem

\bibitem{bruss}
\begin{barticle}[author]
\bauthor{\bsnm{Bruss},~\bfnm{F.~Thomas}\binits{F.~T.}} \AND
  \bauthor{\bsnm{O'Cinneide},~\bfnm{Colm~A.}\binits{C.~A.}}
(\byear{1990}).
\btitle{{On the Maximum and Its Uniqueness for Geometric Random Samples}}.
\bjournal{Journal of Applied Probability}
\bvolume{27}
\bpages{598 -- 610}.
\end{barticle}
\endbibitem

\bibitem{burd}
\begin{barticle}[author]
\bauthor{\bsnm{Burd},~\bfnm{Gregory~A.}\binits{G.~A.}},
  \bauthor{\bsnm{Waymire},~\bfnm{Edward~C.}\binits{E.~C.}} \AND
  \bauthor{\bsnm{Winn},~\bfnm{Ronald~D.}\binits{R.~D.}}
(\byear{2000}).
\btitle{{A Self-Similar Invariance of Critical Binary Galton-Watson Trees}}.
\bjournal{Bernoulli}
\bvolume{6}
\bpages{1 -- 21}.
\end{barticle}
\endbibitem

\bibitem{chavan}
\begin{barticle}[author]
\bauthor{\bsnm{Chavan},~\bfnm{Sagar~R.}\binits{S.~R.}} \AND
  \bauthor{\bsnm{Srinivas},~\bfnm{Venkata~V.}\binits{V.~V.}}
(\byear{2015}).
\btitle{{Effect of DEM source on equivalent Horton–Strahler ratio based GIUH
  for catchments in two Indian river basins}}.
\bjournal{Journal of Hydrology}
\bvolume{528}
\bpages{463 -- 489}.
\bdoi{https://doi.org/10.1016/j.jhydrol.2015.06.049}
\end{barticle}
\endbibitem

\bibitem{chung}
\begin{barticle}[author]
\bauthor{\bsnm{Chung},~\bfnm{Kai~L.}\binits{K.~L.}}
(\byear{1976}).
\btitle{{Excursions in Brownian motion}}.
\bjournal{Arkiv för Matematik}
\bvolume{14}
\bpages{155 -- 177}.
\bdoi{10.1007/BF02385832}
\end{barticle}
\endbibitem

\bibitem{nested}
\begin{barticle}[author]
\bauthor{\bsnm{Curien},~\bfnm{Nicolas}\binits{N.}} \AND
  \bauthor{\bsnm{Haas},~\bfnm{Bénédicte}\binits{B.}}
(\byear{2013}).
\btitle{The stable trees are nested}.
\bjournal{Probability Theory and Related Fields}
\bvolume{157}
\bpages{847 -- 883}.
\end{barticle}
\endbibitem

\bibitem{devroye95}
\begin{barticle}[author]
\bauthor{\bsnm{Devroye},~\bfnm{Luc}\binits{L.}} \AND
  \bauthor{\bsnm{Kruszewski},~\bfnm{Paul}\binits{P.}}
(\byear{1995}).
\btitle{A note on the {H}orton-{S}trahler number for random trees}.
\bjournal{Information Processing Letters}
\bvolume{56}
\bpages{95 -- 99}.
\bdoi{https://doi.org/10.1016/0020-0190(95)00114-R}
\end{barticle}
\endbibitem

\bibitem{drmota}
\begin{barticle}[author]
\bauthor{\bsnm{Drmota},~\bfnm{Michael}\binits{M.}} \AND
  \bauthor{\bsnm{Prodinger},~\bfnm{Helmut}\binits{H.}}
(\byear{2006}).
\btitle{{The Register Function for T-Ary Trees}}.
\bjournal{ACM Transactions on Algorithms}
\bvolume{2}
\bpages{318 -- 334}.
\bdoi{10.1145/1159892.1159894}
\end{barticle}
\endbibitem

\bibitem{duquesne_contour_stable}
\begin{barticle}[author]
\bauthor{\bsnm{Duquesne},~\bfnm{Thomas}\binits{T.}}
(\byear{2003}).
\btitle{{A limit theorem for the contour process of conditioned Galton--Watson
  trees}}.
\bjournal{The Annals of Probability}
\bvolume{31}
\bpages{996 -- 1027}.
\bdoi{10.1214/aop/1048516543}
\end{barticle}
\endbibitem

\bibitem{duquesne09}
\begin{barticle}[author]
\bauthor{\bsnm{Duquesne},~\bfnm{Thomas}\binits{T.}}
(\byear{2009}).
\btitle{{An elementary proof of Hawkes's conjecture on Galton-Watson trees.}}
\bjournal{Electronic Communications in Probability}
\bvolume{14}
\bpages{151 -- 164}.
\bdoi{10.1214/ECP.v14-1454}
\end{barticle}
\endbibitem

\bibitem{levytree_DLG}
\begin{bbook}[author]
\bauthor{\bsnm{Duquesne},~\bfnm{Thomas}\binits{T.}} \AND
  \bauthor{\bsnm{Le~Gall},~\bfnm{Jean-Fran\c{c}ois}\binits{J.-F.}}
(\byear{2002}).
\btitle{Random trees, {L\'evy} processes and spatial branching processes}.
\bseries{Ast\'erisque}
\bvolume{281}.
\bpublisher{Soci\'et\'e math\'ematique de France}.
\bmrnumber{1954248}
\end{bbook}
\endbibitem

\bibitem{duquesne_winkel}
\begin{barticle}[author]
\bauthor{\bsnm{Duquesne},~\bfnm{Thomas}\binits{T.}} \AND
  \bauthor{\bsnm{Winkel},~\bfnm{Matthias}\binits{M.}}
(\byear{2019}).
\btitle{{Hereditary tree growth and L{\'e}vy forests}}.
\bjournal{{Stochastic Processes and their Applications}}
\bvolume{129}
\bpages{3690 -- 3747}.
\end{barticle}
\endbibitem

\bibitem{eisenberg}
\begin{barticle}[author]
\bauthor{\bsnm{Eisenberg},~\bfnm{Bennett}\binits{B.}}
(\byear{2008}).
\btitle{{On the expectation of the maximum of IID geometric random variables}}.
\bjournal{Statistics \& Probability Letters}
\bvolume{78}
\bpages{135 -- 143}.
\bdoi{https://doi.org/10.1016/j.spl.2007.05.011}
\end{barticle}
\endbibitem

\bibitem{esparza}
\begin{binproceedings}[author]
\bauthor{\bsnm{Esparza},~\bfnm{Javier}\binits{J.}},
  \bauthor{\bsnm{Luttenberger},~\bfnm{Michael}\binits{M.}} \AND
  \bauthor{\bsnm{Schlund},~\bfnm{Maximilian}\binits{M.}}
(\byear{2016}).
\btitle{History of {S}trahler {N}umbers — with a {P}reface}.
\bpublisher{International Conference on Language and Automata Theory and
  Applications, 2014}
\bnote{Available online at
  https://archive.model.in.tum.de/um/bibdb/esparza/latarevised16.pdf}.
\end{binproceedings}
\endbibitem

\bibitem{evans}
\begin{bbook}[author]
\bauthor{\bsnm{Evans},~\bfnm{Steven}\binits{S.}}
(\byear{2007}).
\btitle{Probability and Real Trees: {\'E}cole d'{\'E}t{\'e} de Probabilit{\'e}s
  de Saint-Flour XXXV-2005}.
\bseries{Lecture Notes in Mathematics}.
\bpublisher{Springer Berlin Heidelberg}.
\end{bbook}
\endbibitem

\bibitem{facbeneda}
\begin{barticle}[author]
\bauthor{\bsnm{Fac-Beneda},~\bfnm{Joanna}\binits{J.}}
(\byear{2013}).
\btitle{{Fractal structure of the Kashubian hydrographic system}}.
\bjournal{Journal of Hydrology}
\bvolume{488}
\bpages{48 -- 54}.
\bdoi{https://doi.org/10.1016/j.jhydrol.2013.02.033}
\end{barticle}
\endbibitem

\bibitem{feller1971}
\begin{bbook}[author]
\bauthor{\bsnm{Feller},~\bfnm{William}\binits{W.}}
(\byear{1971}).
\btitle{{An Introduction to Probability Theory and Its Applications. Vol. II.
  }}.
\bseries{Second}.
\bpublisher{John Wiley \& Sons Inc.}, \baddress{New York}.
\bmrnumber{MR0270403 (42 \#\#5292)}
\end{bbook}
\endbibitem

\bibitem{flajolet}
\begin{barticle}[author]
\bauthor{\bsnm{Flajolet},~\bfnm{Philippe}\binits{P.}},
  \bauthor{\bsnm{Raoult},~\bfnm{Jean-Claude}\binits{J.-C.}} \AND
  \bauthor{\bsnm{Vuillemin},~\bfnm{Jean~E.}\binits{J.~E.}}
(\byear{1979}).
\btitle{The number of registers required for evaluating arithmetic
  expressions}.
\bjournal{Theoretical Computer Science}
\bvolume{9}
\bpages{99 -- 125}.
\bdoi{https://doi.org/10.1016/0304-3975(79)90009-4}
\end{barticle}
\endbibitem

\bibitem{francon}
\begin{barticle}[author]
\bauthor{\bsnm{Fran{\c{c}}on},~\bfnm{Jean}\binits{J.}}
(\byear{1984}).
\btitle{Sur le nombre de registres n{\'e}cessaires {\`a} l'{\'e}valuation d'une
  expression arithm{\'e}tique}.
\bjournal{RAIRO. Informatique th{\'e}orique}
\bvolume{18}
\bpages{355 -- 364}.
\end{barticle}
\endbibitem

\bibitem{legall_trees}
\begin{barticle}[author]
\bauthor{\bsnm{Gall},~\bfnm{Jean-François~Le}\binits{J.-F.~L.}}
(\byear{2005}).
\btitle{{Random trees and applications}}.
\bjournal{Probability Surveys}
\bvolume{2}
\bpages{245 -- 311}.
\bdoi{10.1214/154957805100000140}
\end{barticle}
\endbibitem

\bibitem{lejan_legall}
\begin{barticle}[author]
\bauthor{\bsnm{Gall},~\bfnm{Jean-Francois~Le}\binits{J.-F.~L.}} \AND
  \bauthor{\bsnm{Jan},~\bfnm{Yves~Le}\binits{Y.~L.}}
(\byear{1998}).
\btitle{Branching Processes in Levy Processes: The Exploration Process}.
\bjournal{The Annals of Probability}
\bvolume{26}
\bpages{213 -- 252}.
\end{barticle}
\endbibitem

\bibitem{horton45}
\begin{barticle}[author]
\bauthor{\bsnm{Horton},~\bfnm{Robert~E.}\binits{R.~E.}}
(\byear{1945}).
\btitle{Erosional Development of streams and their drainage basins ;
  Hydrophysical approach to quantitative morphology}.
\bjournal{GSA Bulletin}
\bvolume{56}
\bpages{275 -- 370}.
\bdoi{10.1130/0016-7606(1945)56[275:EDOSAT]2.0.CO;2}
\end{barticle}
\endbibitem

\bibitem{jacod}
\begin{bbook}[author]
\bauthor{\bsnm{Jacod},~\bfnm{Jean}\binits{J.}} \AND
  \bauthor{\bsnm{Shiryaev},~\bfnm{Albert~N.}\binits{A.~N.}}
(\byear{2003}).
\btitle{Limit Theorems for Stochastic Processes},
\bedition{second} ed.
\bseries{Grundlehren der mathematischen Wissenschaften}
\bvolume{288}.
\bpublisher{Springer Berlin Heidelberg}, \baddress{Berlin, Heidelberg}.
\bdoi{10.1007/978-3-662-05265-5_6}
\end{bbook}
\endbibitem

\bibitem{janson}
\begin{barticle}[author]
\bauthor{\bsnm{Janson},~\bfnm{Svante}\binits{S.}}
(\byear{2006}).
\btitle{Conditioned {G}alton--{W}atson trees do not grow}.
\bjournal{Discrete Mathematics \& Theoretical Computer Science}
\bvolume{DMTCS Proceedings vol. AG, Fourth Colloquium on Mathematics and
  Computer Science Algorithms, Trees, Combinatorics and Probabilities}.
\end{barticle}
\endbibitem

\bibitem{kemp}
\begin{barticle}[author]
\bauthor{\bsnm{Kemp},~\bfnm{Rainer}\binits{R.}}
(\byear{1979}).
\btitle{The average number of registers needed to evaluate a binary tree
  optimally}.
\bjournal{Acta Informatica}
\bvolume{11}
\bpages{363 -- 372}.
\end{barticle}
\endbibitem

\bibitem{kennedy}
\begin{barticle}[author]
\bauthor{\bsnm{Kennedy},~\bfnm{Douglas~P.}\binits{D.~P.}}
(\byear{1976}).
\btitle{{The Distribution of the Maximum Brownian Excursion}}.
\bjournal{Journal of Applied Probability}
\bvolume{13}
\bpages{371 -- 376}.
\end{barticle}
\endbibitem

\bibitem{kesten86}
\begin{barticle}[author]
\bauthor{\bsnm{Kesten},~\bfnm{Harry}\binits{H.}}
(\byear{1986}).
\btitle{Subdiffusive behavior of random walk on a random cluster}.
\bjournal{Annales de l'Institut Henri Poincaré (B) Probability and Statistics}
\bvolume{22}
\bpages{425 -- 487}.
\end{barticle}
\endbibitem

\bibitem{time_place}
\begin{barticle}[author]
\bauthor{\bsnm{Khanfir},~\bfnm{Robin}\binits{R.}}
(\byear{2021}).
\btitle{{Time and place of the maximum for one-dimensional diffusion bridges
  and meanders}}.
\bjournal{Probability Surveys}
\bvolume{18}
\bpages{1 -- 43}.
\bdoi{10.1214/18-PS312}
\end{barticle}
\endbibitem

\bibitem{companion_1}
\begin{barticle}[author]
\bauthor{\bsnm{Khanfir},~\bfnm{Robin}\binits{R.}}
(\byear{2023}).
\btitle{{The Horton-Strahler number of Galton-Watson trees with possibly
  infinite variance}}.
\bjournal{Preprint available on arXiv}.
\bnote{arXiv:2307.05983}.
\end{barticle}
\endbibitem

\bibitem{GHP_correspondence}
\begin{barticle}[author]
\bauthor{\bsnm{Khezeli},~\bfnm{Ali}\binits{A.}}
(\byear{2020}).
\btitle{{Metrization of the Gromov–Hausdorff (-Prokhorov) topology for
  boundedly-compact metric spaces}}.
\bjournal{Stochastic Processes and their Applications}
\bvolume{130}
\bpages{3842 -- 3864}.
\bdoi{https://doi.org/10.1016/j.spa.2019.11.001}
\end{barticle}
\endbibitem

\bibitem{kortchemski_leaves}
\begin{barticle}[author]
\bauthor{\bsnm{Kortchemski},~\bfnm{Igor}\binits{I.}}
(\byear{2012}).
\btitle{Invariance principles for {G}alton--{W}atson trees conditioned on the
  number of leaves}.
\bjournal{Stochastic Processes and their Applications}
\bvolume{122}
\bpages{3126 -- 3172}.
\end{barticle}
\endbibitem

\bibitem{kortchemski_simple}
\begin{bincollection}[author]
\bauthor{\bsnm{Kortchemski},~\bfnm{Igor}\binits{I.}}
(\byear{2013}).
\btitle{A Simple Proof of {D}uquesne's Theorem on Contour Processes of
  Conditioned {G}alton--{W}atson Trees}.
In \bbooktitle{{S{\'e}minaire de Probabilit{\'e}s XLV}},
(\beditor{\bfnm{Catherine}\binits{C.}~\bsnm{Donati-Martin}},
  \beditor{\bfnm{Antoine}\binits{A.}~\bsnm{Lejay}} \AND
  \beditor{\bfnm{Alain}\binits{A.}~\bsnm{Rouault}}, eds.).
\bseries{Lecture Notes in Mathematics}
\bpages{537 -- 558}.
\bpublisher{Springer International Publishing}, \baddress{Heidelberg}.
\bdoi{10.1007/978-3-319-00321-4_20}
\end{bincollection}
\endbibitem

\bibitem{kovchegov23}
\begin{barticle}[author]
\bauthor{\bsnm{Kovchegov},~\bfnm{Yevgeniy}\binits{Y.}},
  \bauthor{\bsnm{Xu},~\bfnm{Guochen}\binits{G.}} \AND
  \bauthor{\bsnm{Zaliapin},~\bfnm{Ilya}\binits{I.}}
(\byear{2023}).
\btitle{{Invariant Galton–Watson trees: metric properties and attraction with
  respect to generalized dynamical pruning}}.
\bjournal{Advances in Applied Probability}
\bpages{1 -- 29}.
\bdoi{10.1017/apr.2022.39}
\end{barticle}
\endbibitem

\bibitem{hortonlaws}
\begin{barticle}[author]
\bauthor{\bsnm{Kovchegov},~\bfnm{Yevgeniy}\binits{Y.}} \AND
  \bauthor{\bsnm{Zaliapin},~\bfnm{Ilya}\binits{I.}}
(\byear{2020}).
\btitle{{Random self-similar trees: A mathematical theory of Horton laws}}.
\bjournal{Probability Surveys}
\bvolume{17}
\bpages{1 -- 213}.
\bdoi{10.1214/19-PS331}
\end{barticle}
\endbibitem

\bibitem{kovchegov}
\begin{barticle}[author]
\bauthor{\bsnm{Kovchegov},~\bfnm{Yevgeniy}\binits{Y.}} \AND
  \bauthor{\bsnm{Zaliapin},~\bfnm{Ilya}\binits{I.}}
(\byear{2021}).
\btitle{{Invariance and attraction properties of Galton–Watson trees}}.
\bjournal{Bernoulli}
\bvolume{27}
\bpages{1789 -- 1823}.
\bdoi{10.3150/20-BEJ1292}
\end{barticle}
\endbibitem

\bibitem{lejan91}
\begin{barticle}[author]
\bauthor{\bsnm{Le~Jan},~\bfnm{Yves}\binits{Y.}}
(\byear{1991}).
\btitle{Superprocesses and projective limits of branching {Markov} process}.
\bjournal{Annales de l'Intitut Henri Poincaré (B) Probability and Statistics}
\bvolume{27}
\bpages{91 -- 106}.
\end{barticle}
\endbibitem

\bibitem{LlogLcriteria}
\begin{barticle}[author]
\bauthor{\bsnm{Lyons},~\bfnm{Russell}\binits{R.}},
  \bauthor{\bsnm{Pemantle},~\bfnm{Robin}\binits{R.}} \AND
  \bauthor{\bsnm{Peres},~\bfnm{Yuval}\binits{Y.}}
(\byear{1995}).
\btitle{{Conceptual Proofs of $L$ Log $L$ Criteria for Mean Behavior of
  Branching Processes}}.
\bjournal{The Annals of Probability}
\bvolume{23}
\bpages{1125 -- 1138}.
\bdoi{10.1214/aop/1176988176}
\end{barticle}
\endbibitem

\bibitem{marchal}
\begin{barticle}[author]
\bauthor{\bsnm{Marchal},~\bfnm{Philippe}\binits{P.}}
(\byear{2008}).
\btitle{{A note on the fragmentation of a stable tree}}.
\bjournal{{Discrete Mathematics \& Theoretical Computer Science}}
\bvolume{{DMTCS Proceedings vol. AI, Fifth Colloquium on Mathematics and
  Computer Science}}.
\bdoi{10.46298/dmtcs.3586}
\end{barticle}
\endbibitem

\bibitem{moussa}
\begin{barticle}[author]
\bauthor{\bsnm{Moussa},~\bfnm{Roger}\binits{R.}} \AND
  \bauthor{\bsnm{Bocquillon},~\bfnm{Claude}\binits{C.}}
(\byear{1996}).
\btitle{Fractal analyses of tree-like channel networks from digital elevation
  model data}.
\bjournal{Journal of Hydrology}
\bvolume{187}
\bpages{157 -- 172}.
\bnote{Fractals, scaling and nonlinear variability in hydrology}.
\bdoi{https://doi.org/10.1016/S0022-1694(96)03093-4}
\end{barticle}
\endbibitem

\bibitem{neveu86}
\begin{barticle}[author]
\bauthor{\bsnm{Neveu},~\bfnm{Jacques}\binits{J.}}
(\byear{1986}).
\btitle{Erasing a branching tree}.
\bjournal{Advances in Applied Probability}
\bvolume{18}
\bpages{101 -- 108}.
\end{barticle}
\endbibitem

\bibitem{peckham}
\begin{barticle}[author]
\bauthor{\bsnm{Peckham},~\bfnm{Scott~D.}\binits{S.~D.}}
(\byear{1995}).
\btitle{{New Results for Self-Similar Trees with Applications to River
  Networks}}.
\bjournal{Water Resources Research}
\bvolume{31}
\bpages{1023 -- 1029}.
\bdoi{https://doi.org/10.1029/94WR03155}
\end{barticle}
\endbibitem

\bibitem{revuz2004continuous}
\begin{bbook}[author]
\bauthor{\bsnm{Revuz},~\bfnm{Daniel}\binits{D.}} \AND
  \bauthor{\bsnm{Yor},~\bfnm{Marc}\binits{M.}}
(\byear{1999}).
\btitle{{Continuous Martingales and Brownian Motion}},
\bedition{third} ed.
\bseries{Grundlehren der mathematischen Wissenschaften}
\bvolume{293}.
\bpublisher{Springer Berlin Heidelberg}.
\end{bbook}
\endbibitem

\bibitem{remy}
\begin{barticle}[author]
\bauthor{\bsnm{Rémy},~\bfnm{Jean-Luc}\binits{J.-L.}}
(\byear{1985}).
\btitle{{Un Procédé Itératif de Dénombrement d'Arbres Binaires et son
  Application à leur Génération Aléatoire}}.
\bjournal{RAIRO Theoretical Informatics and Applications}
\bvolume{19}
\bpages{179 -- 195}.
\end{barticle}
\endbibitem

\bibitem{slack68}
\begin{barticle}[author]
\bauthor{\bsnm{Slack},~\bfnm{R.~S.}\binits{R.~S.}}
(\byear{1968}).
\btitle{A branching process with mean one and possibly infinite variance}.
\bjournal{Zeitschrift f{\"u}r Wahrscheinlichkeitstheorie und Verwandte Gebiete}
\bvolume{9}
\bpages{139 -- 145}.
\end{barticle}
\endbibitem

\bibitem{strahler52}
\begin{barticle}[author]
\bauthor{\bsnm{Strahler},~\bfnm{Arthur~N.}\binits{A.~N.}}
(\byear{1952}).
\btitle{Hypsometric (area-altitude) analysis of erosional topography}.
\bjournal{GSA Bulletin}
\bvolume{63}
\bpages{1117 -- 1142}.
\bdoi{10.1130/0016-7606(1952)63[1117:HAAOET]2.0.CO;2}
\end{barticle}
\endbibitem

\bibitem{Viennot}
\begin{bincollection}[author]
\bauthor{\bsnm{Viennot},~\bfnm{Xavier}\binits{X.}}
(\byear{1990}).
\btitle{Trees}.
In \bbooktitle{Mots, mélanges offert à M.P. Schützenberger}
\bpublisher{Hermès, Paris}
\bnote{Available online at http://www.xavierviennot.org/xavier/}.
\end{bincollection}
\endbibitem

\bibitem{reroot_inv}
\begin{barticle}[author]
\bauthor{\bsnm{Winkel},~\bfnm{Matthias}\binits{M.}},
  \bauthor{\bsnm{Pitman},~\bfnm{Jim}\binits{J.}} \AND
  \bauthor{\bsnm{Haas},~\bfnm{Benedicte}\binits{B.}}
(\byear{2009}).
\btitle{Spinal partitions and invariance under re-rooting of continuum random
  trees}.
\bjournal{Annals of Probability}
\bvolume{37}
\bpages{1381 -- 1411}.
\bdoi{10.1214/08-AOP434}
\end{barticle}
\endbibitem

\bibitem{yamato}
\begin{barticle}[author]
\bauthor{\bsnm{Yamato},~\bfnm{Kosuke}\binits{K.}} \AND
  \bauthor{\bsnm{Yano},~\bfnm{Kouji}\binits{K.}}
(\byear{2020}).
\btitle{Fluctuation scaling limits for positive recurrent jumping-in diffusions
  with small jumps}.
\bjournal{Journal of Functional Analysis}
\bvolume{279}
\bpages{108655}.
\bdoi{https://doi.org/10.1016/j.jfa.2020.108655}
\end{barticle}
\endbibitem

\bibitem{zolotarev}
\begin{barticle}[author]
\bauthor{\bsnm{Zolotarev},~\bfnm{Vladimir~M.}\binits{V.~M.}}
(\byear{1957}).
\btitle{More Exact Statements of Several Theorems in the Theory of Branching
  Processes}.
\bjournal{Theory of Probability \& Its Applications}
\bvolume{2}
\bpages{245 -- 253}.
\bdoi{10.1137/1102016}
\end{barticle}
\endbibitem

\end{thebibliography}


\end{document}